\documentclass[11pt, reqno]{article}

\usepackage{amsmath,amssymb,amsfonts,amsthm}
\usepackage{color}

\usepackage[colorlinks=true,linkcolor=blue,citecolor=blue,urlcolor=blue,pdfborder={0 0 0}]{hyperref}
\usepackage{cleveref}

\usepackage{CJKutf8}

\usepackage{caption}
\usepackage{thmtools}

\usepackage{mathrsfs}

\usepackage{centernot}

\usepackage{crossreftools}

\crefname{theorem}{Theorem}{Theorems}
\crefname{thm}{Theorem}{Theorems}
\crefname{lemma}{Lemma}{Lemmas}
\crefname{lem}{Lemma}{Lemmas}
\crefname{remark}{Remark}{Remarks}
\crefname{prop}{Proposition}{Propositions}
\crefname{defn}{Definition}{Definitions}
\crefname{corollary}{Corollary}{Corollaries}
\crefname{conjecture}{Conjecture}{Conjectures}
\crefname{question}{Question}{Questions}
\crefname{chapter}{Chapter}{Chapters}
\crefname{section}{Section}{Sections}
\crefname{figure}{Figure}{Figures}
\crefname{example}{Example}{Examples}

\theoremstyle{plain}
\newtheorem{thm}{Theorem}[section]

\newtheorem{lemma}[thm]{Lemma}
\newtheorem{theorem}[thm]{Theorem}

\newtheorem{lem}[thm]{Lemma}

\newtheorem{corollary}[thm]{Corollary}

\newtheorem{prop}[thm]{Proposition}

\newtheorem{observation}[thm]{Observation}

\theoremstyle{definition}

\theoremstyle{remark}
\newtheorem{remark}[thm]{Remark}

\numberwithin{equation}{section}

\renewcommand{\P}{\mathbb P}
\newcommand{\proba}{\mathbb{P}}
\newcommand{\E}{\mathbb E}

\newcommand{\R}{\mathbb R}
\newcommand{\Z}{\mathbb Z}
\newcommand{\N}{\mathbb N}

\newcommand{\cvd} {\overset{d}{\longrightarrow}}

\newcommand{\offx}[1]{\, \hbox{{\rm off}}_{#1}\,}

\newcommand{\cB}{\mathcal B}

\newcommand{\cH}{\mathcal H}

\newcommand{\cN}{\mathcal N}

\newcommand{\cS}{\mathcal S}
\newcommand{\cT}{\mathcal T}

\newcommand{\sE}{\mathscr E}
\newcommand{\sF}{\mathscr F}
\newcommand{\sG}{\mathscr G}

\newcommand{\sT}{\mathscr T}
\newcommand{\sU}{\mathscr U}

\newcommand{\sW}{\mathscr W}

\DeclareMathOperator*{\slr}{\leftrightarrow}

\usepackage[margin=0.9in]{geometry}
\usepackage{tikz}
\usepackage{pgfplots}
\usepackage{extarrows}
\usepackage{titling}
\usepackage{commath}
\usepackage{wasysym}
\usepackage{mathtools}
\usepackage{titlesec}
\newcommand{\eps}{\varepsilon}
\usepackage{bbm}
\usepackage{setspace}
\usepackage{enumitem}
\usepackage{booktabs}

\usepackage{cite}

\def\P{\mathbb{P}}

\newcommand{\1}{{\rm 1\hspace*{-0.4ex}%
\rule{0.1ex}{1.52ex}\hspace*{0.2ex}}}

\newcommand{\sS}{\mathscr{S}}

\DeclareMathSymbol{\leqslant}{\mathalpha}{AMSa}{"36} 
\DeclareMathSymbol{\geqslant}{\mathalpha}{AMSa}{"3E} 
\DeclareMathSymbol{\eset}{\mathalpha}{AMSb}{"3F}     

\renewcommand{\epsilon}{\varepsilon}

\makeatletter
\pgfdeclareshape{slash underlined}
{
  \inheritsavedanchors[from=rectangle] 
  \inheritanchorborder[from=rectangle]
  \inheritanchor[from=rectangle]{north}
  \inheritanchor[from=rectangle]{north west}
  \inheritanchor[from=rectangle]{north east}
  \inheritanchor[from=rectangle]{center}
  \inheritanchor[from=rectangle]{west}
  \inheritanchor[from=rectangle]{east}
  \inheritanchor[from=rectangle]{mid}
  \inheritanchor[from=rectangle]{mid west}
  \inheritanchor[from=rectangle]{mid east}
  \inheritanchor[from=rectangle]{base}
  \inheritanchor[from=rectangle]{base west}
  \inheritanchor[from=rectangle]{base east}
  \inheritanchor[from=rectangle]{south}
  \inheritanchor[from=rectangle]{south west}
  \inheritanchor[from=rectangle]{south east}
  \inheritanchorborder[from=rectangle]
  \foregroundpath{
    \southwest \pgf@xa=\pgf@x \pgf@ya=\pgf@y
    \northeast \pgf@xb=\pgf@x \pgf@yb=\pgf@y
    \pgf@xc=\pgf@xa
    \advance\pgf@xc by .5\pgf@xb
    \pgf@yc=\pgf@ya
    \advance\pgf@xc by -1.3pt
    \advance\pgf@yc by -1.8pt
    \pgfpathmoveto{\pgfqpoint{\pgf@xc}{\pgf@yc}}
    \advance\pgf@xc by  2.6pt
    \advance\pgf@yc by  3.6pt
    \pgfpathlineto{\pgfqpoint{\pgf@xc}{\pgf@yc}}
    \pgfpathmoveto{\pgfqpoint{\pgf@xa}{\pgf@ya}}
    \pgfpathlineto{\pgfqpoint{\pgf@xb}{\pgf@ya}}
 }
}
\tikzset{nomorepostaction/.code=\let\tikz@postactions\pgfutil@empty}

\makeatother

\makeatletter
\DeclareFontFamily{OMX}{MnSymbolE}{}
\DeclareSymbolFont{MnLargeSymbols}{OMX}{MnSymbolE}{m}{n}
\SetSymbolFont{MnLargeSymbols}{bold}{OMX}{MnSymbolE}{b}{n}
\DeclareFontShape{OMX}{MnSymbolE}{m}{n}{
    <-6>  MnSymbolE5
   <6-7>  MnSymbolE6
   <7-8>  MnSymbolE7
   <8-9>  MnSymbolE8
   <9-10> MnSymbolE9
  <10-12> MnSymbolE10
  <12->   MnSymbolE12
}{}
\DeclareFontShape{OMX}{MnSymbolE}{b}{n}{
    <-6>  MnSymbolE-Bold5
   <6-7>  MnSymbolE-Bold6
   <7-8>  MnSymbolE-Bold7
   <8-9>  MnSymbolE-Bold8
   <9-10> MnSymbolE-Bold9
  <10-12> MnSymbolE-Bold10
  <12->   MnSymbolE-Bold12
}{}

\let\llangle\@undefined
\let\rrangle\@undefined
\DeclareMathDelimiter{\llangle}{\mathopen}%
                     {MnLargeSymbols}{'164}{MnLargeSymbols}{'164}
\DeclareMathDelimiter{\rrangle}{\mathclose}%
                     {MnLargeSymbols}{'171}{MnLargeSymbols}{'171}
\makeatother

\pretitle{\begin{flushleft}\Large}
\posttitle{\par\end{flushleft}\vskip 0.5em
}
\preauthor{\begin{flushleft}}
\postauthor{
\\
\vspace{0.5em}
\footnotesize{$*$ Laboratoire de Math\'ematiques Rapha\"el Salem, Universit\'e de Rouen Normandie\\
Email: \href{mailto:ablancrenaudiepro@gmail.com}{ablancrenaudiepro@gmail.com}
\\
$\dagger$ The Division of Physics, Mathematics and Astronomy, California Institute of Technology\\ Email: 
\href{mailto:t.hutchcroft@caltech.edu}{t.hutchcroft@caltech.edu}
}
\end{flushleft}}
\predate{\begin{flushleft}}
\postdate{\par\end{flushleft}}
\title{\bf Super-Brownian limits and the $k$-point function for high-dimensional percolation} 

\renewenvironment{abstract}
 {\par\noindent\textbf{\abstractname.}\ \ignorespaces}
 {\par\medskip}

\titleformat{\section}
  {\normalfont\large \bf}{\thesection}{0.4em}{}

\author{{\bf Arthur Blanc-Renaudie$^*$ and Tom Hutchcroft$^\dagger$}}

\begin{document}

\date{\small{\today}}

\maketitle

\setstretch{1.1}

\begin{abstract}
We prove that there exist positive, lattice-dependent constants $A$ and $V$ such that the high-dimensional critical percolation $k$-point function is given by
\[
T_{p_c}(x_1,x_2,\ldots,x_{k}) \sim  V^{k-2} A^{2k-3} \sum_{\mathcal{T}\in \mathsf{Tr}(k)} \sum_{\substack{\Phi:V(\mathcal{T})\to \Z^d \\ \Phi(i)=x_i \forall 1\leq i \leq k}} \prod_{\substack{u,v\in V(\mathcal{T})\\u\sim v}}G(\Phi(u),\Phi(v))
\]
as $\min_{i\neq j}\|x_i-x_j\|\to \infty$, 
where $\mathsf{Tr}(k)$ is a set of isomorphism class representatives of trees with $k$ labelled leaves $\{1,\ldots,k\}$ and unlabelled internal vertices all of which have degree $3$ and $G$ is the lattice Green's function.
This verifies a conjecture of Aizenman and Newman (1984) subject to the usual perturbative conditions needed for convergence of the lace expansion (i.e., that the dimension $d$ is either very large or $d>6$ and one works with a spread-out lattice).
It follows from this theorem that the law of the cluster of the origin, when considered as the counting measure on its range, converges under rescaling to the canonical measure of the integrated super-Brownian excursion. By computing the asymptotics of various more complicated variations on the $k$-point function that involve counting the number of appearances of a ``pattern'' along the geodesic between two points, we also prove the stronger result that the cluster converges \emph{as an embedded metric-measure space} to the continuum random tree equipped with its Brownian embedding into $\R^d$. This convergence holds simultaneously with respect to the chemical distance, pivotal distance, and resistance distance on the cluster, which we prove are asymptotic to constant multiples of each other. This resolves conjectures of Hara and Slade (1998) and van der Hofstad and Slade (2003). As a corollary of our results we prove that there exists a positive constant $C$ such that
$\P_{p_c}(0\leftrightarrow \Z^d \setminus [-r,r]^d)\sim C r^{-2}$, answering a question of Heydenreich and van der Hofstad (2017).
\end{abstract}

\vfill 


\newpage

\tableofcontents

\newpage

\section{Introduction}
\label{sec:intro}

Many models of statistical mechanics are expected to have an \textbf{upper-critical dimension} $d_c$, above which their large-scale critical behaviour becomes identical to an appropriately chosen ``non-interacting'' or ``Gaussian'' model; this is often known as \textbf{mean-field} critical behaviour by analogy with spin systems such as the Ising model. 
For Bernoulli bond percolation, the appropriate comparison to make is between percolation and \emph{branching random walk}.
Indeed, to see that percolation can be thought of as a self-interacting variant of branching random walk, suppose that  we sample the cluster of the origin using a breadth-first search. As we run this search, there is an evolving cloud of ``active'' vertices revealed at the previous step of the algorithm. In branching random walk, each vertex of this evolving cloud of vertices would choose to have a random number of offspring located at adjacent vertices independently of everything else that has happened in the process, while in percolation the different ``offspring'' of each active vertex (i.e., the vertices of the cluster we explore at the next stage of the search) are forbidden from crossing an edge that has already been explored. The self interaction that distinguishes percolation from branching random walk is conjectured to become irrelevant at the critical point in dimensions $d>d_c=6$: Percolation is conjectured to have  the same critical exponents as branching random walk when $d>6$, different critical exponents than branching random walk when $d<6$, and to have logarithmic corrections to mean-field critical behaviour when $d=6$. See \cite{hutchcroft2025dimension} for a recent survey including a heuristic explanation of why percolation has upper critical dimension $d_c=6$ and \cite{MR2239599,heydenreich2015progress} for comprehensive introductions. 


\emph{At the level of critical exponents}, high-dimensional percolation has been well-understood since the milestone 1990 work of Hara and Slade \cite{MR1043524}, who
 proved
 that high-dimensional percolation satisfies the \textbf{triangle condition} $\nabla_{p_c}:=\sum_{x,y}\P_{p_c}(0\leftrightarrow x)\P_{p_c}(x\leftrightarrow y)\P_{p_c}(y\leftrightarrow 0)<\infty$, a sufficient condition for various exponents to take their mean-field values introduced by Aizenman and Newman \cite{MR762034} (see also \cite{MR1127713}).
This implies in particular that 
\[
\P_{p_c}(|K|\geq n) \asymp n^{-1/2} \qquad \text{ and } \qquad \E_{p_c-\eps}|K| \asymp \eps^{-1}
\]
as $n\to\infty$ or $\eps \to 0$ as appropriate, where $K$ denotes the cluster of the origin and $\asymp$ denotes up-to-constants equivalence; see \cite{HutchcroftTriangle} for a modern account of how these estimates follow from the triangle condition.
 Subsequent work has established critical exponents describing various other aspects of high-dimensional percolation, including the Euclidean radius \cite{MR2748397,asselah2025capacity,van2025one,panis2026reversing}, the intrinsic volume growth and random walk~\cite{MR2551766,chatterjee2021subcritical}, clusters in half-spaces \cite{chatterjee2020restricted,panis2025sharp}, clusters in the slightly subcritical regime \cite{chatterjee2021subcritical,hutchcroft2023high,hara1990mean,liu2026crossover}, and finite-size effects \cite{hutchcroft2023high,MR1431856,chatterjee2021subcritical,MR2165583,MR2155704,MR2276449,van2014cycle}.

Hara and Slade proved their result using 
the \emph{lace expansion}, which we introduce in detail in \cref{sec:lace_expansion} and which can be thought of as a systematic method for proving that self interactions localize in sufficiently high dimensions. The lace expansion was originally introduced by Brydges and Spencer \cite{MR782962} to study \emph{weakly self-avoiding walk} and, in addition to percolation, has now been applied to a variety of models including (strictly) self-avoiding walk \cite{MR1174248}, lattice trees and animals \cite{derbez1997lattice}, and the Ising model \cite{sakai2007lace}. One drawback of the lace expansion is that it requires a ``small parameter'' to work, so that these results are restricted either to nearest-neighbor models in very high dimensions or ``spread-out'' models (i.e., the models in which $x$ and $y$ are considered adjacent if $\|x-y\|_1\leq L$ for a large constant $L$) in dimensions $d>d_c$. The record for applying the lace expansion to percolation on the standard hypercubic lattice is currently held by Fitzner and van der Hofstad \cite{fitzner2015nearest}, who treat $\Z^d$ for $d\geq 11$.  (Very recently Duminil-Copin and Panis \cite{duminil2024alternative,duminil2024alternativeSAW} and Duminil-Copin, Markar, Panis, and Slade \cite{duminil2026random} have developed new approaches to high-dimensional critical phenomena that do not use the lace expansion but have similar drawbacks.)


Although high-dimensional percolation has now been well-understood at the level of exponents for several decades, its analysis at the level of \emph{scaling limits} has remained open. Of course, the natural conjecture is that high-dimensional critical percolation clusters have the same scaling limit as critical branching random walk: \emph{super-Brownian excursion} \cite{slade2002scaling,le1999spatial}. Various precise conjectures have been made in this direction, with the oldest (formulated in terms of the first-order asymptotics of the $k$-point function) going back to the 1984 work of Aizenman and Newman \cite{MR762034} and with stronger formulations encompassing both the spatial structure of clusters and their internal geometry given by Hara and Slade
\cite{hara1998incipient} and van der Hofstad and Slade \cite{van2003convergence}.


In this paper we verify these conjectures for models satisfying the perturbative assumptions needed to implement the lace expansion. Precise statements of our results are given over the course of the introduction: We begin by explaining our results concerning the $k$-point function (\cref{thm:k_point}) in \cref{subsec:the_k_point_function_intro} (which are known to imply convergence of critical clusters to integrated super-Brownian excursion when viewed as measures) before discussing our results concerning the full scaling limit (\cref{thm:scaling_limit_intro}) in \cref{subsec:super_brownian_scaling_limits_intro},  where we consider the cluster not just as a measure on its range but as a metric measure space equipped with an embedding into $\R^d$. 
These scaling limit theorems are proven in a sufficiently strong topology (namely the rooted Gromov-Hausdorff-Prokhorov-function topology, see \cref{thm:scaling_limit_GHPf}) that we can deduce e.g.\ first-order asymptotic estimates on one-arm probabilities (\cref{cor:first_order_arm}), answering a further question of Heydenreich and van der Hofstad \cite{heydenreich2015progress}.
In \cref{subsec:intro_axiomatic} we explain the precise inputs we need from the lace expansion, namely the first-order asymptotics of the two-point function \eqref{eq:two_point_assumption} and the ``two-blob'' mixing estimate \eqref{eq:two_blob_assumption} stating that the microscopic geometry of the cluster around two points $x$ and $y$ are asymptotically independent when these two-points are distant and conditioned to be connected. In the same section we also state alternative, non-perturbative versions of our theorems that take these inputs \eqref{eq:two_point_assumption} and \eqref{eq:two_blob_assumption} as hypotheses but do not require any ``small parameter'' assumptions (\cref{thm:k_point_A_and_B}).

\begin{remark}[Long-range percolation]
Analogues of several of our main theorems have recently been established for \emph{long-range percolation} in the parallel works of the second-named author \cite{LRPpaper1,LRPpaper3} (see also the forthcoming work \cite{HutchcroftReeves}), which  also treat logarithmic corrections to super-L\'evy scaling at the upper critical dimension for effectively long-range models. These papers rely on a completely different method from the present paper that does not apply to finite-range models. (The methods of the present paper should also apply to long-range percolation under the perturbative conditions needed to implement the lace expansion, but we do not pursue this here.)
\end{remark}

\begin{remark}[Scaling limits on the torus]
Our results complement the recent work of the first-named author and Nachmias \cite{blanc2025critical}, in which it is proven that critical percolation on high-dimensional tori has the same scaling limit as the critical Erd\H{o}s-R\'enyi random graph \cite{addario2012,MR1434128}. (See also \cite{blanc2024scaling} for similar results for the hypercube.) Although sharing some similarities, both the results and proof methods of these works are rather different to ours since large clusters in the torus do not have any Euclidean structure, while the fine study of this structure is central to this work. (Indeed, a central difficulty in \cite{blanc2025critical} is to prove that large clusters on the torus are equidistributed over the torus in an appropriate sense.)
\end{remark}

\begin{remark}[Independent parallel works]
Several of our results have also been obtained in the independent parallel series of works of Chatterjee, Chinmay, Hanson, and Sosoe \cite{chatterjee2025robust,chatterjee2025limiting,chatterjee2026convergence}. 
Roughly speaking, \cite{chatterjee2026convergence} establishes the special case of \cref{thm:k_point} in which the distances between the points $x_1,\ldots,x_k$ are all of the same order, while \cite{chatterjee2025limiting} establishes the scaling limit of the intrinsic distance between $0$ and a distant vertex conditioned on these two points being connected (a fact that follows from our \cref{thm:geodesic_bb_pointwise} and \cref{thm:CompareDistance}). Meanwhile, the foundational work \cite{chatterjee2025robust} develops a new approach to the high-dimensional IIC that does not rely directly on the lace expansion and can be used to prove mixing estimates for large clusters playing a role analogous to our two-blob estimate (\cref{thm:two_blob}).
Indeed, using the results of \cite{chatterjee2025robust}, it should be possible to prove that our condition \eqref{eq:two_point_assumption} implies our condition \eqref{eq:two_blob_assumption} when $d>6$, so that the hypothesis \eqref{eq:two_blob_assumption} could be omitted from \cref{thm:k_point_A_and_B} and elsewhere; we do not pursue this in this paper to maintain the independence of our two works.

There is also a smaller overlap between our work and the independent work of Cabezas, Fribergh, Heydenreich, and J\'arai \cite{cabezas2025bi}, who used the lace expansion to construct what we call the monochromatic two-arm IIC measure. Indeed, as part of our proof we construct $k$-arm IIC measures for arbitrary $k\geq 2$ in two different senses, where we either require $k$ distinct large clusters (\cref{thm:scheme_function,remark:karmIIC}) or $k$ edge-disjoint connections to distant points (\cref{thm:scheme_function_variantW}); we refer to the latter as the monochromatic $k$-arm IIC measure by analogy with the planar case. Our proof of the existence of these $k$-arm IIC measures follows a different approach to that of \cite{cabezas2025bi}, which relies on the lace expansion much more directly than this work.
\end{remark}

\subsection{The $k$-point function}
\label{subsec:the_k_point_function_intro}
We write $\|\cdot\|$ for the Euclidean norm on $\R^d$, write $\|\cdot\|_1$ for the $\ell^1$ norm, and write $\langle x \rangle = \max\{1,\|x\|\}$ to avoid division by zero.
For $L\geq 1$, let $\Z^d_L$ be the graph with vertex set $\Z^d$ in which distinct vertices $x$ and $y$ are adjacent if $\|x-y\|_1\leq L$. 
Using the lace expansion, it is proven in \cite{MR1959796,MR2393990} that if $d$ is sufficiently large or $d>6$ and we work on the spread-out lattice $\Z^d_L$ with $L$ sufficiently large then the critical two-point function is given asymptotically by
\begin{equation}
\label{eq:two_point_assumption}
\tag{A}
T_{p_c}(x,y):=\P_{p_c}(x\leftrightarrow y) \sim A G(x,y) \sim A\frac{d\Gamma(d/2)}{(d-2) \pi^{d/2}} \langle x-y\rangle^{-d+2} \qquad \text{as $x-y\to\infty$,}
\end{equation}
where $A$ is a positive, lattice-dependent constant and $G$ is the Green's function for simple random walk on $\Z^d$. (Note in particular that $G$ always denotes the Green's function for simple random walk on $\Z^d$, even when we work with percolation on the spread-out lattice $\Z^d_L$.) We sometimes refer to the constant $A$ as the \textbf{amplitude}. 
The \textbf{tree-graph inequalities} of Aizenman and Newman \cite{MR762034} bound the critical \textbf{$k$-point function} by
\begin{equation}
T_{p_c}(x_1,x_2,\ldots,x_k):=\proba_{p_c}(x_1\slr x_2\slr \dots\slr x_k) \leq \sum_{\cT\in \mathsf{Tr}(k)} \sum_{\substack{\Phi:V(\cT)\to \Z^d \\ \Phi(i)=x_i \forall 1\leq i \leq k}} \prod_{u\sim v}T_{p_c}(\Phi(u),\Phi(v)),
\label{eq:tree_graph_intro}
\end{equation}
where $\mathsf{Tr}(k)$ is a set of isomorphism class representatives of trees with $k$ labelled leaves $\{1,\ldots,k\}$ and unlabelled internal vertices all of which have degree $3$. Diagrammatically, these bounds can be represented as
\begin{align*}
  T_{p_c}(x,y,z) &\leq \begin{array}{l}\includegraphics{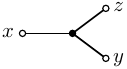}\end{array},\\
  T_{p_c}(x,y,z,w) &\leq \begin{array}{l}\includegraphics{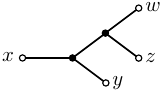}\end{array}+
  \begin{array}{l}\includegraphics{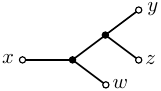}\end{array}
  +
  \begin{array}{l}\includegraphics{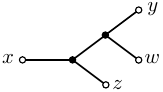}\end{array},
\end{align*}
and so on, 
where each line represents a copy of the two-point function $T_{p_c}$ and black vertices are summed over. (We stress that these diagrammatic expressions represent quantities, not events.  See \cref{sec:lace_expansion} for more details on our use of diagrammatic notation.) The tree-graph inequalities are valid in every dimension but are expected to be of the correct order only in high dimensions. 
Substituting in the asymptotic estimate \eqref{eq:two_point_assumption} and estimating sums by integrals, we obtain in high dimensions that
\[
 T_{p_c}(x_1,x_2,\ldots,x_k) \leq (1+o(1)) A^{2k-3} G(x_1,\ldots,x_k)
\]
as $\min_{i\neq j}\|x_i-x_j\|\to \infty$ and $k\geq 2$ is fixed, where we define the $k$-point Green's function 
 $G(x_1,\ldots,x_k)$ by
\[
G(x_1,\ldots,x_k) := \sum_{\cT\in \mathsf{Tr}(k)} 
\sum_{\substack{\Phi:V(\cT)\to \Z^d \\ \Phi(i)=x_i \forall 1\leq i \leq k}}
 \prod_{u\sim v} G(\Phi(u),\Phi(v)).
\]
In the same 1984 paper mentioned above, Aizenman and Newman \cite[Eq.\ (4.9)]{MR762034} also conjectured that there exists a positive constant $V$, known as the \textbf{vertex factor}, such that the high-dimensional critical percolation $k$-point function satisfies
 \begin{align}
 T_{p_c}(x_1,x_2,\ldots,x_k) &\sim V^{k-2} \sum_{\cT\in \mathsf{Tr}(k)} \sum_{\substack{\Phi:V(\cT)\to \Z^d \\ \Phi(i)=x_i \forall 1\leq i \leq k}} \prod_{u\sim v} T_{p_c}(\Phi(u),\Phi(v))
 \nonumber\\
 \label{eq:k_point_conjecture}
 &\sim V^{k-2} A^{2k-3} G(x_1,\ldots,x_k)
 \end{align}
as $\min_{i\neq j}\|x_i-x_j\|\to \infty$ for each fixed $k\geq 2$. Diagrammatically, this means that 
\begin{align*}
  T_{p_c}(x,y,z) &\sim \begin{array}{l}\includegraphics{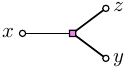}\end{array},\\
  T_{p_c}(x,y,z,w) &\sim \begin{array}{l}\includegraphics{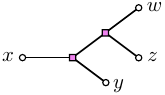}\end{array}+
  \begin{array}{l}\includegraphics{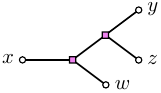}\end{array}
  +
  \begin{array}{l}\includegraphics{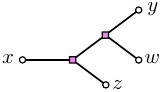}\end{array},
\end{align*}
and so on, where the violet square at each internal vertex indicates that we should both sum over the choice of that vertex and multiply by a copy of
the vertex factor $V$.
As with the constant $A$, the vertex factor $V$ should depend on the microscopic structure of the lattice and is \emph{not} expected to be universal; it encodes the precise interaction between the three different parts of a large cluster that meet at a point where the paths between three macroscopically separated vertices branch off from each other. Note that a similar asymptotic formula is true for critical branching random walks with finite offspring variance, where the vertex factor $V$ is precisely the variance of the offspring distribution and the constant $A$ encodes the diffusivity of the underlying random walk \cite[Proposition 3.1]{angel2021tail}. As we explain in detail below (and in \cref{sec:integrated_super_brownian_motion}), the formula \eqref{eq:k_point_conjecture} easily implies that the critical cluster has integrated super-Brownian excursion as its scaling limit when considered as a random measure.


Substantial partial progress was made on the conjecture \eqref{eq:k_point_conjecture} by Hara and Slade \cite{MR1773141,MR1757958}, who proved among other things that the Fourier transform of the $3$-point function has an asymptotic form consistent with the conjecture \eqref{eq:k_point_conjecture}. Analogues of \eqref{eq:k_point_conjecture} have also been established for the Fourier transform of the critical $k$-point function in \emph{oriented} percolation \cite{van2003convergence}, the \emph{contact process} \cite{van2010convergence}, and \emph{lattice trees} \cite{holmes2008convergence} in high dimensions, but the conjecture has remained open for unoriented finite-range percolation (even at the level of Fourier transforms). 


The first main result of this paper proves the conjectured $k$-point asymptotics \eqref{eq:k_point_conjecture} under the usual perturbative assumptions of the lace expansion. 

\begin{theorem}[The $k$-point function for high-dimensional percolation]
\label{thm:k_point}
 Consider \emph{critical} range-$L$ Bernoulli bond percolation on $\Z^d$. There exist constants $d_0>6$ and $L_0<\infty$ such that if either $d\geq d_0$ or $d>6$ and $L \geq L_0$ then there exist positive constants $A$ and $V$ such that
 \[
 T_{p_c}(x_1,x_2,\ldots,x_k) \sim V^{k-2} A^{2k-3} G(x_1,x_2,\ldots,x_k) 
 \]
 as $\min_{i\neq j}\|x_i-x_j\|\to \infty$ for each fixed $k\geq 2$.
\end{theorem}

\begin{remark}[$k$-point functions at and below the critical dimension]
For $d=d_c=6$ one expects a similar asymptotic estimate to hold but where the vertex factor $V$ is now a slowly decaying function of the separation between points. For example,  extrapolating from the non-rigorous results of \cite{essam1978percolation} leads to the conjecture
\[
T_{p_c}(x,y,z) \sim V( d_{\min}(x,y,z) )\sum_w T_{p_c}(x,w)T_{p_c}(w,y)T_{p_c}(w,z) 
\]
as $d_{\min}(x,y,z)\to \infty$, where $d_{\min}(x,y,z)=\min\{\|x-y\|,\|y-z\|,\|z-x\|\}$ and $V(R) \sim \text{const.}\, (\log R)^{-4/7}$. While related results have recently been established for \emph{long-range} percolation at the critical dimension \cite[Theorem III.1.8]{LRPpaper3}, there are not yet any rigorous results concerning finite-range critical percolation in six dimensions.
 For $d<6$ the relationship between the $2$-point and $3$-point functions is conjectured to have the rather different form \[T_{p_c}(x,y,z)\sim C \sqrt{T_{p_c}(x,y)T_{p_c}(y,z)T_{p_c}(z,x)},\] where $C$ is a universal (dimension-dependent) constant, as has recently been proven for site percolation on the triangular lattice \cite{ang2021integrability,camia2024conformal}. It has also recently been proven that the upper bound $T_{p}(x,y,z)\leq \sqrt{8T_{p}(x,y)T_{p}(y,z)T_{p}(z,x)}$ holds on every graph \cite{gladkov2024percolation}, an inequality that is expected to be sharp for critical percolation on $\Z^d$ if and only if $d<6$. See \cite{LRPpaper2} for related results for long-range percolation in low effective dimension and extensions of the inequality of \cite{gladkov2024percolation} to higher $k$-point functions.
\end{remark}

\subsection{Super-Brownian scaling limits}
\label{subsec:super_brownian_scaling_limits_intro}

 We now state our main theorem concerning the full scaling limit of high-dimensional percolation clusters. These scaling limit theorems will be stated in terms of \textbf{super-Brownian excursion}, a universal limit object describing the scaling limits of critical branching random walks. 

Following Le Gall's \emph{Brownian snake} construction \cite{le1999spatial}, we view super-Brownian excursion as a \emph{continuum random tree} \cite{CRT1,aldous1993continuum,aldous1991continuum} (\cref{fig:CRT_background}) that is embedded into $\R^d$ using Brownian motion; this perspective will be particularly important to us as we will be interested in taking scaling limits that keep track of the intrinsic geometry of clusters in addition to their spatial location: in the branching random walk, the continuum random tree describes the scaling limit of the genealogical structure of the particles while the Brownian embedding describes the limit of how these particles are located in space.

Although the convergence of (unoriented, finite-range) high-dimensional critical percolation to super-Brownian excursion has been a long-standing open problem, analogous convergence theorems have been established for a large number of other statistical mechanics models in high dimensions including branching random walk \cite{watanabe1968limit,borgs1999mean}, lattice trees \cite{holmes2008convergence,derbez1998scaling,derbez1997lattice}, the contact process \cite{durrett1999rescaled}, the voter model \cite{cox1999rescaled}, and oriented percolation \cite{van2003convergence,van2006infinite}; see e.g.\ \cite{perkins2002part,dynkin1994introduction,le1999spatial} for detailed introductions to the  theory of superprocesses and
 \cite{slade2002scaling} for a survey of their applications to high-dimensional statistical mechanics.


 In addition to verifying that critical high-dimensional percolation has super-Brownian excursion as its scaling limit, the following theorem also shows that several natural metrics on critical clusters are equivalent up to constant factors in the scaling limit (see also \cref{lem:uniform_distance_comparison}). The precise metrics considered are the \textbf{chemical distance} $d_\mathrm{chem}$, meaning the graph distance on the cluster, the \textbf{pivotal distance}
  $d_\mathrm{piv}$, where the (semimetric) distance between two points is the number of open pivotals between them, and the \textbf{resistance distance}, where the distance between two points is the effective resistance between them in the cluster (see e.g.\ \cite{LP:book} for background).
 In the statement of this theorem we also write 
$\mathbb{N}$ to denote the canonical measure of the continuum random tree $(\mathscr{T},d,\nu,o)$ equipped with its Brownian embedding $\Phi$ into $\R^d$, which is introduced in more detail below and in \cref{sec:scaling_limit}.

\begin{theorem}[The scaling limit of high-dimensional percolation clusters]
\label{thm:scaling_limit_intro}
 Consider \emph{critical} range-$L$ Bernoulli bond percolation on $\Z^d$. Let $K$ be the cluster of the origin, let
 $\mu_K$ denote the counting measure on $K$, and let $\Phi_K$ denote the inclusion map from $K$ to $\R^d$.
  There exist constants $d_0>6$ and $L_0<\infty$ such that if either $d\geq d_0$ or $d>6$ and $L \geq L_0$ then there exist positive constants $C_\text{\emph{prob}}$, $C_\text{\emph{chem}}$, $C_\text{\emph{piv}}$, $C_\text{\emph{res}}$, and $C_\text{\emph{vol}}$ such that
 \[
   C_\mathrm{prob} r^2 \cdot \P_{p_c}\bigl(|K| \geq \lambda C_\mathrm{vol} r^4\bigr) \sim  \lambda^{-1/2} 
 \]
 as $r\to\infty$ for each $\lambda>0$ and
\begin{multline*}
\P_{p_c}\left( \left(K,\frac{d_\text{\emph{chem}}}{C_\text{\emph{chem}} r^2} , \frac{d_\text{\emph{piv}}}{C_\text{\emph{piv}} r^2}, \frac{d_\text{\emph{res}}}{C_\text{\emph{res}} r^2}, \frac{\mu_K}{C_\text{\emph{vol}}r^4}, \frac{1}{r}\Phi_K,0 \right)\in \cdot \;\Bigg|\; |K| \geq \lambda C_\mathrm{vol} r^4 \right) 
\\\longrightarrow \mathbb{N}\Bigl((\mathscr{T},d,d,d,\nu,\Phi,o) \in \cdot \mid \nu(\sT) \geq \lambda\Bigr)
\end{multline*}
weakly in the weak array topology as $r\to \infty$ for each $\lambda>0$.
 Moreover, the constants $C_\mathrm{prob}$ and $C_\mathrm{vol}$ are given by $C_\mathrm{prob}=AV/\N(\nu(\sT)\geq 1)$ and $C_\mathrm{vol}=A^2V$.
\end{theorem}

Here, \textbf{weak convergence in the weak array topology} means that if, after sampling $K$, we set $X_0=0$ and let $X_1,X_2,\ldots$ be independent uniform random elements of $K$ then
   for each fixed $k\geq 1$ the conditional law given $|K| \geq \lambda C_\mathrm{vol} r^4$ of the finite-dimensional array of random variables
\[
  \left(\frac{|K|}{C_\mathrm{vol}r^4}, \left(\frac{d_\#(X_{i},X_{j})}{C_\# r^2}:\#\in \{\mathrm{chem},\mathrm{piv},\mathrm{res}\}, 0\leq i \leq j \leq k\right), \left(\frac{X_i}{r}\right)_{0\leq i \leq k} \right)
\]
converges weakly to the distribution under the measure $\N( \cdot \mid \nu(\sT) \geq \lambda )$ of the finite-dimensional array
\[
  \left(\nu(\sT), \left(d(Y_{i},Y_{j}):\#\in \{\mathrm{chem},\mathrm{piv},\mathrm{res}\}, 0\leq i \leq j \leq k\right), \left(\Phi(Y_i)\right)_{0\leq i \leq k} \right),
\]
where, conditional on $(\sT,\nu,\Phi,o)$, we set $Y_0=o$ and let $Y_1,Y_2,\ldots$ be independent random samples from the normalized measure $\nu/\nu(\sT)$.
(In the context of measure spaces with a single metric and no additional data this is usually known as the \textbf{Gromov-weak topology}. The fact that $d(Y_i,Y_j)$ does not depend on $\#$ in the second array is not a typo.) This theorem implies in particular that the law of the random measure $\mu_K/C_\mathrm{vol} r^4$ under the conditional measure 
$\P_{p_c}(\;\cdot\; \mid |K|\geq \lambda C_{\mathrm{vol}} r^4)$ 
converges weakly to the law of the appropriate integrated super-Brownian analogue, namely the law of the pushforward $\mu=\Phi_*\nu$ of $\nu$ by $\Phi$ under the conditional measure $\N( \;\cdot\;\mid \nu(\sT) \geq \lambda)$. (Note that the canonical measure $\N$ is not a finite measure, but $\N(\nu(\sT)\geq \lambda)$ is finite for every $\lambda>0$ so that the conditional measure $\N( \;\cdot\;\mid \nu(\sT) \geq \lambda)$ is well-defined as a probability measure.) Versions of the theorem stated with respect to stronger topologies and their consequences are discussed below.



\begin{figure}
\centering
\includegraphics[clip, trim = 1cm 1cm 0.3cm 1.17cm, width=0.53\textwidth]{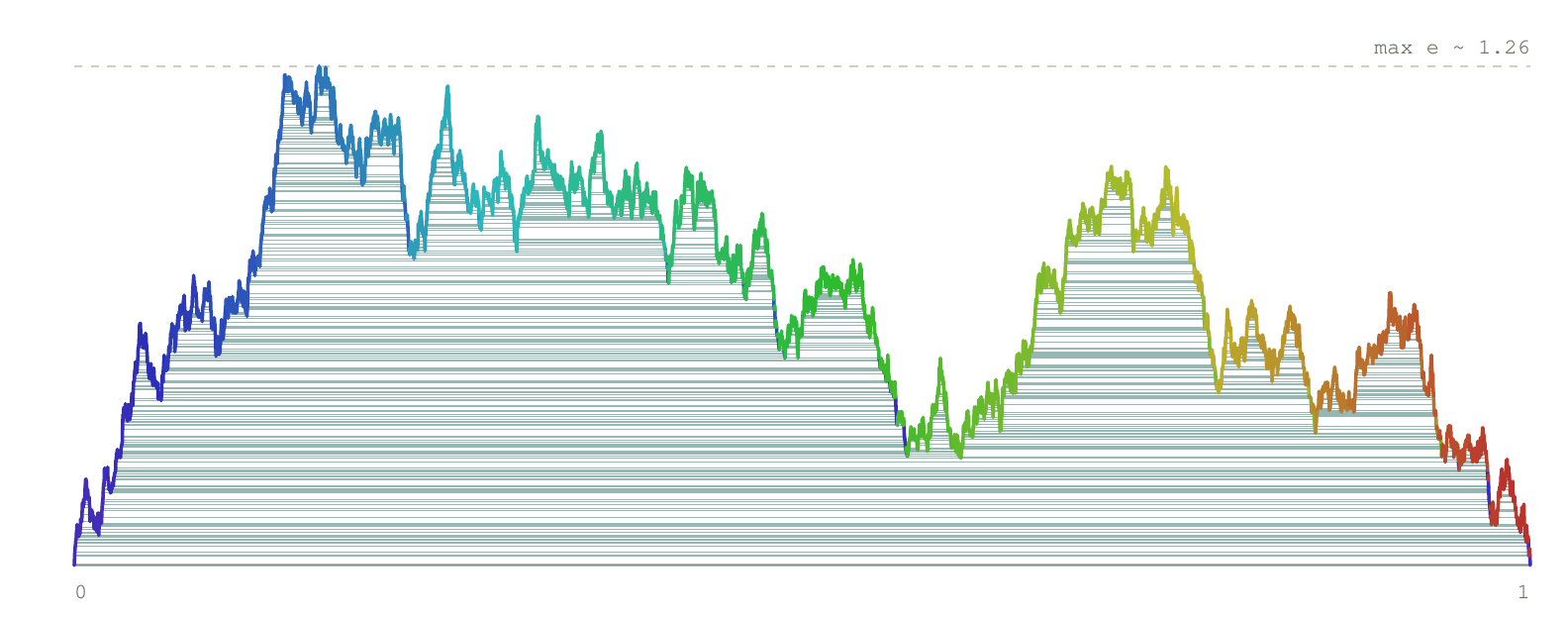} \;
\includegraphics[clip, trim = 0cm 4cm 1cm 3.5cm, width=0.41\textwidth]{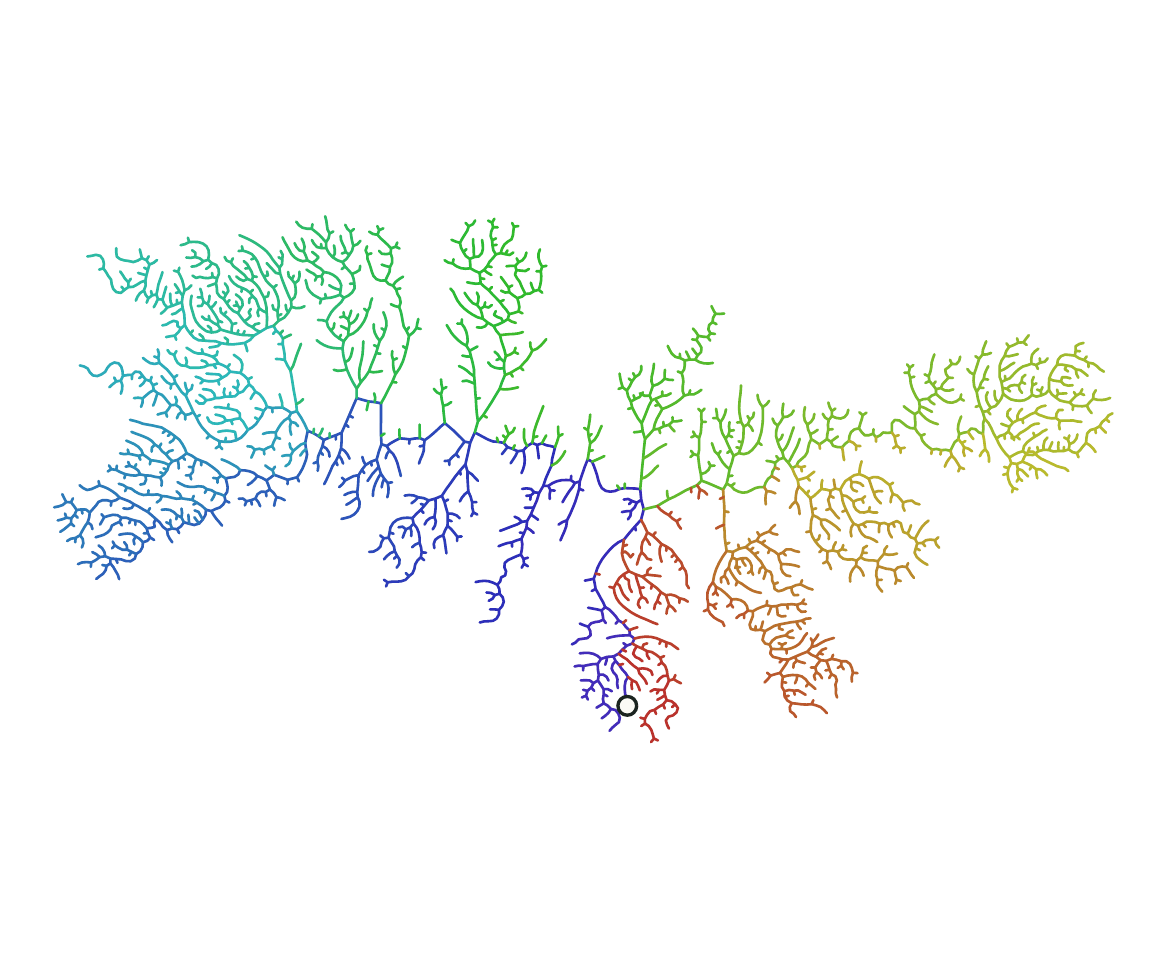}
\caption{Aldous' construction of the continuum random tree \cite{aldous1991continuum}. At the discrete level, there is a bijection between finite plane trees and \emph{Dyck paths} (paths in the natural numbers starting and ending at zero and moving by $\pm 1$ at each step): following the contour around  a plane tree and keeping track of the distance to the root yields a Dyck path, and conversely we can construct a plane tree from a Dyck path $(x_i)_{i=0}^{2n}$ by identifying two points $i$ and $j$ in the path if $x_i=x_j$ and $x_k \geq x_i$ for every $k$ between $i$ and $j$. The continuum random tree (CRT) of mass $1$ can be constructed by applying an analogous quotient operation to
the \emph{Brownian excursion}, i.e., the scaling limit of a uniform random Dyck path of fixed length.
  In our context it is natural to consider trees that have random rather than fixed mass, and more specifically to consider the \emph{canonical measure} on the continuum random tree, which is scale-invariant and arises as the pushforward under the same quotient map of \emph{It\^{o}'s excursion measure}.
   See e.g.\ \cite{le1999spatial,le2005random} for further background.
 }
 \label{fig:CRT_background}
\end{figure}

For our purposes, both the canonical measure\footnote{We warn the reader that several different conventions appear in the literature regarding the precise definition of the canonical measures for the continuum random tree and super-Brownian excursion. Essentially, one can re-scale volumes, intrinsic distances in the tree, and Euclidean distances by different constants to obtain a measure that differs from ours unless the three scaling factors are of the form $(\lambda^4,\lambda^2,\lambda)$. Different conventions for these scaling factors may be more or less natural in different settings. We use the convention that results in the vertex factor being $1$, lengths of ``branches'' in the tree being ``distributed'' as the Lebesgue measure normalized to have $\mathrm{Leb}([0,1])=1$, and the spatial motion being the Brownian motion with covariance matrix $\frac{1}{d}I$, which has independent coordinates and $\E\|B_t\|^2_2=t$, leading to the particular form of the 
moment formula \eqref{eq:canonical_measure_diagrams}. This specific choice of normalization is not important, but it is important to be consistent. In the notation of Le Gall \cite{le1999spatial}, our measure $\N$ is the canonical measure of the $(\frac{1}{2d}\Delta,\psi)$-superprocess with branching mechanism $\psi(u)=\frac{1}{2}u^2$, this choice of $\psi$ making the vertex factor $V=\psi''(0+)$ equal to $1$.} 
 $\N$ and the convergence theorem \cref{thm:scaling_limit_intro} are best understood in terms of (asymptotic) moment formulas extending the $k$-point function asymptotics of \cref{thm:k_point}. Let us now explain how this works when considering the convergence of the normalized counting measure on $K$ to integrated super-Brownian excursion (i.e., when we ignore intrinsic distances),  deferring the corresponding discussion for the convergence as an embedded metric space to \cref{sec:scaling_limit}.  The canonical measure of integrated super-Brownian excursion (i.e., the ``law'' of $\mu=\Phi_* \nu$ under the canonical measure $\mathbb{N}$) has the property that
\begin{equation}
\label{eq:canonical_measure_diagrams}
  \N\left[\prod_{i=1}^n \int f_i(x) \dif \mu(x) \right]
  =
  \idotsint G_\mathrm{SBM}(0,x_1,\ldots,x_n) \prod_{i=1}^n f_i(x_i)\dif x_1 \cdots \dif x_n
\end{equation}
for every collection of continuous, compactly supported functions $f_1,\ldots,f_n:\R^d\to \R$, where
\begin{align*}
  G_\mathrm{SBM}(x_1,\ldots,x_k) &= \lim_{\lambda\to \infty} \lambda^{(d-2)(2k-3)-d(k-2)} G([\lambda x_1],\ldots,[\lambda x_k])
\\
&= \sum_{\cT \in \mathsf{Tr}(k)} \idotsint \prod_{u\sim v} G_\mathrm{BM}(\Phi(u),\Phi(v))  \prod_{u\in V^\circ(\cT)} \dif \Phi(u)
\end{align*}
with $V^\circ(\cT)$ the set of unlabelled vertices of $\cT$, $\Phi(i)=x_i$ for each $1\leq i \leq k$, and where $G_\mathrm{BM}$ is the Green's function for the isotropic Brownian motion on $\R^d$ normalized to have $\E\|B_1\|_2^2=1$ and $[x]$ denotes the closest element of $\Z^d$ to the point $x\in \R^d$.
 When $d>4$ the formula \eqref{eq:canonical_measure_diagrams} uniquely characterises the canonical measure of integrated super-Brownian excursion as shown in \cite[Section I.5.4]{LRPpaper1},  and this characterisation can be taken as the definition for the purposes of this paper. A similar characterisation of the full tuple $(\mathscr{T},d,\nu,\Phi,o)$ is provided by \cref{lem:CRT_SBM_diagram_moment_formula,lem:CRT_SBM_Carleman}.
 It follows immediately from \cref{thm:k_point} and \eqref{eq:canonical_measure_diagrams} that
 \begin{equation}
 \label{eq:spatial_moment_asymptotics_intro}
    AV r^2 \cdot \E_{p_c}\left[  \prod_{i=1}^n \int f_i(x/r) \frac{\dif \mu_K (x)}{A^2V r^4}\right] 
 \to  \N\left[\prod_{i=1}^n \int f_i(x) \dif  \mu(x)
 \right] 
 \end{equation}
 as $r\to \infty$ for each collection of continuous, compactly supported functions $f_1,\ldots,f_n$. Since these moments uniquely characterise the canonical measure of super-Brownian excursion when $d>4$ as discussed above, \cref{thm:k_point} easily implies that if we define the random measure
$\mu_{K,r}=\frac{1}{A^2V r^4}\sum_{x\in K} \delta_{x/r}$ then 
  \[AVr^2 \cdot \P(\mu_{K,r} \in \,\cdot\, ) \to \N( \mu \in \,\cdot\, )\]
for an appropriate topology on an appropriate space of locally finite measures on the space of finite measures on $\R^d$. (We formulate \cref{thm:scaling_limit_intro} in the way we have so that we do not need to introduce such spaces and instead work only with convergence of probability measures.)

\begin{remark}[Other dimensions]
In the critical dimension $d=d_c=6$ it is expected that the super-Brownian scaling limit of \cref{thm:scaling_limit_intro} remains true once appropriate logarithmic corrections to scaling are taken into account. These logarithmic corrections to scaling are expected to be different for the various different intrinsic distances $d_\mathrm{piv}$, $d_\mathrm{chem}$, and $d_\mathrm{res}$; see \cite{essam1978percolation,stenull2003logarithmic} for relevant predictions from the physics literature. Related results have been proven for \emph{long-range} percolation at the critical dimension in \cite{LRPpaper1,LRPpaper3} and the forthcoming work \cite{HutchcroftReeves}. For scaling limits of critical percolation in \emph{two dimensions} we refer the reader to  classical works such as \cite{smirnov2001critical2,Schramm00,sheffield2006exploration} and to the recent work \cite{ambrosio2025tightness,djankovic2026scaling,miller2025existence,miller2025existence_BM} in which scaling limits of critical 2d percolation clusters are constructed \emph{as metric and resistance spaces}.
\end{remark}

\medskip

\noindent
\textbf{Stronger topologies of convergence and their consequences.}
In \cref{sec:tightness} we state and prove a version of \cref{thm:scaling_limit_intro} in which the weak array topology is replaced by a stronger topology, namely the \emph{rooted Gromov-Hausdorff-Prokhorov-function} (rGHPf) \emph{topology} (\cref{thm:scaling_limit_GHPf}). This strengthened theorem has the following corollary regarding \emph{one-arm probabilities}, which answers a question of Heydenreich and van der Hofstad \cite[Open Problem 11.2]{heydenreich2015progress}:

\begin{corollary}[First-order asymptotics of arm probabilities]
\label{cor:first_order_arm}
 Consider \emph{critical} range-$L$ Bernoulli bond percolation on $\Z^d$. There exist constants $d_0>6$ and $L_0<\infty$ such that if either $d\geq d_0$ or $d>6$ and $L \geq L_0$ then there exist positive constants $C_\mathrm{ext-arm}$ and $C_\mathrm{chem-arm}$ such that 
 \begin{align*}
 \P_{p_c}(0\leftrightarrow \Z^d\setminus  [-r,r]^d) \sim C_\mathrm{ext-arm} r^{-2} 
 \qquad\text{ and }\qquad
  \P_{p_c}(\partial B_\mathrm{chem}(0,r) \neq \emptyset) \sim C_\mathrm{chem-arm} r^{-1} 
 \end{align*}
  as $r\to \infty$. Moreover, the constants $C_\mathrm{ext-arm}$ and $C_\mathrm{chem-arm}$ are given by $C_\mathrm{ext-arm}=\N(\Phi(\sT) \cap \R^d\backslash  [-1,1]^d \neq \emptyset)/AV$ and $C_\mathrm{chem-arm} = C_\mathrm{chem}\N(\operatorname{rad}(\sT,o)\geq 1)/AV$ where $\operatorname{rad}(\sT,o)$ denotes the maximum distance of a point of $\sT$ from the root $o$.
\end{corollary}

Here we write $\partial B_\mathrm{chem}(0,r)$ for the set of points having chemical distance exactly $r$ from the origin.
This corollary strengthens theorems of Kozma and Nachmias \cite{MR2748397,MR2551766}, who proved that the high-dimensional one-arm probabilities satisfy the up-to-constants estimates $\P_{p_c}(0\leftrightarrow \partial [-r,r]^d) \asymp r^{-2}$ and $\P_{p_c}(\partial B_\mathrm{chem}(0,r) \neq \emptyset) \asymp r^{-1}$. (Note that these up-to-constants estimates are used in the proof of our scaling limit theorems. See also \cite{asselah2025capacity,van2025one,panis2026reversing} for recent alternative proofs of the up-to-constants Euclidean one-arm estimate.)
The analogous fact that $\P_{p_c}(|K|\geq n) \sim \mathrm{const.}\, n^{-1/2}$ as $n\to \infty$ (which is part of \cref{thm:scaling_limit_intro}) was already proven for nearest-neighbour percolation in very high dimensions by Hara and Slade \cite[Theorem 1.1]{hara1998incipient}, who in fact prove the stronger result that $\P_{p_c}(|K|= n) \sim \mathrm{const.}\, n^{-3/2}$ as $n\to \infty$ under these stronger assumptions.

\begin{remark}
First-order asymptotic estimates on one-arm probabilities play a prominent role in the work of Holmes and Perkins on scaling limits of the range of high-dimensional models \emph{as sets} \cite{holmes2020range}, a form of convergence that is implied by rooted Gromov-Hausdorff-Prokhorov-function convergence.
Their hypotheses include regular variation of the
survival probability in the intrinsic/time parameter and yield extrinsic
one-arm asymptotics; in our work the intrinsic and Euclidean one-arm
asymptotics are instead both deduced as consequences of our rGHPf scaling limit theorem. We also remark that one-arm asymptotics were deduced from certain $k$-point function asymptotics for a large class of high-dimensional models in the work of van der Hofstad and Holmes \cite{van2013survival} (see also \cite{holmes2007weak,van2017criterion}), but their results do not apply directly in our setting since (in our setting of unoriented percolation) they require asymptotic estimates on $k$-point functions \emph{with fixed chemical distance}, which are stronger than the estimates we prove in \cref{thm:k_point,lem:bb_moment_asymptotics}.
\end{remark}

\begin{remark}
It is a consequence of the work of Croydon \cite{croydon2008convergence,croydon2018scaling} that the effective resistance case of our rGHPf scaling limit theorem (\cref{thm:scaling_limit_GHPf}) implies that the simple random walk on a large critical cluster converges in an appropriate sense to Brownian motion on the continuum random tree. We do not pursue this further in this paper since we do not wish to introduce the relevant background material needed to state such a result.
\end{remark}

\subsection{A non-perturbative criterion for super-Brownian limits}
\label{subsec:intro_axiomatic}

Although the proofs of our main results rely on the lace expansion, they do so in a fairly mild way. Besides the two-point function asymptotics stated in \eqref{eq:two_point_assumption}, the only further input we will need from the lace expansion is provided by the following theorem, which can be thought of as a \emph{mixing} property of large critical clusters. Given a cylinder event $F$ and a point $x\in \Z^d$, we write $F_x=\{\omega:\omega-x \in F\}$ for the translation of $F$ by $x$.

\begin{theorem}[The two-blob estimate]
\label{thm:two_blob}
 Consider \emph{critical} range-$L$ Bernoulli bond percolation on $\Z^d$. There exist constants $d_0>6$ and $L_0<\infty$ such that if either $d\geq d_0$ or $d>6$ and $L \geq L_0$ then
for each pair of cylinder events $F^1$ and $F^2$ we have that
\[\P_{p_c}(\{x\leftrightarrow y\} \cap F^1_x \cap F^2_y) \sim \mathbb{P}_\mathrm{IIC}(F^1)\mathbb{P}_\mathrm{IIC}(F^2) \P_{p_c}(x\leftrightarrow y) 
\sim  A\mathbb{P}_\mathrm{IIC}(F^1)\mathbb{P}_\mathrm{IIC}(F^2) G(x,y) 
\]
as $\|x-y\|\to \infty$.
\end{theorem}

Here we recall that the \textbf{incipient infinite cluster} measure $\P_\mathrm{IIC}$ is defined as the weak limit of the conditional measures $\P_{p_c}(\;\, \cdot\; \mid 0\leftrightarrow x)$ as $x\to \infty$ and can be thought of as the law of critical percolation ``conditional on the cluster of the origin being infinite''.
The fact that this measure is well-defined under the hypotheses of \cref{thm:two_blob} was proven using the lace expansion by van der Hofstad and J\'arai \cite{van2004incipient}; see also \cite{heydenreich2014high} for proofs that these measures agree with those derived from various other limiting operations and \cite{Kesten86,Jar03} for analogous results in two dimensions. As we explain in \cref{sec:lace_expansion} (in which we also give a brief introduction to the lace expansion more broadly), \cref{thm:two_blob} can be derived from the lace expansion in a similar way to the original proof of \cite{van2004incipient}, and indeed we will see that the existence of the IIC measure can be extracted directly from the formulation of the lace expansion given in the original work of Hara and Slade \cite{MR1043524}.

\cref{thm:two_blob} can be thought of in terms of the \emph{decorrelation of macroscopic and microscopic information}: conditional on $x$ being connected to the distant vertex $y$, the cluster looks locally like a pair of \emph{independent} IICs around each of the two points $x$ and $y$. We name this property \eqref{eq:two_blob_assumption}, which we define to hold if
\begin{equation}
\label{eq:two_blob_assumption}
\tag{B}
\P_{p_c}(\{x\leftrightarrow y\} \cap  F^1_x \cap F^2_y) \sim \mathbb{P}_\mathrm{IIC}(F^1)\mathbb{P}_\mathrm{IIC}(F^2) \P_{p_c}(x\leftrightarrow y) \qquad \text{as $x-y\to\infty$}
\end{equation}
for each fixed pair of cylinder events $F^1,F^2$.
We will see in \cref{thm:scheme_function} that \eqref{eq:two_point_assumption} and \eqref{eq:two_blob_assumption} together imply many further decorrelation estimates of the same form, in which we specify the topology of connections between some finite set of distant points and the local geometry of clusters around each of these points and find that the small-scale and large-scale geometries decorrelate in an appropriate sense. This decorrelation theorem (\cref{thm:scheme_function}) is the workhorse underlying all our other theorems. In particular, all of our main results hold for high-dimensional models satisfying the two assumptions \eqref{eq:two_point_assumption} and \eqref{eq:two_blob_assumption}, with no further need for ``small parameter'' assumptions as in the derivation of these estimates from the lace expansion.

\begin{theorem}[Two-point and two-blob estimates suffice for super-Brownian scaling]
\label{thm:k_point_A_and_B}
Consider \emph{critical} range-$L$ Bernoulli bond percolation on $\Z^d$. If $d>6$ and \eqref{eq:two_point_assumption} and \eqref{eq:two_blob_assumption} both hold then the conclusions of \cref{thm:k_point}, \cref{thm:scaling_limit_intro}, and \cref{cor:first_order_arm} all hold.
\end{theorem}

We remark that \cref{thm:k_point_A_and_B} has the following immediate corollary for the $k$-point function of the IIC, defined by
\[
  T_\mathrm{IIC}(x_1,x_2,\ldots,x_k) 
=\lim_{z\to\infty}\P_{p_c}(x_1,\ldots,x_k \text{ all connected} \mid x_1 \leftrightarrow z)
  := \lim_{z\to \infty} \frac{T_{p_c}(x_1,x_2,\ldots,x_k,z)}{T_{p_c}(x_1,z)}.
\]
Note that this limit is invariant under permutations of $x_1,\ldots,x_k$ when \eqref{eq:two_point_assumption} holds.
When $x_1=0$, the IIC $k$-point function $T_\mathrm{IIC}(x_1,x_2,\ldots,x_k)$ is equal to the probability that the points $x_2,\ldots,x_k$ all belong to the cluster of the origin in the IIC measure. 

\begin{corollary}[The $k$-point function of the high-dimensional IIC]
\label{thm:k_point_IIC}
 Consider \emph{critical} range-$L$ Bernoulli bond percolation on $\Z^d$. If $d>6$ 
and \eqref{eq:two_point_assumption} and \eqref{eq:two_blob_assumption} both hold then
 \[
 T_{\text{\emph{IIC}}}(x_1,x_2,\ldots,x_k)  \sim V^{k-1} A^{2k-2} G_\mathrm{IIC}(x_1,x_2,\ldots,x_k) 
 \]
 as $\min_{i\neq j}\|x_i-x_j\|\to \infty$, where $V$ and $A$ are the same positive constants appearing in \cref{thm:k_point_A_and_B} and we define
 \[
   G_\mathrm{IIC}(x_1,x_2,\ldots,x_k) = \lim_{z\to \infty}\frac{G(x_1,x_2,\ldots,x_k,z)}{G(x_1,z)}.
 \]
\end{corollary}

\begin{remark}
The fact that $G_\mathrm{IIC}(x_1,x_2,\ldots,x_k)$ is well-defined, finite, and invariant under permutations when $d>4$ is an easy consequence of the estimate $G(x,y) \sim \mathrm{const.} \|x-y\|_2^{-d+2}$ as $\|x-y\|\to \infty$.
\end{remark}

 As before, the asymptotics of \cref{thm:k_point_IIC} can be interpreted in terms of diagrammatic formulae, where the diagrams for the IIC involve trees in which we have a single internal vertex of degree $2$ and all other vertices having degree $3$; this vertex of degree $2$ represents the point where the path to infinity splits off from the other paths between the vertices of interest. When there are $k$ leaves these trees have $2k-2$ edges and $k-1$ internal vertices, leading to the powers of $A$ and $V$ appearing here.

\subsection{Further scaling limits}
\label{subsec:IIC_intro}

The methods we develop in this paper are extremely robust and can be used to prove many further scaling limit theorems for critical high-dimensional percolation conditioned on various tail events. Important examples include scaling limits of the IIC and of clusters conditioned to contain a specific distant point, but one can also consider more degenerate limits. (For example, if $x_1,\ldots,x_k \in \Z^d$ is a collection of well-separated points and we condition the clusters $K_{x_1}$, $\ldots$, $K_{x_k}$ to be distinct and large, they converge to a collection of independent super-Brownian excursions.) 
Indeed, many of the intermediate technical theorems we prove (such as \cref{thm:scheme_function,thm:scheme_function_variantW,pro:CountBranchT}) are stated at a sufficiently high level of generality that these theorems can be deduced from them by exactly the same argument in the proof of \cref{thm:k_point_A_and_B}. This may be relevant to the program of Carpenter and Werner \cite{carpenter2025loops} on the relation between high-dimensional critical percolation and the Brownian loop soup.

To keep the paper at a reasonable length we do not investigate these alternative scaling limit theorems in detail. We do however make note of the following asymptotic estimate for the \emph{derivative} of the two-point function, which is an immediate corollary of our main macroscopic-microscopic decorrelation theorem (\cref{thm:scheme_function}).

\begin{theorem}[Derivative asymptotics]
\label{thm:derivative}
Consider \emph{critical} range-$L$ Bernoulli bond percolation on $\Z^d$. If $d>6$ and \eqref{eq:two_point_assumption} and \eqref{eq:two_blob_assumption} both hold then there exists a constant $G_\mathrm{piv}$, which we call the \textbf{geodesic factor for the pivotal metric}, such that
\[
T_{p_c}'(x,y) \sim p_c^{-1}G_\mathrm{piv} A^2 G^{*2}(x,y)
\]
as $x-y\to \infty$, where $G^{*2}=G*G$ denotes the convolutional square of the lattice Green's function. 
\end{theorem}

The constant $G_\mathrm{piv}$ is related to the constant $C_\mathrm{piv}$ from \cref{thm:scaling_limit_intro} by $C_\mathrm{piv}=AG_\mathrm{piv}$ as shown in the proof of \cref{thm:scaling_limit_intro}.


\subsection{Organization}

The rest of the paper is organized as follows.
\begin{itemize}
  \item
 In \cref{sec:lace_expansion} we prove the two-blob estimate (\cref{thm:two_blob}). We extract the two-blob estimate together with the existence of the IIC measure from a single lace expansion of the same form used in the original work of Hara and Slade \cite{MR1043524}, which we give an exposition of intended for readers not already familiar with the lace expansion. This is the only place where the lace expansion is used, with the rest of the paper relying only on the assumptions \eqref{eq:two_point_assumption} and \eqref{eq:two_blob_assumption}. 
 \item In \cref{sec:scheme_function} we introduce the formalism of \emph{scheme functions} and prove our main decorrelation theorem, which states that scheme functions asymptotically factorize at large scales (\cref{thm:scheme_function}).
This theorem shows very generally that the microscopic geometry of clusters decorrelates from macroscopic connection events, even when we condition on a large number of connection events possibly involving multiple arms emanating from the same microscopic regions. The results of this section imply in particular that ``$k$-arm IIC measures'' involving $k$ large \emph{distinct} clusters are well-defined for every $k\geq 1$, complementing related results of \cite{cabezas2025bi,blanc2025critical}.
  \item In \cref{sec:k_point_proof} we apply the results of \cref{sec:scheme_function} to prove our main result on $k$-point function asymptotics (\cref{thm:k_point}), the main technical step being the localization of the overcount factors relating the tree-graph inequality to the true $k$-point function.
As part of this section we also prove ``monochromatic'' analogues of our scheme-function factorization theorem involving edge-disjoint connections that need not belong to different clusters and deduce that ``monochromatic $k$-arm IIC measures'' are well-defined for all $k\geq 1$ (\cref{thm:scheme_function_variantW}).
   \item In \cref{sec:scaling_limit} we prove our scaling limit results in the weak array topology (\cref{thm:scaling_limit_intro}). We begin by explaining in detail how the pointwise $k$-point function asymptotics of \cref{thm:k_point} imply the convergence of the counting measure on the cluster of the origin to integrated super-Brownian excursion via the method of moments (\cref{thm:integrated_SBM}). We then prove an analogue of our pointwise $k$-point function theorems in which some points are required to be a pivotal edge on the path between two given points (\cref{thm:geodesic_bb_pointwise}), and use this to deduce weak array convergence for the pivotal metric (together with the embedding into space) via a further application of the method of moments (\cref{thm:scaling_limit_bb}).
We then prove \cref{thm:scaling_limit_intro} by showing that the three different metrics we consider are asymptotically equivalent (in the weak array sense) once appropriate constant scaling factors are taken into account, these factors arising from the microscopic geometry of the cluster around a typical pivotal between two distant points (\cref{thm:CompareDistance,cor:CompareDistance}).
    \item In \cref{sec:tightness} we strengthen the convergence in our scaling limit theorem to the rooted Gromov--Hausdorff--Prokhorov-function topology and deduce the first-order one-arm asymptotics (\cref{cor:first_order_arm}); here we also prove a strong form of the equivalence of different metrics (\cref{lem:uniform_distance_comparison}). The proofs in this section are adapted from the work of Archer, Nachmias, and Shalev \cite{archer2024ghp} on the GHP tightness of the uniform spanning tree of high-dimensional tori.

     \item Finally, in \Cref{sec:positivity_of_local_factors} we collect the finite-energy arguments used to characterize when the various limiting constants appearing in our theorems are positive (i.e., which cylinder events have positive probability for the various $k$-arm IIC measures we construct). These proofs, which are topological in nature, are collected in this appendix to avoid breaking the flow of the main argument.
 \end{itemize}

\section{Lace expansion analysis of the two-blob function}
\label{sec:lace_expansion}

 In this section we explain how the lace expansion can be used to prove \cref{thm:two_blob}. A secondary goal of the section is to show how the existence of the high-dimensional IIC measure, originally established using a slightly different form of the lace expansion in \cite{van2004incipient}, can also be extracted directly from the original Hara--Slade approach to the lace expansion \cite{MR1043524}. That is, one can carry out a single lace expansion analysis that yields both the two-point function estimate \eqref{eq:two_point_assumption} \emph{and} the existence of the IIC measure. Since all delicate quantitative aspects of the lace expansion have already been handled in the classical works \cite{MR1043524,MR1959796}, we will be able to focus on formal aspects, making our analysis fairly simple. (In particular, since we can assume that the two-point function estimate \eqref{eq:two_point_assumption} holds, we do not need to engage with the ``bootstrap analysis'' of the lace expansion or associated deconvolution problems.)

\subsection{A crash course on the lace expansion}


\medskip

To formulate the lace expansion in a way that will enable us to read off both the two-point function asymptotics \eqref{eq:two_point_assumption} and the existence of the IIC measure from a single expansion, we will use slightly different notation than is typical. In particular, we will elevate the probability
\[
\mathbf{T}_W(x,y):=\P(x \leftrightarrow y \text{ only via } W) =  \P(x \leftrightarrow y) - \P(x \leftrightarrow y \text{ off } W)
\]
to a position of central importance, where $W \subseteq V\cup E$ is a set of edges and vertices, ``$x \leftrightarrow y$ off  $W$'' means that $x$ is connected to $y$ by an open path that does not use any edges or vertices of $W$, and ``$x \leftrightarrow y$ only via $W$'' means that $x$ is connected to $y$ but not off $W$, so that any open path from $x$ to $y$ must use an edge or vertex of $W$. For sets of \emph{vertices} these ``only via'' probabilities already occur prominently in the Hara--Slade approach to the lace expansion (where they are referred to as connections \emph{through} $W$). Allowing sets of edges also has the following advantage: If we can prove that for each non-empty finite set $W \subseteq V \cup E$ there exists a constant $A_W$ such that
\begin{equation}
\label{eq:T_W_asymptotics}
\mathbf{T}_W(0,x) \sim A_W G(0,x)
\end{equation}
then we deduce both the asymptotic estimate for the two-point function \eqref{eq:two_point_assumption}  \emph{and the existence of the IIC measure}. The first claim about the estimate \eqref{eq:two_point_assumption} follows immediately by taking $W=\{0\}$ (or any other set containing $0$). For the existence of the IIC measure, note that if $W$ is a finite set of edges then we have by \eqref{eq:T_W_asymptotics} that
\begin{multline*}
\P(0 \leftrightarrow x, W \text{ closed}) = (1-p)^{|W|}\P(0 \leftrightarrow x \text{ off $W$}) = (1-p)^{|W|}(\P(0 \leftrightarrow x)-\P(0 \leftrightarrow x \text{ only via $W$})) \\=  (1-p)^{|W|} (A - A_W\pm o(1)) G(0,x),
\end{multline*}
so that
\[
\P(W \text{ closed}\mid 0 \leftrightarrow x) \to \frac{A - A_W}{A} \cdot (1-p)^{|W|} \qquad \text{ as $x\to\infty$}.
\]
This shows that the IIC measure is well-defined for all cylinder events of the form $\{W$ closed$\}$, and it follows by inclusion-exclusion that the IIC measure is defined for all cylinder events. Note moreover that the asymptotic formula
\[
\P(0 \leftrightarrow x, W \text{ closed}) \sim \mathbb{P}_\mathrm{IIC}(W \text{ closed})\P(0\leftrightarrow x)
\]
holds even when $\mathbb{P}_\mathrm{IIC}(W \text{ closed})=0$: Indeed,  \cref{lem:IIC_positivity} states that $\mathbb{P}_\mathrm{IIC}(W \text{ closed})=0$ if and only if $W$ contains a cut set separating $0$ from infinity, in which case $\P(0 \leftrightarrow x, W \text{ closed})=0$ for all but finitely many $x$.


\medskip

In order to facilitate our later application to the two-blob function, it will be important to work with percolation on a general graph. (Specifically, we will later apply the lace expansion to the graph formed by deleting some finite set of edges from $\Z^d$.)

\medskip

\noindent
\textbf{A linear-algebraic formulation of the lace expansion.} Let $G=(V,E)$ be a simple, countable graph and let $P:E\to [0,1]$ be a function assigning an inclusion probability to each edge. (Here a graph is said to be simple if it does not include any loops or multiple edges. This restriction is for notational purposes only.) 
Let $\mathscr{M}=\R^{V\times V}$ denote the set of real matrices indexed by $V\times V$, and let $\mathscr{O}=\mathscr{M}^{\mathscr{P}(V \cup E)}$ denote\footnote{This symbol is \texttt{\string\mathscr\{O\}}.} the set of families of matrices of the form $\mathbf{S}=(\mathbf{S}_W)_{W\subseteq V \cup E}$, where each $\mathbf{S}_W$ belongs to $\mathscr{M}$ and $\mathscr{P}(V\cup E)$ denotes the set of all sets of edges and vertices. The set $\mathscr{O}$ is naturally a vector space when equipped with componentwise addition and scalar multiplication.
Important elements of $\mathscr{O}$ for our purposes include 
\begin{align*}\mathbf{I}_W(x,y)&:=\mathbbm{1}(x=y \in W) \qquad \text{ and } \qquad
\mathbf{T}_W(x,y) := \P(x\leftrightarrow y \text{ only via $W$}).
\end{align*}
We will also think of $P$ as an element of $\mathscr{M}$ defined by $P(x,y)=0$ if $x$ and $y$ are not adjacent and $P(x,y)=P(e)$ for the unique edge connecting $x$ and $y$ otherwise. 

Given $x,y$ and a set $W$ such that the event $\{x\leftrightarrow y$ only via $W\}$ occurs, we define the \textbf{cutting bond} to be the first open pivotal $e$ on the path from $x$ to $y$ such that $x$ is connected to $e^-$ only via~$W$. (See \cref{fig:cutting_bond}.) If no such edge exists we say that ``$x\leftrightarrow y$ only via $W$ with no cutting bond''. A nice feature of this definition is that there are no constraints on the manner in which the part of the cluster between $e$ and $y$ intersects the set $W$. We define an element $\mathbf{\Pi}_0$ of $\mathscr{O}$ by
\[
\mathbf{\Pi}_{0,W}(x,y):=\mathbbm{1}(x\neq y)\P(x\leftrightarrow y \text{ only via $W$ with no cutting bond}).
\]
The basic idea of the Hara--Slade percolation lace expansion is to expand all ``only via'' connection probabilities in terms of the location of the cutting bond. This expansion is encapsulated by the following lemma. For each edge $e$ and vertex $x$, we write $K_e(x)$ for the set of vertices that are connected to $x$ by an open path that does not include the edge $e$.

\begin{figure}[t]
\centering
\includegraphics{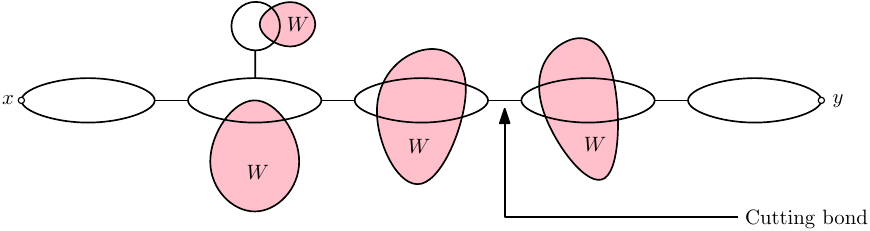}
\caption{Example of a cutting bond for the event $\{x \leftrightarrow y$ only via $W\}$. The elliptical regions represent 2-connected components of the cluster, while straight lines represent bridges. The pink shaded regions represent the set $W$. The cutting bond is unique when it exists.}
\label{fig:cutting_bond}
\end{figure}

\begin{lemma}[One step of lace expansion {\cite[Lemma 10.1]{MR2239599}}] Let $G=(V,E)$ be a countable graph and let $P:E\to[0,1]$ be an assignment of edge inclusion probabilities. 
For each pair of distinct points $x,y\in V$ and set $W \subseteq V \cup E$ we have that
\begin{align*}
&\mathbf{T}_W(x,y) = \mathbf{I}_W(x,y)+\mathbf{\Pi}_{0,W}(x,y)+ \sum_e (\mathbf{I}_W(x,e^-)+\mathbf{\Pi}_{0,W}(x,e^-)) P(e)T(e^+, y)\\
&\hspace{4cm}-\sum_{e\in E^\rightarrow} \E \left[\mathbbm{1}(x \leftrightarrow e^- \text{ \emph{only via $W$ with no cutting bond}}) P(e) \mathbf{T}_{K_e(x)}(e^+,y)\right].
\end{align*}
In particular, all relevant sums and expectations converge absolutely. 
\end{lemma}

In this lemma, the first two terms represent the probability that $x$ is connected to $y$ only via $W$ with no cutting bond. In the third term, we have approximated the probability that $x$ is connected to $y$ only via $W$ and $e$ is the first cutting bond on the open path by the product of the probabilities that $x$ is connected to $e^-$ only via $W$ with no cutting bond, that $e$ is open, and that $e^+$ is connected to $y$. Finally, the fourth term gives a convenient expression for the error in this approximation as a kind of iterated expectation. We should expect this expansion to be useful if this remainder term is much smaller than the probability it approximates, which roughly means that the parts of the cluster on either side of $e$ are approximately independent of one another; it is reasonable to expect this to be the case for critical percolation in high dimensions but not in low dimensions.

\medskip

This equation can be written more elegantly by introducing some further notation. Note that we can multiply elements of $\mathscr{O}$ by elements of $\mathscr{M}$, defining $\mathbf{S}U$ for $\mathbf{S}\in \mathscr{O}$ and $U\in \mathscr{M}$ by
\[
(\mathbf{S}U)_W(x,y) = \sum_{z\in V} S_W(x,z)U(z,y)
\]
whenever all relevant sums converge absolutely.
We define the \textbf{remainder operator} $\mathsf{R}$ to be the linear operator mapping its domain in $\mathscr{O}$ to $\mathscr{O}$ via
\[
\mathsf{R}[\mathbf{S}]_A(x,y) =  \mathbbm{1}(x\neq y)\sum_e \E \left[\mathbbm{1}(x \leftrightarrow e^- \text{ only via $A$, no cutting bond})P(e) \mathbf{S}_{K_e(x)}(e^+,y)\right]
\]
whenever the relevant sums and expectations converge absolutely.
With this notation in place, the above lemma may be written simply as
\[\mathbf{T}=\mathbf{I}+\mathbf{\Pi}_0+(\mathbf{I}+\mathbf{\Pi}_0)PT - \mathsf{R}[\mathbf{T}],\]
where $T(x,y)=\P(x\leftrightarrow y)$ denotes the usual (unrestricted) two-point function. From this perspective, the ``lace expansion'' is nothing other than the solution to this equation using the geometric series expansion  $(I+\mathsf{R})^{-1}=\sum_{i=0}^\infty (-1)^{i} \mathsf{R}^{i}$. Of course, most of the difficulty in implementing the lace expansion lies in proving that this series converges (in some appropriate sense) and using this to deduce asymptotic properties of the two-point function $T$.

\medskip

While infinite series expansions such as $(I+\mathsf{R})^{-1}=\sum_{i=0}^\infty (-1)^{i} \mathsf{R}^{i}$ require careful justification, we can always use finite-step versions of these expansions and seek to bound the error terms that arise. The remainder operator $\mathsf{R}$ has the helpful property that if $\mathbf{S}\in \mathscr{O}$ and $U\in \mathscr{M}$ then
\[
\mathsf{R}[\mathbf{S} U] =  \mathsf{R}[\mathbf{S}]U 
\]
whenever all relevant sums converge absolutely (or all terms in the sum are non-negative). Using this together with the linearity of $\mathsf{R}$, we can write 
\begin{align*}
\mathbf{T}&=\mathbf{I}+\mathbf{\Pi}_0+(\mathbf{I}+\mathbf{\Pi}_0)PT - \mathsf{R}[\mathbf{T}]\\
&=\mathbf{I}+\mathbf{\Pi}_0+(\mathbf{I}+\mathbf{\Pi}_0)PT - \mathsf{R}[\mathbf{I}+\mathbf{\Pi}_0+(\mathbf{I}+\mathbf{\Pi}_0)PT - \mathsf{R}[\mathbf{T}]]\\
&=\mathbf{I}+\mathbf{\Pi}_1+(\mathbf{I}+\mathbf{\Pi}_1)PT + \mathsf{R}^2[\mathbf{T}],
\end{align*}
where $\mathbf{\Pi}_1 = \mathbf{\Pi}_0-\mathsf{R}[\mathbf{I}+\mathbf{\Pi}_0]$.
Iteratively, if we define $\mathbf{\Pi}^{(0)}=\mathbf{\Pi}_0$ and define $\mathbf{\Pi}_{i}$ and $\mathbf{\Pi}^{(i)}$ for $i\geq 1$ by
\[
\mathbf{\Pi}^{(i)} = \mathsf{R}^i[\mathbf{I}+\mathbf{\Pi}_0] \qquad \text{ and } \qquad \mathbf{\Pi}_i = \mathbf{\Pi}_{i-1} + (-1)^{i} \mathbf{\Pi}^{(i)}
 = \mathbf{\Pi}_0+\sum_{j=1}^i (-1)^j \mathsf{R}^j[\mathbf{I}+\mathbf{\Pi}_0]
\]
then
\begin{equation}
\label{eq:finite_step_lace_expansion}
\mathbf{T}=\mathbf{I}+\mathbf{\Pi}_i + (\mathbf{I}+\mathbf{\Pi}_i)PT + (-1)^{i+1} \mathsf{R}^{i+1}[\mathbf{T}] =: \mathbf{I}+\mathbf{\Pi}_i + (\mathbf{I}+\mathbf{\Pi}_i)PT + (-1)^{i+1} \mathbf{R}_{i}
\end{equation}
for every $i\geq 0$, where we write $\mathbf{R}_{i}:=\mathsf{R}^{i+1}[\mathbf{T}]$. 
(Note that each of the operators $\mathbf{R}_i$ and $\mathbf{\Pi}^{(i)}$ is written as a sum of non-negative terms, so that these operators are always well-defined if one allows infinite entries. Similarly their products with non-negative matrices are well-defined if one allows infinite entries.)
To proceed, we would like to establish conditions under which (a) the remainder $\mathbf{R}_i$ converges to zero in some appropriate sense as $i\to\infty$, and (b) the limiting operator $\mathbf{\Pi}=\lim_i \mathbf{\Pi}_i$ is well-behaved (i.e., decays rapidly at infinity).

\medskip

\textbf{Expressing the lace expansion operators as iterated expectations.}
To bound the operators $\mathbf{R}_i$ and $\mathbf{\Pi}^{(i)}$, it is convenient to go back to their standard expressions in terms of iterated expectations. We will use the physicist's notation for expectations $\langle X \rangle :=\E[X]$ and use the shorthand
\[
\mathbbm{1}_\text{ncb}(x,y;W) = \mathbbm{1}(x\leftrightarrow y \text{ only via $W$ with no cutting bond}).
\]
We will write $\langle \,\cdot\,\rangle_{(0)},\langle\,\cdot\, \rangle_{(1)},\ldots,\langle \,\cdot\,\rangle_{(N)}$ for expectations taken with respect to $N+1$ independent percolation configurations $\omega_{(0)},\omega_{(1)},\ldots,\omega_{(N)}$, but where we will be interested in taking $\langle\,\cdot\, \rangle_{(i)}$-expectations with respect to random variables defined in terms of the subsequent $N-i$ configurations. Given an edge $e$ and a vertex $u$, we write $K_e^{(i)}(u)$ for the cluster of $u$ in $\omega_{(i)}\setminus\{e\}$, and given vertices $x,y$ and a sequence of oriented edges $e_1,\ldots,e_N$ we define
\[
\mathbbm{1}_\text{ncb}^i := \mathbbm{1}(e_i^+\leftrightarrow e_{i+1}^- \text{ in $\omega_{(i)}$ only via $K_{e_i}^{(i-1)}(e_{i-1}^+)$ with no cutting bond})
\]
for each $1\leq i \leq N$, where we write $e_0^+=x$ and $e_{N+1}^-=y$.
The operator $\mathbf{\Pi}^{(N)}$ can be written as
\begin{multline*}
\mathbf{\Pi}_{W}^{(N)}(x,y)= \sum_{e_1,\ldots,e_N \in E^\rightarrow}\langle \mathbbm{1}_\text{ncb}(x,e_1^-;W) P(e_1) \langle \mathbbm{1}_\text{ncb}^1 P(e_2) \langle \mathbbm{1}_\text{ncb}^2 P(e_3) 
\\\cdots \langle \mathbbm{1}_\text{ncb}^{N-1} P(e_{N}) \langle \mathbbm{1}_\text{ncb}^N  \rangle_{(N)} \rangle_{(N-1)} \cdots \rangle_{(2)} \rangle_{(1)} \rangle_{(0)}.
\end{multline*}
This is just an alternative interpretation of the same algebraic manipulations we carried out by taking powers of $\mathsf{R}$ above, and indeed is equivalent\footnote{NB: In the original Hara--Slade paper \cite{MR1043524} they use the letter $g$ for expressions of this form (without the distinguished set $W$) and use $\Pi$ to denote the analogue of $\mathbf{\Pi}P$ in our notation. Our notation is modelled after that of \cite{MR2239599}.} to the expressions given in the original paper \cite{MR1043524}. The remainder operator $\mathbf{R}_N$ can also be written in a similar fashion as
\begin{multline*}
\mathbf{R}_{N,W}(x,y):= \sum_{e_1,\ldots,e_N,e_{N+1} \in E^\rightarrow}\langle \mathbbm{1}_\text{ncb}(x,e_1^-;W) P(e_1) \langle \mathbbm{1}_\text{ncb}^1 P(e_2) \langle \mathbbm{1}_\text{ncb}^2 P(e_3) 
\\\cdots \langle \mathbbm{1}_\text{ncb}^{N-1} P(e_{N}) \langle \mathbbm{1}_\text{ncb}^N P(e_{N+1}) \mathbf{T}_{K_{e_{N+1}}^{(N)}(e_N^+)}(e_{N+1}^+,y)  \rangle_{(N)} \rangle_{(N-1)} \cdots \rangle_{(2)} \rangle_{(1)} \rangle_{(0)},
\end{multline*}
and bounding $\mathbf{T}_{K_{e_{N+1}}^{(N)}} \leq T$ yields that
\[
\mathbf{R}_{N,W}=|\mathbf{R}_{N,W}| \leq \mathbf{\Pi}_{W}^{(N)}P T.
\]
Thus, bounds on $\mathbf{\Pi}_{W}^{(N)}$ can be used both to prove that the expansion converges and to show that the limiting operator $\mathbf{\Pi}=\lim_{i\to\infty}\mathbf{\Pi}_i$ decays rapidly at infinity. 


\medskip

\noindent \textbf{Diagrammatic notation and estimates.} We now explain how the BK inequality can be used to bound $\mathbf{\Pi}_{W}^{(N)}$ in terms of appropriate \emph{diagrammatic sums}. These diagrammatic sums will be expressed in Feynman diagram notation, where each line\footnote{NB: In the original Hara--Slade paper \cite{MR1043524}, wavy lines are used for this purpose.} (or curve; the distinction is only used to keep the pictures tidy) represents a copy of the two-point matrix $T(\cdot,\cdot)$, black (unlabelled) vertices represent points that are summed over, and white (unlabelled) vertices represent vertices that are not summed over, but rather appear as variables in the definition of the diagram. When there are only two labelled vertices, we will often omit the labels and implicitly label the left-most white vertex $x$ and the right-most labelled vertex $y$. Similarly, when there is only one labelled vertex we will often omit the label and implicitly label this vertex $x$. For example,
\[
\P(x \leftrightarrow y) =  \begin{array}{l}\includegraphics{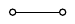}\end{array} \qquad \text{ and } \qquad \sum_{y} \P(x\leftrightarrow y) = \begin{array}{l}\includegraphics{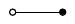}\end{array}.
\]
(Since $T$ is symmetric, we do not need to orient the lines in our diagrams.)
If two vertices are connected not by a line but by a broken line cut by a small perpendicular, this line represents a copy of $PT$ rather than $T$:
\[
\begin{array}{l}\includegraphics{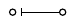}\end{array} := PT(x,y) = \sum_{z\in V} P(x,z)T(z,y).
\]
Since $PT$ need not be symmetric, the orientation is now important\footnote{In fact for percolation on the full lattice $\Z^d$ the two operators $P$ and $T$ commute, and the orientation is not really important. As mentioned previously we will want to also work with subgraphs of $\Z^d$ where the order is potentially meaningful.}: the line break is next to the vertex that appears as an argument of the $P$ matrix in the relevant sum.
For the lace expansion, we want to consider quantities defined with respect to a distinguished set $W \subseteq V \cup E$, and for this it is convenient to introduce the following extension of the standard diagrammatic notation.
Given a set $W\subseteq V \cup E$, we write $V(W)$ for the set of vertices that either belong to $W$ or are an endpoint of an edge in $W$. When a vertex in our diagram is represented by a small red box, this means that this vertex should be summed \emph{only over $V(W)$}, rather than over the entire vertex set as would be the case for a black vertex. Thus, for example, we may write
\begin{align*}
\begin{array}{l}\includegraphics{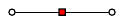}\end{array} :=\sum_{w\in V(W)} T(x,w)T(w,y).
\end{align*}
(This convention is new to this paper.)
There is one more important convention regarding the use of diagrammatic notation in the lace expansion: When a collection of curves encloses a polygon, we shade the polygon to denote that we allow all vertices in the sum to be equal, and leave the polygon unshaded to denote that we sum only over those configurations of vertices in which at least two of the vertices in the boundary of the polygon are distinct.
 The distinction between shaded and unshaded loops is important in the proof of convergence of the lace expansion, since e.g.\
\[\begin{array}{l}\includegraphics{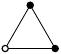}\end{array}:=\sum_{y,z \in V} \mathbbm{1}(x,y,z \text{ not all equal}) T(x,y)T(y,z)T(z,x)\] is small for critical percolation in very high dimensions whereas
\[\begin{array}{l}\includegraphics{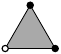}\end{array}:=\sum_{y,z \in V}  T(x,y)T(y,z)T(z,x) = 1+\begin{array}{l}\includegraphics{Example_Diagram_triangle_unshaded.pdf}\end{array} \]
 is always at least $1$. 

This diagrammatic notation is often used in conjunction with the BK inequality (and the union bound) to bound quantities of interest. For example:
\begin{align*}
\P(x \leftrightarrow y, \text{ no open pivotals}) &\leq\begin{array}{l}\includegraphics{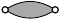}\end{array}  &&:= T(x,y)^2\\
\E\#\{\text{open pivotals for $x \leftrightarrow y$}\} &\leq   \begin{array}{l}\includegraphics{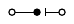}\end{array} &&:=  \sum_{z_1\sim z_2\in V} T(x,z_1)P(z_1,z_2) T(z_2,y)\\
\P(x\leftrightarrow y, \text{ at least one open pivotal}) &\leq  \begin{array}{l}\includegraphics{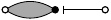}\end{array} &&:= \sum_{z_1 \sim z_2\in V} T(x,z_1)^2P(z_1,z_2)T(z_2,y)\\
\P(x \leftrightarrow y \text{ only via $W$}) &\leq \begin{array}{l}\includegraphics{Example_Diagram_only_via_W.pdf}\end{array} &&:=\sum_{w\in V(W)} T(x,w)T(w,y).
\end{align*}
(Verifying each of these estimates is a worthwhile exercise for readers unfamiliar with diagrammatic notation and the BK inequality.)
%
%
%
%
The use of diagrammatic estimates in the lace expansion will be based primarily on the following simple lemma, which again is a consequence of the BK inequality.

\begin{lemma}[\!\!{\cite[Lemma 2.5]{MR1043524}}]
\label{lem:basic_diagrammatic_bound}
 Let $G=(V,E)$ be a countable graph and let $P:E\to[0,1]$ be an assignment of edge inclusion probabilities.  The diagrammatic bounds
\begin{align}
\P(x \leftrightarrow y \text{ only via $W$ with no cutting bond}) &\leq \begin{array}{l}\includegraphics{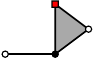}\end{array}\label{eq:cutting_bond_diagram}
\\
\P(\{x \leftrightarrow y \text{ only via $W$ with no cutting bond}\} \cap \{x\leftrightarrow v \}) &\leq \begin{array}{l}\includegraphics{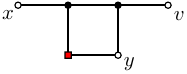}\end{array}+\begin{array}{l}\includegraphics{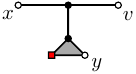}\end{array}\label{eq:ladder_diagram_pieces}
\end{align}
hold for every $x,y,v\in V$ and $W \subseteq V \cup E$.
\end{lemma}



The diagrammatic expressions in \cref{lem:basic_diagrammatic_bound} can be composed in a natural way to get bounds on the iterated expectations appearing in the definition of the operators $\mathbf{\Pi}^{(i)}$. For example, it follows from these bounds that
\[
\mathbf{\Pi}_W^{(0)} \leq \overline{\mathbf{\Pi}}_W^{(0)}:= \begin{array}{l}\includegraphics{Cutting_Bond_Diagram.pdf}\end{array} \qquad \text{ and } \qquad
\mathbf{\Pi}_W^{(1)} \leq \overline{\mathbf{\Pi}}_W^{(1)}:= \begin{array}{l}\includegraphics{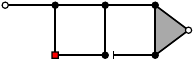}\end{array}+\begin{array}{l}\includegraphics{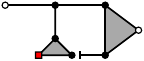}\end{array},
\]
where the two broken lines account for the appearance of the $P(e_1)$ term in the definition of $\mathbf{\Pi}^{(1)}_W$. In each case, we use overlines to denote the operator defined by a diagrammatic sum that bounds the operator of interest. (When composing these diagrams, note that the presence of a broken edge forces the vertices around the polygon to be distinct, so that a region formed this way is always unshaded.)
Similarly, iterating a second time yields the diagrammatic bound
\begin{multline*}
\mathbf{\Pi}_W^{(2)} \leq \overline{\mathbf{\Pi}}_W^{(2)} := \begin{array}{l}\includegraphics{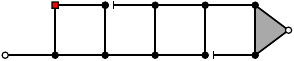}\end{array}+\begin{array}{l}\includegraphics{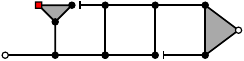}\end{array}
\\+\begin{array}{l}\includegraphics{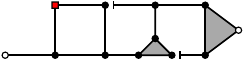}\end{array}+\begin{array}{l}\includegraphics{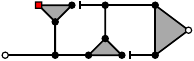}\end{array}.
\end{multline*}
(To think of these diagrams as being formed by composing those in the expression for $\overline{\mathbf{\Pi}}^{(1)}$ in the left with the two diagrams in \eqref{eq:ladder_diagram_pieces}, one should first mentally reflect these two diagrams through the horizontal axis.)
At each stage of the iteration, the free vertex $v$ in the two diagrams in \eqref{eq:ladder_diagram_pieces} gets paired with the red vertex in $\overline{\mathbf{\Pi}}_W^{(i)}$ and a line is broken to account for the appearance of a $P$ term; see \cite{MR1043524} for further details. 
 In particular, $\overline{\mathbf{\Pi}}_W^{(i)}$ is a sum of $2^i$ diagrammatic expressions for each $i\geq 0$.
Note that these diagrammatic expressions are extremely lengthy to write out using standard notation! 
The operators $\overline{\mathbf{\Pi}}^{(i)}$ are defined by iterating this pattern inductively, and satisfy the inequalities
\[
\mathbf{\Pi}^{(i)}=|\mathbf{\Pi}^{(i)}| \leq \overline{\mathbf{\Pi}}^{(i)} \qquad \text{ and } \qquad |\mathbf{\Pi}_N| \leq \overline{\mathbf{\Pi}}_N := \sum_{i=0}^N \overline{\mathbf{\Pi}}^{(i)}
\]
for every $i,N\geq 0$. We also define $\overline{\mathbf{R}}_{N,W}:=\overline{\mathbf{\Pi}}{}^{(N)}_WPT$, so that $|\mathbf{R}_{N,W}|\leq \overline{\mathbf{R}}_{N,W}$. To establish convergence and rapid decay properties of $\mathbf{\Pi}=\sum_{i=0}^\infty \mathbf{\Pi}^{(i)}$ it suffices to establish analogous bounds on $\overline{\mathbf{\Pi}}^{(i)}$ and $\overline{\mathbf{\Pi}}$.

\medskip

\noindent
\textbf{The main bounds.}
Now that we have established all the notation needed to understand the lace expansion, we are ready to state the main technical results of \cite{MR1043524,MR1959796} as they apply in our framework.
 The following theorem, essentially due to Hara, van der Hofstad, and Slade \cite{MR1959796}, encapsulates one of the most important aspects of the lace expansion:

\begin{theorem}
\label{thm:lace_expansion_pi_bar}
 Consider \emph{critical} range-$L$ Bernoulli bond percolation on $\Z^d$. There exists $d_0,L_0<\infty$ such that if either $d\geq d_0$ or $d>6$ and $L \geq L_0$ then for each finite non-empty set of edges and vertices $W$ we have that $\overline{\mathbf{R}}_{N,W}(x,y)\to 0$ as $N\to\infty$ for every $x,y$ and that there exists a constant $C_W<\infty$ such that
\[
\overline{\mathbf{\Pi}}_W(0,x) \leq C_W  \langle x\rangle^{-2d+6}
\]
for every $x\in \Z^d$.
\end{theorem}

\begin{proof}[Proof of \cref{thm:lace_expansion_pi_bar}]
This is essentially equivalent to the estimates proven in \cite{MR1959796}, except that their bound is only stated for (diagrammatic estimates on) the operator $\Pi$, which is similar to our $\mathbf{\Pi}$ but appears when expanding the two-point function $T$ rather than the ``only via'' two-point function $\mathbf{T}_W$. While \cref{thm:lace_expansion_pi_bar} could be proven by working back through the proofs of \cite{MR1959796}, we will instead explain how it can be deduced as a \emph{consequence} of the results of \cite{MR1959796}. In that paper, the authors work with the operators $\Pi^{(N)}$, $\Pi_N$, and $\Pi$ defined by
\begin{multline*}
\Pi^{(N)}(x,y)= \sum_{e_1,\ldots,e_N \in E^\rightarrow}\langle \mathbbm{1}(x \text{ is $2$-edge-connected to $e_1^-$}) P(e_1) \langle \mathbbm{1}_\text{ncb}^1 P(e_2) \langle \mathbbm{1}_\text{ncb}^2 P(e_3) 
\\\cdots \langle \mathbbm{1}_\text{ncb}^{N-1} P(e_{N}) \langle \mathbbm{1}_\text{ncb}(e_N^+,y;K_{e_N}^{(N-1)}(e_{N-1}^+))  \rangle_{(N)} \rangle_{(N-1)} \cdots \rangle_{(2)} \rangle_{(1)} \rangle_{(0)},
\end{multline*}
$\Pi_N=\sum_{i=0}^N (-1)^i \Pi^{(i)}$, and $\Pi=\sum_{i=0}^\infty (-1)^i \Pi^{(i)}$ (if the sum converges). As in our framework, there is also a remainder operator $R_N$ satisfying $R_N \leq \Pi^{(N)}PT$ such that
\[
T= I+\Pi_N + (I+\Pi_N)PT + R_N
\]
for every $N\geq 0$, where $\Pi_N=\sum_{i=0}^N (-1)^i \Pi^{(i)}$. 
The operators $\Pi^{(i)}$ admit the diagrammatic estimates
\begin{align*}
\Pi^{(0)} &\leq \overline{\Pi}^{(0)} := \begin{array}{l}\includegraphics{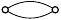}\end{array},\\
\Pi^{(1)} &\leq \overline{\Pi}^{(1)} := \begin{array}{l}\includegraphics{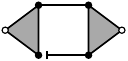}\end{array},\\
\Pi^{(2)} &\leq \overline{\Pi}^{(2)} := \begin{array}{l}\includegraphics{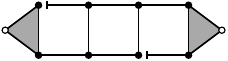}\end{array}+\begin{array}{l}\includegraphics{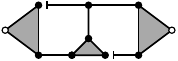}\end{array},\\
\Pi^{(3)} &\leq \overline{\Pi}^{(3)} := \begin{array}{l}\includegraphics{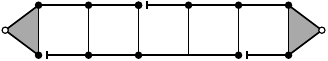}\end{array}+\begin{array}{l}\includegraphics{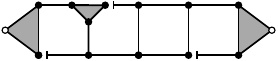}\end{array}\\
&\hspace{6cm}+\begin{array}{l}\includegraphics{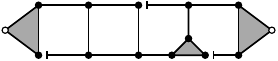}\end{array}+\begin{array}{l}\includegraphics{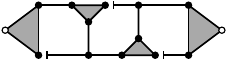}\end{array},
\end{align*}
and so on. It is proven in \cite[Theorem 1.2 and Prop 1.8(c)]{MR1959796} that, if $d$ is sufficiently large or $d>6$ and $L$ is sufficiently large and we consider critical percolation on $\mathbb{Z}^d_L$, then $\overline{\Pi}=\sum_{i=0}^\infty \overline{\Pi}^{(i)}$ converges and there exists a constant $C<\infty$ such that
\begin{equation}
\label{eq:original_Pi_bound}
\overline{\Pi}(x,y) \leq C \|x-y\|^{-2d+6}
\end{equation}
for every $x,y\in \Z^d$.
Now, for each $i\geq 1$ the operator $\overline{\Pi}^{(i)}$ can be obtained by composing our operator $\overline{\mathbf{\Pi}}^{(i-1)}$ on the left with a diagram of the form \includegraphics[height=0.44cm]{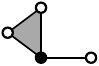}, matching the rightmost white vertex of this diagram with the red vertex in $\overline{\mathbf{\Pi}}^{(i-1)}$ and breaking a line to indicate a $P$ term as appropriate. In particular, considering the contribution to $\overline{\Pi}^{(i+1)}$ from configurations of vertices in which the lefthand shaded triangle collapses to a point and the second-left black vertex on the top row belongs to $V(W)$ yields that
\[
\overline{\Pi}^{(i+1)}(x,y) \geq \left[\min_{w\in V(W)}T(x,w) \min_{e\in E} P(e)\right] \overline{\mathbf{\Pi}}_W^{(i)}(x,y),
\]
so that for each $x\in \Z^d$ and each finite non-empty set $W \subseteq V\cup E$ there exists a constant $C_{x,W}<\infty$ such that
\[
\overline{\mathbf{\Pi}}_W^{(i)}(x,y) \leq C_{x,W} \overline{\Pi}^{(i+1)}(x,y)
\]
for every $y\in \Z^d$ and $i\geq 0$. Summing over $i\geq 0$ yields that $\overline{\mathbf{\Pi}}_W(x,y)\leq C_{x,W} \overline{\Pi}(x,y)$ for every $x,y\in \Z^d$ and every finite non-empty set $W\subseteq V\cup E$, so that the claim follows from \eqref{eq:original_Pi_bound}. (Of course it is also possible to bound our operators $\overline{\mathbf{\Pi}}^{(i)}_W$ directly using the methods of \cite{MR1959796}.)
\end{proof}

\begin{remark}
Note that there is nothing specific to $\Z^d$ in this proof: It shows in complete generality that bounds on $\overline{\Pi}$ imply bounds on $\overline{\mathbf{\Pi}}_W$ for $W$ finite whenever $\inf_{e\in E} P(e)>0$. To deduce the following corollary, which implies that the IIC is well-defined as the limit $\lim_{x\to\infty}\P(\,\cdot \mid 0\leftrightarrow x)$ as discussed above, one needs the following additional properties:
\begin{itemize}
\item The bounds must be strong enough to give that the sum $\sum_{z}[(I+\Pi)P](x,z)T(z,y)$ is dominated by terms with $z=O(1)$. This requires the dimension to be strictly larger than six.
\item The two-point function $T$ must satisfy $T(0,y)\sim T(x,y)$ as $y\to\infty$ and $x$ is fixed. In the case that $T$ is asymptotic to a constant multiple of the Green's function, this is equivalent to the base graph $G$ having the Liouville property (i.e., the random walk on $G$ being tail-trivial). Without this property, the limit $\lim_{x\to\infty}\P(\,\cdot \mid 0\leftrightarrow x)$ should depend on the ``direction'' in which $x$ goes to infinity, as is certainly the case on the $k$-regular tree for $k\geq 3$. (Note that one also expects directional dependence of the limit for \emph{subcritical} percolation on $\Z^d$ \cite{liu2026crossover,MR1905854}, which is related to the fact that the massive Green's function $G_\lambda$ does not satisfy $T(0,y)\sim T(x,y)$ as $y\to\infty$ and $x$ is fixed.)
\end{itemize}
\end{remark}

\begin{corollary}
\label{cor:IIC_point_to_blob}
 Consider \emph{critical} range-$L$ Bernoulli bond percolation on $\Z^d$. There exists $d_0,L_0<\infty$ such that if either $d\geq d_0$ or $d>6$ and $L \geq L_0$ then for each finite non-empty set of edges and vertices $W$ there exists a positive constant $A_W$ such that
 \[
\mathbf{T}_W(0,x) \sim A_W G(0,x)
 \]
 as $x\to \infty$.
\end{corollary}

As explained at the beginning of this section, it is a consequence of this corollary that the IIC measure exists as originally proven in \cite{van2004incipient}. (In \cite{van2004incipient} it is also shown that the IIC measure can be defined equivalently via various other natural limiting procedures; see also \cite{heydenreich2014high}.)

\begin{remark}
This corollary is much easier to prove than the two-point function estimate \eqref{eq:two_point_assumption}, for the simple reason that we are allowed to assume \eqref{eq:two_point_assumption} and the associated bound of \cref{thm:lace_expansion_pi_bar} in its proof. This lets us avoid all of the bootstrapping analysis and deconvolution problems that appear when analyzing percolation from scratch using the lace expansion. 
\end{remark}

\begin{proof}[Proof of \cref{cor:IIC_point_to_blob}]
Let $W \subseteq V\cup E$ be finite and non-empty.
It follows from \cref{thm:lace_expansion_pi_bar} that we can take $i\to\infty$ in \eqref{eq:finite_step_lace_expansion} and write
\[\mathbf{T}_W(0,x)=\mathbf{I}_W(0,x)+\mathbf{\Pi}_W(0,x)+(\mathbf{I}_W+\mathbf{\Pi}_W)PT(0,x).\]
We also have by \eqref{eq:two_point_assumption} that $T(x,y)\sim AG(x,y)$ as $x-y\to \infty$. The rest of the proof is just elementary analysis of the resulting convolution. Before we carry out this analysis, let us note that an easy finite-energy argument (\cref{lem:only_via_positivity})
yields that there exist positive constants $c_W,R_W>0$ such that $\mathbf{T}_W(0,x)\geq c_W T(0,x)$ for every $x\in \Z^d$ with $\|x\|_2\geq R_W$, so that terms of order smaller than $G(0,x) \asymp \|x\|^{-d+2}$ can always be safely ignored when computing first-order asymptotics. 

\medskip

To proceed with the analysis, we let $R\geq 1$ be a large integer and write
\begin{multline*}
\mathbf{T}_W(0,x) = [\mathbf{I}_W+\mathbf{\Pi}_W](0,x) + \sum_{\|y\| \leq R}[(\mathbf{I}_W+\mathbf{\Pi}_W)P](0,y)T(0,x) \\+ \sum_{\|y\| \leq R}[(\mathbf{I}_W+\mathbf{\Pi}_W)P](0,y)(T(y,x)-T(0,x)) + \sum_{\|y\| > R}[(\mathbf{I}_W+\mathbf{\Pi}_W)P](0,y)T(y,x).
\end{multline*}
Applying \cref{thm:lace_expansion_pi_bar} again ensures that the first term is $o(G(0,x))$, while the fact that $T(y,x) \sim AG(y,x) \sim A G(0,x)$ as $x\to \infty$ for each fixed $y\in \Z^d$ ensures that the third term is $o(G(0,x))$, so that
\begin{equation*}
\mathbf{T}_W(0,x) \sim \sum_{\|y\| \leq R}[(\mathbf{I}_W+\mathbf{\Pi}_W)P](0,y)T(0,x)  + \sum_{\|y\| > R}[(\mathbf{I}_W+\mathbf{\Pi}_W)P](0,y)T(y,x)
\end{equation*}
for each fixed $R<\infty$.
Finally,  \cref{thm:lace_expansion_pi_bar} also implies that the sum
\[
A_W:= A \sum_{y\in \Z^d} [(\mathbf{I}_W+\mathbf{\Pi}_W)P](0,y)
\]
converges absolutely. 
Thus, for each $\eps>0$ there exists $R<\infty$ such that
\[
\left|A\sum_{\|y\|\leq R} [(\mathbf{I}_W+\mathbf{\Pi}_W)P](0,y) - A_W\right| \leq \frac{1}{2}\eps A \quad \text{ and } \quad \sum_{\|y\|>R} [(\mathbf{I}_W+\mathbf{\Pi}_W)P](0,y)T(y,x) \leq \frac{1}{2} \eps A G(0,x)
\]
for every $x\in \Z^d$, where the second estimate follows from the elementary bound
\[
\sup_{x\in \Z^d\setminus \{0\}} \|x\|^{d-2} \sum_{y\in \Z^d, \|y\|\geq R} \|y\|^{-2d+6} \|x-y\|^{-d+2} \to 0 \qquad \text{ as $R\to \infty$ when $d>6$.}
\]
Putting everything together, it follows that
\[
1-\eps \leq  \liminf_{x\to \infty} \frac{\mathbf{T}_W(0,x)}{A_W G(0,x)}\leq  \limsup_{x\to \infty} \frac{\mathbf{T}_W(0,x)}{A_W G(0,x)} \leq 1+\eps
\]
for every $\eps>0$, which is equivalent to the claim.
\end{proof}

\subsection{The two-blob function}

We are now ready to explain how the two-blob function asymptotics of \cref{thm:two_blob} follows from the theory developed in the previous subsection. We begin by noting that the diagrammatic sums we have just introduced have the following important monotonicity property:

\begin{observation}
\label{observation:monotonicity}
The operators $\overline{\mathbf{\Pi}}_W^{(i)}$ depend increasingly on the edge inclusion probabilities $(P(e))_{e\in E}$. In particular, they decrease when edges are deleted from the graph $G$. Consequently, the same is true of the operators $\overline{\mathbf{R}}_N$ and $\overline{\mathbf{\Pi}} = \sum_{i=0}^\infty \overline{\mathbf{\Pi}}^{(i)}$.
\end{observation}

This observation will allow us to use results about the lace expansion on $\Z^d$ ``for free'' when analyzing percolation on graphs formed by deleting some finite collection of edges from $\Z^d$ in our proof of the two-blob estimate.

\begin{proof}[Proof of \cref{thm:two_blob}]
It suffices to prove the claim in the case that the cylinder events $F^1$ and $F^2$ are both of the form $\{W^i \text{ closed}\}$ for some finite sets of edges $W^1$ and $W^2$, the general case following by inclusion-exclusion.
Let $y\in \Z^d$, and let $\Z^d_{L,y}$ be the graph formed by deleting the set of edges $y+W^2$ from $\Z^d_L$. Let $T_y$, $P_y$, $\mathbf{T}_y=(\mathbf{T}_{y,W})_W$, and $\mathbf{\Pi}_y=(\mathbf{\Pi}_{y,W})_W$ be the operators defined in the previous section, but applied to this graph, so that
\[
T_y(0,y) = \P(0\leftrightarrow y \mid y+W^2 \text{ closed}), \qquad \mathbf{T}_{y,W}(0,y) = \P(0\leftrightarrow y \text{ only via $W$}\mid y+W^2 \text{ closed})
\]
and 
\begin{equation}
\mathbf{T}_{y,W}  = \mathbf{I}_W + \mathbf{\Pi}_{y,W} + (\mathbf{I}_W + \mathbf{\Pi}_{y,W})P_y T_y
\label{eq:lace_expansion_TyW}
\end{equation}
for every set of edges and vertices $W$. The fact that the operators $\mathbf{\Pi}_{y,W}$ converge and satisfy this equality follows from \Cref{observation:monotonicity} together with the convergence of $\overline{\mathbf{R}}_N$ to zero established by \cref{thm:lace_expansion_pi_bar}. Moreover, \cref{thm:lace_expansion_pi_bar} and \Cref{observation:monotonicity} also yield that if $W$ is a finite non-empty set of edges and vertices then
\begin{equation}
\label{eq:Pi_bound_with_deleted_edges}
\overline{\mathbf{\Pi}}_{y,W}(0,z) \leq C_W \|z\|^{-2d+6}
\end{equation}
for some constant $C_W<\infty$ that does not depend on $y$ or the choice of $W^2$. Meanwhile, the existence of the IIC measure and the asymptotic estimate \eqref{eq:two_point_assumption} ensure that
\begin{multline}
\label{eq:two_point_asymptotic_y}
T_y(z,y) = \P(z\leftrightarrow y \mid y+W^2\text{ closed})\\ \sim  \frac{\P_\mathrm{IIC}(W^2 \text{ closed})}{\P(W^2 \text{ closed})} \P(z\leftrightarrow y) \sim A\frac{\P_\mathrm{IIC}(W^2 \text{ closed})}{\P(W^2 \text{ closed})}G(z,y)
\end{multline}
 as $z-y\to\infty$. (This estimate holds even when the right-hand side is identically zero, as explained at the beginning of this section.)
As in the proof of \cref{cor:IIC_point_to_blob}, we let $R\geq 1$ and use \eqref{eq:lace_expansion_TyW} to write
\begin{multline*}
\mathbf{T}_{y,W^1}(0,y) =
  \sum_{\|z\| \leq R}[(\mathbf{I}_{W^1}+\mathbf{\Pi}_{W^1})P](0,z)T_y(0,y) 
+ \sum_{\|z\| > R}[(\mathbf{I}_{W^1}+\mathbf{\Pi}_{y,W^1})P_y](0,z)T_y(z,y) 
\\+  [\mathbf{I}_{W^1}+\mathbf{\Pi}_{y,W^1}](0,y)
+\sum_{\|z\| \leq R}[(\mathbf{I}_{W^1}+\mathbf{\Pi}_{y,W^1})P_y](0,z)(T_y(z,y)-T_y(0,y)) \\
+\sum_{\|z\| \leq R}[(\mathbf{I}_{W^1}+\mathbf{\Pi}_{y,W^1})P_y-(\mathbf{I}_{W^1}+\mathbf{\Pi}_{W^1})P](0,z)T_y(0,y).
\end{multline*}
As before, \eqref{eq:Pi_bound_with_deleted_edges} and \eqref{eq:two_point_asymptotic_y} imply that the two terms on the second line are $o(\|y\|^{-d+2})$ when $R$, $W^1$, and $W^2$ are fixed and $y\to\infty$. If we can show that $[(\mathbf{I}_{W^1}+\mathbf{\Pi}_{W^1})P - (\mathbf{I}_{W^1}+\mathbf{\Pi}_{y,W^1})P_y](0,z)\to 0$ pointwise when $z$, $W^1$, and $W^2$ are fixed, it will follow that the same is true for the term on the third line. Once this is done, it will follow by the same reasoning as in the proof of \cref{cor:IIC_point_to_blob} that
 \begin{multline*}
\mathbf{T}_{y,W^1}(0,y) \sim A\frac{\P_\mathrm{IIC}(W^2 \text{ closed})}{\P(W^2 \text{ closed})}\left[ \sum_{z\in \Z^d} [(\mathbf{I}_{W^1}+\mathbf{\Pi}_{W^1})P](0,z) \right] G(0,y) \\= A_{W^1}\frac{\P_\mathrm{IIC}(W^2 \text{ closed})}{\P(W^2 \text{ closed})} G(0,y)
 \end{multline*}
when $W^1$ and $W^2$ are fixed and $y\to\infty$ and hence that
\begin{align*}
&\P(0 \leftrightarrow y, W^1\text{ and }y+W^2 \text{ closed}) = (1-p)^{|W^1|}\P(0 \leftrightarrow y \text{ off $W^1$, $y+W^2$ closed})
\\ &\hspace{3cm}= (1-p)^{|W^1|}(\P(0 \leftrightarrow y, y+W^2 \text{ closed})-\P(0 \leftrightarrow y \text{ only via $W^1$, $y+W^2$ closed})) \\ &\hspace{3cm}\sim \P_\mathrm{IIC}(W^2 \text{ closed}) (1-p)^{|W^1|}(A - A_{W^1}) G(0,y)
\\&\hspace{3cm}=A\P_\mathrm{IIC}(W^1 \text{ closed})\P_\mathrm{IIC}(W^2 \text{ closed}) G(0,y)
\end{align*}
when $W^1$ and $W^2$ are fixed and $y\to\infty$ as desired. 

It remains to prove that $[(\mathbf{I}_{W^1}+\mathbf{\Pi}_{W^1})P - (\mathbf{I}_{W^1}+\mathbf{\Pi}_{y,W^1})P_y](0,z)\to 0$ when $z$, $W^1$, and $W^2$ are fixed and $y\to\infty$. It suffices to prove that this holds for $\mathbf{\Pi}_{W^1}-\mathbf{\Pi}_{y,W^1}$, the dependence of the other terms on $y$ being very simple. Now, each term of the form
\begin{multline*}
\langle \mathbbm{1}_\text{ncb}(x,e_1^-;W^1) P(e_1) \langle \mathbbm{1}_\text{ncb}^1 P(e_2) \langle \mathbbm{1}_\text{ncb}^2 P(e_3) 
\\\cdots \langle \mathbbm{1}_\text{ncb}^{N-1} P(e_{N}) \langle \mathbbm{1}_\text{ncb}(e_N^+,z;K_{e_N}^{(N-1)}(e_{N-1}^+))  \rangle_{(N)} \rangle_{(N-1)} \cdots \rangle_{(2)} \rangle_{(1)} \rangle_{(0)}
\end{multline*}
 in the sums defining $\mathbf{\Pi}_{W^1}^{(N)}(0,z)$ and $\mathbf{\Pi}_{y,W^1}^{(N)}(0,z)$ can be thought of as the \emph{probability} of a certain event defined in terms of a collection of $N$ independent percolation configurations and $N$ independent coin flips for each of the edges $e_1,\ldots,e_N$, where $P$ should be replaced by $P_y$ for $\mathbf{\Pi}_{y,W^1}^{(N)}(0,z)$ and where each of the expectations taken depends implicitly on the choice of $y$ and $W^2$ when evaluating $\mathbf{\Pi}_{y,W^1}^{(N)}(0,z)$. Since the difference between these two expressions is equivalent to whether or not we condition on all the edges of $y+W^2$ being closed (in all of the independent configurations), it follows by bounded convergence that each such term used to evaluate $\mathbf{\Pi}_{y,W^1}^{(N)}(0,z)$ converges to the corresponding term used to evaluate $\mathbf{\Pi}_{W^1}^{(N)}(0,z)$ as $y\to \infty$ (i.e., when we fix the choice of $e_1,\ldots,e_N$, $W^1$, $W^2$, and $z$ before taking $y\to \infty$). 
For each fixed expansion depth and each fixed choice of oriented edges appearing in the iterated expectation, it follows that the associated summand for $\mathbf{\Pi}_{y,W^1}^{(N)}(0,z)$ converges to the corresponding summand for $\mathbf{\Pi}_{W^1}^{(N)}(0,z)$. Moreover, the absolute value of each such summand is bounded by the corresponding summand in $\overline{\Pi}$, uniformly in $y$. These diagrams  are summable uniformly in $y$ by \cref{thm:lace_expansion_pi_bar} and \cref{observation:monotonicity}. Dominated convergence first in the summation over oriented edges and then in the expansion depth gives
\[
[(\mathbf{I}_{W^1}+\mathbf{\Pi}_{W^1})P -
(\mathbf{I}_{W^1}+\mathbf{\Pi}_{y,W^1})P_y](0,z)\to 0
\]
 as $y\to\infty$ when $z$, $W^1$, and $W^2$ are fixed as claimed. \qedhere

\end{proof}

\section{Scheme function asymptotics}
\label{sec:scheme_function}

In this section we prove that the two-point and two-blob estimates \eqref{eq:two_point_assumption} and \eqref{eq:two_blob_assumption} together imply a far-reaching generalization of the two-blob estimate, \cref{thm:scheme_function}, which allows us to estimate the probability of connections via distinct clusters of arbitrary finite collections of points, with prescribed microscopic geometry of the configuration around each point. This theorem will be shown to imply the $k$-point asymptotics of \cref{thm:k_point} in \cref{sec:k_point_proof} and all our results about scaling limits in \cref{sec:scaling_limit}; it will also immediately imply our theorem on the asymptotics of the derivative as we explain below.

\medskip

We begin with some relevant definitions. Throughout the section, we will write $\P=\P_{p_c}$ for the law of critical percolation on the spread-out lattice $\Z^d_L$.
We define a \textbf{$k$-blob}  to be a tuple $B=(W,b_1,\ldots,b_k)$ consisting of a (possibly empty) finite set of edges $W$ and an ordered collection of $k$ distinct points $b_1,\ldots,b_k$. We define a $k$-\textbf{plan} to be a pair $(H,\ell)$ consisting of a finite graph $H$ with vertex set $\{1,\ldots,k\}$, which may have multiple edges but not self-loops or isolated vertices, together with a labelling function $\ell$ assigning the labels $\{1,\ldots,\deg(v)\}$ to the edges incident to each vertex $v$. Thus, we may think of the edges of $H$ as unordered pairs $\{(i,n),(j,m)\}$, where for each $1\leq i \leq k$ and $1\leq n \leq \deg(i)$ there is a unique pair $(j,m)$ with $1\leq j\neq i \leq k$ and $1\leq m \leq \deg(j)$ such that $\{(i,n),(j,m)\}$ is an edge of $H$. We define a $k$-\textbf{scheme} $S=(H,\ell,B)$ to be a $k$-plan $(H,\ell)$ together with a function $B$ assigning a $\deg(i)$-blob $B_i=(W_i,b_{i,1},\ldots,b_{i,\deg(i)})$ to each vertex $i$ of the plan. Given a scheme $S$, we write $H(S)$ for the plan associated to $S$, write $V(S)$ and $E(S)$ for the edge and vertex sets of this plan, and define the \textbf{scheme function}, also denoted $S$, by
\begin{align*}
S(x_1,\ldots,x_k) &= \P(x_i+b_{i,n} \leftrightarrow x_j+b_{j,m} \text{ if and only if $\{(i,n),(j,m)\}\in E(S)$} \\&\hspace{8cm}\text{ and } x_i+W_i \text{ closed for every $1\leq i \leq k$})\\
&=:\P(\mathscr{E}(S;x_1,\ldots,x_k)),
\end{align*}
where $\mathscr{E}(S;x_1,\ldots,x_k)$ denotes the event whose probability defines $S(x_1,\ldots,x_k)$.
Since plans are not permitted to contain self-loops, the event $\mathscr{E}(S;x_1,\ldots,x_k)$ includes the constraint that all of the marked points $\{x_i+b_{i,j} : 1\leq j \leq \deg(i)\}$ in the blob $B_i$ are in distinct clusters for each $1\leq i \leq k$. A simple example of a calculation with scheme functions is given in the proof of \cref{thm:derivative} below.

\begin{figure}
\centering
\includegraphics{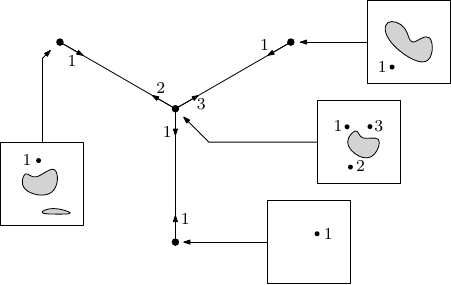}
\caption{Schematic illustration of a \textbf{scheme}: a \textbf{plan} together with assignments of \textbf{blobs} to each vertex. In this example, the scheme is a $3$-scheme whose associated plan has one vertex of degree $3$ and three vertices of degree $1$, so that the degree $3$ vertex is assigned a $3$-blob and each of the three degree $1$ vertices are assigned $1$-blobs.} 
\end{figure}

\medskip

The main result of this section is the following asymptotic estimate on scheme functions, which can be thought of as a far-reaching generalization of the two-blob estimate \eqref{eq:two_blob_assumption}. Given a scheme $S$, we write $E^{1/2}(S)$ for the set of oriented edges of the form $((i,n),(j,m))$ with $i<j$.

\begin{thm}[Asymptotic factorization of scheme functions]
\label{thm:scheme_function}
Let $d>6$ and $L\geq 1$ be such that critical percolation on $\mathbb{Z}^d_L$ satisfies the estimates \eqref{eq:two_point_assumption} and \eqref{eq:two_blob_assumption}. There exists a function $\mathbf{W}$ assigning non-negative real numbers to blobs, invariant under permuting the vertices of a blob, such that if $S$ is a $k$-scheme then
\[
S(x_1,\ldots,x_k) \sim \prod_{i=1}^k \mathbf{W}(B_i) \prod_{((i,n),(j,m)) \in E^{1/2}(S)} AG(x_i,x_j)
\]
as $\min_{i\neq j}\|x_i-x_j\|\to\infty$. Moreover, $\mathbf{W}(B)$ is positive for a blob $B=(W,b_1,\ldots,b_k)$ if and only if there exists a collection of vertex-disjoint paths connecting $b_1,\ldots,b_k$ to infinity in $\Z^d_L$ that are disjoint from $W$.
\end{thm}

We call $\mathbf{W}$ the \textbf{blob factor}, by analogy with the vertex factor.
As with the two-blob estimate~\eqref{eq:two_blob_assumption}, a key feature of this estimate is that the dependencies on the microscopic geometry of the blobs approximately factorize from each other and from the large scale connection probabilities between blobs. 
As usual, the asymptotic estimate of \cref{thm:scheme_function} is valid even when the right-hand side is zero, in which case it means that $S(x_1,\ldots,x_k)$ is identically zero whenever $\min_{i\neq j}\|x_i-x_j\|$ is sufficiently large. Note that the blob $(\emptyset,0)$ must have $\mathbf{W}((\emptyset,0))=1$ for \cref{thm:scheme_function} to be consistent with \eqref{eq:two_point_assumption}. More generally, if $W$ is any finite set of edges then $\mathbf{W}((W,0))$ is equal to the probability that $W$ is closed in the IIC measure.

\begin{remark}
\label{remark:karmIIC}
It is a corollary of \cref{thm:scheme_function} that if $b_1,\ldots,b_k$ are such that there exist vertex-disjoint paths $\gamma_1,\ldots,\gamma_k$ in $\Z^d_L$ with $\gamma_i$ connecting $b_i$ to infinity, then there is a well-defined \textbf{$k$-arm IIC measure} defined by conditioning $b_1,\ldots,b_k$ to be connected in distinct clusters to some distant points $x_1,\ldots,x_k$: The probability under this measure that a finite set of edges $W$ is closed is precisely $\mathbf{W}((W,b_1,\ldots,b_k))/\mathbf{W}((\emptyset,b_1,\ldots,b_k))$, with the probabilities of all other cylinder events determined by inclusion-exclusion. Related $k$-arm IIC measures were also constructed via a different limiting procedure (working with $p\uparrow p_c$ limits) in the work of the first author and Nachmias \cite{blanc2025critical}; it should be possible to prove that these two limiting constructions are equivalent as for the usual IIC \cite{van2004incipient} but we do not pursue this here. See also the results of \cite{cabezas2025bi} regarding a slightly different notion of $2$-arm IIC we call the \emph{monochromatic $2$-arm IIC} and study in \cref{Sec:TreeScheme}.
\end{remark}

\medskip

Before proceeding further, let us note how \cref{thm:scheme_function} immediately yields our theorem on the asymptotics of the derivative.

\begin{proof}[Proof of \cref{thm:derivative} given \cref{thm:scheme_function}]
For each neighbour $z$ of $0$ in $\Z^d_L$, let $B_z$ be the $2$-blob $(\emptyset,0,z)$, let $B_\emptyset$ denote the $1$-blob $(\emptyset,0)$, and let $S_z$ be the scheme whose plan is a path of length $2$ with middle vertex associated to the blob $B_z$ and with endpoint vertices both associated to $B_\emptyset$. By Russo's formula,
we have that
 \[T'_{p_c}(x,y)=\frac{1}{1-p_c}\sum_{a\in \Z^d} \sum_{b \sim a}\P_{p_c}(x \leftrightarrow a \nleftrightarrow b \leftrightarrow y)
 = \frac{1}{1-p_c}\sum_{a\in \Z^d} \sum_{z\sim 0} S_z(x,a,y).\]
It follows from \cref{thm:scheme_function} that there exists a positive constant $G_\mathrm{piv}= \frac{p_c}{1-p_c} \sum_{z\sim 0} \mathbf{W}(B_z)$ such that for each $\eps>0$ there exists $R(\eps)<\infty$ such that
\[
  \Bigl|\frac{1}{1-p_c}\sum_{z\sim 0} S_z(x,a,y) - \frac{1}{p_c}G_\mathrm{piv} A^2G(x,a)G(a,y)\Bigr| \leq \eps G(x,a)G(a,y)
\]
whenever $\min\{\|x-y\|,\|x-a\|,\|y-a\|\}\geq R$. (We use this normalization since we later use $G_\mathrm{piv}$ to count \emph{open} pivotals.) On the other hand, using the BK inequality 
to bound $S_z(x,a,y) \leq T_{p_c}(x,a)T_{p_c}(a+z,y)$ easily shows that the contribution to the sum from points $a$ that have distance at most $R$ from $\{x,y\}$ is bounded by
 $C(R)\|x-y\|^{-d+2}$ for some $R$-dependent constant $C(R)$, and it follows easily that $T'_{p_c}(x,y)\sim p_c^{-1} G_\mathrm{piv} A^2 G*G(x,y)$ as $\|x-y\|\to \infty$ as claimed.
\end{proof}


\subsection{Proof of \cref{thm:scheme_function}}

In this section we prove \cref{thm:scheme_function} (apart from the characterisation of which blobs have positive blob factors which is deferred to \cref{sec:positivity_of_local_factors}). The proof can be summarised as follows:

\begin{itemize}
  \item \textbf{Coupling and the basic identity.}
  We first couple our critical percolation configuration with a collection of \emph{independent} percolation clusters, with one independent cluster for each edge of the plan. To construct this coupling we first fix an orientation of each edge of the plan and an enumeration of the edges. We use this coupling to rewrite the scheme function in terms of this collection of independent clusters in
  \eqref{eq:InductiveAizenmanOff}, which reduces the problem to understanding the probability that each independent cluster connects the two points it is required to connect off of the sets $x_i+W_i$ and the clusters revealed at earlier stages of the exploration.

  \item \textbf{Localization of the interaction (\cref{prop:BlobOffApprox}).} To analyze the reformulation of the scheme function given in \eqref{eq:InductiveAizenmanOff}, we consider a modified version of the procedure used to turn our many independent clusters into the actual clusters in the single percolation configuration in which we only consider intersections that occur locally near our marked points $x_1,\ldots,x_k$, and only between the clusters that are  required to enter the vicinity of these points. We prove that the probability our single percolation configuration makes all the connections we require is equal, up to a negligible error, to the analogous quantity  for this modified collection of clusters with localized interactions.
   In other words, up to negligible errors, the only way the coupled clusters can fail to make their required connections is through some \emph{local} obstruction in the vicinity of one of the points $x_1,\ldots,x_k$; there are no ``unforced interactions'' between the clusters at macroscopic scales. As part of the proof, we also show that configurations are negligibly unlikely to include ``extra arms'' around the points $x_1,\ldots,x_k$ that are not forced to exist by the scheme event (\cref{lem:2Arm}), a fact we will use again in \cref{subsec:localization_of_overcount_factors}.

  \item \textbf{A non-factorized asymptotic formula (\cref{lem:weaker_scheme_function}).}
  Once the connection event has been localized, the two-blob estimate \eqref{eq:two_blob_assumption} implies that the events describing the local geometry around distinct points $x_i$ become asymptotically independent as $\min_{i\neq j}\|x_i-x_j\|\to\infty$ (\cref{lem:BlobOffApproxCV}). The conditional probability of interest is therefore well-approximated by an expression of the desired form where the constant prefactor is written as a product 
   of local factors, each depending on the blob, the enumeration chosen for the edges of the scheme (restricted to the edges incident to the relevant vertex of the plan), and the localization parameters.
  Taking a double limit in which we first send $\min_{i\neq j}\|x_i-x_j\|\to\infty$ and then send the localization parameters to infinity, we deduce (\cref{lem:weaker_scheme_function}) that the products of these localized blob factors must form a Cauchy sequence as the localization factors become large, and hence that for each scheme $S$ there exists a well-defined constant $\Pi(S)$ such that
  \[
    S(x_1,\ldots,x_k) \sim \Pi(S)
      (1-p)^{|W_1|+\cdots+|W_k|}\,A^{|E(S)|}\,G_S(x_1,\ldots,x_k)
  \]
as $\min_{i\neq j}\|x_i-x_j\|\to \infty$, where $G_S$ is the product of Green's functions appearing in \cref{thm:scheme_function}.
 (The proof gives a constant that may also depend on the enumeration and orientation of the edges of the plan, but since these choices do not appear on the left hand side of this estimate $\Pi(S)$ cannot really depend on them.)

  \item \textbf{Factorization of the constant.}
  It remains to identify the constant $\Pi(S)$ with a product of local blob factors. At this point we know that 
the limiting factors exist for any product of localized factors over all the blobs in a scheme and must upgrade this to existence for each individual factor.
   This is done by noting that the trivial blob $(\emptyset,0)$ has its blob factor determined by the two-point asymptotic estimate \eqref{eq:two_point_assumption} and that every other blob can appear in a scheme in which all other blobs are the trivial blob, so that we can deduce convergence of individual blob factors from the convergence of products we have already established. 
\end{itemize}

We now begin working towards the proof.
For the rest of this section we will fix $d>6$ and $L\geq 1$ such that critical percolation on $\Z^d_L$ satisfies \eqref{eq:two_point_assumption} and \eqref{eq:two_blob_assumption} and fix a scheme $S$. 
Recall that for each $x\in \Z^d$ we let $K_x$ denote the cluster of $x$ in critical percolation. We think of $K_x$ as a subgraph of $\Z^d_L$: its vertices are the vertices connected to $x$, and its edges are the open edges between them. For every subgraph $G$ of $\Z^d_L$, set of vertices $V$ in $\Z^d_L$, and vertex $x$ of $G$, let $G \offx x V$ be the subgraph of $G$ obtained as the connected component of $x$ in $G$ after closing all the edges with at least one endpoint in $V$. If we are also given a set of edges $E$ we write $G \offx x (V,E)$ for the connected component of $x$ in $G$ after closing all the edges that either belong to $E$ or have at least one endpoint in $V$.

\medskip

Let us also introduce some notation for balls and their boundaries.
For each $r\geq 0$ and $x\in\Z^d$, we write
  $B_r(x)=\{y\in\Z^d:\|y-x\|\leq r\}$
and define $\partial B_r(x)$ to be the inner vertex boundary of $B_r(x)$ in $\Z^d_L$, i.e., the set of vertices in $B_r(x)$ adjacent to at least one vertex outside $B_r(x)$. We also write $B_r^{\mathbf e}(x)$ for the set of edges with both endpoints in $B_r(x)$. If $A$ is a set of vertices, percolation on $\Z^d_L\setminus A$ means percolation on the subgraph induced by $\Z^d\setminus A$. We say that two sets of vertices  \textbf{touch} if they either intersect or are adjacent.

\medskip

Let $N=\# E^{1/2}(S)$ be the number of edges in the plan of $S$ and let $\prec$ be an arbitrary total order on $E^{1/2}(S)$; we will refer to $\prec$ as an \textbf{order on $S$}.
 We enumerate the edges of $E^{1/2}(S)$ as $e_1\prec e_2\prec \dots \prec e_N$ and for each $1\leq n \leq N$ write $e_n=(e_n^-,e_n^+)=((\alpha^-_n,\beta^-_n),(\alpha^+_n,\beta^+_n))$, 
 $y_n^-:=x_{\alpha^-_n}+b_{\alpha^-_n,\beta^-_n}$, and $y_n^+:=x_{\alpha^+_n}+b_{\alpha^+_n,\beta^+_n}$. 

\medskip

Consider the set of edges $E_W=E_W(x_1,\ldots,x_k):=\bigcup_{1\leq i \leq k} (x_i+W_i)$.
Let $(\tilde K_n)_{n=1}^N$ be \emph{independent} random subgraphs of $\Z^d_L$, where $\tilde K_n$ is distributed as the cluster of $y_n^-$ in critical percolation on $\Z^d_L$ for each $1\leq n \leq N$.
We now couple these independent clusters with a single instance of critical percolation on $\Z^d_L$.
  Define a sequence of subgraphs $(K_n)_{n=1}^N$ inductively by
\[ K_n:= \begin{cases}
\tilde K_n \offx {y_n^-} \left (\bigcup_{1\leq m<n} K_m, E_W \right )
& \text{ if $y_n^- \notin \bigcup_{1\leq m<n} K_m$}
\\
\varnothing & \text{otherwise},
\end{cases}\]
for each $1\leq n \leq N$, where $\varnothing$ denotes the empty subgraph.
We can easily verify by induction that, for each $1\leq n \leq N$, the union of the clusters $(K_1,\ldots,K_n)$ has the same distribution as the union of the clusters of $(y_1^-,\ldots,y_n^-)$ conditioned on the set $E_W$ being closed. In particular, the clusters $(K_1,\ldots,K_n)$ can be coupled with the critical percolation configuration conditioned on $E_W$ being closed so that $K_n$ is equal to the cluster of $y_n^-$ on the event that $y_n^-$ does not belong to $\bigcup_{m<n} K_m$. On the other hand, if $y_n^-$ does belong to $\bigcup_{m<n} K_m$ then $K_n$ is empty. These considerations are summarized by the identity
\begin{equation} S(x_1,x_2,\dots, x_k)=\proba(E_W \text{ closed})\proba(y_n^+\in K_n \text{ for every } 1\leq n \leq N),
\label{eq:InductiveAizenmanOff}
\end{equation}
which holds for every scheme $S$, order $\prec$, and $x_1,x_2,\dots, x_k$.

\medskip

In this section, we will often consider events $E(R_1,R_2,(x_i)_{1\leq i \leq k})$ depending on parameters $R_1$, $R_2$, and $(x_i)_{1\leq i \leq k}$, defined on the same probability space as $(\tilde K_n)_{1\leq n \leq N}$. We say that such an ``event'' (which is really a family of events) holds $(S,\prec)$\textbf{-usually} (or simply that the event is $(S,\prec)$-usual) if 
\[ \liminf_{R_1\to \infty} \liminf_{R_2\to\infty} \liminf_{\min_{1\leq i \neq j \leq k}\|x_i-x_j\|\to \infty} \proba \left (E(R_1,R_2,(x_i)_{1\leq i \leq k}) \; \middle | \; \forall n, y_n^+\in \tilde K_n \right )=1. \]
We say that such an event  holds \textbf{$(S,\prec)$-rarely} (or simply that the event is $(S,\prec)$-rare) if its complement holds $(S,\prec)$-usually.
Note that if $E$ and $E'$ both hold $(S,\prec)$-usually then the events $E\cup E'$ and $E\cap E'$ also hold $(S,\prec)$-usually, so we can use  standard logical tools to manipulate such events.

\medskip

For each $1\leq i\leq k$ and $1\leq j\leq \deg(i)$ let $n(i,j)$ denote the unique integer such that $(i,j)\in \{e_{n(i,j)}^-,e_{n(i,j)}^+\}$ 
and define $\prec_i$ to be the order on $\{1\leq j \leq \deg(i)\}$ such that $j\prec_i j'$ if and only if $n(i,j)<n(i,j')$.
We approximate the rightmost event in \eqref{eq:InductiveAizenmanOff} as follows: for every $R\geq 1$, $1\leq i \leq k$, and $1\leq j \leq \deg(i),$ let 
\[ \tilde K^{R}_{i,j}:= \tilde K_{n(i,j)} \offx {x_i+b_{i,j}} (\Z^d\backslash B_R(x_i)) \]
to be the part of $\tilde K_{n(i,j)}$ reachable from $x_i+b_{i,j}$ within $B_R(x_i)$. For each $1\leq i \leq k$ we define inductively, along the order $\prec_i$, for every $1\leq j\leq \deg(i)$:
\[ K^R_{i,j}:=\tilde K_{i,j}^R \offx {x_i+b_{i,j}} \left (\bigcup_{j':j'\prec_i j} K^R_{i,j'}, (x_i+W_i) \right ).\]
Thus, the definition of $K^R_{i,j}$ is the same as $K_{n(i,j)}$ except that we only take account of intersections between different clusters (and different parts of $W$) locally in the vicinity of $x_i$.
The main step of the proof is the following proposition, which localizes the interaction between different clusters to the vicinity of the points $x_1,\ldots,x_k$.

\begin{prop} \label{prop:BlobOffApprox} For every scheme $S$ and order $\prec$ on $S$, the symmetric difference
\[
 \bigl\{ y_n^+\in K_n \text{ for every $1\leq n \leq N$}\bigr\} \,\Delta\, \bigl\{K^{R_2}_{i,j} \cap \partial B_{R_1}(x_i) \neq \emptyset \text{ for every $1\leq i \leq k$ and $1\leq j \leq \deg(i)$} \bigr\}  \]
 holds $(S,\prec)$-rarely. 
\end{prop}

In other words, on the event that all the independent clusters $\tilde K_n$ make the desired connections (that is, $y_n^+ \in \tilde K_n$ for each $1\leq n \leq N$), the event that the clusters $K_n$ make the same connections is roughly the same as the event that none of the clusters $K_n$ are blocked from becoming large by some local obstruction near one of the points $y_n^-$ or $y_n^+$.

\medskip


The basic reason \cref{prop:BlobOffApprox} is useful is that the events $\{K^{R_2}_{i,j} \cap \partial B_{R_1}(x_i) \neq \emptyset$ for all $1\leq j \leq \deg(i)\}$ are easily seen to decorrelate when $\min_{i\neq j} \|x_i-x_j\|\to\infty$ with $R_1,R_2$ fixed as established in the following lemma.

\begin{lemma} \label{lem:BlobOffApproxCV} For every $R_1,R_2$, there exists a function $\widehat{\mathbf{W}}^{R_1,R_2}$ from pairs consisting of a blob $(W,b_1,\ldots,b_k)$ and a total order of $\{b_1,\ldots,b_k\}$ to non-negative real numbers such that for every scheme $S$ and order $\prec$, we have that
\[ \proba\left ( \forall i,j,  K^{R_2}_{i,j} \cap \partial B_{R_1}(x_i) \neq \emptyset \; \middle | \; \forall n, y_n^+\in \tilde K_n \right ) \longrightarrow \prod_{1\leq i \leq k} \widehat{\mathbf{W}}^{R_1,R_2}(B_i,\prec_i) \]
as $\min_{i\neq j} \|x_i-x_j\|\to \infty$.
\end{lemma}

After this lemma is proven, we will write $\mathbf{W}^{R_1,R_2}(W,b_1,\ldots,b_k)=(1-p)^{|W|} \widehat{\mathbf{W}}^{R_1,R_2}(W,b_1,\ldots,b_k)$, which is more relevant to the statement of \cref{thm:scheme_function} (since our scheme function includes the condition that the sets $x_i+W_i$ are closed).
We will see at the end of the section that $\mathbf{W}^{R_1,R_2}$ and $\widehat{\mathbf{W}}^{R_1,R_2}$ do not actually  depend on the order we place on the blob (and are invariant under permuting the order of the vertices $b_1,\ldots,b_k$ in the blob).

\begin{proof}[Proof of \cref{lem:BlobOffApproxCV}]  Since $\{ \tilde K_n\}_{1\leq n \leq N}$ are independent, it follows from the two-blob estimate \eqref{eq:two_blob_assumption} that for each fixed $R_2\geq 1$
we have the distributional convergence of conditional measures
\[ (\tilde K_{i,j}^{R_2}-x_i)_{i,j} \;\Big|\; \text{$y_n^+\in \tilde K_n$ for every $1\leq n\leq N$}  \qquad \cvd \qquad (K_{\mathrm{IIC}}^{R_2}(b_{i,j})+b_{i,j})_{i,j},\]
where $K_{\mathrm{IIC}}^{R_2}(b_{i,j})$ is distributed as the connected component of $0$ under the incipient infinite cluster measure after closing every edge adjacent to $\Z^d\backslash B_{R_2}$ and the different clusters $(K_{\mathrm{IIC}}^{R_2}(b_{i,j}))_{i,j}$ are taken to be independent. (This independence is immediate when $(i,j)$ and $(i',j')$ correspond to different edges of the plan and follows from the two-blob assumption when $(i,j)$ and $(i',j')$ correspond to different orientations of the same edge of the plan.)
Moreover, if $y_n^+\in \tilde K_n$ for every $1\leq n \leq N$, then for every $1\leq i\leq k$ we have by induction (using the order $\prec_i$ on $1\leq j \leq \deg(i)$) that 
$K^{R_2}_{i,j}$ is obtained as a function of $W_i$ and the truncated clusters
$(\tilde K^{R_2}_{i,j'})_{j'\preceq_i j}$.
As such, the collection of subgraphs $\{K^{R_2}_{i,j}-x_i \}_{1\leq i \leq k, 1\leq j \leq \deg(i)}$ converges  in distribution as $\min_{i\neq j} \|x_i-x_j\|\to \infty$ to a limiting collection of random subgraphs in which the subgraphs associated to different values of $1\leq i \leq k$ are independent and have laws determined by $(B_i,\prec_i)$. The claim follows with $\widehat{\mathbf{W}}^{R_1,R_2}(B_i,\prec_i)$ given by the relevant probability in the limit law associated to the pair $(B_i,\prec_i)$.
\end{proof} 

As a first step towards the proof of \cref{prop:BlobOffApprox}, 
let us deal with an event that will repeatedly appear in our argument. We say that the event $2$-Arms$(R_1,R_2)$ holds if for some $1\leq i \leq k$, $1\leq j \leq \deg(i)$ there exist two edge-disjoint paths from $\partial B_{R_1}(x_i)$ to $\partial B_{R_2}(x_i)$ in $\tilde K_{n(i,j)}$.

\begin{lemma} \label{lem:2Arm} The event $2\emph{-Arms}(R_1,R_2)$ holds $(S,\prec)$-rarely. 
\end{lemma}

\begin{proof}[Proof of \cref{lem:2Arm}]
We may assume that $R_1$ is larger than the maximum of the finitely many numbers $\|b_{i,j}\|$.
It suffices to prove under this assumption that
\begin{multline*}
  \!\!\!\lim_{R_2 \to \infty} \limsup_{\min_{a\neq b} \|x_a-x_b\|\to \infty}\\\P\Bigl(\text{$\exists$ two edge-disjoint paths from $\partial B_{R_1}(x_i)$ to $\partial B_{R_2}(x_i)$ in $\tilde K_{n(i,j)}$}\mid y_{n(i,j)}^+ \in \tilde K_{n(i,j)}\Bigr)
  =0
\end{multline*}
for each fixed $1\leq i\leq k$, $1\leq j\leq \deg(i)$.
Fix one such choice of $i$, $j$, and $R_1$, let $i',j'$ be such that $\{(i,j),(i',j')\}\in E$ and let $n=n(i,j)$. Write
$u=x_i+b_{i,j}$ and  $v=x_{i'}+b_{i',j'}$
so that $\{u,v\}=\{y_n^-,y_n^+\}$. Since $\tilde K_n$ is the cluster of $y_n^-$, the event $y_n^+\in \tilde K_n$ is exactly the event that there is a path from $u$ to $v$ in $\tilde K_n$.
Suppose that this event holds and that $\tilde K_{n}$ contains two edge-disjoint paths $\gamma, \gamma'$ from $\partial B_{R_2}(x_i)$ to $\partial B_{R_1}(x_i)$. Let $\Gamma$ be a path from $v$ to $u$ in $\tilde K_n$.
 By stopping $\Gamma$ at the first time it intersects $\gamma\cup\gamma'$ or $\partial B_{R_1}(x_i)$ and then following the corresponding path if necessary, we find that there exists a pair of edge-disjoint paths connecting $\partial B_{R_1}(x_i)$ to $v$ and $\partial B_{R_1}(x_i)$ to $\partial B_{R_2}(x_i)$ respectively.
  Hence, by the BK inequality, the conditional probability in the display above is at most
\[  \frac{\proba(\partial B_{R_1}(x_i)\slr v)\proba(\partial B_{R_1}(x_i)\slr \partial B_{R_2}(x_i))}{\proba(u \leftrightarrow v)}.\]
It follows by a simple finite-energy argument (or by the two-point estimate \eqref{eq:two_point_assumption} and a union bound) that the ratio $\proba(\partial B_{R_1}(x_i)\slr v)/\proba(u \leftrightarrow v)$
is bounded by an $R_1$-dependent constant, and the claim follows since $\proba(\partial B_{R_1}(x_i)\slr \partial B_{R_2}(x_i))\to 0$ as $R_2\to \infty$ with $R_1$ fixed (since there are no infinite clusters at criticality).
\end{proof}

To proceed, we first define yet another intermediate collection of clusters. For each $1\leq i \leq k$ let $\cN_i:=\{n(i,j)\}_{1\leq j \leq \deg(i)}$. Inductively, for $1\leq i \leq k, 1\leq j \leq \deg(i)$ we define the collection of subgraphs
\[ K_{i,j}' := \tilde K_{n(i,j)} \offx {x_i+b_{i,j}} \left (\bigcup_{j'\prec_i j} K_{i,j'}' \cup \bigcup_{\substack{1\leq m<n(i,j)\\ m\notin \cN_i}} K_m, E_W \right ). \]
%
Note that the only difference between the definitions of $K_{n(i,j)}$ and $K_{i,j}'$ is that instead of considering the connected component of $y_{n(i,j)}^-$ we consider the component of $x_i+b_{i,j}\in \{y_{n(i,j)}^-, y_{n(i,j)}^+\}$. (In other words, $K_{i,j}'$ is what we would have computed for $K_{n(i,j)}$ if we had first modified the choice of $E^{1/2}(S)$ to take all the edges incident to $i$ to be oriented out of $i$.)


\begin{lemma} \label{lem:BlobOffApproxA} 
The event $\{ K_{i,j}' \cap B_{R_1}(x_i)=  K^{R_2}_{i,j} \cap B_{R_1}(x_i)$ for every $1\leq i \leq k$ and $1\leq j \leq \deg(i)\}$ holds $(S,\prec)$-usually.
\end{lemma}

\begin{proof}[Proof of \cref{lem:BlobOffApproxA}]
It suffices to prove that $\{ K_{i,j}' \cap B_{R_1}(x_i)=  K^{R_2}_{i,j} \cap B_{R_1}(x_i)\}$ holds $(S,\prec)$-usually for each $1\leq i \leq k$ and $1\leq j \leq \deg(i)$.
We prove this by induction on $j$ for each $i$. Fix $i,j$ and assume that the result holds for all $j'\prec_i j$. 
Let $R_{3/2}$ be an intermediate parameter between $R_1$ and $R_2$.
The induction hypotheses can be written as stating that
\begin{equation} 
\liminf_{R_{3/2}\to\infty} \liminf_{R_{2}\to\infty} \liminf_{\min\|x_a-x_b\|\to \infty}\proba\left (\forall j'\prec_i j, K_{i,j'}'\cap B_{R_{3/2}}(x_i)=  K_{i,j'}^{R_2}\cap B_{R_{3/2}}(x_i) \;\middle |\; \forall n, y_n^+\in \tilde K_n \right )=1. \label{eq:4LiminfBlobA}\end{equation}
By definition, both $K_{i,j}^{R_2}$ and $K_{i,j}'$ are connected components of $x_i+b_{i,j}$ in $\tilde K_{n(i,j)}$ after removing some vertices and edges. So, for a vertex $z$ to be in the symmetric difference $K_{i,j}'\Delta K_{i,j}^{R_2}$ there must exist a path $\gamma$ in $K_{n(i,j)}$ from $x_i+b_{i,j}$ to $z$ that uses a vertex or edge  that is permitted for one of $K_{i,j}'$ or $K_{i,j}^{R_2}$ but not for the other. Thus it is enough to prove that $(S,\prec)$-rarely there exists a simple path $\gamma$ in $\tilde K_{n(i,j)}$ from $x_i+b_{i,j}$ to a vertex $z\in B_{R_1}(x_i)$ and a vertex $u\in \gamma$ such that at least one of the following conditions hold: 
\begin{itemize}
\item[(a)] $u\in\Z^d\backslash B_{R_2}(x_i)$,
\item[(b)]  $u$ is the endpoint of an edge of $E_W\backslash(x_i+W_i)$,
\item[(c)]  
$u\in K_{i,j'}^{R_2} \Delta K_{i,j'}'$ for some $j'\prec_i j$,
\item[(d)]  $u\in K_m$ for some $1\leq m < n$ that does not belong to $\cN_i$.
\end{itemize}
We will say ``case (a)'' to mean the event that there exists such a pair $(\gamma,u)$ with $u\in \Z^d\setminus B_{R_2}(x_i)$, with similar terminology in use for other cases.
If case (a) holds then inside $\gamma$  we have two edge-disjoint paths from $\partial B_{R_1}(x_i)$ to $\partial B_{R_2}(x_i)$, which holds $(S,\prec)$-rarely by \cref{lem:2Arm}. If case (b) holds then case (a) also holds provided that $\min_{i\neq j} \|x_i-x_j\|$ is large enough as a function of $R_2$, so that case (b) holds $(S,\prec)$-rarely also. Similarly, for case (c), either the event appearing in \eqref{eq:4LiminfBlobA} does not hold (which is $(S,\prec)$-rare by our induction hypothesis) or, as in the proof of case (b), there exist two edge-disjoint paths from $\partial B_{R_{3/2}}(x_i)$ to $\partial B_{R_1}(x_i)$, and the fact that case (c) holds $(S,\prec)$ rarely follows by another application of \cref{lem:2Arm}.

\medskip

For case $(d)$, if $K_m\cap B_{R_2}(x_i)=\emptyset$ for every $1\leq m < n$ that does not belong to $\cN_i$ we are again in case (a). As such, it suffices to prove that the event that there exists $1\leq m <n$ with $m\notin \cN_i$ such that $K_m \cap B_{R_2}(x_i)\neq \emptyset$ holds $(S,\prec)$-rarely. Fix one such $m\notin \cN_i$. We have by definition that $K_m\subset \tilde K_m$ and hence by the BK inequality that
\[ \proba \left (K_m\cap B_{R_2}(x_i)\neq \emptyset \;\middle  |\; \forall n, y_n^+\in \tilde K_n \right ) \leq \sum_{v\in \Z^d} \frac{\proba (y_m^-\slr v )\proba (y_m^+\slr v )\proba  (v\slr B_{R_2}(x_i) )}{\proba  (y_m^- \slr y_m^+ )}. \]
It follows by a standard finite-energy argument (or by \eqref{eq:two_point_assumption} and a union bound) that there exists a constant  $C(R_2,S)$ 
 depending only on $R_2$ and $S$ such that
\begin{equation}
\label{eq:irrelevant_ball} \proba \left (K_m\cap B_{R_2}(x_i)\neq \emptyset \;\middle  |\; \forall n, y_n^+\in \tilde K_n \right ) \leq C(R_2,S) \sum_{v\in \Z^d} \frac{\proba (x_{\alpha^-_m}\slr v)\proba (x_{\alpha^+_m}\slr v )\proba \left (v\slr x_i\right )}{\proba (x_{\alpha^-_m} \slr x_{\alpha^+_m} )}. \end{equation}
Moreover, since $G(x,y)$ is of order $(1+\| x-y\|)^{-d+2}$, it follows from \eqref{eq:two_point_assumption} and the triangle inequality that there exists a constant $C$ such that
\begin{equation} \min(\proba(a\slr b),\proba(b\slr c))\leq C\proba(a\slr c) \label{eq:ConvenientSplitFrac}\end{equation}
for every $a,b,c\in \Z^d$.
Substituting this estimate into \eqref{eq:irrelevant_ball} yields that there exists a constant $C'(R_2,S)$ such that
\[ \proba \left (K_m\cap B_{R_2}(x_i)\neq \emptyset \;\middle  |\; \forall n, y_n^+\in \tilde K_n \right ) \leq C'(R_2,S) \sum_{v \in \Z^d} (\proba (x_{\alpha^-_m}\slr v)+\proba (x_{\alpha^+_m}\slr v ))\proba \left (v\slr x_i\right ). \]
Since $d>6 > 4$ and $\P(x\leftrightarrow y)$ is of order $\|x-y\|^{-d+2}$, the convolution of the two-point function with itself is of order $\|x-y\|^{-d+4}$ and therefore converges to zero as $\|x-y\|\to \infty$.
It follows that the final sum converges to $0$ as $\min_{1\leq a\neq b\leq k} \|x_a-x_b\|\to \infty$ and hence that case (d) holds $(S,\prec)$-rarely as claimed.  \qedhere
\end{proof}

\begin{proof}[Proof of \cref{prop:BlobOffApprox}] Recall that the only difference in the definitions of $ K_{n(i,j)}$ and $K_{i,j}'$ is that instead of considering the connected component of $y_n^-$ we consider the component of $x_i+b_{i,j}\in \{y_{n(i,j)}^-, y_{n(i,j)}^+\}$. As such, on the event that $y_n^+ \in K_n$ for every $1\leq n\leq N$ we also have that $K_{i,j}'=K_{n(i,j)}$ for every $1\leq i\leq k$ and $1\leq j \leq \deg(i)$.
Applying \cref{lem:BlobOffApproxA}, it follows that the implications 
\begin{equation}  \{ \forall n, y_n^+\in K_n\} \Longrightarrow  \{\forall i,j, K_{n(i,j)}' \cap \partial B_{R_1}(x_i) \neq \emptyset \}\Longrightarrow \{\forall i,j,  K^{R_2}_{i,j} \cap \partial B_{R_1}(x_i) \neq \emptyset \} \label{eq:prop:BlobOffApprox=>}\end{equation}
hold $(S,\prec)$-usually in the sense that the event that the left-hand event holds but the right-hand event does not is $(S,\prec)$-rare.

\medskip

To conclude, we need to show that the reverse implication holds $(S,\prec)$-usually in the same sense. 
Consider the event
\[ \sG(R_1,R_2)= 
\left\{
 y_{n(i,j)}^+\in \tilde K_{n(i,j)},\, K^{R_2}_{i,j} \cap B_{R_1}(x_i)= K_{i,j}' \cap B_{R_1}(x_i), \text{ and }  K^{R_2}_{i,j} \cap \partial B_{R_1}(x_i)\neq \emptyset   \quad \forall i,j
 \right\}
 \]
and for each $1\leq i \leq k$ and $1\leq j\leq \deg(i)$ consider the event 
\[ \sF_{i,j}=\left \{ y_{n(i,j)}^+\notin K_{n(i,j)} \, \text{ and } \, \forall m <n(i,j),\,y_{m}^+\in K_{m} \right \}.\]
To show the converse of \eqref{eq:prop:BlobOffApprox=>} it is enough to show that for every $i,j$ the event $\sG(R_1,R_2)\cap \sF_{i,j}$ holds $(S,\prec)$-rarely. Fix $(i,j)$, let $(i',j')$ be such that $\{(i,j),(i',j')\}\in E$, and suppose that $\sG(R_1,R_2)\cap \sF_{i,j}$ holds. 
Since $\sG(R_1,R_2)$ holds, there exists a simple path $\gamma$ from $x_i+b_{i,j}$ to $\partial B_{R_1}(x_i)$ in $K_{i,j}^{R_2}$ and a simple path $\gamma'$ from $x_{i'}+b_{i',j'}$ to $\partial B_{R_1}(x_{i'})$ inside $K_{i',j'}^{R_2}$, and a simple path $\Gamma$ from $x_i+b_{i,j}$ to $x_{i'}+b_{i',j'}$ in $\tilde K_{n(i,j)}$.
By taking the portion of $\Gamma$ between the last place it intersects $\gamma$ and the first place it intersects $\gamma'$, it follows that there exists a simple path $\Gamma'$ in $\tilde K_{n(i,j)}$ that is edge-disjoint from $\gamma$ and $\gamma'$, starts from a vertex $z\in \gamma$ and ends at a vertex $z'\in \gamma'$. 

Fix some measurable way of choosing the paths $\gamma$, $\gamma'$, and $\Gamma'$ as a function of the various configurations under consideration, the precise choice of this assignment being unimportant, and let $M$ be the minimal distance the path $\Gamma'$ attains from the points $\{x_i,x_{i'}\}$.
Note that for each constant $r\geq 1$, if $R_1$ is sufficiently large and 
$\Gamma'$ comes within distance $r$ of $x_i$ then the paths $\Gamma'$ and $\gamma$, which are disjoint and both contained in $\tilde K_{n(i,j)}$, must each 
contain paths from $\partial B_r(x_i)$ to $\partial B_{R_1}(x_i)$. Using this observation together with the analogous observation around $x_{i'}$ and
applying \cref{lem:2Arm}, it follows that
\begin{equation} \limsup_{R_{1}\to\infty} \limsup_{R_{2}\to\infty} \limsup_{\min\|x_a-x_b\|\to \infty}\proba\left (\sG(R_1,R_2)\cap \sF_{i,j}\cap \{ M\leq r\} \; \middle | \; \forall n, y_n^+\in \tilde K_n \right )= 0.\label{eq:BlobOffStopCloseA} \end{equation} 
for each $r\geq 1$.
It remains to prove an analogous estimate in the case that $M$ is large. Take $r\geq 1$ sufficiently large that $\Gamma'$ is necessarily disjoint from $x_i+W_i$ and $x_{i'}+W_{i'}$ when $M\geq r$.
When $\sF_{i,j}$ holds, we have by the same induction argument used at the start of the proof that
  $K_{a,b}'=K_{n(a,b)}$ for every $a,b$ such that $n(a,b)< n(i,j)$. Thus, since $\sG(R_1,R_2)$ and $\sF_{i,j}$ both hold, the two paths $\gamma$, $\gamma'$ (which are contained in $B_{R_1}(x_i)$ and $B_{R_1}(x_{i'})$ and open in $K_{i,j}^{R_2}$ and $K_{i',j'}^{R_2}$ respectively) must also be open in $K_{n(i,j)}$. Since $y_{n(i,j)}^+$ does \emph{not} belong to $K_{n(i,j)}$ on the event $\sF_{i,j}$, the path $\Gamma'$ must \emph{not} be open in $K_{n(i,j)}$, which means that either
\begin{enumerate}
\item[(a)] There exists $v\in \Gamma'$ such that $v\in K_m$ for some $m<n(i,j)$, or
\item[(b)] $\Gamma'$ visits the set $E_W$.
\end{enumerate}
We bound from above the conditional probability  that (a) occurs given $y_n^+\in \tilde K_n$ for every $1\leq n\leq N$  using the tree-graph inequality:
\[ \sum_{1\leq a,a'\leq k} \sum_{u\in \Z^d} \sum_{\substack{v\in \Z^d\\\|v-x_i\|\geq r\\\|v-x_{i'}\|\geq r}} \frac{\proba(x_a\slr u)\proba(u\slr x_{a'})}{\proba(x_a\slr x_{a'})}\proba(u\slr v)\frac{\proba(x_i \slr v)\proba(v\slr x_{i'})}{\proba(x_i \slr x_{i'})}.\]
 Using the universal bound \eqref{eq:ConvenientSplitFrac} we may bound this sum from above by 
\[ C\sum_{1\leq a,a'\leq k} \sum_{u\in \Z^d} \sum_{\substack{v\in \Z^d\\\|v-x_i\|\geq r\\\|v-x_{i'}\|\geq r}} (\proba(x_a\slr u)+\proba(u\slr x_{a'}))\proba(u\slr v)(\proba(x_i \slr v)+\proba(v\slr x_{i'})). \]
(Note that this sum includes terms with $x_a$ or $x_{a'}$ possibly coinciding with $x_i$ or $x_{i'}$, which is why we need to take $r$ large to make it small.)
By \eqref{eq:two_point_assumption} and the assumption that $d>6$, the convolutional cube of the two-point function is finite, which implies that this sum converges to zero uniformly in $(x_1,\ldots,x_k)$ as $r\to \infty$ and we deduce that
\begin{multline} \limsup_{r\to\infty}\limsup_{R_{1}\to\infty} \limsup_{R_{2}\to\infty} \limsup_{\min\|x_a-x_b\|\to \infty}\\\proba\left (\sG(R_1,R_2)\cap \sF_{i,j}\cap \{ M\geq r\}\cap \{\text{(a) occurs}\} \; \middle | \; \forall n, y_n^+\in \tilde K_n \right )= 0.\label{eq:BlobOffStopFarA} \end{multline} 
For the case (b), our above assumption on $r$ ensures that the vertex $v$ 
must belong to $x_a+W_a$ for some $a \notin \{i,i'\}$. 
Since the path $\Gamma'$ is contained in the path $\Gamma$ which connected $x_i+b_{i,j}$ to $x_{i'} + b_{i',j'}$ in $\tilde K_{n(i,j)}$, we can apply the BK inequality and a union bound 
to bound the conditional probability of this event given $y_n^+\in \tilde K_n$ for every $1\leq n \leq N$ by something of the form
\[ C \sum_{a\notin\{i,i'\}} \frac{\proba(x_i\slr x_a)\proba(x_{i'}\slr x_a)}{\proba(x_i\slr x_{i'})} \]
with the constant $C$ depending only on the scheme $S$. Since this goes to zero 
as $\min_{1\leq a\neq a'\leq k} \|x_a-x_{a'}\|\to\infty$, we deduce that
\begin{multline} \limsup_{R_{1}\to\infty} \limsup_{R_{2}\to\infty} \limsup_{\min\|x_a-x_b\|\to \infty}\\\proba\left (\sG(R_1,R_2)\cap \sF_{i,j}\cap \{ M\geq r\}\cap \{\text{(b) occurs}\} \; \middle | \; \forall n, y_n^+\in \tilde K_n \right )= 0.\label{eq:BlobOffStopFarB} \end{multline}
for each sufficiently large $r$. Putting together \eqref{eq:BlobOffStopCloseA}, \eqref{eq:BlobOffStopFarA}, and \eqref{eq:BlobOffStopFarB} completes the proof that the event $\sG(R_1,R_2) \cap \sF_{i,j}$ holds $(S,\prec)$-rarely and hence the proof of the proposition. \qedhere

\end{proof} 
We now prove a slightly weaker version \cref{thm:scheme_function}, from which we will deduce \cref{thm:scheme_function}. To lighten notation we write 
\[ \Pi^{R_1,R_2}(S,\prec):=\prod_{1\leq i \leq k} \widehat{\mathbf{W}}^{R_1,R_2}(B_i,\prec_i) \quad \text{ and } \quad G_S(x_1,x_2,\dots,x_k):=\prod_{((i,n),(j,m)) \in E^{1/2}(S)} G(x_i,x_j).\]

\begin{lem} \label{lem:weaker_scheme_function} The double limit $\Pi(S)=\lim_{R_1\to \infty}\lim_{R_2\to \infty}\Pi^{R_1,R_2}(S,\prec)$ exists independently of the choice of order $\prec$ and the scheme function satisfies the asymptotic formula
\[ S(x_1,x_2,\dots,x_k)=(\Pi(S)+o(1))(1-p)^{|W_1|+|W_2|+\dots+|W_k|} A^{|E(S)|} G_S(x_1,\dots,x_k) \]
as $\min \|x_i-x_j\|\to \infty$. 
\end{lem}

The only difference between this lemma and \cref{thm:scheme_function} is that we have not yet shown $\Pi(S)$ can be written in the form $\prod_{1\leq i \leq k} \widehat{\mathbf{W}}(B_i)$.

\begin{proof}[Proof of \cref{lem:weaker_scheme_function}]
To lighten notation let us write
\[
  \overline{\lim} \qquad \text{ to denote } \qquad \limsup_{R_1\to \infty}\limsup_{R_2\to \infty} \limsup_{\min_{i\neq j}\|x_i-x_j\|\to \infty}.
\]
In this notation, \cref{prop:BlobOffApprox} states that
\[ 
\overline{\lim} \left | \proba \left (\forall n, y_n^+\in K_n \;\middle |\;\forall n,y_n^+\in \tilde K_n\right ) 
- \proba \left (\forall i,j,  K^{R_2}_{i,j} \cap \partial B_{R_1}(x_i) \neq \emptyset \;\middle |\;\forall n,y_n^+\in \tilde K_n \right) \right |= 0 \label{eq:RecallBlobA} \]
while \cref{lem:BlobOffApproxCV} states that
\[
\overline{\lim}\left | \proba \left (\forall i,j,  K^{R_2}_{i,j} \cap \partial B_{R_1}(x_i) \neq \emptyset \;\middle |\;\forall n,y_n^+\in \tilde K_n \right ) 
- \Pi^{R_1,R_2}(S,\prec) \right | = 0. 
\]
The assumption \eqref{eq:two_point_assumption} ensures that $\proba(\forall n,y_n^+\in \tilde K_n)\sim A^{|E(S)|} G_S(x_1,\dots,x_k)$ as $\min_{i\neq j}\|x_i-x_j\|\to \infty$, and hence by \eqref{eq:InductiveAizenmanOff} that
\[ \overline{\lim}\left |\frac{S(x_1,x_2,\dots,x_k)}{A^{|E(S)|}G_S(x_1,\dots,x_k)}- \proba(E_W \text{ is closed})\proba \left (\forall n, y_n^+\in K_n \;\middle |\; \forall n,y_n^+\in \tilde K_n\right )  \right | =0. \label{eq:RecallBlobB} \]
Putting these estimates together, we have by the triangle inequality that
\begin{equation} \overline{\lim}\left |\frac{S(x_1,x_2,\dots,x_k)}{A^{|E(S)|}G_S(x_1,\dots,x_k)} - \proba(E_W \text{ is closed}) \Pi^{R_1,R_2}(S,\prec) \right |= 0. 
\label{eq:PiR1R2}
\end{equation}
Since the ratio $S(x_1,x_2,\dots,x_k)/(A^{|E(S)|}G_S(x_1,\dots,x_k))$ does not depend on the parameters $R_1$ and $R_2$ or the order $\prec$, it follows that $(\Pi^{R_1,R_2}(S,\prec))_{R_1,R_2}$ satisfies a Cauchy criterion in the double limit as we first send $R_2\to \infty$ and then send $R_1\to \infty$, and therefore converges toward some limit $\Pi(S)$ in the same double limit which does not depend on the order $\prec$. This is easily seen to imply the claim in conjunction with  \eqref{eq:PiR1R2}.
\end{proof} 

We are now ready to complete the proof of \cref{thm:scheme_function}, modulo the topological characterisation of the positivity of the blob factor which is deferred to \Cref{sec:positivity_of_local_factors}.

\begin{proof}[Proof of \cref{thm:scheme_function}] It remains to show that we may write $\Pi(S)$ in the form $\prod_{1\leq i \leq k} \widehat{\mathbf{W}}(B_i)$ and that $\mathbf{W}(B):=(1-p)^{|W|}\widehat{\mathbf{W}}(B)>0$ for a blob $B=(W,b_1,\ldots,b_k)$ if and only if there exist vertex-disjoint paths from $b_1,\ldots,b_k$ to $\infty$ in $\Z^d_L\setminus W$. The proof of the second part is deferred to \cref{lem:blob_factor_positivity}.
For the first part,
it suffices by \eqref{eq:PiR1R2} to show that the limit
\[
  \widehat{\mathbf{W}}(B) = \lim_{R_1\to \infty}\lim_{R_2\to \infty} \widehat{\mathbf{W}}^{R_1,R_2}(B,\prec)
\]
exists and does not depend on the order $\prec$ for each blob $B$. (The fact that this limit exists independently of the choice of order immediately implies that it is invariant under permuting the indices of the vertices in the blob.)

We first show this is true for the
 trivial blob  $B_\emptyset=(\emptyset,0)$ with $\widehat{\mathbf{W}}(B_\emptyset)=1$. Let $\prec_\emptyset$ be the trivial order on $\{1\}$, let $(H_1,\ell_1)$ be the plan with two vertices and one edge, and consider the scheme $S_2=(H_1,\ell_1, B^1)$ where $B^1$ maps $1$ and $2$ to the trivial blob $B_\emptyset$, so that $S_2(x_1,x_2)=T_{p_c}(x_1,x_2)$ is the usual two-point function.  Let $\prec_2$ be an arbitrary order for $S_2$. By \cref{lem:weaker_scheme_function}, we have that the limit 
\[
  \Pi(S_2) = \lim_{R_1\to \infty} \lim_{R_2\to \infty} \Pi^{R_1,R_2}(S_2,\prec_2) = \lim_{R_1\to \infty} \lim_{R_2\to \infty} \widehat{\mathbf{W}}^{R_1,R_2}(B_\emptyset,\prec_\emptyset)^2
\]
exists. Since $S_2(x_1,x_2)=T_{p_c}(x_1,x_2)\sim AG(x_1,x_2)$ by \eqref{eq:two_point_assumption} we must have $\Pi(S_2) =1$ and hence that
\begin{equation}
\label{eq:trivial_blob_factor}
  \lim_{R_1\to \infty} \lim_{R_2\to \infty} \widehat{\mathbf{W}}^{R_1,R_2}(B_\emptyset,\prec_\emptyset)=1.
\end{equation}
In particular, the limit exists as required. (Here independence of order is automatic since there is only one order on the trivial blob.)
%

Now fix $k\geq 1$ and a $k$-blob $B$ and let $\prec_B$ be an order on $B$. Let $H_k$ be the graph with a single vertex $k+1$ of degree $k$ and leaves $\{1,\dots,k\}$, let $\ell_k$ be any corresponding labelling function, and consider the scheme $S_B$ defined by mapping each of the vertices $1,\ldots,k$ to the trivial blob and the vertex $k+1$ to the blob $B$.
If we fix an order $\prec$ on $S_B$ such that the induced order $\prec_{k+1}$ is equal to $\prec_B$, it follows from \cref{lem:weaker_scheme_function} that the limit
\[
 \Pi(S_B)= \lim_{R_1\to \infty}\lim_{R_2\to\infty} \Pi^{R_1,R_2} (S_B,\prec) = \lim_{R_1\to \infty}\lim_{R_2\to\infty} (\widehat{\mathbf{W}}^{R_1,R_2}(B_\emptyset,\prec_\emptyset))^k \widehat{\mathbf{W}}^{R_1,R_2}(B,\prec_B)
\]
is well-defined independently of $\prec$. It follows from this together with \eqref{eq:trivial_blob_factor} that 
\[
 \widehat{\mathbf{W}}(B):=\lim_{R_1\to \infty}\lim_{R_2\to\infty}  \widehat{\mathbf{W}}^{R_1,R_2}(B,\prec_B) =\Pi(S_B)
\]
is also defined independently of the choice of $\prec_B$ as required.
\end{proof}

\section{$k$-point function asymptotics} \label{sec:k_point_proof}

In this section we apply the scheme function factorization theorem (\cref{thm:scheme_function}) to prove our main theorem for $k$-point function asymptotics, \cref{thm:k_point}.

\medskip

Before proceeding further, let us set up some relevant notation.
Let $k\geq 3$ and suppose $x_1,\ldots,x_k$ are distinct points in $\Z^d$. If $x_1,\ldots, x_k$ all belong to the same percolation cluster in $\Z^d_L$, there must exist (not necessarily distinct) points $x_{k+1},\ldots, x_{2k-2}$ in $\Z^d$, a tree $\cT \in \mathsf{Tr}(k)$, and a collection of edge-disjoint open paths $(\gamma_e)_{e\in \cT}$ in $\Z^d_L$ such that $\gamma_{i,j}$ connects $x_i$ and $x_j$ for each $\{i,j\}\in E(\cT)$. We denote the event that this holds for a particular $\cT\in \mathsf{Tr}(k)$ and collection of points $x_{k+1},\ldots, x_{2k-2}$ by 
\[\Psi_\cT(x_1,\ldots,x_{2k-2})=\bigcirc_{i\sim j} \{x_i \leftrightarrow x_j\},\]
where the adjacency relation $i\sim j$ is in the tree $\cT$ and where $\bigcirc$ denotes the disjoint occurrence of all the relevant connection events. (Recall that for notational convenience we take each isomorphism-class-representative $\cT\in\mathsf{Tr}(k)$ to have vertex set $\{1,\ldots,2k-2\}$, with leaves $\{1,\ldots,k\}$ labelled by the prescribed external points and internal vertices $\{k+1,\ldots,2k-2\}$ given fixed auxiliary labels which are not relevant for determining the isomorphism class.) Note that the event $\Psi_\cT(x_1,\ldots,x_m)$ is defined for any $m\geq 1$, tree with vertex set $\{1,\ldots,m\}$, and points $x_1,\ldots,x_m \in \Z^d$. 

\medskip

Recall that the tree-graph inequality \eqref{eq:tree_graph_intro} is proven by using the BK inequality to bound the probability
$\P_p(\Psi_\cT(x_1,\ldots,x_{2k-2}))$ by $\prod_{i\sim j} \P_p(x_i\leftrightarrow x_j)$
before summing over the possible choices of $\cT$ and $x_{k+1},\ldots, x_{2k-2}$ in a union bound. 
We wish to understand the precise manner in which this inequality is wasteful.
%
%
If we define the \textbf{overcount factor}
\[
  N(x_1,\ldots,x_k) = \# \{(\cT,x_{k+1},\ldots,x_{2k-2}): \cT \in \mathsf{Tr}(k), x_{k+1},\ldots,x_{2k-2} \in \Z^d,\; \text{$\Psi_\cT(x_1,\ldots,x_{2k-2})$ holds}\}
\]
then we can write
\[ \1(x_1,\ldots,x_k \text{ all connected}) = \sum_{\cT\in \mathsf{Tr}(k)} \sum_{x_{k+1},\dots,x_{2k-2}} \frac{\1(\Psi_\cT(x_1,\dots,x_{2k-2}))}{N(x_1,\dots, x_k)},\]
where we define the ratio appearing here to be zero when $\Psi_\cT(x_1,\dots,x_{2k-2})$ does not hold, 
so that the $k$-point function can be expressed as
\begin{equation} T_{p_c}(x_1,\ldots,x_k)=\sum_{\cT\in \mathsf{Tr}(k)} \sum_{x_{k+1},\dots,x_{2k-2}} \E\left [\frac{\1(\Psi_\cT(x_1,\dots,x_{2k-2}))}{N(x_1,\dots, x_k)}\right ]. \label{eq:ExpOvercountBranch}\end{equation}
We will use this formula as the basis of our analysis of the $k$-point function.

%

\medskip

The rest of this section is summarised as follows:
\begin{itemize}
  \item \textbf{Asymptotics for tree-shaped edge-disjoint connections.} In \cref{Sec:TreeScheme} we prove an asymptotic formula for the probability that a fixed collection of points $x_1,\ldots,x_k$ are connected by \emph{edge-disjoint} open paths with the combinatorics of a prescribed tree $\cT$ (\cref{thm:scheme_function_variant}), together with a stronger theorem also letting us control the local geometry around each point (\cref{thm:scheme_function_variantW}). This theorem is a close relative of the scheme function factorization theorem (\cref{thm:scheme_function}), the main difference being that the connections are now required only to be edge-disjoint rather than to belong to distinct clusters. 
   It follows from this theorem that the \emph{monochromatic $k$-arm IIC measures} are well-defined for every $k\geq 1$.

  \item \textbf{Localization of the overcount factor.}
   In \cref{subsec:localization_of_overcount_factors}, we prove (\cref{pro:CountBranchT}) that, on the event $\Psi_\cT(x_1,\ldots,x_{2k-2})$ and up to negligible errors, the branching trees for $(x_1,\ldots,x_k)$ are precisely those obtained by keeping the same combinatorial tree $\cT$ and moving each branch point $x_{k+1},\ldots,x_{2k-2}$ within the set of nearby vertices that have three long disjoint arms emanating from them. As a consequence, the overcount factor $N(x_1,\ldots,x_k)$ also localizes in the sense that it is well-approximated by a product of local quantities, one for each branch point.  The main input here is that we are negligibly unlikely to have additional arms around the branch points that are not forced by the point being a branch point. 


  \item \textbf{Negligibility of degenerate branching trees.}
  The analysis of \cref{subsec:localization_of_overcount_factors} is carried out only in the case that all distances between the points $x_1,\ldots,x_{2k-2}$ are large, whereas for applications to \cref{thm:k_point} we can assume this only for $x_1,\ldots,x_k$. In \cref{subsec:negligibility_of_degenerate_branching_trees} we prove that the contribution to \eqref{eq:ExpOvercountBranch} from \emph{degenerate} branching trees in which two of the points $x_1,\ldots,x_{2k-2}$ are at bounded distance is negligible (\cref{cor:DegenerateTreeBK}). 

  \item \textbf{Proof of \cref{thm:k_point}.}
  Finally,  in \cref{subsec:proof_of_the_k_point_function} we put the previous steps together to conclude the proof of \cref{thm:k_point}. As part of this, we prove a lower bound on the three-point function of the correct order (\cref{lem:three_point_lower}) to establish that the vertex factor is positive.
\end{itemize}

\subsection{Asymptotics of branching tree probabilities} \label{Sec:TreeScheme}

The goal of this section is to prove the following asymptotic formula for edge-disjoint connections with the combinatorics of a given tree. Note that the tree in this theorem is \emph{not} required to have binary branching, and the set $\{1,\ldots,k\}$ is now the entire vertex set of the tree rather than just the set of leaves.

\begin{theorem}
\label{thm:scheme_function_variant}
Let $d>6$ and $L\geq 1$ be such that critical percolation on $\mathbb{Z}^d_L$ satisfies the estimates \eqref{eq:two_point_assumption} and \eqref{eq:two_blob_assumption}. There exists a function $\mathbf{C}$ from the positive integers to $[0,\infty)$ such that 
\[
\P_{p_c}\bigl(\Psi_\cT(x_1,x_2,\dots, x_k)\bigr) :=\P_{p_c}\Bigl(\bigcirc_{i\sim j} \{x_i \leftrightarrow x_j\}\Bigr) \sim \prod_{i=1}^k \mathbf{C}(\deg(i)) \prod_{i\sim j} A G(x_i,x_j) 
\]
as $\min_{i\neq j}\|x_i-x_j\|\to\infty$ for each tree $\cT$ with vertex set $\{1,\ldots,k\}$. 
Moreover, the constant $\mathbf{C}(m)$ is positive if and only if there exist $m$ edge-disjoint paths from $0$ to $\infty$ in $\Z^d_L$.
\end{theorem}


\medskip

For applications later in the paper, it will be useful to prove a more general version of the theorem in which we also specify the microscopic geometry around the points $x_1,\ldots,x_k$. To this end, given a tree $\cT$ with vertex set $\{1,\ldots,k\}$ and a collection $W=(W_i)_{i=1}^k$ of finite sets of edges in $\Z^d_L$, we write $\Psi_{\cT,W}(x_1,\ldots,x_k)$ for the event that $\Psi_{\cT}(x_1,\ldots,x_k)$ holds and that $x_i+W_i$ is closed for every $1\leq i \leq k$. The following theorem generalizes \cref{thm:scheme_function_variant}.

\begin{thm} \label{thm:scheme_function_variantW} Let $d>6$ and $L\geq 1$ be such that critical percolation on $\mathbb{Z}^d_L$ satisfies the estimates \eqref{eq:two_point_assumption} and \eqref{eq:two_blob_assumption}. There exists a function $\mathbf{C}$ from pairs $(m,W)$ of positive integers and finite sets of edges in $\Z^d_L$ to $[0,\infty)$ such that 
\[
\P_{p_c}\bigl(\Psi_{\cT,W}(x_1,x_2,\dots, x_k)\bigr) \sim \prod_{i=1}^k \mathbf{C}(\deg(i),W_i) \prod_{i\sim j} A G(x_i,x_j) 
\]
as $\min_{i\neq j}\|x_i-x_j\|\to\infty$ for each tree $\cT$ with vertex set $\{1,\ldots,k\}$ and collection $(W_i)_{i=1}^k$ of finite sets of edges in $\Z^d_L$.  Moreover, the constant $\mathbf{C}(m,W)$ is positive if and only if there exist $m$ edge-disjoint paths from $0$ to $\infty$ in $\Z^d_L\setminus W$.
\end{thm} 

As usual, the asymptotic estimates of \cref{thm:scheme_function_variant,thm:scheme_function_variantW} are valid even when the right-hand side is zero, in which case the left-hand side is also identically zero when $\min_{i\neq j}\|x_i-x_j\|$ is sufficiently large.
We interpret this proposition as establishing two things in addition to the asymptotics of \cref{thm:scheme_function_variant}:
\begin{enumerate}
  \item The existence of a \textbf{monochromatic\footnote{We use this terminology by analogy with the planar case.} $k$-arm IIC measure} describing the local limit of the percolation configuration at the origin when the origin is conditioned to be connected by edge-disjoint paths to $k$ distant vertices that have a large distance from the origin and from each other, provided that $k$ is not so large that such connections are impossible. The probability that a finite set of edges $W$ is closed in this limit measure is precisely $\mathbf{C}(k,W)/\mathbf{C}(k)$, with the probabilities of all other cylinder events determined by inclusion-exclusion. We remark that the monochromatic $2$-arm IIC measure was constructed in the independent recent work \cite{cabezas2025bi} using different techniques (relying more heavily on the lace expansion).
  \item When we condition on the points $x_1,\ldots,x_k$ to be connected via edge-disjoint paths according to the combinatorial tree $\cT$, in the limit as the points become well-separated we see \emph{independent} copies of the monochromatic $\deg(i)$-arm IIC around each point $x_i$.
\end{enumerate}

Let us start by introducing some relevant notation.
Fix a tree $\cT$ with vertex set $\{1,\ldots,k\}$ and write
\[ G_\cT(x_1,x_2,\dots, x_k):= \prod_{i\sim j} G(x_i,x_j),\]
so that $\P(\Psi_{\cT}(x_1,\ldots,x_k)) \leq (1+o(1)) A^{k-1} G_\cT(x_1,x_2,\dots, x_k)$ as $\min_{i\neq j}\|x_i-x_j\|\to \infty$ by \eqref{eq:two_point_assumption} and the BK inequality.
We say that an event $E(R_1,R_2,x_1,\dots, x_k)$ with parameters $R_1,R_2,$ and $x_1,\ldots,x_k$ is $\Psi_{\cT}$\textbf{-usual} if 
\[
\lim_{R_1\to \infty}\lim_{R_2\to \infty}\lim_{\min_{i\neq j}\|x_i-x_j\|\to\infty} \frac{\proba_{p_c}(\Psi_{\cT}(x_1,x_2,\dots, x_k)\backslash E(R_1,R_2,x_1,\dots, x_k))}{G_\cT(x_1,x_2,\dots,x_k)}=0.
\]
We might intuitively think that an event is $\Psi_\cT$-usual if and only if it holds with high probability conditional on $\Psi_\cT(x_1,x_2,\dots, x_k)$, but we cannot yet conclude this is the case as we have not proven $\proba(\Psi_\cT(x_1,x_2,\dots, x_k))$ is of the same order as $G_\cT(x_1,\ldots,x_k)$. (Indeed, this is not even true when $\cT$ contains vertices of high degree and $\Psi_{\cT}(x_1,\ldots,x_k)$ has probability zero.) 

\medskip


The first step towards \cref{thm:scheme_function_variantW} is the following lemma.

\begin{lemma}[No unforced connections] \label{lem:PsiArms} We have $\Psi_{\cT}$-usually: 
\begin{itemize}
\item[(a)] Every cluster in $\Z^d\backslash \bigcup_{1\leq i \leq k} B_{R_1}(x_i)$  touches at most two of the balls $(B_{R_2}(x_i))_{1\leq i \leq k}$.
\item[(b)] For each pair of adjacent vertices $i$ and $j$ in $\cT$ there is exactly one cluster in $\Z^d\backslash \bigcup_{1\leq i \leq k} B_{R_1}(x_i)$ that touches both $B_{R_1}(x_i)$ and $B_{R_1}(x_j)$. Moreover,  no cluster in $\Z^d\backslash \bigcup_{1\leq i \leq k} B_{R_1}(x_i)$ touches both $B_{R_1}(x_i)$ and $B_{R_1}(x_j)$ when $i$ and $j$ are not adjacent in $\cT$.
\item[(c)] For each $1\leq i \leq k$ there are no more than $\deg_{\cT}(i)$ edge-disjoint open paths from $\partial B_{R_1}(x_i)$ to $\partial B_{R_2}(x_i)$.
\end{itemize}
\end{lemma}

The proof of this lemma will rely on the following abstract path rerouting lemma. We say that a path is \textbf{edge-simple} if it does not cross the same edge twice.

\begin{lemma}[Rerouting around spare arms]
\label{lem:rerouting_spare_arm}
Let $G$ be a graph. Let $m\geq n\geq  1$ and suppose that 
$\Gamma_1,\ldots,\Gamma_{m}$ 
and $\gamma_1,\ldots,\gamma_n$ are two collections of oriented finite,  edge-simple paths such that the paths $\gamma_1,\ldots,\gamma_n$ are pairwise edge-disjoint and the paths $\Gamma_1,\ldots,\Gamma_{m}$ are pairwise edge-disjoint.
There exists a subset $A$ of $\{1,\ldots,n\}$, a  bijection $\pi$ from $A$ to a subset $B$ of $\{1,\ldots,m\}$, and a collection of pairwise edge-disjoint edge-simple paths
$\gamma'_1,\ldots,\gamma'_n$ such
that:
\begin{enumerate}
  \item If $i\notin A$ then $\gamma_i'=\gamma_i$.
\item If $i\in A$ then $\gamma'_i$ is equal to the concatenation of some initial segment of $\gamma_i$ with some terminal segment of $\Gamma_{\pi(i)}$.
\item If $j\notin B$ then $\Gamma_j$ is edge-disjoint from all of the paths $\{\gamma_i':1\leq i \leq n\}$.
\end{enumerate}
In particular, if $m>n$ then there exists $1\leq j_0 \leq m$ such that $\Gamma_{j_0}$ is edge-disjoint from all of the paths $\gamma_1',\ldots,\gamma_n'$. 
\end{lemma}

\begin{proof}[Proof of \cref{lem:rerouting_spare_arm}]
We will use an algorithmic procedure in which at each step we have some paths $\gamma_1',\ldots,\gamma'_n$, $\Gamma'_1,\ldots,\Gamma'_{m}$, and an \textbf{assignment} mapping some subcollection of the paths $\gamma_i'$ bijectively to some subcollection of the paths $\Gamma'_j$ (and with some paths being unassigned.) We will keep track of which paths $\gamma_1',\ldots,\gamma_i'$ are \textbf{active} at each step.
Start with $\gamma_i'=\gamma_i$ active and unassigned for each $1\leq i \leq n$ and $\Gamma'_i=\Gamma_i$ for every $1\leq i \leq m$. At each step of the algorithm, the following conditions will hold:
\begin{itemize}
 \item $\Gamma_i'$ is a terminal segment of $\Gamma_i$ for every $1\leq i \leq m$. ($\Gamma_i'$ may have length zero, meaning that it is just the terminal endpoint of $\Gamma_i$.)
 \item Every active path $\gamma_i'$ is unassigned and equal to $\gamma_i$. 
 \item If the path $\gamma_i'$ is assigned to the path $\Gamma_j'$ then $\gamma_i'$ is the concatenation of an initial segment of $\gamma_i$ with the path $\Gamma_j'$. (These segments may consist of a single vertex and no edges.) 
 \item If the path $\gamma_i'$ is both inactive and unassigned, it is equal to $\gamma_i$ and is edge-disjoint from the all of the paths $\Gamma_j'$.
 \item The inactive paths among $\gamma'_1,\ldots,\gamma'_n$ are pairwise edge-disjoint.
\end{itemize}
At step $k\geq 1$ of the algorithm, inspect the current value of $\gamma'_{i}$ for $i=k \mod n$, where we set $0 \mod n =  n \mod n = n$ for notational convenience. If $\gamma'_i$ is inactive, do nothing. Otherwise, $\gamma_i'=\gamma_i$. In this case, find the first edge crossed by $\gamma_i$ that is also crossed 
by one of the paths $\Gamma_j'$. If no such edge exists, keep $\gamma'_i=\gamma_i$ and declare $\gamma_i'$ to be inactive. If $\gamma_i$ first meets the set of edges crossed by the union of the paths at $\Gamma_j'$ at some edge $e$, let $1\leq j \leq m$ be the value of $j$ for which $\Gamma_j'$ passes through $e$ (which is unique since the $\Gamma_j$ are edge-disjoint).
We set the new value of $\gamma_i'$ to be the concatenation of the initial segment of $\gamma_i$ up until the first time it visits an endpoint of $e$, the terminal segment of $\Gamma_j'$ after it last visits an endpoint of $e$, putting the edge $e$ in the middle of this concatenation if the two paths cross $e$ in the same direction and omitting it otherwise. We also set $\gamma_i'$ to be assigned to $\Gamma_j'$ and redefine $\Gamma_j'$ to be the terminal segment of $\Gamma_j'$ after its last visit to an endpoint of $e$, and declare $\gamma_i'$ to be inactive. If another path $\gamma_{i'}'$ was assigned to $\Gamma_j'$ prior to this step, we set $\gamma_{i'}'=\gamma_{i'}$ and declare $\gamma_{i'}'$ to be unassigned and active.
We easily verify by induction that the output of each stage of the algorithm satisfies the conditions above. We terminate the algorithm if and when every path is inactive.

Note that every time we reassign a path in this algorithm, the length of the associated path $\Gamma_j'$ decreases by at least one. As such, at most finitely many such reassignment steps can occur and the algorithm must eventually terminate with all paths inactive. Moreover, once the index $1\leq j \leq m$ is such that $\Gamma_j'$ becomes assigned to a path $\gamma_i'$, the path $\Gamma_j'$ remains assigned to some path $\gamma_{i'}'$ at all future steps. 
If we define the sets $A$ and $B$ to be the sets of indices $1\leq i \leq n$ and $1\leq j \leq m$ that are assigned to each other when the algorithm terminates and define $\pi$ to be this assignment, 
the properties of the algorithm discussed above ensure that this data and the paths $\gamma_1',\ldots,\gamma_n'$ have the desired properties.
\end{proof}

\begin{proof}[Proof of \cref{lem:PsiArms}] To lighten notation we write
$\Psi_\cT=\Psi_\cT(x_1,\ldots,x_k)$ and
 $\cB_r:=\bigcup_{1\leq i \leq k} B_r(x_i)$ for each $r>0$. We will write $C$ for a constant that does not depend on the parameters $R_1,R_2,x_1,\ldots,x_k$, the exact value of which may vary from line to line. For constants depending on some but not all of these parameters we will write e.g. $C(R_1)$ or $C(R_1,R_2)$ as appropriate.

\medskip

\noindent
\textbf{Proof of (a).} 
Suppose that $\Psi_\cT$ holds and let $(\gamma_{i,j})_{\{i,j\}\in E}$ be  an associated collection of edge-disjoint paths from $x_i$ to $x_j$ indexed by the edges of $\cT$, taking $\gamma_{j,i}$ to be the reversal of $\gamma_{i,j}$. It suffices to prove that the three events
\begin{align*}
E_1 &= \Bigl\{ \forall \{i,j\}, B_{R_2}(x_i) \text{ is not connected to } B_{R_2}(x_j) \text{ disjointly from $(\gamma_{i,j})_{\{i,j\}\in E}$}\Bigr\},\\
E_2 &=\Bigl\{ \forall \{i_1,j_1\}\neq\{i_2,j_2\}  \text{ the paths } \gamma_{i_1,j_1} \text{ and } \gamma_{i_2,j_2} \text{ are not connected in }\Z^d\backslash \cB_{R_1}\Bigr\},
\intertext{and}
E_3&=\Bigl\{\forall \{i_1,j_1\}\in E, i_2\notin \{i_1,j_1\},  \text{ the path } \gamma_{i_1,j_1} \text{ is not connected to } B_{R_2}(x_{i_2}) \text{ in }\Z^d\backslash \cB_{R_1}\Bigr\} 
\end{align*}
are all $\Psi_\cT$-usual, the event we are interested in being contained in the intersection $E_1\cap E_2 \cap E_3$. (Note that to make the event $E_1$ well-defined we should choose in advance a measurable way of choosing the paths $(\gamma_{i,j})_{\{i,j\}\in E}$ on the event $\Psi_{\cT}$. The precise way that we do this is not important.) The fact that $E_1$ holds $\Psi_{\cT}$-usually is a trivial consequence of the BK inequality, which together with \eqref{eq:two_point_assumption} yields that
\[
  \frac{\P(\Psi_{\cT}\setminus E_1)}{G_\cT(x_1,\ldots,x_k)} \leq C(R_2) \sum_{i,j} \P(x_i\leftrightarrow x_j),
\]
which converges to zero as $\min_{i\neq j} \|x_i-x_j\|\to \infty$ with $R_1$ and $R_2$ fixed.

We now prove that $E_2$ is $\Psi_{\cT}$-usual.
If $E_2$ does \emph{not} hold then there exists $v_1\in  \gamma_{i_1,j_1}\backslash \cB_{R_1}$ and $v_2\in \gamma_{i_2,j_2}\backslash \cB_{R_1}$ such that disjointly from $(\gamma_{i,j})_{\{i,j\}\in E}$ we have $v_1\slr v_2$. We can upper bound the probability that such a path exists for each fixed $(i_1,j_1)$ and $(i_2,j_2)$ using the BK inequality and the two-point estimate \eqref{eq:two_point_assumption} to obtain a bound of the form
\[ CG_\cT(x_1,\dots,x_k) \sum_{v_1,v_2\notin \cB_{R_1}} \frac{\langle x_{i_1}-v_1\rangle^{-d+2}\langle v_1-x_{j_1}\rangle^{-d+2}}{\langle x_{i_1}-x_{j_1}\rangle^{-d+2}}\frac{\langle x_{i_2}-v_2\rangle^{-d+2}\langle v_2-x_{j_2}\rangle^{-d+2}}{\langle x_{i_2}-x_{j_2}\rangle^{-d+2}} \langle v_1-v_2\rangle^{-d+2}. \]
 As in the proof of \eqref{eq:ConvenientSplitFrac}, we have by the triangle inequality that $\max(\|x_{i_1}-v_1\|,\|v_1-x_{j_1}\|)\geq \|x_{i_1}-x_{j_1}\|/2$ and similarly for $(i_2,j_2), v_2$, so that we may bound the sum appearing here (without the $G_\cT$ term) by 
\[ C \sum_{v_1,v_2\notin \cB_{R_1}} (\langle x_{i_1}-v_1\rangle^{-d+2}+\langle v_1-x_{j_1}\rangle^{-d+2})(\langle x_{i_2}-v_2\rangle^{-d+2}+\langle v_2-x_{j_2}\rangle^{-d+2})\langle v_1-v_2\rangle^{-d+2}. \]
As in the proof of \cref{lem:BlobOffApproxA}, by distributing the products above and using that the open triangle condition holds we deduce that this sum converges uniformly to $0$ as $R_1\to \infty$, so that $E_2$ is $\Psi_\cT$-usual as claimed.

Finally we prove that $E_3$ is $\Psi_\cT$-usual.
Observe that if $E_3$ does not hold but $E_1$ does, then there exists an edge $\{i_1,j_1\}$ of $\cT$, a vertex $i_2 \notin \{i_1,j_1\}$ of $\cT$, and a point $v\in  \gamma_{i_1,j_1}$ such that $v$ is connected to $B_{R_2}(x_{i_2})$ disjointly from the paths $(\gamma_{i,j})_{\{i,j\}\in E}$. Using the BK inequality and \eqref{eq:two_point_assumption}, we may bound the probability this happens for fixed $i_1,j_1,i_2$ by 
\[ C(R_2)G_{\cT}(x_1,\dots,x_k) \sum_{v} \frac{\langle x_{i_1}-v\rangle^{-d+2}\langle v-x_{j_1}\rangle^{-d+2}}{\langle x_{i_1}-x_{j_1}\rangle^{-d+2}}\langle v-x_{i_2}\rangle^{-d+2}. \]
As before we may bound the sum appearing here (without the $C(R_2)G_{\cT}$ term) from above by 
\[ C  \sum_{v} (\langle x_{i_1}-v\rangle^{-d+2} + \langle v-x_{j_1}\rangle^{-d+2})\langle v-x_{i_2}\rangle^{-d+2}. \]
This last sum converges to $0$ as $\min \|x_i-x_j\|\to \infty$ with $R_1,R_2$ fixed, so that $E_3$ is $\Psi_\cT$-usual as claimed. This completes the proof of part (a).


\medskip

\noindent
\textbf{Proof of (b).}
Since the events $E_1$ and $E_2$ are both $\Psi_\cT$-usual, it follows that, $\Psi_\cT$-usually, every connected component of $\Z^d\backslash \cB_{R_1}$ intersects at most one of the paths $(\gamma_{i,j})_{\{i,j\}\in E}$ and every such component touching at least two of the balls $(B_{R_2}(x_i))_{1\leq i \leq k}$ intersects exactly one of these paths. Thus, by a simple bijective argument, it suffices to show that the event 
\begin{equation*} 
E_4= \Bigl\{
\forall i,j,  \gamma_{i,j} \text{ intersects at most one cluster of } \Z^d\backslash \cB_{R_1} \text{ touching two of the balls } (B_{R_2}(x_i))_{1\leq i \leq k}\Bigr\}
\end{equation*}
is $\Psi_\cT$-usual.
 If $E_4$ does not hold but $E_3$ does then
the following events must occur disjointly:
 \begin{itemize} \item disjointly for every $\{i,j\}\in E(\cT)$ we have $B_{R_2}(x_i)\slr B_{R_2}(x_j)$
\item for some $i\neq j$ we have $B_{R_2}(x_i)\slr B_{R_2}(x_j)$.
\end{itemize}
Thus, it follows by the BK inequality and \eqref{eq:two_point_assumption} 
that
 \[ 
 \frac{\P((\Psi_\cT \cap E_3)\setminus E_4)}{G_\cT(x_1,\ldots,x_k)} \leq C(R_2) 
  \sum_{i\neq j} \proba(x_i\slr x_j).
  \]
  Since this bound converges to zero as $\min \|x_i-x_j\|\to \infty$ with $R_1$ and $R_2$ fixed it follows that $E_4$ is $\Psi_\cT$-usual as claimed, concluding the proof of part (b).

\medskip

\noindent
\textbf{Proof of (c).} We now bound the probability that there is an additional arm event. Suppose that for some  vertex $u$ there exist $\deg_\cT(u)+1$ edge-disjoint open paths from $\partial B_{R_1}(x_u)$ to $\partial B_{R_2}(x_u)$. Write $q=\deg_\cT(u)$ and enumerate the neighbours of $u$  in $\cT$  as $v_1,\ldots,v_q$ . Since the event $E_3$ holds $\Psi_\cT$-usually, we may  restrict to  the case that  $\gamma_{i,j}$ does not intersect  $B_{R_2}(x_u)$  for every edge $\{i,j\}$ of $\cT$ not incident to  $u$.
 Fixing some collection of $q+1$ edge-disjoint open crossings of the annulus, we may apply \cref{lem:rerouting_spare_arm} to these crossings (oriented from $\partial B_{R_2}(x_u)$ toward $\partial B_{R_1}(x_u)$) and the $q$ paths of the form $\gamma_{v_s,u}$ (oriented from $x_{v_s}$ toward $x_u$) to obtain a modified collection of edge-disjoint paths from the vertices $x_{v_s}$  to  $\partial B_{R_1}(x_u)$ together with a further path from $\partial B_{R_2}(x_u)$ to $\partial B_{R_1}(x_u)$ that is edge-disjoint from all these modified paths.  Thus, on the event under consideration, a modified version of the event $\Psi_\cT$ in which the connections incident to $u$ end at $\partial B_{R_1}(x_u)$ occurs disjointly from an additional crossing of the annulus $B_{R_2}(x_u)\setminus B_{R_1}(x_u)$ whenever the $\Psi_\cT$-usual event $E_3$ holds.
 It  follows by the BK inequality and \eqref{eq:two_point_assumption} that the probability that $E_3$ holds and we have an extra arm is at most
 \[
C(R_1)\sum_{u=1}^k
 \P(\partial B_{R_1}(x_u)\leftrightarrow \partial B_{R_2}(x_u))
 G_\cT(x_1,\dots,x_k)
\]
for some constant $C(R_1)$. This is easily seen to imply the claim. \qedhere
\end{proof}

Toward proving \cref{thm:scheme_function_variant}, we next show the following result:

\begin{lemma} \label{lem:FindScheme}  $\Psi_\cT$-usually there exists a collection 
$\{b_{i,j}\}_{1\leq i,j \leq k: \{i,j\}\in E}$ of points in $\partial B_{R_1}(0)$ such that the following hold:
\begin{enumerate}
\item For each $\{i,j\}\in E$, the points $x_i+b_{i,j}$ and $x_j+b_{j,i}$ are connected by an open path whose internal vertices lie in $\Z^d\backslash \bigcup_{1\leq i \leq k} B_{R_1}(x_i)$.
\item  For each $\{i,j\}\neq \{i',j'\}$, the points $x_i+b_{i,j}$ and $x_{i'}+b_{i',j'}$ are \emph{not} connected by an open path whose internal vertices lie in $\Z^d\backslash \bigcup_{1\leq i \leq k} B_{R_1}(x_i)$.
\item For every $1\leq i \leq k$, there exists a collection of edge-disjoint open paths $(\Gamma_{i,j})_{j:\{i,j\}\in E}$ connecting $x_i$ to $\partial B_{R_2}(x_i)$ in $B_{R_2}(x_i)$ such that for each $j$ with $\{i,j\}\in E$, the last vertex of $\Gamma_{i,j}$ inside $B_{R_1}(x_i)$ is $x_i+b_{i,j}$.
\end{enumerate}
\end{lemma} 

\begin{proof}[Proof of \cref{lem:FindScheme}]
 We keep the notations of the last lemma and proof: Recall that for each $\{i,j\}\in E$, $\gamma_{i,j}$ is a path from $x_i$ to $x_j$ and that $\gamma_{j,i}$ is the reversal of $\gamma_{i,j}$. For each $\{i,j\}\in E$ let $x_i+b_{i,j}$ be the last vertex in $\gamma_{i,j}\cap B_{R_1}(x_i)$. The first item is direct from definition of the collection $(\gamma_{i,j})$, while the second item holds $\Psi_\cT$ usually by part (a) of \cref{lem:PsiArms}. 
Finally, for the last statement, it suffices to note that, $\Psi_\cT$-usually, the corresponding part of the path $\gamma_{i,j}$ stays inside the ball $B_{R_2}(x_i)$. If this part of $\gamma_{i,j}$ did \emph{not} stay in $B_{R_2}(x_i)$ then this would create an extra $(R_1,R_2)$ arm, which is prevented by part (c) of Lemma \ref{lem:PsiArms}. 
\end{proof}

We next show a converse to \cref{lem:FindScheme}, stating that the existence of the structure guaranteed by that lemma is \emph{equivalent} to the event $\Psi_\cT(x_1,\ldots,x_k)$ up to a negligible error term.

\begin{lemma} \label{lem:ConverseFindScheme} If we define $\Psi^*_\cT(x_1,\ldots,x_k;R_1,R_2)$ to be the event that there exists $\{b_{i,j}\}_{1\leq i,j \leq k: \{i,j\}\in E}$ as in \cref{lem:FindScheme}, then
\[
  \lim_{R_1\to \infty}\lim_{R_2\to \infty}\lim_{\min_{i\neq j}\|x_i-x_j\|\to\infty}\frac{\P(\Psi^*_\cT(x_1,\ldots,x_k;R_1,R_2) \setminus \Psi_\cT(x_1,\ldots,x_k))}{G_\cT(x_1,\ldots,x_k)} =0.
\]
\end{lemma}

\begin{proof}[Proof of \cref{lem:ConverseFindScheme}] By directly adapting the proof of Lemma \ref{lem:PsiArms}, we can prove that the probability that $\Psi^*_\cT(x_1,\ldots,x_k)$ holds but at least one of the events in Lemma \ref{lem:PsiArms} does not hold is $o(G_\cT(x_1,\dots,x_k))$ in the relevant triple limit; we omit the details as all arguments are identical.

We now  show  that if $\Psi_\cT^*(x_1,\ldots,x_k)$ holds and the events in Lemma \ref{lem:PsiArms} hold then  $\Psi_\cT(x_1,\dots,x_k)$  also holds. To this end, consider a collection of edge-disjoint paths $(\gamma_{i,j})_{\{i,j\}\in E}$ from $x_i+b_{i,j}$ to $x_j+b_{j,i}$ whose internal vertices lie in $\Z^d\backslash \bigcup_{1\leq i \leq k} B_{R_1}(x_i)$, where as usual we take $\gamma_{i,j}$ to be the reversal of $\gamma_{j,i}$.  For each $i$, apply \cref{lem:rerouting_spare_arm} inside $B_{R_2}(x_i)$ to the arms supplied by the definition of $\Psi_\cT^*$ and to the incident paths $(\gamma_{i,j})_{j:\{i,j\}\in E}$. Part  (c) of \cref{lem:PsiArms}  rules out the spare-arm alternative, so the incident paths can be rerouted edge-disjointly to end  at $x_i$ instead of $x_i+b_{i,j}$ . Performing this local modification independently for each $i$ yields edge-disjoint paths realizing  $\Psi_\cT(x_1,\ldots,x_k)$. \qedhere 
\end{proof}

We have one more technical lemma to prove before completing the proof of \cref{thm:scheme_function_variant,thm:scheme_function_variantW}. We say that a collection of vertices $\{b_{i,j}\}_{1\leq i,j \leq k: \{i,j\}\in E}$ \textbf{witnesses} $\Psi_\cT^*(x_1,\ldots,x_k;R_1,R_2)$ if the conclusions of \cref{lem:FindScheme} hold with this choice of $\{b_{i,j}\}_{1\leq i,j \leq k: \{i,j\}\in E}$.

\begin{lemma} \label{lem:DescribeScheme} The following holds $\Psi_\cT$-usually: 
If $\{b_{i,j}\}_{1\leq i,j \leq k: \{i,j\}\in E}$ is a collection of points witnessing $\Psi_\cT^*(x_1,\ldots,x_k;R_1,R_2)$, then a second collection of points $\{b'_{i,j}\}_{1\leq i,j \leq k: \{i,j\}\in E}$ also witnesses $\Psi_\cT^*(x_1,\ldots,x_k;R_1,R_2)$ if and only if $x_i+b_{i,j}$ is connected to $x_{i}+ b'_{i,j}$ by an open path whose internal vertices lie in $B_{R_2}(x_i)\backslash B_{R_1}(x_i)$ for every $\{i,j\} \in E$.
\end{lemma}

\begin{proof}[Proof of \cref{lem:DescribeScheme}] Fix two collections of vertices  $\{b_{i,j}\}_{1\leq i,j \leq k: \{i,j\}\in E}$ and $\{b'_{i,j}\}_{1\leq i,j \leq k: \{i,j\}\in E}$. If the collection $\{b_{i,j}\}_{1\leq i,j \leq k: \{i,j\}\in E}$ witnesses $\Psi_\cT^*(x_1,\ldots,x_k;R_1,R_2)$ and $x_i+b_{i,j}$ is connected to $ x_{i}+ b'_{i,j}$ by an open path whose internal vertices lie in $B_{R_2}(x_i)\backslash B_{R_1}(x_i)$ for every $\{i,j\}\in E$ then we trivially have that $\{b'_{i,j}\}_{1\leq i,j \leq k: \{i,j\}\in E}$ witnesses $\Psi_\cT^*(x_1,\ldots,x_k;R_1,R_2)$ also on the $\Psi_\cT$-usual event from \cref{lem:PsiArms}.
Conversely, 
suppose that the two collections $\{b_{i,j}\}_{1\leq i,j \leq k: \{i,j\}\in E}$ and $\{b'_{i,j}\}_{1\leq i,j \leq k: \{i,j\}\in E}$ both witness $\Psi_\cT^*(x_1,\ldots,x_k;R_1,R_2)$.
By part (b) of \cref{lem:PsiArms},  there is $\Psi_\cT$-usually exactly one connected  component $C_{i,j}$ of $\Z^d\backslash \bigcup_{1\leq i \leq k} B_{R_1}(x_i)$ touching both $B_{R_1}(x_i)$ and $B_{R_1}(x_j)$ for each $\{i,j\}\in E$, so that $x_i+b_{i,j}$ and $x_i+b'_{i,j}$ both touch $C_{i,j}$ for each $\{i,j\}\in E$. Moreover by parts (b) and (c) of \cref{lem:PsiArms}  there is, $\Psi_\cT$-usually, no two edge-disjoint paths from $\partial B_{R_1}(x_i)$ to $\partial B_{R_2}(x_i)$ using vertices of $C_{i,j}$ for each $\{i,j\}\in E$, so that every path through $C_{i,j}$ from $x_i+b_{i,j}$ to $x_i+b'_{i,j}$ has all its internal vertices in $B_{R_2}(x_i)$ for each $\{i,j\}\in E$. This concludes the proof.
\end{proof}

\begin{proof}[Proof of \cref{thm:scheme_function_variant} and \cref{thm:scheme_function_variantW}]
Fix $\cT$ and $W=(W_i)_{1\leq i \leq k}$ and let $\sW(x_1,\ldots,x_k)$
 be the event that the edges of $x_1+W_1,\dots, x_k+W_k$ are all closed, so that $\Psi_{\cT,W}(x_1,\ldots,x_k)=\Psi_{\cT}(x_1,\ldots,x_k)\cap \sW(x_1,\ldots,x_k)$. For each $R_1\geq 0$, fix an arbitrary order $\prec_{R_1}$ on the set $\partial B_{R_1}(0)$, and let 
\[ \Psi^{**}_\cT((x_i)_{1\leq i \leq k}, (b_{i,j})_{1\leq i,j \leq k: \{i,j\}\in E};R_1,R_2)\]
be the event that $(b_{i,j})_{1\leq i,j \leq k: \{i,j\}\in E}$ witnesses $\Psi_\cT^*(x_1,\ldots,x_k;R_1,R_2)$ and that for each $\{i,j\}\in E$ there do not exist $b\prec_{R_1} b_{i,j}\in \partial B_{R_1}(0)$ such that $x_i+b_{i,j}$ is connected to $x_i+b$ by an open path whose internal vertices lie in $B_{R_2}(x_i)\backslash B_{R_1}(x_i)$. By \cref{lem:DescribeScheme}, this is $\Psi_\cT$-usually equivalent to the event that $(b_{i,j})_{1\leq i,j \leq k: \{i,j\}\in E}$ is the \emph{minimal} witness for $\Psi_\cT(x_1,\ldots,x_k)$ with respect to the order $\prec_{R_1}$ in the natural sense, so that $\Psi_\cT$-usually there exists exactly one collection $(b_{i,j})_{1\leq i,j \leq k: \{i,j\}\in E}$ such that $\Psi^{**}_\cT((x_i)_{1\leq i \leq k}, (b_{i,j})_{1\leq i,j \leq k: \{i,j\}\in E};R_1,R_2)$ holds whenever $\Psi_\cT^*(x_1,\ldots,x_k;R_1,R_2)$ holds. Together with \cref{lem:FindScheme,lem:ConverseFindScheme}, this shows that
\[
 \lim_{R_1\to \infty} \lim_{R_2\to\infty} \lim_{\min_{i\neq j} \|x_i-x_j\|\to \infty} \frac{\P\Bigl(\Psi_\cT(x_1,\ldots,x_k) \Delta \{ \exists! (b_{i,j})_{i,j} \text{ s.t. } \Psi_\cT^{**}((x_i)_i,(b_{i,j})_{i,j}) \text{ holds}\}\Bigr)}{G_\cT(x_1,\ldots,x_k)} =0.
\]
Next, let $\sU_\cT(x_1,\dots,x_k)$ be the event that there exists a \emph{unique} tuple $(b_{i,j})_{i,j}$ such that $\Psi_\cT^{**}((x_i)_i,(b_{i,j})_{i,j})$ holds. By dividing the last inequality according to the value of the tuple $(b_{i,j})_{i,j}$ we get 
\begin{equation} \proba(\Psi_{\cT,W}(x_1,\dots,x_k))=\sum_{(b_{i,j})_{i,j}}  \proba(\Psi^{**}_\cT((x_i)_i, (b_{i,j})_{i,j}) \cap \sU_\cT((x_i)_i)\cap \sW((x_i)_i))+o(G_\cT(x_1,\dots,x_k)),  \label{eq:SplitWtreeA_HalfStep} \end{equation}
where the little-$o$ notation appearing here refers to our usual triple limit in which we first send $\min_{i\neq j}\|x_i-x_j\|\to\infty$, then send $R_2\to \infty$, and then finally send $R_1\to \infty$.

We now get rid of the event $\sU_\cT(x_1,\dots,x_k)$ in \eqref{eq:SplitWtreeA_HalfStep}. Assume that $\Psi^{**}_\cT((x_i)_i, (b_{i,j})_{i,j})$ holds but $\sU_\cT((x_i)_i)$ does not.
Let $(b'_{i,j})_{i,j}\neq (b_{i,j})_{i,j}$ be such that $\Psi^{**}_\cT((x_i)_i, (b'_{i,j})_{i,j})$ also holds. Let $I,J$ be such that $b'_{I,J}\neq b_{I,J}$. 
By definition of $\Psi^*$, we have disjointly for every edge  $\{i,j\}$ of $\cT$ that the points $x_i+b_{i,j}$ and $x_j+b_{j,i}$ are connected by an open path $\gamma_{i,j}$ whose internal vertices lie in $\Z^d\backslash \bigcup_{1\leq i \leq k} B_{R_1}(x_i)$. 
By \cref{lem:DescribeScheme}, there must exist (as long as $\min \|x_i-x_j\|_2$ is large enough depending on $R_2$) a path $\Gamma_{I,J}$ from $x_I+b'_{I,J}$ to $\partial B_{R_2}(x_I)$ whose internal vertices lie in $B_{R_2}(x_I)\backslash B_{R_1}(x_I)$. The path $\Gamma_{I,J}$ cannot intersect $\gamma_{I,J}$ since otherwise it would contradict the minimality of either $b_{I,J}$ or $b'_{I,J}$ for the order $\prec_{R_1}$, which is imposed by definition of $\Psi^{**}$. 
Also $\Gamma_{I,J}$ cannot intersect any of the paths $\gamma_{i,j}$ for $\{i,j\}\neq \{I,J\}$ since it would otherwise contradict item 2 of Lemma \ref{lem:FindScheme}. We deduce by the BK inequality,  
\[ \sum_{(b_{i,j})_{i,j}}  \proba(\Psi^{**}_\cT((x_i)_i, (b_{i,j})_{i,j}) \cap \sW((x_i)_i) \backslash \sU_\cT((x_i)_i))\leq C(R_1)\proba(\partial B_{R_1}(0)\slr \partial B_{R_2}(0)) G_\cT(x_1,\dots,x_k),\]
which is $o(G_\cT(x_1,\dots,x_k))$ in the usual triple limit. Thus, by \eqref{eq:SplitWtreeA_HalfStep} and a union bound, we have
\begin{equation} \proba(\Psi_{\cT,W}(x_1,\dots,x_k))=\sum_{(b_{i,j})_{i,j}}  \proba(\Psi^{**}_\cT((x_i)_i, (b_{i,j})_{i,j})\cap \sW((x_i)_i))+o(G_\cT(x_1,\dots,x_k)).  \label{eq:SplitWtreeA} \end{equation}

Next, for each collection $O=(O_i)_{1\leq i \leq k}$ of sets of edges in $B_{R_1}^{\mathbf e}(0)$, we define
\[\Psi^{***}_{\cT}((x_i)_{1\leq i \leq k}, (b_{i,j})_{1\leq i,j \leq k: \{i,j\}\in E},O;R_1,R_2)\]
to be the event that the following conditions all hold:
\begin{itemize}
\item[(a)] The points $x_i+b_{i,j}$ and $x_j+b_{j,i}$ are connected by an open path whose internal vertices lie in $\Z^d\backslash \bigcup_{\ell=1}^k B_{R_1}(x_\ell)$ for every $\{i,j\}\in E$;
\item[(b)] For every $\{i,j\}\neq \{i',j'\}\in E$, the points $x_i+b_{i,j}$ and $x_{i'}+b_{i',j'}$ are \emph{not} connected by an open path whose internal vertices lie in $\Z^d\backslash \bigcup_{\ell=1}^k B_{R_1}(x_\ell)$;
\item[(c)] All the edges in $\bigcup_{1\leq i \leq k} B_{R_1}^{\mathbf e}(x_i)$ are closed;
\item[(d)] After opening $x_\ell+O_\ell$ for every $1\leq \ell\leq k$, for each $\{i,j\}\in E$ there does not exist $b\prec_{R_1} b_{i,j}\in \partial B_{R_1}(0)$ such that $x_i+b_{i,j}$ is connected to $x_i+b$ by an open path whose internal vertices lie in $B_{R_2}(x_i)\backslash B_{R_1}(x_i)$; 
\item[(e)] After opening $x_i+O_i$ for every $1\leq i \leq k$, there exists a collection of edge-disjoint open paths $(\gamma_{i,j})_{j:\{i,j\}\in E}$ from $x_i+b_{i,j}$ to $x_i$ in $B_{R_2}(x_i)$.
\end{itemize}
It is straightforward from the definitions to check that as long as $\inf \|x_i-x_j\|_\infty>2R_2$:
\begin{itemize}
\item If $\Psi^{***}_\cT((x_i)_i, (b_{i,j})_{i,j},(O_i)_i;R_1,R_2)$ holds then after opening the edges $x_i+O_i$ for every $i$ the event $\Psi^{**}_\cT((x_i)_i, (b_{i,j})_{i,j};R_1,R_2)$ also holds.
\item Conversely, if $\Psi_\cT^{**}((x_i)_i, (b_{i,j})_{i,j};R_1,R_2)$ holds and $x_1+O_1,\dots, x_k+O_k$ are the sets of edges open in $B_{R_1}(x_1),\dots, B_{R_1}(x_k)$, then after closing these sets of edges, the event $\Psi^{***}_\cT((x_i)_i, (b_{i,j})_{i,j},(O_i)_i)$ holds. 
\end{itemize}
Therefore, as long as $\min \|x_i-x_j\|_\infty> 2R_2 \geq R_1$ and $x_i+W_i\subseteq B_{R_1}^{\mathbf e}(x_i)$ for every $i$, we have that 
\begin{equation} 
\proba(\Psi^{**}_\cT((x_i)_i, (b_{i,j})_{i,j}) \cap \sW((x_i)_i))
 =  \hspace{-2em} \sum_{\substack{ O_1,\dots, O_k \\ \forall 1\leq i \leq k, W_i\cap O_i=\emptyset}}  
 \left[\prod_{i=1}^k \left ( \frac{p}{1-p} \right )^{|O_i|} \right]\proba \left ( \Psi_\cT^{***}((x_i)_i, (b_{i,j})_{i,j},(O_i)_i) \right ), \label{eq:SplitWtreeB}\end{equation}
where the sum is over sets of edges in $B_{R_1}^{\mathbf e}(0)$.
Since the conditions (c), (d), and (e) above depend only on the status of edges in $B_{R_2}(x_1),\dots, B_{R_2}(x_k)$, we may apply \cref{thm:scheme_function} (and inclusion-exclusion) to deduce that
 there exists a non-negative function $\mathbf{C}^{R_1,R_2}$ such that 
\[ \proba \left ( \Psi^{***}_\cT((x_i)_i, (b_{i,j})_{i,j},(O_i)_i;R_1,R_2) \right ) = A^{k-1} G_\cT(x_1,x_2,\dots,x_k)\left (\prod_{i=1}^k \mathbf{C}^{R_1,R_2}((b_{i,j})_{j:\{i,j\}\in E},O_i)+o(1) \right ) \]
as $\min_{i\neq j}\|x_i-x_j\|\to \infty$ with $R_1,R_2$ fixed.
Thus, by summing over $(O_i)_i,(b_{i,j})_{i,j}$, and by using \eqref{eq:SplitWtreeB}, the sum in \eqref{eq:SplitWtreeA} is equal to
\[A^{k-1}G_\cT(x_1,\dots,x_k) \left[\prod_{i=1}^k \left ( \sum_{(b_{i,j})_{j: \{i,j\}\in E}} \sum_{O_i: W_i\cap O_i=\emptyset}\mathbf{C}^{R_1,R_2} \left ((b_{i,j})_{j:\{i,j\}\in E},O_i \right )  \right )+o(1)\right] \]
as $\min_{i\neq j} \|x_i-x_j\|\to \infty$ with $R_1,R_2$ fixed. 
Note that each term in the product appearing here depends only on $R_1,R_2,\deg(i)$, and $W_i$, and we denote it by $\mathbf{C}^{R_1,R_2}(\deg(i),W_i)$.
Putting everything together, we have that
\begin{equation} \lim_{R_1\to \infty}\lim_{R_2\to \infty}\lim_{\min_{i\neq j}\|x_i-x_j\|\to\infty} \left | \frac{\proba(\Psi_{\cT,W}(x_1,\dots,x_k))}{A^{k-1}G_\cT(x_1,\dots,x_k) } - \prod_{i=1}^k \mathbf{C}^{R_1,R_2}(\deg(i),W_i) \right |= 0.  \label{eq:EndOfProof_Thm:Sheme_function_variant} \end{equation}

We finish exactly as at the end of the proof of \cref{thm:scheme_function}; we go faster as no argument changes. Since the first term above does not depend on $R_1,R_2$, we have a Cauchy criterion for $(\prod_{i=1}^k \mathbf{C}^{R_1,R_2}(\deg(i),W_i))_{R_1,R_2}$ and hence that the double-limit 
\[\mathbf{C}(\cT,W)=\lim_{R_1\to \infty} \lim_{R_2\to \infty} \prod_{i=1}^k \mathbf{C}^{R_1,R_2}(\deg(i),W_i)\] is well-defined for every tree $\cT$ and collection $(W_i)_i$. Considering the case where the tree $\cT$ consists of a single edge and using \eqref{eq:two_point_assumption}, we find that $\mathbf{C}^{R_1,R_2}(1,\emptyset)$ converges to $1$. Considering the case where $\cT$ has a single vertex $m+1$ of degree $m$ and leaves $\{1,2,\dots,m\}$ and where we set $W_1=W_2=\dots=W_m=\emptyset$ and $W_{m+1}=W$, we deduce that the limit $\mathbf{C}(m,W)=\lim_{R_1\to\infty}\lim_{R_2\to\infty}\mathbf{C}^{R_1,R_2}(m,W)$ is well-defined also. The claim now follows from \eqref{eq:EndOfProof_Thm:Sheme_function_variant}. The proof of the claim regarding when $\mathbf{C}(m,W)$ is positive is deferred to  \cref{lem:variant_positivity}.\end{proof}

\subsection{Localization of overcount factors}
\label{subsec:localization_of_overcount_factors}

Recall from 
\eqref{eq:ExpOvercountBranch} that the $k$-point function may be expressed as
\begin{equation*} T_{p_c}(x_1,\ldots,x_k)=\sum_{\cT\in \mathsf{Tr}(k)} \sum_{x_{k+1},\dots,x_{2k-2}} \E\left [\frac{\1(\Psi_\cT(x_1,\dots,x_{2k-2}))}{N(x_1,\dots, x_k)}\right ], \end{equation*}
where $N(x_1,\ldots,x_k)$ is the \textbf{overcount factor}
\[
  N(x_1,\ldots,x_k) = \# \{(\cT,x_{k+1},\ldots,x_{2k-2}): \cT \in \mathsf{Tr}(k), x_{k+1},\ldots,x_{2k-2} \in \Z^d,\; \text{$\Psi_\cT(x_1,\ldots,x_{2k-2})$ holds}\}.
\]
We say that the pair $(\cT,(x_{k+1},\ldots,x_{2k-2}))$ is a \textbf{branching tree} for $(x_1,\ldots,x_k)$ if $\cT\in \mathsf{Tr}(k)$ and $\Psi_\cT(x_1,\ldots,x_{2k-2})$ holds, so that $N(x_1,\ldots,x_k)$ is the number of  branching trees for $(x_1,\ldots,x_k)$.  
The main goal of this section is to relate the overcount factors to the localized quantities 
 $|\cS^\Psi_{R_1,R_2}(x)|$, where we define $\cS^\Psi_{R_1,R_2}(x)$ to be the set of vertices $y\in B_{R_1}(x)$ such that there exist three edge-disjoint open paths from $y$ to $\partial B_{R_2}(x)$.

\begin{prop} \label{pro:CountBranchT} Given $\cT\in \mathsf{Tr}(k)$ and $x_1,\ldots,x_{2k-2}$, we have $\Psi_\cT$-usually that $(\cT',y_{k+1},\ldots,y_{2k-2})$ with $\cT'\in \mathsf{Tr}(k)$ is a branching tree for $(x_1,\ldots,x_k)$ if and only if $\cT'=\cT$ and $(y_{k+1},\dots, y_{2k-2}) \in \prod_{i=k+1}^{2k-2} \cS^\Psi_{R_1,R_2}(x_{i})$. As a consequence, we have $\Psi_\cT$-usually that
\[ N(x_1,\dots, x_k)=\prod_{i=k+1}^{2k-2} |\cS^\Psi_{R_1,R_2}(x_i)|. \]
\end{prop}

Note that, by the definition of being $\Psi_\cT$-usual, this proposition has content only in the limit in which the branch points $x_{k+1}$, $\ldots$, $x_{2k-2}$ become very well-separated. We will return to the issue of estimating the contribution to the $k$-point function from branching trees with some branch points at bounded distance from one another in \cref{subsec:negligibility_of_degenerate_branching_trees}.

\medskip

Before proving \cref{pro:CountBranchT} we first prove two auxiliary lemmas:

\begin{lemma} \label{lem:PathInBranchA} 
Given $\cT\in \mathsf{Tr}(k)$, the following holds $\Psi_\cT$-usually: For every $1\leq i\neq j\leq k$ and every open path $\gamma$ from $x_i$ to $x_j$, if $i_0,i_1,\dots, i_a$ is the longest sequence such that
\begin{itemize}
\item  $i_t\neq i_{t+1}$ for every $0\leq t<a$ and
\item there exists $\tau_0<\tau_1<\dots<\tau_a$ such that $\gamma(\tau_t)\in B_{R_1}(x_{i_t})$ for every $t$,
\end{itemize}
then $(i_0,i_1,\dots, i_a)$ is the geodesic path in $\cT$ from $i$ to $j$. 
\end{lemma}

\begin{proof}[Proof of \cref{lem:PathInBranchA}]
By parts (a) and (b) of Lemma \ref{lem:PsiArms}, we have $\Psi_\cT$-usually that $(i_t,i_{t+1})\in E(\cT)$  for every $0\leq t<a$ and every open path $\gamma$ from $x_i$ to $x_j$. Moreover, $\Psi_\cT$-usually  we have that $i_t\neq i_{t+2}$ for every $0\leq t<a-1$ for every sequence $i_0,\ldots,i_a$ of the form appearing in the statement of the lemma, since otherwise there would exist $i\neq j$ and two edge-disjoint paths from $B_{R_1}(x_i)$ to $B_{R_1}(x_j)$, which is prevented by part (b) of \cref{lem:PsiArms}. It follows that $(i_0,i_1,\dots, i_a)$ must be the geodesic path from $i$ to $j$ in $\cT$.
\end{proof} 

For each $x_1,x_2,x_3\in \Z^d$ we write 
$\Upsilon(x_1,x_2,x_3) := \{ x : (x_1\slr x)\circ (x_2\slr x)\circ (x_3\slr x)\}$.

\begin{lemma} \label{lem:PositionBranch}If $\cT\in \mathsf{Tr}(k)$ then
\[ \Upsilon(x_{i_1},x_{i_2},x_{i_3})\subset B_{R_1}(x_{j(i_1,i_2,i_3)}) \]
 $\Psi_\cT$-usually for every $i_1<i_2<i_3$, 
where $j(i_1,i_2,i_3)$ is the unique integer at which the three geodesics between $i_1$, $i_2$, and $i_3$ in $\cT$ all meet.
\end{lemma}

\begin{proof}[Proof of \cref{lem:PositionBranch}]
It suffices to prove the result with $R_1$ replaced by $R_2$. Let $j=j(i_1,i_2,i_3)$ and, assuming $\Psi_\cT(x_1,\ldots,x_{2k-2})$ holds, let $\gamma_1$,  $\gamma_2$, and $\gamma_3$ be edge-disjoint paths from $x_{i_1}$, $x_{i_2}$, and $x_{i_3}$ to some vertex $z$. By \cref{lem:PathInBranchA}, $\Psi_\cT$-usually, the paths $\gamma_1\cup \gamma_2$, $\gamma_1\cup \gamma_3$,  and $\gamma_2\cup \gamma_3$ all pass through $B_{R_1}(x_{j})$. As such, at least two of the paths $\gamma_1,\gamma_2$, and $\gamma_3$ intersect $B_{R_1}(x_j)$. Without loss of generality let us assume that $\gamma_1$ and $\gamma_2$ both intersect $B_{R_1}(x_j)$. If the point $z$ does \emph{not} belong to $B_{R_2}(x_j)$, there must exist within each of $\gamma_1$ and $\gamma_2$ two edge-disjoint paths from $\partial B_{R_1}(x_j)$ to $\partial B_{R_2}(x_j)$, yielding a total of four edge-disjoint paths from $\partial B_{R_1}(x_j)$ to $\partial B_{R_2}(x_j)$. This is ruled out by part (c) of \cref{lem:PsiArms}, completing the proof.
\end{proof}

\begin{proof}[Proof of \cref{pro:CountBranchT}] We first prove that, $\Psi_\cT$-usually, every branching tree $(\cT',(y_{k+1},\dots, y_{2k-2}))$ must be of the desired form. 
Suppose that $\Psi_\cT(x_1,\ldots,x_{2k-2})$ holds and that $(\cT',(y_{k+1},\dots, y_{2k-2}))$ is also a branching tree for $(x_1,\ldots,x_k)$.
For each distinct triple $i_1,i_2,i_3 \in \{1,\ldots,k\}$, the assumption that $(\cT',(y_{k+1},\dots, y_{2k-2}))$ is a branching tree for $(x_1,\ldots,x_k)$ guarantees that there exists $\ell \in \{k+1,\dots, 2k-2\}$ such that $y_\ell \in \Upsilon(x_{i_1},x_{i_2},x_{i_3})$. Applying \cref{lem:PositionBranch}, it follows that, $\Psi_\cT$-usually, the point $y_\ell$ must belong to $B_{R_1}(x_{j(i_1,i_2,i_3)})$ where $j(i_1,i_2,i_3)$ is the unique integer at which the three geodesics between $i_1$, $i_2$, and $i_3$ in $T$ all intersect.
 Moreover, since every internal vertex of $\cT$ has degree $3$, for each $j\in \{k+1,\dots, 2k-2\}$ there exists $i_1,i_2,i_3\in \{1,\ldots,k\}$ such that $j(i_1,i_2,i_3)=j$.
Since the two trees $\cT$ and $\cT'$ have the same number of internal vertices and the balls $B_{R_1}(x_j)$ and $B_{R_1}(x_{j'})$ are disjoint for all $j\neq j'$ in the relevant triple limit, it follows that, $\Psi_\cT$-usually, there exists (for every branching tree $(\cT',(y_{k+1},\dots, y_{2k-2}))$) a unique bijection $\ell:\{k+1,\ldots,2k-2\}\to \{k+1,\ldots,2k-2\}$ mapping the internal vertices of $\cT$ to the internal vertices of $\cT'$ such that $y_{\ell(j)}\in B_{R_1}(x_j)$ for every $j\in \{k+1,\ldots,2k-2\}$. Moreover, since $(\cT',(y_{k+1},\ldots,y_{2k-2}))$ is a branching tree for $(x_1,\ldots,x_k)$, we must have (when $\min_{i\neq j}\|x_i-x_j\| > 3R_1$) that $y_{\ell(j)} \in \cS^\Psi_{R_1,R_2}(x_j)$ for every $j\in \{k+1,\ldots,2k-2\}$.

It thus remains  to show that $\cT'=\cT$. Since $\mathsf{Tr}(k)$ is a set of isomorphism class representatives, it suffices to prove that $\ell$ extends to an isomorphism of leaf-labelled trees. 
We proceed by contradiction. Assume that there exists $\{i,j\}\in E(\cT')$ such that $\{\ell^{-1}(i),\ell^{-1}(j)\}$ is not an edge of $\cT$, where we extend $\ell$ to $\{1,\ldots,2k-2\}$ by setting $\ell(i)=i$ for every $1\leq i\leq k$. Since $\{\ell^{-1}(i),\ell^{-1}(j)\}\notin E(\cT)$ there exists $\ell^{-1}(z)\notin\{\ell^{-1}(i),\ell^{-1}(j)\}$ that is on the path from $\ell^{-1}(i)$ to $\ell^{-1}(j)$ in $\cT$. Without loss of generality we may assume that $z$ is closer to $i$ than $j$ in $\cT'$ (it cannot be equally close to both as they are neighbours). As a result, we may find two leaves $L(i), L(j)$ in $\cT'$ such that the path  from $L(i)$ to $L(j)$ in $\cT'$ passes through $L(i),z,i,j$, and $L(j)$ in that order (as well as possibly some other vertices in between). Since $(\cT',(y_{k+1},\dots, y_{2k-2}))$ is a branching tree for $(x_1,x_2,\dots, x_k)$ it follows that there exists an open path $\gamma$ from $x_{L(i)}$ to $x_{L(j)}$ that visits the points $x_{L(i)},y_z,y_i,y_j$, and $x_{L(j)}$ in the same order. However, since $\ell^{-1}(z)$ is between $\ell^{-1}(i)$ and $\ell^{-1}(j)$ in the tree $\cT$, we have by \cref{lem:PathInBranchA} that, $\Psi_\cT$-usually, there does not exist any pair of leaves $a,b\in \{1,\ldots,k\}$ and an open path $\gamma$ from $x_a$ to $x_{b}$ such that $\gamma$ goes through $B_{R_1}(x_{\ell^{-1}(z)})$ before going through both $B_{R_1}(x_{\ell^{-1}(i)})$ and $B_{R_1}(x_{\ell^{-1}(j)})$. This yields a contradiction, so that we must have $\cT'=\cT$. (More precisely, the event that $\cT'=\cT$ for every branching tree $(\cT',(y_{k+1},\ldots,y_{2k-2}))$ holds $\Psi_\cT$-usually.)

\medskip
 
 It remains to show conversely that, $\Psi_\cT$-usually, $(\cT,(y_{k+1},\ldots,y_{2k-2}))$ is a branching tree for $(x_1,\ldots,x_k)$ for every
\[ (y_{k+1},\ldots,y_{2k-2})\in \prod_{i=k+1}^{2k-2} \cS^\Psi_{R_1,R_2}(x_{i}).\]
 Suppose that $\Psi_\cT(x_1,\ldots,x_{2k-2})$ holds, let $(y_{k+1},\ldots,y_{2k-2})\in \prod_{i=k+1}^{2k-2} \cS^\Psi_{R_1,R_2}(x_{i})$, and let $(\gamma_{i,j})_{\{i,j\}\in E(\cT)}$ be a collection of edge-disjoint open paths where $\gamma_{i,j}$ connects $x_i$ to $x_j$ and where we define $\gamma_{j,i}$ to be the reversal of $\gamma_{i,j}$ for every $\{i,j\}\in E$. We write $y_i=x_i$ for $1\leq i \leq k$. We wish to modify each path $\gamma_{i,j}$ locally inside the balls $B_{R_2}(x_i)$ and $B_{R_2}(x_j)$ to obtain a new family of edge-disjoint paths $(\gamma'_{i,j})_{\{i,j\}\in E(\cT)}$ with $\gamma'_{i,j}$ connecting $y_i$ to $y_j$ for each $\{i,j\}\in E(\cT)$. Since these modifications will be performed locally, it suffices to prove we can modify our family of paths to replace a single $x_i$ with $y_i$ for some $k+1\leq i \leq 2k-2$; the claim will follow by applying the same construction iteratively for all $i\in \{k+1,\ldots,2k-2\}$.
  Let $j\in \{k+1,\dots, 2k-2\}$ and let $i_1\neq i_2\neq i_3$ be three distinct vertices such that $\{i_1,j\}$, $\{i_2,j\}, \{i_3,j\}$ are all edges of $\cT$. Let $\gamma_1,\gamma_2,\gamma_3$ be edge-disjoint paths from $x_{i_1}$, $x_{i_2}$, and $x_{i_3}$ to $x_j$ and let
   $\Gamma_1,\Gamma_2,\Gamma_3$ be three edge-disjoint paths from $\partial B_{R_2}(x_j)$ to $y_j$, which exist since $y_j \in \cS^\Psi_{R_1,R_2}(x_j)$. 
When applying \cref{lem:rerouting_spare_arm} to these two collections of paths, we must obtain $\Psi_\cT$-usually that the assignment is total (i.e., $A=B=\{1,2,3\}$), since otherwise there would be an extra $(R_1,R_2)$-arm around $x_j$, which is not the case $\Psi_\cT$-usually by \cref{lem:PsiArms}. As such, we must $\Psi_\cT$-usually be able to reroute the three paths $\gamma_1,\gamma_2,\gamma_3$ to obtain edge-disjoint paths from $x_{i_1}$, $x_{i_2}$, and $x_{i_3}$ to $y_j$ as required. \qedhere
\end{proof}

\subsection{Negligibility of degenerate branching trees}
\label{subsec:negligibility_of_degenerate_branching_trees}

Before proving \cref{thm:k_point}, we deal here with a last annoying case. We say that a branching tree $(\cT,\{x_{k+1},\dots, x_{2k-2}\})$ for $(x_1,\dots, x_k)$ is \textbf{$r$-degenerate} if there exists $1\leq i <j\leq 2k-2$ such that $\|x_i-x_j\|\leq r$. We write $N_{\leq r}(x_1,\dots, x_k)$ for the number of $r$-degenerate branching trees for $(x_1,x_2,\dots, x_k)$. The aim of this section is to prove that degenerate branching trees are negligible by  establishing upper bounds on $\E[N_{\leq r}(x_1,\ldots,x_k)]$. We prove this claim for trees that may have labelled internal vertices or internal vertices of degree different than $3$ for applications beyond the $k$-point function.

\begin{lemma} \label{lem:DegenerateSum}
Let $\cT$ be a tree with vertex set $\{1,\ldots,k+\ell\}$ for some $k,\ell\geq 1$ and suppose that the set $\{1,\ldots,k\}$ contains all the leaves of $\cT$. For each $r\geq 1$ we have that

\[\sum_{x_{k+1},\ldots,x_{k+\ell}\in \Z^d}\1(\exists i\neq j \text{ s.t.\ } \|x_i-x_j\|\leq r) G_\cT(x_1,\ldots, x_{k+\ell}) =o\left(\sum_{x_{k+1},\ldots,x_{k+\ell}\in \Z^d}G_\cT(x_1,\ldots,x_{k+\ell})\right) \]
 as $\min_{i\neq j} \|x_i-x_j\|\to \infty$, where we take this statement to hold vacuously when the sum on the right-hand side is infinite.
\end{lemma}

Applying a union bound and the BK inequality, \cref{lem:DegenerateSum} has the following immediate corollary.

\begin{corollary} \label{cor:DegenerateTreeBK} $\E[N_{\leq r}(x_1,x_2,\dots, x_k)]=o(G(x_1,\ldots,x_k))$ as $\min_{1\leq i\neq j \leq k} \|x_i-x_j\|\to \infty$ for every $k\geq 2$ and $r\geq 1$.
\end{corollary}

We warn the reader that while \cref{lem:DegenerateSum} may seem obvious, the proof is a little more delicate than one might expect.

\begin{proof}[Proof of \cref{lem:DegenerateSum}] 
Fix $k,\ell\geq 1$ and the tree $\cT$. We first prove that for each $\eps>0$ there exists $\delta>0$ such that
\begin{multline}\sum_{x_{k+1},\ldots,x_{k+\ell}\in \Z^d}\1\left(\langle x_j-x_{j'}\rangle\leq \delta \inf_{i\sim j,i\neq j'} \langle x_i-x_{j'}\rangle\right) G_\cT(x_1,\ldots,x_{k+\ell}) \\\leq \eps \sum_{x_{k+1},\ldots,x_{k+\ell}\in \Z^d}G_\cT(x_1,\ldots,x_{k+\ell}).\label{eq:sumLocalVsGlobal} \end{multline}
for every $x_1,\ldots,x_k\in \Z^d$, $k+1\leq j \leq k+\ell$, and $1\leq j' \leq k+\ell$ with $j'\neq j$.
To prove this, first note that if 
$\langle x_j-x_{j'}\rangle \leq \frac{1}{4}\inf_{i\sim j,i\neq j'} \langle x_i-x_{j'}\rangle$ then $\langle x_i -x_{j'}\rangle$ is of the same order as $\langle x_i-x_{j}\rangle$ for every $i$ that is adjacent to $j$ but not equal to $j'$, so that if we replace each term $G(x_i,x_j)$ with $i\neq j'$ in the product defining $G_\cT(x_1,\ldots,x_{k+\ell})$ with $G(x_i,x_{j'})$ then the order of this product changes by at most a constant factor and this modified product either has no dependence on $x_j$ (if $j'$ is not adjacent to $x_j$) or depends on $x_j$ only through the single term $G(x_j,x_{j'})$ (if $j'$ is adjacent to $j$).
 Since we also have that
\[\frac{|\{x \in \Z^d : \langle x \rangle \leq r\}|}{|\{x\in \Z^d: \langle x \rangle \leq R\}|} = O((r/R)^d) \qquad \text{ and } \qquad 
\frac{\sum_{\langle x \rangle \leq r} G(0,x)}{\sum_{\langle x \rangle \leq R} G(0,x)} = O((r/R)^2)
\]
for every $R\geq r >0$ with $R\geq 2$, it follows that there exists a constant $C$ (depending on the choice of tree $\cT$) such that if $\delta \leq 1/4$ then
\begin{multline}\sum_{x_{k+1},\ldots,x_{k+\ell}\in \Z^d}\1\left(\langle x_j-x_{j'}\rangle\leq \delta \inf_{i\sim j,i\neq j'} \langle x_i-x_{j'}\rangle\right) G_\cT(x_1,\ldots,x_{k+\ell}) \\\leq C \delta^2
\sum_{x_{k+1},\ldots,x_{k+\ell}\in \Z^d}\1\left(\langle x_j-x_{j'}\rangle\leq \frac{1}{4} \inf_{i\sim j,i\neq j'} \langle x_i-x_{j'}\rangle\right) G_\cT(x_1,\ldots,x_{k+\ell})
\\
\leq C \delta^2
\sum_{x_{k+1},\ldots,x_{k+\ell}\in \Z^d} G_\cT(x_1,\ldots,x_{k+\ell})
 \end{multline}
for every $x_1,\ldots,x_k\in \Z^d$, $k+1\leq j \leq k+\ell$, and $1\leq j' \leq k+\ell$ with $j'\neq j$ as required.

For each $\delta>0$, define $\mathscr{B}_\delta(x_1,\ldots,x_k)$ to be the set of sequences $(x_{k+1},\ldots,x_{k+\ell})$ in $\Z^d$ such that there exists $k+1\leq j\leq k+\ell$ and $1\leq j' \leq k+\ell$ such that $j\neq j'$ and  $\langle x_j-x_{j'}\rangle\leq \delta \inf_{i\sim j,i\neq j'} \langle x_i-x_{j'}\rangle$.
Summing \eqref{eq:sumLocalVsGlobal} over all $j,j'$, we obtain that for each $\eps>0$ there exists $\delta>0$ such that
\begin{equation}\sum_{(x_{k+1},\ldots,x_{k+\ell})\in \mathscr{B}_\delta(x_1,\ldots,x_k)} G_\cT(x_1,\ldots,x_{k+\ell}) \\\leq \eps \sum_{x_{k+1},\ldots,x_{k+\ell}\in \Z^d}G_\cT(x_1,\ldots,x_{k+\ell})\label{eq:sumLocalVsGlobal2} \end{equation}
for every $x_1,\ldots,x_k \in \Z^d$. Thus, to conclude the proof, it suffices to prove that for each $\delta>0$ and $r\geq 1$ there exists $R<\infty$ such that if $\min_{1 \leq i < j\leq k}\|x_i-x_j\|\geq R$ and $\min_{1 \leq i < j\leq k+\ell}\|x_i-x_j\|\leq r$ then $(x_{k+1},\ldots,x_{k+\ell})\in \mathscr{B}_\delta(x_1,\ldots,x_k)$.

To proceed, we will suppose that $x_1,\ldots,x_{k+\ell}\in \Z^d$ are such that $\min_{1 \leq i < j\leq k}\|x_i-x_j\|\geq R$ and $\min_{1 \leq i < j\leq k+\ell}\|x_i-x_j\|\leq r$ but $(x_{k+1},\ldots,x_{k+\ell})\notin \mathscr{B}_\delta(x_1,\ldots,x_k)$, and show that this yields an upper bound on $R$ depending only on $r$, $\delta$, and $\cT$ (and hence only on $r$, $\eps$, and $\cT$ since $\delta$ was chosen as a function of $\eps$ and $\cT$). We may assume that $R>r$.
 By assumption, there must exist a pair $i\neq j$ with $\|x_{i}-x_{j}\|\leq r$, and since the indices $i,j$ cannot both belong to $\{1,\ldots,k\}$ we may without loss of generality take $k+1\leq j_0\leq k+\ell$ such that
 $\|x_{j_0}-x_{i_0}\|\leq r$ for some $1\leq i_0 \leq k+\ell$.
 Since $(x_{k+1},\ldots,x_{k+\ell})\notin \mathscr{B}_\delta(x_1,\ldots,x_k)$ there must exist a neighbour $j_1$ of $j_0$ distinct from $i_0$ such that
$\langle x_{j_1}-x_{i_0}\rangle \leq \delta^{-1} \langle x_{j_0}-x_{i_0}\rangle$. If $k+1\leq j_1 \leq k+\ell$ then similarly there must exist a neighbour $j_2$ of $j_1$ distinct from $j_0$ such that $\langle x_{j_2}-x_{j_0}\rangle \leq \delta^{-1} \langle x_{j_1}-x_{j_0}\rangle$, and we may continue this construction inductively until finding $a \geq 1$ such that $1\leq j_a\leq k$. Similarly, we may construct a sequence $j_{-1},\ldots,j_{-b}$ with $1\leq b \leq k+l$ such that $j_{-b}\in \{1,\ldots,k\}$, $j_{-i}$ is adjacent to $j_{-i+1}$ and distinct from $j_{-i+2}$ for each $1\leq i \leq b$ and $\langle x_{j_{-i}} -x_{j_{-i+2}}\rangle \leq \delta^{-1} \langle x_{j_{-i+1}} -x_{j_{-i+2}}\rangle$ for every $1\leq i \leq b$. The sequence $j_{-b},\ldots,j_a$ must define a geodesic between two distinct vertices in $\{1,\dots,k\}$, and the triangle inequality implies that the endpoints of this geodesic satisfy
\[
  \|x_{j_{-b}}-x_{j_{a}}\| \leq \sum_{i=-b}^{a-1} \langle x_{j_i}-x_{j_{i+1}}\rangle \leq C_1(\delta,\cT) \langle x_{j_0}-x_{i_0}\rangle \leq C_2(\delta,\cT) r
\]
for every $r\geq 1$ and some constants $C_1(\delta,\cT)$ and $C_2(\delta,\cT)$.
We deduce that in this scenario we must have that $R \leq C_2(\delta,\cT) r$, concluding the proof as explained above. \qedhere
\end{proof}

\subsection{Proof of \cref{thm:k_point}}
\label{subsec:proof_of_the_k_point_function}

We are now almost ready to prove \cref{thm:k_point}. The proof will require one more lemma establishing a lower bound on the three-point function of the correct order.

\begin{lemma}
\label{lem:three_point_lower}
Consider critical percolation on $\Z^d_L$. If $d>6$ and \eqref{eq:two_point_assumption} and \eqref{eq:two_blob_assumption} both hold then there exists a positive constant $c$ such that
\[
  T_{p_c}(x,y,z) \geq c\, G(x,y,z)
\]
for all triples of distinct points $x,y,z\in \Z^d$.
\end{lemma}

\begin{proof}[Proof of \cref{lem:three_point_lower}]
Let $e_1$ be a basis vector in $\Z^d$. Since $0,e_1,-e_1$ are connected to $\infty$ by three vertex-disjoint paths, 
 it follows from \cref{thm:scheme_function} that there exist positive constants $c$ and $R$ such that
\begin{multline*}
  \P(x \leftrightarrow w, w+e_1 \leftrightarrow y, w-e_1 \leftrightarrow z, \text{ and } w,w+e_1,w-e_1 \text{ are all in distinct clusters}) 
  \\\geq c G(x,w)G(w,y)G(w,z)
\end{multline*}
whenever all pairwise distances between $x,y,z,$ and $w$ are at least $R$. If we modify a configuration in which this event holds by forcing the edges $\{w,w+e_1\}$ and $\{w,w-e_1\}$ to be open, we obtain a configuration in which $x$ is connected to $y$ and $z$ and $w$ is the only vertex of $\Z^d_L$ that is connected to $x$, $y$, and $z$ by paths that are vertex disjoint other than their endpoints.
If we define $\mathscr{A}(x,y,z;w)$ to be the event that $w$ is the unique vertex with this property, we deduce that
\begin{equation*}
    \P(\mathscr{A}(x,y,z;w)) 
  \geq \frac{c p_c^2}{(1-p_c)^2} G(x,w)G(w,y)G(w,z)
\end{equation*}
whenever all pairwise distances between $x,y,z,$ and $w$ are at least $R$. The claim follows by summation over $w$ since the events $\mathscr{A}(x,y,z;w)$ are disjoint for different choices of $w$. (As usual there are some easy further details to handle regarding points at bounded distance from each other. When one of the distances between $x,y,z$ is at most $R$, it suffices to use Harris-FKG to bound the three-point function from below by the product of two two-point functions, where one of the two-point functions is between the two points of minimal distance. When $x,y,z$ all have pairwise distance at least $R$, the contribution to $G(x,y,z)=\sum_w G(x,w)G(w,y)G(w,z)$ from points $w$ with distance at least $R$ from all three points $x,y,z$ is easily seen to be of the same order as the entire sum $G(x,y,z)$, so that the bound we have proven here suffices to deduce the claim.)
\end{proof}

\begin{proof}[Proof of \cref{thm:k_point}]
We analyze the formula \eqref{eq:ExpOvercountBranch}.
Fix $k\geq 2$ and $\eps>0$. By \cref{pro:CountBranchT} we may take $R_1,R_2,R_3$ such that if $\min_{1\leq i \neq j \leq 2k-2} \|x_i-x_j\|\geq R_3$ then
\[ \proba\left (\Psi_\cT(x_1,\dots,x_{2k-2}) \text{ holds and }  N(x_1,\dots, x_k)\neq  \prod_{i=k+1}^{2k-2} |\cS^\Psi_{R_1,R_2}(x_i)|\right ) \leq \eps G_\cT(x_1,x_2,\dots, x_{2k-2})\]
for every $\cT\in \mathsf{Tr}(k)$.
It follows that we can bound
\begin{align*}  \Bigg | T_{p_c}(x_1,x_2&,\ldots, x_k)- \sum_{\cT\in \mathsf{Tr}(k)} \sum_{x_{k+1},\dots,x_{2k-2}} \E\left [\frac{\1(\Psi_\cT(x_1,\dots,x_{2k-2}))}{\prod_{i=k+1}^{2k-2} |\cS^\Psi_{R_1,R_2}(x_i)|}\right ] \Bigg | 
\\ & \leq \eps \sum_{\cT \in \mathsf{Tr}(k)} \sum_{x_{k+1},\dots,x_{2k-2}} G_\cT(x_1,\dots, x_{2k-2})  
+ \sum_{\cT\in \mathsf{Tr}(k)} \sum_{\substack{x_{k+1},\dots,x_{2k-2}\\ \exists i\neq j :\|x_i-x_j\|\leq R_3}} \proba(\Psi_\cT(x_1,x_2,\dots,x_{2k-2})) 
\\  & = \eps G(x_1,\ldots,x_k) +  \E[N_{\leq R_3}(x_1,x_2,\dots, x_k)].
\end{align*}
Since $\eps>0$ was arbitrary, it follows from this and \cref{cor:DegenerateTreeBK}
that 
\begin{equation} T_{p_c}(x_1,\ldots,x_k) = \sum_{\cT\in \mathsf{Tr}(k)} \sum_{x_{k+1},\dots,x_{2k-2}} \E\left [\frac{\1(\Psi_\cT(x_1,\dots,x_{2k-2}))}{\prod_{i=k+1}^{2k-2} |\cS^\Psi_{R_1,R_2}(x_i)|}\right ]+o(G(x_1,\ldots,x_k))\label{eq:ApproximateCountingBranchingTree} \end{equation}
in the triple limit in which we first send $\min_{i\neq j}\|x_i-x_j\|\to\infty$, then send $R_2\to \infty$, then send $R_1\to \infty$.
 On the other hand, since $|\cS^\Psi_{R_1,R_2}(x_i)|$ only depends on the status of the edges inside $B_{R_2}(x_i)$, we may apply  \cref{thm:scheme_function_variant} to deduce (by bounded convergence) that for each $R_2\geq R_1$ there exists a constant $C(R_1,R_2)$   such that
\[ \E\left [\frac{\1(\Psi_\cT(x_1,\dots,x_{2k-2}))}{\prod_{i=k+1}^{2k-2} \left |\cS^\Psi_{R_1,R_2}(x_i) \right |}\right ]
= \Big(C(R_1,R_2)^{k-2}A^{2k-3}+o(1) \Big ) G_\cT(x_1,x_2,\dots,x_{2k-2})\]
 as $\min\|x_i-x_j\| \to \infty$ with $R_1,R_2$ fixed. It follows from this and a second application of \cref{cor:DegenerateTreeBK} that
\begin{align} \sum_{\cT\in \mathsf{Tr}(k)} \sum_{x_{k+1},\dots,x_{2k-2}} \E\left [\frac{\1(\Psi_\cT(x_1,\dots,x_{2k-2}))}{ \prod_{i=k+1}^{2k-2} \left |\cS^\Psi_{R_1,R_2}(x_i) \right |}\right ] &
 = [C(R_1,R_2)+o(1)]^{k-2}A^{2k-3}  G(x_1,\ldots,x_k) \label{eq:LocalMiscountBranchingTree}\end{align}
 as $\min_{i\neq j}\|x_i-x_j\|\to\infty$ with $R_1,R_2$ fixed.
Finally, combining \eqref{eq:ApproximateCountingBranchingTree} and \eqref{eq:LocalMiscountBranchingTree}, we deduce that
\[  T_{p_c}(x_1,\ldots,x_k)=(C(R_1,R_2)^{k-2}A^{2k-3} +o(1)) G(x_1,\ldots,x_k)\]
in the triple limit in which we first send $\min_{i\neq j}\|x_i-x_j\|\to\infty$, then send $R_2\to\infty$, then send $R_1\to \infty$.
Since the left-hand side does not depend on $R_1$ and $R_2$ and since $A>0$, the constant $C(R_1,R_2)$ must converge to some $V$ in the double limit in which we first send $R_2\to \infty$ and then send $R_1\to \infty$. As a consequence, for this constant $V$ we must have
\[  T_{p_c}(x_1,\ldots,x_k)=(V^{k-2}A^{2k-3} +o(1)) G(x_1,\ldots,x_k)\]
as $\min_{i\neq j}\|x_i-x_j\|\to\infty$.
It follows from \cref{lem:three_point_lower} that $V>0$, completing the proof.
\end{proof}

\section{Scaling limits}
\label{sec:scaling_limit}

\subsection{Integrated super-Brownian excursion}
\label{sec:integrated_super_brownian_motion}

In this section we prove the part of our main scaling limit theorem in which we consider the cluster of the origin only as a measure in $\R^d$.

\begin{thm}\label{thm:integrated_SBM}
Let $d>6$ and $L\geq 1$ be such that critical percolation on $\mathbb{Z}^d_L$ satisfies the estimates \eqref{eq:two_point_assumption} and \eqref{eq:two_blob_assumption}. There exist positive constants $C_\text{\emph{prob}}$ and $C_\text{\emph{vol}}$ such that
 \[
   C_\mathrm{prob} r^2 \cdot \P_{p_c}\bigl(|K| \geq \lambda C_\mathrm{vol} r^4\bigr) \sim  \lambda^{-1/2} 
 \]
 and
\begin{equation*}
\P_{p_c}\left( \frac{1}{C_\mathrm{vol} r^4}\sum_{x\in K} \delta_{x/r} \in \cdot \;\Bigg|\; |K| \geq \lambda C_\mathrm{vol} r^4 \right) 
\\\longrightarrow \mathbb{N}\Bigl(\mu \in \cdot \mid \mu(\R^d) \geq \lambda\Bigr)
\end{equation*}
weakly in the weak topology as $r\to \infty$ for each $\lambda>0$, where $\N$ denotes the canonical measure of the continuum random tree equipped with its standard Brownian embedding into $\R^d$ and $\mu$ is the pushforward of the mass measure on the tree under this embedding (i.e., $\mu$ is integrated super-Brownian excursion). Moreover, the constants $C_\mathrm{prob}$ and $C_\mathrm{vol}$ are given in terms of the amplitude $A$ and the vertex factor $V$ by $C_\mathrm{prob}=AV/\N(\mu(\R^d)\geq 1)$ and $C_\mathrm{vol}=A^2V$.
\end{thm}

The proof will make use of the following easy consequence of the Kozma-Nachmias \cite{MR2748397} one-arm estimate and the fact that the expected number of connections from the origin to the boundary of a box is bounded \cite{hutchcroft2023high,panis2025sharp}.

\begin{lemma}
\label{lem:two_point_conditioned_on_arm}
Let $d>6$ and $L\geq 1$ be such that critical percolation on $\mathbb{Z}^d_L$ satisfies the estimate \eqref{eq:two_point_assumption}. There exists a constant $C$ such that
\[
  \P_{p_c}(0\leftrightarrow x \text{ and } 0 \leftrightarrow \Z^d\setminus [-r,r]^d) \leq C \min \{\langle x\rangle^{-d+2}, r^{-2} \langle x\rangle^{-d+4}\}
\]
for every $x\in \Z^d$ and $r\geq 1$.
\end{lemma}

\begin{proof}[Proof of \cref{lem:two_point_conditioned_on_arm}]
We will use the fact that if $\partial [-r,r]^d$ denotes the set of oriented edges of $\Z^d_L$ with  $e^-\in [-r,r]^d$ and $e^+\notin [-r,r]^d$ then there exists a constant $C_1$ such that
\begin{equation}
  \sum_{e\in \partial [-r,r]^d} \P(0\leftrightarrow e^- \text{ in $[-r,r]^d$}) \leq C_1
\label{eq:box_boundary}
\end{equation}
for every $r\geq 1$: this follows from \cite[Theorem 2.3]{hutchcroft2023high}, and the recent work \cite{panis2026reversing} shows that in fact this holds with $[-r,r]^d$ replaced by any finite set. (Both results apply when $d>6$ and \eqref{eq:two_point_assumption} holds.)

It suffices to consider the case $r\geq 4 \|x\|_\infty$, the claim following trivially from \eqref{eq:two_point_assumption} otherwise. 
We use the union bound 
\begin{multline}
\P_{p_c}(0\leftrightarrow x \text{ and } 0 \leftrightarrow \Z^d\setminus [-r,r]^d) \leq 
\P_{p_c}(0\leftrightarrow x \text{ only via $\Z^d\setminus [-r/2,r/2]^d$})
\\+ 
\P_{p_c}(0\leftrightarrow x \text{ in $[-r/2,r/2]^d$} \text{ and } 0 \leftrightarrow \Z^d\setminus [-r,r]^d).
\label{eq:two_point_arm_union_bound}
\end{multline}
For the first event, we use a further union bound and the BK inequality to write
\begin{multline}
  \P_{p_c}(0\leftrightarrow x \text{ only via $\Z^d\setminus [-r/2,r/2]^d$}) \\\leq \sum_{e\in \partial [-r/2,r/2]^d} \P(0\leftrightarrow e^- \text{ in $[-r/2,r/2]^d$}) \P(e^+\leftrightarrow x) \leq C_2 r^{-d+2}
  \label{eq:two_point_arm1}
\end{multline}
for some constant $C_2$, where we used \eqref{eq:two_point_assumption} and the assumption that $\|x\|_\infty\leq r/4$ to bound $\P(e^+\leftrightarrow x)=O(r^{-d+2})$ and used \eqref{eq:box_boundary} to bound the remaining sum over $e$.
For the other term in \eqref{eq:two_point_arm_union_bound}, we note that if the origin is connected to $x$ inside $[-r/2,r/2]^d$ and the origin is connected to the complement of $[-r,r]^d$ then there must exist a vertex $w$ in $[-r/2,r/2]^d$ such that the events $\{0\leftrightarrow w\}$, $\{w\leftrightarrow x\}$, and $\{w\leftrightarrow \Z^d\setminus [-r,r]^d\}$ occur disjointly. (Indeed, take a simple open path from $0$ to $x$ in $[-r/2,r/2]^d$ and from $0$ to $\Z^d\setminus [-r,r]^d$ and consider the last place the second path intersects the first path.) It follows by a union bound and the BK inequality that
\begin{multline}
  \P_{p_c}(0\leftrightarrow x \text{ in $[-r/2,r/2]^d$ and $0\leftrightarrow \Z^d\setminus [-r,r]^d$}) \\ \leq \sum_{w\in [-r/2,r/2]^d} T_{p_c}(0,w)T_{p_c}(w,x) \P_{p_c}(w\leftrightarrow \Z^d\setminus [-r,r]^d)
  \leq C_3 r^{-2}\langle x\rangle^{-d+4}
  \label{eq:two_point_arm2}
\end{multline}
for some constant $C_3$, where in the final inequality we used that $ \P_{p_c}(w\leftrightarrow \Z^d\setminus [-r,r]^d)=O(r^{-2})$ for every $w\in [-r/2,r/2]^d$ and bounded the remaining sum over $w$ by $\sum_{w\in [-r,r]^d} T_{p_c}(0,w)T_{p_c}(w,x) = O(G^{*2}(0,x))=O(\langle x \rangle^{-d+4})$.
The claim follows from \eqref{eq:two_point_arm_union_bound}, \eqref{eq:two_point_arm1}, and \eqref{eq:two_point_arm2} since $r^{-d+2}=O(r^{-2}\langle x \rangle^{-d+4})$ when $r\geq 4\|x\|_\infty$.
\end{proof}

\begin{proof}[Proof of \cref{thm:integrated_SBM}]
 We keep the proof brief as similar arguments are well-known in the literature; we refer the reader to e.g.\ \cite[Section I.5.4]{LRPpaper1} for a more detailed treatment of a similar theorem.
We write $\P=\P_{p_c}$ and $\E=\E_{p_c}$. 
As explained in \eqref{eq:canonical_measure_diagrams} and \eqref{eq:spatial_moment_asymptotics_intro}, it follows immediately from \cref{thm:k_point} that if we define $\mu_{K,r} = (A^2V r^4)^{-1}\sum_{x\in K} \delta_{x/r}$ and write $\mu[f]=\int f(x)\dif \mu(x)$ then
\begin{multline}
  AV r^2 \cdot \E\left[\prod_{i=1}^n \mu_{K,r}[f_i]\right] =
  AV r^2 \cdot \E\left[\prod_{i=1}^n \int f_i(x/r) \frac{\dif \mu_K}{A^2V r^4}\right] \\\to \idotsint G_\mathrm{SBM}(0,x_1,\ldots,x_n) \prod_{i=1}^n f_i(x) \dif x_1 \cdots \dif x_n 
  = \N\left[ \prod_{i=1}^n \mu[f_i] \right]
  \label{eq:spatial_moment_asymptotic_restate}
\end{multline}
as $r\to \infty$ for each collection of continuous, compactly supported functions $f_1,\ldots,f_n$. 
Although $\N$ is not a finite measure, the measure $\N_f$ with Radon-Nikodym derivative $\mu[\phi]/\N[\mu[\phi]]$ is a well-defined probability measure for each continuous, compactly supported, non-negative function $\phi$ that is not identically zero. If we define $\E_{\phi,r}$ by biasing $\E$ by $\mu_{K,r}[\phi]$ for the same function $\phi$, it follows from \eqref{eq:spatial_moment_asymptotic_restate} that 
\begin{equation}
  \E_{\phi,r}\left[\prod_{i=1}^n \mu_{K,r}[f_i]\right]=\frac{\E\left[\mu_{K,r}[\phi] \prod_{i=1}^n \mu_{K,r}[f_i]\right]}{\E[\mu_{K,r}[\phi]]} \to \N_\phi\left[ \prod_{i=1}^n \mu[f_i] \right]
\label{eq:phi_biased}
\end{equation}
as $r\to \infty$ for each sequence of continuous, compactly supported functions $f_1,\ldots,f_n$. Since $d>6>4$, we have by \cite[Lemma I.5.34]{LRPpaper1} that the canonical measure $\N$ satisfies the exponential-type estimate $\N[|\mu(f)|^m] \leq C(f)^m m!$ for each continuous compactly supported function $f$ and integer $m\geq 1$. Thus, by Carleman's criterion for finite-dimensional arrays \cite{petersen1982relation}, it follows from \eqref{eq:phi_biased} (applied to sequences in which we repeat each $f_i$ an arbitrary number of times in order to consider arbitrary mixed moments) that if $f_1,\ldots,f_n$ is any sequence of compactly supported continuous functions then the joint law of $(\mu_{K,r}[f_1],\ldots,\mu_{K,r}[f_n])$ under $\E_{\phi,r}$ converges to the joint law of $(\mu[f_1],\ldots,\mu[f_n])$ under $\N_\phi$.
On the other hand, integrating the estimate of \cref{lem:two_point_conditioned_on_arm} shows that
for each continuous, compactly supported, non-negative function $\phi$ that is not identically zero, there exists a constant $C(\phi)$ such that
\begin{equation}
  \P_{\phi,r}(\mu_{K,r}(\{x\in \R^d: \|x\|\geq \lambda \})\neq 0) \leq  C(\phi) \lambda^{-2}
\label{eq:arm_exponent_tightness}
\end{equation}
for every $\lambda,r \geq 1$. 
This provides the tightness needed to deduce that the law of the random measure $\mu_{K,r}$ under $\E_{\phi,r}$ converges weakly as $r\to\infty$ to the law of the random measure $\mu$ under $\N_\phi$ for each continuous, compactly supported, non-negative function $\phi$ that is not identically zero. 

Fix $\lambda>0$.
To see that this implies the claim concerning the volume tail, observe that if $M<\infty$ and $\phi_M$ is compactly supported, non-negative, bounded by 1, and equal to $1$ on $\{x \in \R^d :\|x\|\leq M\}$ then we have by the one-arm estimate $\P_{p_c}(0\leftrightarrow \Z^d\setminus [-r,r]^d) \asymp r^{-2}$ that
\[
  |\P(|K|\geq \lambda A^2 V r^4)-\P(\mu_{K,r}[\phi_M] \geq \lambda)| \leq \P(0\leftrightarrow \{x:\|x\|\geq Mr\})=O\left(\frac{1}{M^2 r^2}\right),
\]
while \eqref{eq:spatial_moment_asymptotic_restate} and the weak convergence established above yield 
(since  $\N_{\phi_M}(\mu[\phi_M]= \lambda)=0$) that
\begin{multline*}
\P(\mu_{K,r}[\phi_M] \geq \lambda) = \E\left[\mu_{K,r}[\phi_M]\right]\E_{\phi_M,r}\left[\frac{\1(\mu_{K,r}[\phi_M] \geq \lambda)}{\mu_{K,r}[\phi_M]}\right] 
\\
\sim \frac{1}{AVr^2} \N\left[\mu[\phi_M]\right]\N_{\phi_M} \left[\frac{\1(\mu[\phi_M] \geq \lambda)}{\mu[\phi_M]}\right] = \frac{\N(\mu[\phi_M]\geq \lambda )}{AV r^2}.
\end{multline*}
Since $M<\infty$ was arbitrary, it follows easily from these two estimates that
\begin{equation}
  \P(|K|\geq \lambda A^2 V r^4) \sim \frac{\N(\mu(\R^d)\geq \lambda )}{AVr^2} = \frac{\lambda^{-1/2}\N(\mu(\R^d)\geq 1)}{AVr^2}
\label{eq:volume_tail_asymptotics_proof}
\end{equation}
as $r\to\infty$, so that the claimed asymptotic estimate on the volume tail holds with $C_\mathrm{vol}=A^2 V$ and $C_\mathrm{prob}= AV/\N(\mu(\R^d)\geq 1)$.

The fact that the conditional law of $\mu_{K,r}$ on the event $\{|K|\geq \lambda A^2 V r^4\}$ converges to the conditional measure $\N(\mu \in \cdot \mid \mu(\R^d) \geq \lambda )$ follows easily from \eqref{eq:volume_tail_asymptotics_proof} and the weak convergence of biased measures established above. Indeed, for each $r\geq 1$, $M,\lambda>0$, and each bounded continuous function $F$ we can write
\begin{align*}
  \E[F(\mu_{K,r}) \mid \mu_{K,r}[\phi_M] \geq \lambda]&=
  \frac{\E\left[\1(\mu_{K,r}[\phi_M] \geq \lambda)F(\mu_{K,r})\right]}{\P(\mu_{K,r}[\phi_M] \geq \lambda)} \\
  &=\frac{\E[\mu_{K,r}[\phi_M]]}{\P(\mu_{K,r}[\phi_M] \geq \lambda)}\E_{\phi_M} \left[\frac{\1(\mu_{K,r}[\phi_M] \geq \lambda)F(\mu_{K,r})}{\mu_{K,r}[\phi_M]}\right]
  \\
&\sim \frac{\N[\mu[\phi_M]]}{\N(\mu[\phi_M] \geq \lambda)}\N_{\phi_M} \left[\frac{\1(\mu[\phi_M] \geq \lambda)F(\mu)}{\mu[\phi_M]}\right] = \N(F(\mu)\mid \mu[\phi_M]\geq \lambda)
\end{align*}
as $r\to\infty$ with $M,\lambda>0$ fixed, where we used that $\N(\mu[\phi_M] = \lambda)=0$ to apply the portmanteau theorem in the last line. The claim follows since the symmetric difference $\{|K|\geq \lambda A^2Vr^4\}\Delta \{\mu_{K,r}[\phi_M]\geq \lambda\}$ has probability much smaller than the probability of either $\{|K|\geq \lambda A^2Vr^4\}$ or $\{\mu_{K,r}[\phi_M]\geq \lambda\}$ when $M$ is a large constant and $r\to \infty$.
\end{proof}

\subsection{Scaling limits with the pivotal metric}

The goal of this section is to prove the following scaling limit theorem for the pivotal metric, which we use to prove \cref{thm:scaling_limit_intro} in the next section.

\begin{theorem}[Scaling limits for the pivotal metric]
\label{thm:scaling_limit_bb}
Let $d>6$ and $L\geq 1$ be such that critical percolation on $\mathbb{Z}^d_L$ satisfies the estimates \eqref{eq:two_point_assumption} and \eqref{eq:two_blob_assumption}. There exist positive constants $C_\text{\emph{prob}}$, $C_\text{\emph{vol}}$, and $C_\text{\emph{piv}}$ such that
 \[
   C_\mathrm{prob} r^2 \cdot \P_{p_c}\bigl(|K| \geq \lambda C_\mathrm{vol} r^4\bigr) \sim  \lambda^{-1/2} 
 \]
 and
\begin{equation*}
\P_{p_c}\left( \left(K,\frac{d_\text{\emph{piv}}}{C_\text{\emph{piv}} r^2}, \frac{\mu_K}{C_\text{\emph{vol}}r^4}, \frac{1}{r}\Phi_K,0 \right)\in \cdot \;\Bigg|\; |K| \geq \lambda C_\mathrm{vol} r^4 \right) 
\longrightarrow \mathbb{N}\Bigl((\sT,d,\nu,\Phi,o) \in \cdot \mid \nu(\sT) \geq \lambda\Bigr)
\end{equation*}
weakly in the weak array topology as $r\to \infty$ for each $\lambda>0$, where $\mathbb{N}$ denotes the canonical measure of the continuum random tree $(\sT,d,\nu,o)$ equipped with its Brownian embedding $\Phi$ into $\R^d$,  $\mu_K$ denotes the counting measure on $K$, and $\Phi_K$ denotes the inclusion map from $K$ to $\R^d$. The constants $C_\mathrm{prob}$ and $C_\mathrm{vol}$ are as in \cref{thm:integrated_SBM}, while $C_\mathrm{piv}=AG_\mathrm{piv}$ where $A$ is the amplitude and $G_\mathrm{piv}$ is the pivotal geodesic factor from \cref{thm:derivative}.
\end{theorem}

As with \cref{thm:integrated_SBM}, we will deduce this theorem from an appropriate \emph{pointwise} asymptotic estimate analogous to \cref{thm:k_point}.
Given $k\geq 2$ and $\ell \geq 0$, we define $\mathsf{TrSub}(k,\ell)$ to be a set of isomorphism class representatives of trees with leaves labelled $1,\ldots,k$, internal vertices of degree $2$ labelled $1,\ldots,\ell$, and $k-2$ unlabelled internal vertices of degree $3$.
Such a tree can be obtained from a tree in $\mathsf{Tr}(k)$ by subdividing edges $\ell$ times and labeling the new vertices that are created. As usual, for notational convenience we take every tree in $\mathsf{TrSub}(k,\ell)$ to have vertex set $\{1,\ldots,2k+\ell-2\}$ with the first $k$ vertices being the leaves in order of labeling and the next $\ell$ vertices being the degree $2$ internal vertices in order of labeling.
For each tree $\cT \in \mathsf{TrSub}(k,\ell)$ we will choose once and for all an ordering of the two neighbours of each internal vertex of degree two (this information is not part of the data used to determine isomorphism classes, but is used only for notational convenience below).

 Given $\cT \in \mathsf{TrSub}(k,\ell)$, vertices $x_1,\ldots,x_{k+\ell}$ and oriented edges $e_1,\ldots,e_\ell$ in $\Z^d_L$, 
  we define the event $\Psi_\cT^\mathrm{sub}(x_1,\ldots,x_{k+\ell},e_1,\ldots,e_\ell)$ to hold if there exist $x_{k+\ell+1},\ldots,x_{2k+\ell-2}$ such that $\Psi_\cT(x_1,\ldots,x_{2k+\ell-2})=\bigcirc_{i\sim j}\{x_i \leftrightarrow x_j\}$ holds and if $k+i$ is an internal vertex of $\cT$ of degree $2$ with ordered pair of neighbours $(j_1,j_2)$ as above, then $x_{k+i}+e_i$ is an open pivotal for the connection between $x_{j_1}$ and $x_{j_2}$. (Note that for technical convenience we do \emph{not} specify the direction in which the path from $x_{j_1}$ to $x_{j_2}$ crosses the edge $x_{k+i}+e_i$.)
We also define $\Psi_\cT^\mathrm{sub}(x_1,\ldots,x_{2k+\ell-2},e_1,\ldots,e_\ell)$ to be the event that this holds with a specific choice of $x_{k+\ell+1},\ldots,x_{2k+\ell-2}$.
We also define 
\[
  G_\cT^\mathrm{sub}(x_1,\ldots,x_{k+\ell}) := \sum_{x_{k+\ell+1},\ldots,x_{2k+\ell-2}\in \Z^d} \prod_{i\sim j}G(x_i,x_j) 
\]
and
\[
  G_{\cT,\mathrm{SBM}}^\mathrm{sub}(x_1,\ldots,x_{k+\ell}) := \idotsint \prod_{i\sim j}G_\mathrm{BM}(x_i,x_j)   \dif x_{k+\ell+1} \cdots \dif x_{2k+\ell-2}
\]
for each $\cT\in \mathsf{TrSub}(k,\ell)$ and each $x_1,\ldots,x_{k+\ell}$ in $\Z^d$ or $\R^d$ respectively.

\begin{thm}
\label{thm:geodesic_bb_pointwise}
Let $d>6$ and $L\geq 1$ be such that critical percolation on $\mathbb{Z}^d_L$ satisfies the estimates \eqref{eq:two_point_assumption} and \eqref{eq:two_blob_assumption}. For each $k\geq 2$, $\ell \geq 0$, $\cT \in \mathsf{TrSub}(k,\ell)$, and sequence of oriented edges $e_1,\ldots,e_\ell$ emanating from the origin
we have that
\begin{equation*}
  \P_{p_c}(\Psi^{\text{\emph{sub}}}_{\cT}(x_1,\ldots,x_{k+\ell},e_1,\ldots,e_\ell))
  \sim \left[\prod_{i=1}^\ell \frac{2p_c}{1-p_c}\mathbf{W}(\emptyset,0,e_i^+)\right] V^{k-2} A^{2k-3+\ell}G_\cT^\mathrm{sub}(x_1,\ldots,x_{k+\ell})
\end{equation*}
as $\min_{i\neq j}\|x_i-x_j\|\to \infty$, where $\mathbf{W}$ is the blob factor from \cref{thm:scheme_function}.
\end{thm}

\begin{proof}[Proof of \cref{thm:geodesic_bb_pointwise}]
For each $r\geq 1$ and oriented edge $e$ of $\Z^d_L$, let $\mathscr{P}_r(e)$ be the event that
 $e$ is open and the endpoints of $e$ are connected to the complement of $e^-+[-r,r]^d$ by edge-disjoint open paths disjoint from $e$, but
closing the edge $e$ and the edges outside $e^-+[-r,r]^d$ disconnect $e^-$ and $e^+$.
Let $\Psi_{\cT}^{\mathrm{sub},r}(x_1,\ldots,x_{k+\ell},e_1,\ldots,e_\ell)$ be the event obtained by modifying the definition of $\Psi_{\cT}^{\mathrm{sub}}$ by replacing the pivotality condition on the edge $x_{k+i}+e_i$ with the local pivotality condition.
Define $\Psi_{\cT}^{\mathrm{sub},r}(x_1,\ldots,x_{2k+\ell-2},e_1,\ldots,e_\ell)$ similarly.
As in  \eqref{eq:ExpOvercountBranch}, we can write
\begin{multline*}
  \P_{p_c}(\Psi^{\mathrm{sub},r}_{\cT}(x_1,\ldots,x_{k+\ell},e_1,\ldots,e_\ell)) \\= \!\!\!\!\sum_{x_{k+\ell+1},\ldots,x_{2k+\ell-2}\in \Z^d}  \E_{p_c}\left[\frac{\1\bigl(\Psi_{\cT}(x_1,\ldots,x_{2k+\ell-2}) \cap \bigcap_{i=1}^\ell \mathscr{P}_r(x_{k+i}+e_i)\bigr)}{N^{\mathrm{sub,r}}_\cT(x_1,\ldots,x_{k+\ell})}\right]
\end{multline*}
where $N^{\mathrm{sub,r}}_\cT(x_1,\ldots,x_{k+\ell})$ is the number of sequences of points $x_{k+\ell+1},\ldots,x_{2k+\ell-2}$ such that the event $\Psi_{\cT}^{\mathrm{sub},r}(x_1,\ldots,x_{2k+\ell-2})$ holds. Exactly the same proof used to establish \cref{pro:CountBranchT} also shows that 
$N^{\mathrm{sub,r}}_\cT(x_1,\ldots,x_{k+\ell}) = \prod_{i=k+\ell+1}^{2k+\ell-2} |\cS^\Psi_{R_1,R_2}(x_i)|$ with high probability on the event  $\Psi_{\cT}(x_1,\ldots,x_{2k+\ell-2})$, in our usual triple limit in which we first send $\min_{i\neq j}\|x_i-x_j\|\to\infty$, then send $R_2\to \infty$, then send $R_1\to \infty$. (Indeed, all that has changed is that the tree $\cT$ does not belong to $\mathsf{Tr}(k)$, which was not relevant to that proof.)

Since $\mathscr{P}_r(x_{k+i}+e_i)$ is a cylinder event, it follows from this together with \cref{thm:scheme_function_variantW} and \cref{lem:DegenerateSum} exactly as in the conclusion of the proof of \cref{thm:k_point} that there exist constants $\mathbf{C}'_r(e_i)$ such that
\[
  \P_{p_c}(\Psi^{\mathrm{sub},r}_{\cT}(x_1,\ldots,x_{k+\ell},e_1,\ldots,e_\ell)) \sim 
  \left[\prod_{i=1}^\ell \mathbf{C}'_r(e_i)\right] V^{k-2} A^{2k-3+\ell}G_\cT^\mathrm{sub}(x_1,\ldots,x_{k+\ell})
\]
as $\min_{i\neq j} \|x_i-x_j\|\to \infty$ with the radius $r\geq 1$ fixed.
 Moreover, if we define the event $\mathscr{P}^i=\mathscr{P}^i_\cT(x_1,\ldots,x_{k+\ell},e_1,\ldots,e_\ell)$ to hold when  $x_{k+i}+e_i$ is an open pivotal for the connection between the relevant vertices $x_{j_1}$ and $x_{j_2}$ as in the definition of $\Psi_{\cT}^\mathrm{sub}$, it follows straightforwardly from \cref{lem:PsiArms} that  
\[
  \lim_{r\to \infty}\lim_{\min_{i\neq j} \|x_i-x_j\|\to\infty}\frac{\P_{p_c}(\Psi_{\cT}(x_1,\ldots,x_{2k+\ell-2}) \cap (\mathscr{P}_r(x_{k+i_0}+e_{i_0}) \Delta \mathscr{P}^{i_0})}{G_\cT(x_1,\ldots,x_{2k+\ell-2})} =0
\]
for each $1\leq i_0\leq \ell$, where $\Delta$ denotes the symmetric difference,
since the event in question must entail an unforced connection of the form ruled out by \cref{lem:PsiArms}. It follows from this and \cref{lem:DegenerateSum} that the limit $\lim_{r\to \infty} \prod_{i=1}^\ell \mathbf{C}'_r(e_i)$ is well-defined and satisfies
\[
  \P_{p_c}(\Psi^{\mathrm{sub}}_{\cT}(x_1,\ldots,x_{k+\ell},e_1,\ldots,e_\ell)) \sim 
  \left[\lim_{r\to \infty}\prod_{i=1}^\ell \mathbf{C}'_r(e_i)\right] V^{k-2} A^{2k-3+\ell}G_\cT^\mathrm{sub}(x_1,\ldots,x_{k+\ell})
\]
as $\min_{i\neq j} \|x_i-x_j\|\to \infty$. The fact that $\mathbf{C}'(e)=\lim_{r\to\infty}\mathbf{C}'_r(e)$ is well-defined and equal to the normalized blob factor $\frac{2p_c}{1-p_c}\mathbf{W}(\emptyset,0,e^+)$ follows by comparing the case of this formula in which $\cT$ is a path of length two with that obtained by applying \cref{thm:scheme_function} as in the proof of \cref{thm:derivative}, with the factor of $2$ arising since we did not specify the direction in which the edge $x_{k+i}+e_i$ was crossed (and the blob factor $\mathbf{W}(\emptyset,0,e^+)$ is invariant under switching $0$ and $e^+$). \qedhere
\end{proof}

We now begin to explain how \cref{thm:geodesic_bb_pointwise}  can be used to prove  \cref{thm:scaling_limit_bb} via the analysis of appropriate moments.
A technical problem arises that these moments are often infinite when we work in the full-space, with e.g.\ $\E[ d_{\mathrm{piv}}(0,x)^m ]$ governed by the $(m+1)$-gon diagram $G^{*(m+1)}(0,x)$ and therefore infinite when $2(m+1)>d$.
To circumvent this issue, given $x,y\in \Z^d$ and a continuous, compactly-supported function $\phi:\R^d\to \R$, we write 
\[d^{(\phi,r)}_\mathrm{piv}(x,y)=\sum_{e\in E(\Z^d_L)}\1(\text{$e$ an open pivotal for $x\leftrightarrow y$})\phi(e^-/r).\]
Similarly, given an instance of $(\sT,o,\nu,\Phi)$ of the continuum random tree together with its Brownian embedding $\Phi$, we write
\[
  d^{(\phi)}(x,y) = \int_0^{d(x,y)} \phi(\Phi(\gamma_{x,y}(t))) \dif t
\]
where $\gamma_{x,y}$ is the geodesic between $x$ and $y$ in $\sT$ parameterised by arc-length. We first write down a formula for moments involving both the embedding and intrinsic distances for the scaling limit, and then derive the same formulas asymptotically for critical percolation using \cref{thm:geodesic_bb_pointwise}.

\medskip

Given $k\geq 2$, let $I_k=\{(i,j):1\leq i < j \leq k\}$, equipped with the lexicographic order. We define a \textbf{$k$-multi-index} to be a function $\alpha:I_k\to \{0,1,2,\ldots\}$. Given a $k$-multi-index $\alpha$, we define $|\alpha|=\sum_{(i,j)\in I_k} \alpha_{i,j}$ and for each $1\leq m \leq |\alpha|$ define $(L^-(m),L^+(m)) $ to be the minimal element of $I_k$ with $\sum_{(i,j) \leq (L^-(m),L^+(m))} \alpha_{i,j} \geq m$. We say that a tree $\cT\in \mathsf{TrSub}(k,|\alpha|)$ is \textbf{compatible} with $\alpha$ if for each $1\leq i \leq |\alpha|$, the internal degree-two vertex of $\cT$ labelled $i$ lies on the geodesic between the leaves $L^-(i)$ and $L^+(i)$ in $\cT$, and write $\mathsf{TrSub}(\alpha)$ for the set of trees in $\mathsf{TrSub}(k,|\alpha|)$ that are compatible with the multi-index $\alpha$. 

\begin{lemma}
\label{lem:CRT_SBM_diagram_moment_formula}
Suppose that $d>4$.
 Let $k\geq 2$, let $\alpha:I_k\to \{0,1,\ldots\}$ be a multi-index, and let $\phi$ and $f_1,\ldots,f_k$ be continuous, compactly supported functions. Setting $x_1$ to be the origin in $\R^d$, we have that
\begin{multline*}
  \N\left[\idotsint \prod_{1\leq i < j\leq k} d^{(\phi)}(x_i,x_j)^{\alpha_{i,j}}  \prod_{i=2}^k  f_i(\Phi(x_i)) \dif \nu(x_2) \cdots \dif \nu(x_k)\right]
  \\ = 
  \sum_{\cT\in \mathsf{TrSub}(\alpha)} \idotsint G_{\cT,\mathrm{SBM}}^\mathrm{sub}(x_1,\ldots,x_{k+|\alpha|}) \prod_{i=2}^{k+|\alpha|} f_i(x_i) \dif x_2 \cdots \dif x_{k+|\alpha|}
\end{multline*}
where we set $f_i=\phi$ for every $k+1 \leq i \leq k+|\alpha|$.
\end{lemma}

\begin{proof}[Proof of \cref{lem:CRT_SBM_diagram_moment_formula}]
It follows from \cite[Chapter III, Theorem 4]{le1999spatial} that 
 if we ``sample'' $(\sT,\nu,o)$ from the biased measure $\N^k$ defined by $\dif \N^k/\dif \N=\frac{1}{|\mathsf{Tr}(k)|}\nu(\sT)^{k-1}$ then sample points $X_2,\ldots,X_k \in \sT$ independently at random from the probability measure $\nu/\nu(\sT)$, the subtree of $\sT$ spanned by the geodesics between the points $X_=o,\ldots,X_k$ has the same ``distribution'' as if we picked a combinatorial tree $\cT\in \mathsf{Tr}(k)$ uniformly at random then assigned each edge an independent Lebesgue$[0,\infty)$ length. (The scare quotes here are to remind the reader that $\N$ is not a probability measure, so we cannot really ``sample'' from it.) The lemma is a direct manifestation of this relationship.
\end{proof}

\begin{lemma}
\label{lem:bb_moment_asymptotics}
Let $d>6$ and $L\geq 1$ be such that critical percolation on $\mathbb{Z}^d_L$ satisfies the estimates \eqref{eq:two_point_assumption} and \eqref{eq:two_blob_assumption}.
 Let $k\geq 2$, let $\alpha:I_k\to \{0,1,\ldots\}$ be a multi-index, and let $\phi$ and $f_1,\ldots,f_k$ be continuous, compactly supported functions. If we set $x_1=0$ then
\begin{multline*}
  \E_{p_c}\left[ \frac{1}{(A^2V r^4)^{k-1}}\sum_{x_2,\ldots,x_k\in K} \prod_{1\leq i < j\leq k} \left(\frac{d^{(\phi,r)}_\mathrm{piv}(x_i,x_j)}{AG_\mathrm{piv}r^2}\right)^{\alpha_{i,j}} \prod_{i=2}^k  f_i(x_i/r)\right]
  \\ = \frac{1}{AVr^2} \left(\N\left[\idotsint \prod_{1\leq i < j\leq k} d^{(\phi)}(x_i,x_j)^{\alpha_{i,j}}  \prod_{i=2}^k  f_i(\Phi(x_i)) \dif \nu(x_2) \cdots \dif \nu(x_k)\right] +o(1)\right)
\end{multline*}
as $r\to \infty$.
\end{lemma}

\begin{proof}[Proof of \cref{lem:bb_moment_asymptotics}]
Let $E_0^\rightarrow$ denote the set of oriented edges emanating from the origin in $\Z^d_L$.
Expanding the powers of the pivotal distance writes the left-hand side as a sum over ordered collections of marked pivotal vertices lying on the backbones between the sampled cluster points. 
Grouping these terms according to the tree obtained after subdividing the corresponding geodesic segments gives
\begin{align*}
  &\E_{p_c}\left[ \frac{1}{(A^2V r^4)^{k-1}}\sum_{x_2,\ldots,x_k\in K} \prod_{1\leq i < j\leq k} \left(\frac{d^{(\phi,r)}_\mathrm{piv}(x_i,x_j)}{AG_\mathrm{piv}r^2}\right)^{\alpha_{i,j}} \prod_{i=2}^k  f_i(x_i/r)\right]
\\
&= (1+o(1))
\sum_{\cT \in \mathsf{TrSub}(\alpha)}\sum_{x_2,\ldots,x_{k+|\alpha|}\in \Z^d} \sum_{e_1,\ldots,e_{|\alpha|}\in E_0^\rightarrow}
\frac{
\P_{p_c}(\Psi^{\text{sub}}_{\cT}(x_1,\ldots,x_{k+|\alpha|},e_1,\ldots,e_{|\alpha|}))
}
{
2^{|\alpha|}A^{2k+|\alpha|-2} V^{k-1} G_\mathrm{piv}^{|\alpha|} r^{4k+2|\alpha|-4}
}
\prod_{i=2}^{k+|\alpha|} f_i\!\left(\frac{x_{i}}{r}\right) 
  \\&=
  \frac{1+o(1)}{AV}\sum_{\cT\in \mathsf{TrSub}(\alpha)} 
  \frac{1}{r^{4k+2|\alpha|-4}}\sum_{x_2,\ldots,x_{k+|\alpha|}} G_{\cT}^\mathrm{sub}(x_1,\ldots,x_{k+|\alpha|})\prod_{i=2}^{k+|\alpha|} f_i\!\left(\frac{x_{i}}{r}\right)
    \\&=
   \frac{1+o(1)}{AVr^2}\sum_{\cT\in \mathsf{TrSub}(\alpha)}
  \idotsint G_{\cT,\mathrm{SBM}}^\mathrm{sub}(x_1,\ldots,x_{k+|\alpha|})\prod_{i=2}^{k+|\alpha|} f_i(x_i) 
  \dif x_2\cdots \dif x_{k+|\alpha|}
\end{align*}
 with $f_i=\phi$ for $k+1\leq i\leq k+|\alpha|$ as above, where we applied \cref{thm:geodesic_bb_pointwise} and \cref{lem:DegenerateSum} in the penultimate line and write $E_0^\rightarrow$ for the set of oriented edges emanating from the origin in $\Z^d_L$. Here the the $o(1)$ in the second line accounts for the situation in which different pivotal edges coincide or the choice of tree is ambiguous, factor of $2^{|\alpha|}$ accounts for the fact we did not specify the direction in which the edge $x_{k+i}+e_i$ is crossed in the definition of 
$\Psi^{\text{sub}}_{\cT}(x_1,\ldots,x_{k+|\alpha|},e_1,\ldots,e_{|\alpha|})$ and
  $AG_\mathrm{piv} = \frac{p_cA}{1-p_c} \sum_{e\in E_0^\rightarrow} \mathbf{W}(\emptyset,0,e^+)$ is the normalizing constant arising from \cref{thm:geodesic_bb_pointwise}; the terms in which two marked vertices have small distance are negligible by \cref{lem:DegenerateSum}. The preceding display is then identified with the right-hand side of the lemma by \cref{lem:CRT_SBM_diagram_moment_formula}.
\end{proof}

Our next lemma gives bounds on moments in the scaling limit implying via Carleman's condition that this scaling limit is determined by these moments.

\begin{lemma} If $d>4$ then for each continuous, compactly-supported function $\phi:\R^d \to \R$ there exists a constant $C(\phi)$ such that
\label{lem:CRT_SBM_Carleman}
\begin{equation*}
  \N\left[\idotsint \prod_{1\leq i < j\leq k} d^{(\phi)}(x_i,x_j)^{\alpha_{i,j}}  \prod_{i=2}^k  \phi(\Phi(x_i)) \dif \nu(x_2) \cdots \dif \nu(x_k)\right]
  \\ \leq C(\phi)^{k+|\alpha|} (k+|\alpha|)!
\end{equation*}
for every $k\geq 2$ and multi-index $\alpha:I_k \to \{0,1,\ldots\}$.
\end{lemma}

\begin{proof}[Proof of \cref{lem:CRT_SBM_Carleman}]
It suffices to consider the case that $\phi$ is non-negative and bounded by $1$. Let $\operatorname{Supp}(\phi)$ be the support of $\phi$, and note that there exists a constant $C_1(\phi)=C_1(\operatorname{Supp}(\phi))$ such that
\[
  \int_{\operatorname{Supp}(\phi)} G_\mathrm{BM}(x,z) G_\mathrm{BM}(z,y) \dif z \leq C_1(\phi) G_\mathrm{BM}(x,y)
\]
for every $x,y\in \R^d$. Indeed, if we decompose $\operatorname{Supp}(\phi)=\{z\in \operatorname{Supp}(\phi): \|z-x\| \leq \|z-y\| \}\cup\{z\in \operatorname{Supp}(\phi): \|z-y\| \leq \|z-x\| \}= S_1 \cup S_2$ according to whether a point is closer to $x$ or to $y$ then 
\begin{multline*}
  \int_{\operatorname{Supp}(\phi)} G_\mathrm{BM}(x,z) G_\mathrm{BM}(z,y) \dif z \\ \leq C G_\mathrm{BM}(x,y) \int_{S_1} G_\mathrm{BM}(x,z) \dif z + C G_\mathrm{BM}(x,y) \int_{S_2} G_\mathrm{BM}(y,z) \dif z \leq C_1(\phi) G_\mathrm{BM}(x,y)
\end{multline*}
for some constant $C_1(\phi)$ as claimed (with $C$ a constant depending only on $d$). As such, when we expand the moments of $\N$ using the diagrammatic formula of \cref{lem:CRT_SBM_diagram_moment_formula}, we can always eliminate each internal subdivided edge at the cost of introducing a factor of $C_1(\phi)$, meaning that
\begin{multline*}
  \idotsint G_{\cT,\mathrm{SBM}}^\mathrm{sub}(x_1,\ldots,x_{k+|\alpha|}) \prod_{i=2}^{k+|\alpha|} \phi(x_i) \dif x_2 \cdots \dif x_{k+|\alpha|} \\\leq C_1(\phi)^{|\alpha|} \idotsint G_{\cT',\mathrm{SBM}}(x_1,\ldots,x_{k}) \prod_{i=2}^{k} \phi(x_i) \dif x_2 \cdots \dif x_{k}
\end{multline*}
for every $k\geq 2$, multi-index $\alpha:I_k\to \{0,1,\ldots\}$, and $\cT\in \mathsf{TrSub}(\alpha)$,
where $\cT'\in \mathsf{Tr}(k)$ is the tree obtained from $\cT \in \mathsf{TrSub}(\alpha)$ by replacing each maximal path of degree-two internal vertices by a single edge. Now, for each tree $\cT'\in \mathsf{Tr}(k)$ there are at most $(2k-3)(2k-2)\cdots(2k+|\alpha|-4) \leq 2^{|\alpha|} \frac{(k+|\alpha|)!}{k!}$ trees $\cT\in \mathsf{TrSub}(\alpha)$ obtained as subdivisions of $\cT$, and it follows that
\begin{multline*}
  \N\left[\idotsint \prod_{1\leq i < j\leq k} d^{(\phi)}(x_i,x_j)^{\alpha_{i,j}}  \prod_{i=2}^k  \phi(\Phi(x_i)) \dif \nu(x_2) \cdots \dif \nu(x_k)\right] \\\leq 2^{|\alpha|}C_1(\phi)^{|\alpha|} \frac{(k+|\alpha|)!}{k!}
  \N\left[\idotsint  \prod_{i=2}^k  \phi(\Phi(x_i)) \dif \nu(x_2) \cdots \dif \nu(x_k)\right].
\end{multline*}
(So far we have only used that $d>2$.)
It is established in \cite[Lemma I.5.34]{LRPpaper1} that the moment $\N[\idotsint  \prod_{i=2}^k  \phi(\Phi(x_i)) \dif \nu(x_2) \cdots \dif \nu(x_k)]$ is bounded by an expression of the form $C_2(\phi)^k k!$ when $d>4$, from which the result now follows.
\end{proof}

\begin{proof}[Proof of \cref{thm:scaling_limit_bb}]
Fix $\lambda>0$.
Let $C_\mathrm{piv}=G_\mathrm{piv}A$ where $G_\mathrm{piv}$ is as in \cref{thm:derivative} and $A$ is as in \eqref{eq:two_point_assumption}.
We begin by proving a compactly supported version of the claim. Let  $\mu_{K,r}=(A^2Vr^4)^{-1}\sum_{x\in K}\delta_{x/r}$, recall that $\mu$ denotes the pushforward $\mu=\Phi_*\nu$, and let $\phi$ be continuous, compactly supported, non-negative, and not identically zero.
For each integer $k\geq 2$, define a probability measure $\P_{\phi,r,k}$ by biasing $\P_{p_c}$ by $\mu_{K,r}[\phi]^{k-1}$, and define $\N_{\phi,k}$ analogously by biasing $\N$ by $\mu[\phi]^{k-1}$. Conditional on the cluster $K$, let $X_2,X_3,\ldots,X_k$ be sampled independently from the probability measure on $K$ with probability mass function proportional to $\phi(x/r)$. Similarly, conditional on the sample $(\sT,d,\nu,\Phi,o)$ from $\N_{\phi,k}$, let $Y_2,Y_3,\ldots,Y_k$ be chosen independently at random from the probability measure obtained by multiplying $\nu$ by the density $\phi\circ \Phi$. We also set $X_1=Y_1=0$.

Observe that moments involving these random points can be written in terms of moments of the unbiased measure in which points are summed over rather than chosen at random, with e.g.\
\begin{align*}
  &\E_{\phi,r,k}\left[\prod_{1\leq i < j \leq k} \left(\frac{d^{(\phi,r)}_\mathrm{piv}(X_i,X_j)}{C_\mathrm{piv}r^2}\right)^{\alpha_{i,j}} \prod_{i=2}^k P_i(X_i/r)\right] 
  \\
  &\hspace{3cm}=  
  \E_{\phi,r,k}\left[ \frac{(A^2 V r^4)^{-k+1}}{\mu_{K,r}[\phi]^{k-1}} \sum_{x_2,\ldots,x_k \in K} \prod_{1\leq i < j \leq k} \left(\frac{d^{(\phi,r)}_\mathrm{piv}(x_i,x_j)}{C_\mathrm{piv}r^2}\right)^{\alpha_{i,j}} \prod_{i=2}^k P_i(x_i/r) \phi(x_i/r) \right] 
  \\
  &\hspace{3cm}=
 \frac{(A^2 V r^4)^{-k+1}}{\E_{p_c}[\mu_{K,r}[\phi]^{k-1}]} \E_{p_c}\left[ \sum_{x_2,\ldots,x_k \in K} \prod_{1\leq i < j \leq k} \left(\frac{d^{(\phi,r)}_\mathrm{piv}(x_i,x_j)}{C_\mathrm{piv}r^2}\right)^{\alpha_{i,j}} \prod_{i=2}^k P_i(x_i/r) \phi(x_i/r) \right]
\end{align*}
for each $k\geq 2$ and polynomials $P_2,\ldots,P_k$.
As such, for each fixed $k\geq 2$, it follows from \cref{lem:CRT_SBM_diagram_moment_formula,lem:bb_moment_asymptotics} 
 that all the moments of the array
%
$((d^{(\phi,r)}_\mathrm{piv}(X_i,X_j)/(C_\mathrm{piv}r^2)_{1\leq i\leq j\leq k},(X_i/r)_{1\leq i\leq k})$
under the measure $\P_{\phi,r,k}$ converge to those of the array
$((d^{(\phi)}(Y_i,Y_j))_{1\leq i\leq j\leq k},(\Phi(Y_i))_{1\leq i\leq k})$
under the measure $\N_{\phi,k}$. Moreover, the same is true for the arrays
\[
\left(\mu_{K,r}[\phi],\,\left(\frac{d^{(\phi,r)}_\mathrm{piv}(X_i,X_j)}{C_\mathrm{piv}r^2}\right)_{1\leq i\leq j\leq k},\,\left(\frac{X_i}{r}\right)_{1\leq i\leq k}\right)
\text{ and }
\left(\nu[\phi],\,\left(d^{(\phi)}(Y_i,Y_j)\right)_{1\leq i\leq j\leq k},\,(\Phi(Y_i))_{1\leq i\leq k}\right)
\]
under the measures $\P_{\phi,r,k}$ and $\N_{\phi,k}$ as can be seen by summing over additional points in the equalities above. Using \cref{lem:CRT_SBM_Carleman}, it follows from this and Carleman's criterion that the same convergence statement holds weakly as well as in the sense of moments. Since $\N_{\phi,k}$ is supported on configurations with $\mu[\phi]>0$ and division by $\mu[\phi]^{k-1}$ is continuous, it follows by the same argument used in the proof of \cref{thm:integrated_SBM} that the same convergence of arrays holds with respect to the measures $\P^{\phi,\lambda,r}$ and $\N^{\phi,\lambda}$ defined by conditioning $\P_{p_c}$ and $\N$ on the events $\{\mu_{K,r}[\phi] \geq \lambda\}$ and $\{\mu[\phi] \geq \lambda\}$ respectively, with the number of points $k$ in the array remaining an arbitrary integer $k\geq 2$ and with the convergence holding both weakly and in the sense of moments.

 We now remove the compact support. 
  Let $\phi_M:\R^d\to [0,1]$ be a continuous bump function, equal $1$ on $\{x:\|x\|\leq M\}$ and vanishing on $\{x:\|x\|\geq 2M\}$. The Kozma-Nachmias one-arm estimate \cite{MR2748397} and the volume-tail asymptotic \eqref{eq:volume_tail_asymptotics_proof} imply that, after conditioning on $|K|\geq \lambda A^2Vr^4$ , the  probability that $K$ is not contained in the ball $\{x:\|x\|\leq Mr\}$ is $O(\lambda^{1/2} M^{-2})$, uniformly in  $r\geq 1$. On the complementary event in which the cluster is contained in this ball, choosing a uniform point from the cluster $K$ is equivalent to choosing a point of $K$ proportional to $\phi_M(x/r)$ and the truncated pivotal distance $d_\mathrm{piv}^{(\phi_M,r)}$ is equal to the true pivotal distance $d_\mathrm{piv}$ for all pairs of points in the cluster. Since we also have that $\mathbb N(\mu(\{x:\|x\|\geq M\}) \neq 0\mid \nu(\mathscr T)\geq\lambda) = O(\lambda^{1/2} M^{-2})$, this is easily seen to imply the desired weak convergence in conjunction with the weak convergence of the conditional measures $\P^{\phi,\lambda,r}$ as in the proof of \cref{thm:integrated_SBM}.
\end{proof}

\subsection{Comparing intrinsic distances}

In this section we complete the proof of \cref{thm:scaling_limit_intro} by proving the asymptotic equivalence of different intrinsic distances (after appropriate constant rescaling) on large critical clusters.

\begin{thm} \label{thm:CompareDistance} 
Let $d>6$ and $L\geq 1$ be such that critical percolation on $\mathbb{Z}^d_L$ satisfies the estimates \eqref{eq:two_point_assumption} and \eqref{eq:two_blob_assumption}.
For each $\# \in \{\mathrm{piv},\mathrm{res},\mathrm{chem}\}$ there exists a positive constant $G_\#$, which we call the \textbf{geodesic factor} associated to the distance $d_\#$, such that
\[
  \E_{p_c} \left[d_\#(x,y) \mathbbm{1}(x\leftrightarrow y)\right] \sim G_\# A^2 \cdot G^{*2}(x,y), \qquad \E_{p_c} \left[d_\#(x,y)^2 \mathbbm{1}(x\leftrightarrow y)\right] \sim 2G_\#^2 A^3 \cdot G^{*3}(x,y)
\]
and
\[ \E_{p_c}\left [\left (d_\#(x,y)-\frac{G_\#}{G_\mathrm{piv}} \cdot d_\mathrm{piv}(x,y)\right)^2 \mathbbm{1}(x\leftrightarrow y) \right ]=o(G^{*3}(x,y)) \]
as $\|x-y\|\to \infty$.
\end{thm}

We will apply this theorem primarily through the following corollary, which immediately implies \cref{thm:scaling_limit_intro} in conjunction with \cref{thm:scaling_limit_bb} and which is proven at the end of this section.

\begin{corollary}
\label{cor:CompareDistance}
Let $d>6$ and $L\geq 1$ be such that critical percolation on $\mathbb{Z}^d_L$ satisfies the estimates \eqref{eq:two_point_assumption} and \eqref{eq:two_blob_assumption}. Given $K$, let $X_2,X_3,\ldots$ be independent uniform random elements of $K$ and let $X_1=0$. Then
\[
  \lim_{r\to\infty} \P_{p_c}\left( \max_{1\leq i,j \leq k}|G_\# d_\mathrm{piv}(X_i,X_j)- G_\mathrm{piv} d_\#(X_i,X_j)| \geq \eps r^2 \;\middle|\; |K| \geq \lambda C_\mathrm{vol} r^4\right)=0
\]
for every $\eps,\lambda>0$, $k<\infty$ and $\#\in \{\mathrm{chem},\mathrm{res}\}$.
\end{corollary}

\begin{remark}
In \cref{lem:uniform_distance_comparison} we strengthen the conclusion of this lemma to show that 
\[\lim_{r\to\infty}\P\left(\sup_{u,v\in K}
|G_\# d_\mathrm{piv}(u,v)-G_\mathrm{piv}d_\#(u,v)| \geq \eps r^2 \;\middle|\; |K|\geq r^4 \right)=0\]
for every $\eps>0$ and $\#\in \{\mathrm{chem},\mathrm{res}\}$.
\end{remark}

We now begin working towards the proof of \cref{thm:CompareDistance}.
 For each pair of vertices $x$ and $y$ belonging to the same cluster, we write $\mathrm{Piv}(x,y)$ for the set of open pivotal edges for the connection $x\leftrightarrow y$. We enumerate this set as $\mathrm{Piv}(x,y)=\{\mathrm{Piv}_1(x,y),\mathrm{Piv}_2(x,y),$ $\ldots,\mathrm{Piv}_{d_\mathrm{piv}(x,y)}(x,y)\}$ according to the order in which an open simple path from $x$ to $y$ crosses the pivotal edges, and for each $1\leq i \leq d_\mathrm{piv}(x,y)$ let $\mathrm{Piv}_i(x,y)^-$ and $\mathrm{Piv}_i(x,y)^+$ denote the endpoints of the pivotal closer to $x$ and closer to $y$ respectively on an open path from $x$ to $y$.
For each $0\leq i \leq d_\mathrm{piv}(x,y)$ we also define the \textbf{$i$th sausage} $\sS_i(x,y)$ to be the connected component of $K \setminus \mathrm{Piv}(x,y)$ containing both $\mathrm{Piv}_i(x,y)^+$ and $\mathrm{Piv}_{i+1}(x,y)^-$, where we set $\mathrm{Piv}_0(x,y)^+=x$ and $\mathrm{Piv}_{d_\mathrm{piv}(x,y)+1}(x,y)^-=y$, and define $R_i(x,y)$ to be the diameter of $\sS_i(x,y)$ in $\Z^d$. (Note that in this terminology sausages are not necessarily $2$-edge-connected.)

  Each of the distances we consider can be written as a sum over distances between successive open pivotals:
\[
  d_\#(x,y) = \sum_{0\leq i \leq d_\mathrm{piv}(x,y)} \left[d_\#(\mathrm{Piv}_i(x,y)^+,\mathrm{Piv}_{i+1}(x,y)^-)+\mathbbm{1}(i\geq 1)\right].
\]
We decompose the intrinsic distance $d_{\#}$ as $d_{\#}^{>r}+d_{\#}^{\leq r}$, where 
\begin{align*} d^{\leq r}_{\#}(x,y) := \sum_{0\leq i \leq d_\mathrm{piv}(x,y)} \left[d_\#(\mathrm{Piv}_i(x,y)^+,\mathrm{Piv}_{i+1}(x,y)^-) +\mathbbm{1}(i\geq 1)\right] \mathbbm{1}(R_i \leq r)
\end{align*}
and
\begin{align*}
d^{> r}_{\#}(x,y) := 
\sum_{0\leq i \leq d_\mathrm{piv}(x,y)} \left[d_\#(\mathrm{Piv}_i(x,y)^+,\mathrm{Piv}_{i+1}(x,y)^-) +\mathbbm{1}(i\geq 1)\right] \mathbbm{1}(R_i > r).
  \end{align*}
  To prove \cref{thm:CompareDistance} we will prove a similar result for $d_\#^{\leq r}$ for each fixed $r\geq 1$ and prove separately that
$d_{\#}^{>r}$ is asymptotically negligible when $r$ is large.

\begin{lemma}
\label{lem:two_intrinsic_distances_truncated}
Let $d>6$ and $L\geq 1$ be such that critical percolation on $\mathbb{Z}^d_L$ satisfies the estimates \eqref{eq:two_point_assumption} and \eqref{eq:two_blob_assumption}. For each $\# \in \{\mathrm{piv},\mathrm{chem},\mathrm{res}\}$ and $r\geq 1$ there exists a constant $G_{\#}^{\leq r}$ such that
\[
\E_{p_c}\left[d_\#^{\leq r}(x,y)\mathbbm{1}(x\leftrightarrow y)\right] \sim G_{\#}^{\leq r} A^2 \cdot G^{*2}(x,y) \]
and
\[
\E_{p_c}\left[d_\#^{\leq r}(x,y)d_{\#'}^{\leq r}(x,y)\mathbbm{1}(x\leftrightarrow y)\right] \sim 2 G_{\#}^{\leq r} G_{\#'}^{\leq r} A^3 \cdot G^{*3}(x,y)
\]
as $\|x-y\|\to \infty$ for each fixed $\#,\#'\in \{\mathrm{piv},\mathrm{chem},\mathrm{res}\}$ and $r\geq 1$.
\end{lemma}

Since $d_\#^{>r}(x,y) \leq d_\mathrm{chem}^{>r}(x,y)$ for each $\#\in \{\mathrm{piv},\mathrm{chem},\mathrm{res}\}$, it suffices to state and prove our negligibility lemma only for $\#=\mathrm{chem}$.

\begin{lemma}
\label{lem:sausage_functional_truncation_removal}
Let $d>6$ and $L\geq 1$ be such that critical percolation on $\mathbb{Z}^d_L$ satisfies the estimates \eqref{eq:two_point_assumption} and \eqref{eq:two_blob_assumption}. For each $\eps>0$ there exists $r<\infty$ such that
\[
  \E_{p_c}\left[d_\mathrm{chem}^{> r}(x,y)\mathbbm{1}(x\leftrightarrow y)\right] \leq \eps G^{*2}(x,y) \quad \text{ and } \quad \E_{p_c}\left[d_\mathrm{chem}^{> r}(x,y)d_\mathrm{chem}(x,y)\mathbbm{1}(x\leftrightarrow y)\right] \leq \eps G^{*3}(x,y)
\]
for every $x\neq y$ in $\Z^d$.
\end{lemma}

Let us now see how these two lemmas imply \cref{thm:CompareDistance}.

\begin{proof}[Proof of \cref{thm:CompareDistance}]
It follows from \cref{lem:sausage_functional_truncation_removal} that the constants $G_{\#}^{\leq r}$ introduced in \cref{lem:two_intrinsic_distances_truncated} are Cauchy as $r\to \infty$ and hence that the limit $G_\# =\sup_{r\geq 1} G_{\#}^{\leq r}$ is well-defined and satisfies the claimed asymptotic estimate 
$\E \left[d_\#(x,y) \mathbbm{1}(x\leftrightarrow y)\right] \sim G_\# A^2 \cdot G^{*2}(x,y)$ as $\|x-y\|\to \infty$ for each $\# \in \{\mathrm{piv},\mathrm{chem},\mathrm{res}\}$ as claimed. 
Since
\[
  \E[d_\#^{>r}(x,y)d_{\#'}(x,y) \mathbbm{1}(x\leftrightarrow y)] \leq \E[d_\mathrm{chem}^{>r}(x,y)d_{\mathrm{chem}}(x,y) \mathbbm{1}(x\leftrightarrow y)],
\]
for every $\#\in \{\mathrm{piv},\mathrm{chem},\mathrm{res}\}$, 
the claims regarding second moments and covariances also follow easily from the second part of \cref{lem:two_intrinsic_distances_truncated} and the second estimate of \cref{lem:sausage_functional_truncation_removal}.
\end{proof}

We now prove \cref{lem:two_intrinsic_distances_truncated,lem:sausage_functional_truncation_removal}. We begin with \cref{lem:two_intrinsic_distances_truncated}, which is an easy consequence of \cref{thm:scheme_function}.

\begin{proof}[Proof of \cref{lem:two_intrinsic_distances_truncated}]
For each $r\geq 1$, let $\mathscr{C}_r$ be the set of pairs $\cH=(H,e_1,e_2)$ consisting of connected subgraphs $H$ of $\Z^d_L$ of diameter at most $r$ and edges $e_1,e_2$ in the boundary of $H$ (i.e., with one endpoint in the vertex set of $H$ and the other outside of the vertex set of $H$) with the endpoint of $e_1$ contained in the vertex set of $H$ being equal to the origin and with the endpoints of $e_1$ and $e_2$ contained in $H$ being $2$-edge-connected in $H$; this normalization ensures that the set $\mathscr{C}_r$ is finite. For each such triple $\cH=(H,e_1,e_2)$, 
we define an associated blob $(W,b_1,b_2)$ in which $W$ is the set of edges incident to but not contained in $H$ and $b_1$ and $b_2$ are the endpoints of $e_1$ and $e_2$ not belonging to the vertex set of $H$. If we also define $\Delta_\#(\cH)$ to be the relevant intrinsic distance in $H$ between the endpoints of $e_1$ and $e_2$ belonging to $H$, we can write $\E_{p_c} d_\#^{\leq r}(x,y)\mathbbm{1}(x\leftrightarrow y)$ as a weighted sum of scheme functions
\begin{multline*}
  \E_{p_c} d_\#^{\leq r}(x,y)\mathbbm{1}(x\leftrightarrow y) = \sum_{\cH \in \mathscr{C}_r} \frac{p_c^{|E(H)|+2}}{(1-p_c)^2} (\Delta_\#(\cH)+1) \sum_{z \in \Z^d} S_{\cH}(x,z,y)
  \\
   +(\text{contributions from first and last sausages})
\end{multline*}
where we write $S_{\cH}$ for the scheme whose plan is a path with three vertices and with associated blobs being the trivial blob at the two endpoints and the blob $(W,b_1,b_2)$ defined from $\cH$ associated to the midpoint, and where we have split off the contribution from the first and last sausages since they do not fit as neatly into our scheme-functional  language. Given this formula, the claim about the first moment is an immediate consequence of \cref{thm:scheme_function} together with the usual argument showing that the contribution from $z$ within bounded distance of $x,y$ is negligible as $\|x-y\|\to \infty$. (The contributions from the first and last sausage are also easily seen to be of order $G(x,y)$ and hence negligible.)

For the second moment, we have similarly that if $\#,\#'\in \{\mathrm{piv},\mathrm{chem},\mathrm{res}\}$ then
\begin{multline*}
  \E_{p_c} \left[d_\#^{\leq r}(x,y)d_{\#'}^{\leq r}(x,y)\mathbbm{1}(x\leftrightarrow y)\right]= \sum_{\cH\in \mathscr{C}_r} \frac{p_c^{|E(H)|+2}}{(1-p_c)^2} (\Delta_\#(\cH)+1) (\Delta_{\#'}(\cH)+1) \sum_{z \in \Z^d} S_{\cH}(x,z,y)
  \\
  +2\sum_{\cH,\cH' \in \mathscr{C}_r} \frac{p_c^{|E(H)|+|E(H')|+4}}{(1-p_c)^4} (\Delta_\#(\cH)+1)(\Delta_{\#'}(\cH')+1)
  \sum_{z_1,z_2 \in \Z^d} S_{\cH,\cH'}(x,z_1,z_2,y)
  \\
  +(\text{contributions from first and last sausages})
\end{multline*}
for every $x,y\in \Z^d$ and $r\geq 1$, 
where $S_{\cH,\cH'}$ is the scheme whose plan is a path with vertex set $\{1,2,3,4\}$, with endpoints $1$ and $4$ associated to the trivial $1$-blob $(\emptyset,0)$, and with interior vertices $2$ and $3$ associated to the blobs $(W,b_1,b_2)$ and $(W',b_1',b_2')$ defined from $\cH=(H,e_1,e_2)$ and $\cH'=(H',e_1',e_2')$ respectively. The claimed asymptotics now follow from \cref{thm:scheme_function} by a similar argument to the first moment, where we also use that the first term (which counts the contribution to the second moment from choosing the same sausage twice) is of order $G^{*2}(x,y)=o(G^{*3}(x,y))$ and hence negligible.
\end{proof}

We now begin to work towards the proof of \cref{lem:sausage_functional_truncation_removal}. We begin with the following lemma which shows that the contribution from large $2$-edge-connected components is negligible.

\begin{lemma} \label{lem:TriangleCountCycle} 
Let $d>6$ and $L\geq 1$ be such that critical percolation on $\mathbb{Z}^d_L$ satisfies the estimates \eqref{eq:two_point_assumption} and \eqref{eq:two_blob_assumption}. For each $\eps>0$ there exists $\ell<\infty$ such that
\[\sum_{z\in \Z^d} \proba(\{x\slr z\}\circ \{z\slr y\} \text{ and $z$ belongs to an edge-simple cycle of diameter at least $\ell$})\leq \eps G^{*2}(x,y) \]
for every $x,y\in \Z^d$.
\end{lemma}

\begin{proof}[Proof of \cref{lem:TriangleCountCycle}] Fix $\eps>0$. Taking a union bound over the first place the path from $x$ to $z$ and the last place the path from $y$ to $z$ intersects the large loop and applying the BK inequality, the relevant sum is bounded by 
\begin{equation} \sum_{u,v,z \in \Z^d} \proba(x\slr u) \proba(v\slr y)\proba( \text{$u,v,z$ belong to some edge-simple cycle of diameter at least $\ell$}). \label{eq:BKopenCycle2arm} 
\end{equation}
If $u,v,z$ all belong to the same edge-simple cycle, the three arcs between them define disjoint witnesses of the events $\{u\leftrightarrow v\}$, $\{v\leftrightarrow z\}$, and $\{z\leftrightarrow u\}$. Thus, the probability that $u,v,z$ all belong to the same edge-simple cycle is bounded by the product of the two-point functions $\proba(u\slr v)\proba(v\slr z)\proba(z\slr u) \preceq \langle u-v\rangle^{-d+2}\langle v-z\rangle^{-d+2}\langle z-u\rangle^{-d+2}$. Since $d>6$, it follows by standard convolution estimates that there exists $r<\infty$ such that
\begin{multline}
  \sum_{u,v,z \in \Z^d} \mathbbm{1}(d_\mathrm{max}(u,v,z) \geq r) \proba(x\slr u) \proba(v\slr y)\proba( \text{$u,v,z$ belong to some edge-simple cycle}) \\\leq \frac{\eps}{2} G^{*2}(x,y).
  \label{eq:long_loop_separated}
\end{multline}
On the other hand, if $u,v,z$ all have pairwise distances at most $r$ but lie in an edge-simple cycle of diameter at least $\ell$ we must have that $K_u$ has diameter at least $\ell$, and since this event has probability tending to zero as $\ell\to \infty$ it follows that there exists $\ell<\infty$ (depending on our earlier choice of $r$) such that
\begin{multline}
  \sum_{\substack{u,v,z \in \Z^d\\d_\mathrm{max}(u,v,z) \leq r}} \proba(x\slr u) \proba(v\slr y)\proba( \text{$u,v,z$ belong to some edge-simple cycle of diameter  $\geq \ell$}) 
  \\\leq
 \sum_{\substack{u,v,z \in \Z^d\\d_\mathrm{max}(u,v,z) \leq r}}  \proba(x\slr u) \proba(v\slr y)\proba(K_u \text{ has diameter } \geq\ell)
 \leq 
   \frac{\eps}{2} G^{*2}(x,y).
   \label{eq:long_loop_close}
\end{multline}
(Note that this bound is extremely coarse but suffices for our purposes.) The two estimates \eqref{eq:long_loop_separated} and \eqref{eq:long_loop_close} yield the claim. \qedhere
\end{proof}

\begin{proof}[Proof of \cref{lem:sausage_functional_truncation_removal}]
We first prove the claim regarding 
$\E[d_\mathrm{chem}^{>r}(x,y) \mathbbm{1}(x\leftrightarrow y)]$. 
Fix $\eps>0$.
When $x$ is connected to $y$, define the \textbf{backbone} $\mathrm{bb}(x,y)$ between $x$ and $y$ to be the set of vertices $z\in \Z^d$ such that $\{x\leftrightarrow z\}\circ\{z\leftrightarrow y\}$ holds. Each point $z$ on the backbone $\mathrm{bb}(x,y)$ belongs to one of the sausages $\mathscr{S}_i(x,y)$, and we define $R(z)$ to be the diameter of this sausage in $\Z^d$, so that
\begin{equation}
\E[d_\mathrm{chem}^{>r}(x,y) \mathbbm{1}(x\leftrightarrow y)]\leq \E \#\{z\in \mathrm{bb}(x,y):R(z) > r\}.
\label{eq:sausage_truncation1}
\end{equation}
 We also define the graphs $\mathscr{B}_0(x,y),\ldots,\mathscr{B}_{d_\mathrm{piv}(x,y)}(x,y)$ to be the subgraphs of the appropriate sausages $\mathscr{S}_0(x,y),\ldots,\mathscr{S}_{d_\mathrm{piv}(x,y)}(x,y)$ induced by vertices of the backbone, so that each of the graphs $\mathscr{B}_i(x,y)$ is $2$-edge-connected, and for each $z\in \mathrm{bb}(x,y)$ define $R^{\mathrm{bb}}(z)$ to be the diameter in $\Z^d$ of the associated graph $\mathscr{B}_i(x,y)$. Since each of the graphs $\mathscr{B}_i(x,y)$ is $2$-edge-connected, if $R^{\mathrm{bb}}(z)\geq r$ then $z$ must belong to an edge-simple cycle of diameter at least $r$. As such, it follows from \cref{lem:TriangleCountCycle} that there exists $r_1<\infty$ such that
 \begin{equation}\E \#\{z\in \mathrm{bb}(x,y):R^{\mathrm{bb}}(z) > r_1\} \leq \frac{\eps}{2} G^{*2}(x,y)
\label{eq:sausage_truncation2}
 \end{equation}
 for every $x,y\in \Z^d$. On the other hand, if $z\in \mathrm{bb}(x,y)$ and $R^{\mathrm{bb}}(z) \leq r_1$ but $R(z)\geq r_2$ for some $r_2\geq r_1$, then there must exist a crossing of the annulus $\{w \in \Z^d: r_1\leq \|w-z\| \leq r_2\}$ that is disjoint from the backbone, so that the event $\{x\leftrightarrow z\}\circ \{z \leftrightarrow y\}$ occurs disjointly from this annulus crossing event. Thus, the BK inequality implies if $r_2$ is sufficiently large (depending on $\eps$ and the parameter $r_1$ that was already chosen) then 
\begin{equation}\E \#\{z\in \mathrm{bb}(x,y):R^{\mathrm{bb}}(z) \leq r_1, R(z)\geq r_2\} \leq \frac{\eps}{2} G^{*2}(x,y).
\label{eq:sausage_truncation3}
 \end{equation}
 The claim regarding the first moment follows from \eqref{eq:sausage_truncation1}, \eqref{eq:sausage_truncation2}, and \eqref{eq:sausage_truncation3}. 
For the second estimate, we observe as before that
\begin{multline*}
  \E[d_\mathrm{chem}^{>r}(x,y)d_\mathrm{chem}(x,y) \mathbbm{1}(x\leftrightarrow y)] \\\leq \E \sum_{z,w} \mathbbm{1}(\{x\leftrightarrow z\}\circ \{z\leftrightarrow w\} \circ \{w\leftrightarrow y\})[\mathbbm{1}(R(z)>r)+\mathbbm{1}(R(w)>r)]
\end{multline*}
and note that this can be bounded using essentially the same argument as for the first moment.
\end{proof}

\begin{proof}[Proof of \cref{cor:CompareDistance}]
It follows by translation invariance (and more specifically the mass-transport principle) that the conditional distributions of the pair $(d_\mathrm{piv}(X_i,X_j),d_\#(X_i,X_j))$ given $|K| \geq \lambda C_\mathrm{vol} r^4$ are all equal to the conditional distribution of the pair $(d_\mathrm{piv}(0,X_2),d_\#(0,X_2))$ given $|K| \geq \lambda C_\mathrm{vol} r^4$, so that it suffices to prove that $r^{-2}|G_\# d_\mathrm{piv}(0,X)-G_\mathrm{piv} d_\#(0,X)|$ is small with high probability conditioned on $|K| \geq \lambda C_\mathrm{vol} r^4$ when $X$ is a uniform random element of $K$ and $r$ is large. It follows immediately from \cref{thm:CompareDistance} that
\begin{align*}
\E \left[\sum_{\|x\|\leq M r} \mathbbm{1}(0\leftrightarrow x) (G_\# d_\mathrm{piv}(0,x)-G_\mathrm{piv} d_\#(0,x))^2\right] = o(r^6)
\end{align*}
as $r\to \infty$ for each $M>0$. Thus, it follows from Markov's inequality that
\begin{align*}
\P\left(\#\{\|x\|\leq M r: 0\leftrightarrow x \text{ and } |G_\# d_\mathrm{piv}(0,x)-G_\mathrm{piv} d_\#(0,x)| \geq \eps r^2\} \geq \eps r^4\right) = o(r^{-2})
\end{align*}
as $r\to \infty$ for each fixed $M,\eps>0$. The claim follows from this and the fact that $K$ is unlikely to contain a point with $\|x\|> Mr$ under the conditional measure $\P_{p_c}( \,\cdot\, \mid |K|\geq \lambda C_\mathrm{vol} r^4)$ when $M$ is much larger than $\lambda^{1/4}$; this last fact is an easy consequence of the estimates $\P(|K| \geq n) \asymp n^{-1/2}$ and $\P(0\leftrightarrow \Z^d \setminus [-r,r]^d) \asymp r^{-2}$ that was already used in the proof of \cref{thm:scaling_limit_bb}.
\end{proof}

\begin{proof}[Proof of \cref{thm:scaling_limit_intro}]
The theorem is an immediate consequence of \cref{thm:scaling_limit_bb} and \cref{cor:CompareDistance} together with the fact that \eqref{eq:two_point_assumption} and \eqref{eq:two_blob_assumption} hold under the perturbative hypotheses of the theorem. (Of course the conclusions of the theorem hold whenever $d>6$ and \eqref{eq:two_point_assumption} and \eqref{eq:two_blob_assumption} both hold as stated in \cref{thm:k_point_A_and_B}.)
\end{proof}

\section{Stronger topologies of convergence}
\label{sec:tightness}

In this section we state and prove a stronger version of \cref{thm:scaling_limit_intro} in which the weak array topology is replaced by the \emph{rooted Gromov-Hausdorff-Prokhorov-function} (rooted GHPf) topology as introduced in \cite{ambrosio2025tightness,djankovic2026scaling} (see also \cite{croydon2008convergence,croydon2018scaling}). 
 The main advantage of this improved scaling limit theorem is that it enables studying the limiting distribution of quantities such as the \emph{diameter} of large clusters, which are not continuous with respect to the weak array topology. See e.g.\ the introduction of \cite{archer2024ghp} for a detailed discussion of the benefits of GHP convergence over Gromov-weak convergence in the context of high-dimensional random graphs converging to the continuum random tree.
 We write $\bar d_\#(x,y)=d_\#(x,y)+\mathbbm{1}(x\neq y)$ for each $\#\in \{\mathrm{piv},\mathrm{chem},\mathrm{res}\}$ to make $\bar d_\mathrm{piv}$ a metric rather than a semimetric.

\begin{theorem}[The scaling limit of high-dimensional percolation clusters]
\label{thm:scaling_limit_GHPf}
Let $d>6$ and $L\geq 1$ be such that critical percolation on $\mathbb{Z}^d_L$ satisfies the estimates \eqref{eq:two_point_assumption} and \eqref{eq:two_blob_assumption}. For each $\# \in \{\mathrm{chem},\mathrm{piv},\mathrm{res}\}$ the convergence
\begin{equation*}
\P_{p_c}\left( \left(K,\frac{\bar d_\#}{C_\# r^2}, \frac{\mu_K}{C_\text{\emph{vol}}r^4}, \frac{1}{r}\Phi_K,0 \right)\in \cdot \;\Bigg|\; |K| \geq \lambda C_\mathrm{vol} r^4 \right) 
\\\longrightarrow \mathbb{N}\Bigl((\mathscr{T},d,\nu,\Phi,o) \in \cdot \mid \nu(\sT) \geq \lambda\Bigr)
\end{equation*}
holds weakly with respect to the rooted Gromov-Hausdorff-Prokhorov-function topology as $r\to \infty$ for each $\lambda>0$.
\end{theorem}

We now recall the relevant definitions. Fix a target metric space $(Y,d_Y)$, which we will always take to be $\R^d$ equipped with the Euclidean metric. We will define a metric on isomorphism classes of tuples $(X,d,\mu,f,o)$ where $(X,d)$ is a compact metric space, $\mu$ is a finite Borel measure on $X$, $f:X\to Y$ is continuous, and $o\in X$ is a point; an isomorphism of this data is required to simultaneously preserve the metric, the measure, and the choice of root. Note that the measure $\mu$ is \emph{not} required to have full support in $X$.
Given two such (isomorphism class representatives of) tuples $(X_1,d_1,\mu_1,f_1,o_1)$ and  $(X_2,d_2,\mu_2,f_2,o_2)$, we define the rooted GHPf distance
$d_\mathrm{rGHPf}((X_1,d_1,\mu_1,f_1,o_1),(X_2,d_2,\mu_2,f_2,o_2))$ 
 to be the infimum over all isometric embeddings $\psi_i:X_i\to W$ into a common compact metric space $(W,d_W)$ of the quantity
\[
 d_H^W(\psi_1(X_1),\psi_2(X_2))
 +d_P^W((\psi_1)_*\mu_1,(\psi_2)_*\mu_2)
 +d_W(\psi_1(o_1),\psi_2(o_2))
 +\Delta_f(f_1\circ \psi^{-1}_1,f_2\circ \psi_2^{-1}),
\]
where $d_H^W$ and $d_P^W$ denote the Hausdorff and Prokhorov distances in $W$ and $\Delta_f(f_1\circ \psi^{-1}_1,f_2\circ \psi_2^{-1})$ is the infimal value of $\eps>0$ such that for each $i\in\{1,2\}$ and $x_i\in X_i$ there exists $x_{3-i}\in X_{3-i}$ with
\[
 d_W(\psi_i(x_i),\psi_{3-i}(x_{3-i}))\leq \eps
 \quad\text{and}\quad
 \|f_i(x_i)-f_{3-i}(x_{3-i})\|\leq \eps.
\]
Note that convergence of a sequence of tuples $(X_n,d_n,\mu_n,f_n,o_n)$ to a tuple $(X,d,\mu,f,o)$ in the rooted GHPf topology implies the convergence of the rooted metric-measure spaces $(X_n,d_n,\mu_n,o_n)$ to $(X,d,\mu,o)$ in the Gromov-Hausdorff-Prokhorov topology and convergence of the set-measure pairs $(f_n(X_n),(f_n)_*\mu_n)$ to the set-measure pair $(f(X),f_*\mu)$ in the Hausdorff-Prokhorov topology on $Y$.

\begin{remark}
\cref{thm:scaling_limit_GHPf} implies in particular that the cluster of the origin converges \emph{as a set} to the relevant super-Brownian limit object as $r\to \infty$ when we condition on $|K| \geq \lambda C_\mathrm{vol} r^4$. This convergence was established for a number of other high-dimensional models by Holmes and Perkins \cite{holmes2020range}, but remained open for (unoriented) percolation.
\end{remark}


\subsection{Tightness criteria and their verification}

In this section we deduce \cref{thm:scaling_limit_GHPf} from the corresponding statement for the weak array topology (\cref{thm:scaling_limit_intro,thm:k_point_A_and_B}). The main additional step required is to prove \emph{tightness} in the rooted GHPf topology, the proof of which will rely on the following characterisation of precompact sets in rooted GHPf spaces.
Recall that a metric space is said to be \textbf{proper} if all its closed bounded subsets are compact.

\begin{prop}[Compactness in rooted GHPf]
\label{prop:GHPf_compactness}
Let $(Y,d_Y)$ be a proper metric space and let $\mathcal{K}$ be a collection of isomorphism class representatives of tuples $(X,d,\mu,f,o)$ of rooted compact metric-measure spaces equipped with continuous functions $f:X\to Y$. The set $\mathcal{K}$ has compact closure in the rooted GHPf topology if and only if the following conditions all hold:
\begin{enumerate}
  \item The metric spaces $(X,d)$ appearing in $\mathcal{K}$ have bounded diameter and are uniformly totally bounded in the sense that for each $\eps>0$ there exists $N_\eps<\infty$ such that every metric space $(X,d)$ appearing in $\mathcal{K}$ can be covered by $N_\eps$ balls of radius $\eps$.
  \item The measures $\mu$ have bounded total mass.
  \item The functions $f$ appearing in $\mathcal{K}$ are bounded and equicontinuous in the sense that the set $\{f(x):(X,d,\mu,f,o)\in \mathcal{K},\, x\in X\}$ is bounded in $Y$ and for every $\eps>0$ there exists $\delta$ such that $d_Y(f(x),f(y)) \leq \eps$ for every $(X,d,\mu,f,o)\in \mathcal{K}$ and every $x,y\in X$ with $d(x,y) \leq \delta$.
\end{enumerate}
\end{prop}

\begin{proof}[Proof of \cref{prop:GHPf_compactness}]
This follows easily from \cite[Proposition A.4]{ambrosio2025tightness}, which gives the same statement for the rooted GHf topology (i.e., without considering the measure $\mu$).
\end{proof}

We will apply \cref{prop:GHPf_compactness} mostly via the following sufficient condition for compactness, essentially due to \cite{athreya2016gap} and known as the \emph{lower mass condition}. (Note that item 1 of this sufficient condition is implied by items 2 and 3 for geodesic metric spaces.)

\begin{prop}\label{prop:GHPf_compactness_lower_mass_version}
Let $(Y,d_Y)$ be a proper metric space and let $\mathcal{K}$ be a collection of isomorphism class representatives of tuples $(X,d,\mu,f,o)$ of rooted compact metric-measure spaces equipped with continuous functions $f:X\to Y$. Suppose moreover that the following conditions hold:
\begin{enumerate}
   \item The metric spaces $(X,d)$ appearing in $\mathcal{K}$ have bounded diameter.
    \item The measures $\mu$ have bounded total mass.
  \item For each radius $r>0$ there exists $\eps>0$ such that the ball $B(x,r)$ has $\mu(B(x,r)) \geq \eps$ for every $(X,d,\mu,f,o) \in \mathcal{K}$ and $x\in X$.
  \item The functions $f$ appearing in $\mathcal{K}$ are bounded and equicontinuous in the sense that the set $\{f(x):(X,d,\mu,f,o)\in \mathcal{K},\,x\in X\}$ is bounded in $Y$ and for every $\eps>0$ there exists $\delta$ such that $d_Y(f(x),f(y)) \leq \eps$ for every $(X,d,\mu,f,o)\in \mathcal{K}$ and every $x,y\in X$ with $d(x,y) \leq \delta$.
\end{enumerate}
Then $\mathcal{K}$ has compact closure in the rooted GHPf topology and every tuple $(X,d,\mu,f,o)$ in the closure $\overline{\mathcal{K}}$ has the property that $\mu$ has support $X$.
\end{prop}

In order to apply this compactness criterion, we next prove the following three lemmas: the first gives the lower mass estimate on chemical balls,  the second yields equicontinuity of the embedding with respect to the chemical metric, and the third improves \cref{cor:CompareDistance} to a uniform comparison of different intrinsic distances holding for all pairs of points in the cluster.

\begin{lemma}[Lower mass condition] \label{lem:TightGHP}  
Let $d>6$ and $L\geq 1$ be such that critical percolation on $\mathbb{Z}^d_L$ satisfies the estimate \eqref{eq:two_point_assumption}. We have that
\[
\lim_{\eps\downarrow0}\limsup_{r\to\infty}
\proba_{p_c}\left (\exists v\in K \text{ such that } |B_\mathrm{chem}(v,\delta r^2)|\leq \eps r^4 \;\middle |\;  |K| \geq  r^4 \right )=0 .
\]
for every  $\delta>0$.
\end{lemma}

\begin{lemma}[Chemical equicontinuity of the embedding] \label{lem:equicontinuity} 
Let $d>6$ and $L\geq 1$ be such that critical percolation on $\mathbb{Z}^d_L$ satisfies the estimate \eqref{eq:two_point_assumption}. We have that
\[ 
\lim_{\delta \downarrow0}\limsup_{r\to\infty} \proba_{p_c} \left (\exists u,v\in K \text{ such that } d_\mathrm{chem}(u,v) \leq \delta r^2 \text{ and } \|u-v\|\geq \eps r \; \middle| \; |K|\geq  r^4 \right )= 0\]
for every $\eps>0$.
\end{lemma}

\begin{lemma}[Uniform comparison of different intrinsic distances]
\label{lem:uniform_distance_comparison}
Let $d>6$ and $L\geq 1$ be such that critical percolation on $\mathbb{Z}^d_L$ satisfies the estimates \eqref{eq:two_point_assumption} and \eqref{eq:two_blob_assumption}.  For each $\#\in \{\mathrm{piv},\mathrm{res}\}$ we have that
\[
  \lim_{r\to\infty}\P_{p_c}\left( \exists u,v \in K \text{ such that } |C_\# d_\mathrm{chem}(u,v)-C_\mathrm{chem} d_\#(u,v)| \geq \eps r^2 \mid  |K| \geq r^4\right)=0
\]
for every  $\eps>0$.
\end{lemma}

\begin{remark}
The proofs of \cref{lem:TightGHP,lem:equicontinuity} show that, in order for the relevant probabilities to be small, it suffices to  take $\eps$ of order $\delta^{2} (-\log \delta)^{-1}$ in \cref{lem:TightGHP} and to take $\delta$ of order $\eps^2 (-\log \eps)^{-2}$ in \cref{lem:equicontinuity}. These bounds are presumably quite close to optimal.
\end{remark}

We begin by proving \cref{lem:TightGHP}. The proof is a straightforward adaptation of \cite[Proposition 5.1]{blanc2024scaling}, which in turn draws on the methods of \cite{archer2024ghp}.

\begin{proof}[Proof of \cref{lem:TightGHP}]
This proof applies to any transitive graph satisfying the triangle condition.
Since $\P(|K|\geq n) \asymp n^{-1/2}$ when the triangle condition holds, it suffices to prove that 
\[
\lim_{\eps\downarrow0}\limsup_{r\to\infty}
\proba \left (\exists v\in K \text{ such that } |B_\mathrm{chem}(v,\delta r^2)|\leq \eps r^4 \;\middle |\; r^4\leq |K|\leq \kappa r^4 \right )=0 
\]
for every $\delta>0$ and $\kappa>1$ sufficiently large that $\P(n\leq |K| \leq \kappa n) \geq \frac{1}{2}\P(|K|\geq n)$ for all $n\geq 1$.

Fix one such $\delta>0$ and $\kappa>1$.
It follows from \cite[Proposition 3.6 and Remark 3.7]{hutchcroft2020slightly} (which applies to any transitive graph satisfying the triangle condition; see also \cite[Lemma 6.2]{MR2653185}) that there exist positive constants $c$ and $C$ such that
\begin{equation}
  \P\left(|B_\mathrm{chem}(x,\ell)|\leq M, \partial B_\mathrm{chem}(x,\ell) \neq \emptyset \right) \leq \frac{C}{\ell} \exp\left[-\frac{c \ell^2}{M}\right]
\label{eq:skinny}
\end{equation}
for every $\ell,M \geq 1$.
As such, if $\eps<1$ then we can use the mass-transport principle to write
\begin{align}
&\E \left[\mathbbm{1}(r^4\leq |K|\leq \kappa r^4)  \sum_{v\in K} \mathbbm{1}(|B_\mathrm{chem}(v,\delta r^2)|\leq \eps r^4)\right]
\nonumber\\&\hspace{8cm}= \E \left[|K| \mathbbm{1}(r^4\leq |K|\leq \kappa r^4) \mathbbm{1}(|B_\mathrm{chem}(0,\delta r^2)|\leq \eps r^4)\right] 
\nonumber\\&\hspace{8cm}\leq \kappa r^4 \P(|B_\mathrm{chem}(0,\delta r^2)|\leq \eps r^4,\, \partial B_\mathrm{chem}(0,\delta r^2)\neq \emptyset)
\nonumber\\&\hspace{8cm}\leq C \delta^{-1} \kappa r^2 e^{-c \delta^2/\eps}
\label{eq:skinny_useful}
\end{align}
for every $r\geq 1$, $\delta>0$, $\kappa>1$, and $0<\eps<1$.
Fix $r\geq 1$, $\kappa>1$, and $0<\delta,\eps <1$ and for each $k\geq 0$ define the event
\[ E_k:=\{r^4\leq |K|\leq \kappa r^4\}\cap\{\exists v\in K \text{ such that } |B_\mathrm{chem}(v,2^{-k}\delta r^2)|\leq 8^{-k}\eps r^4\}, \]
noting that this event cannot possibly hold when $k \geq 4 \log_8 r$.
Observe that if  $k\geq 0$ is such that $E_k$ holds but $E_{k+1}$ does not, then there must exist a vertex $v\in K$ such that $|B_\mathrm{chem}(v,2^{-k}\delta r^2)|\leq 8^{-k}\eps r^4$ and $|B_\mathrm{chem}(v,2^{-k-1}\delta r^2)|\geq 8^{-k-1}\eps r^4$, so that $E_k \setminus E_{k+1}$ is contained in the event
\[ F_k:=\{r^4\leq |K|\leq \kappa r^4\}\cap\{\#\{v\in K : |B_\mathrm{chem}(v,2^{-k-1}\delta r^{2})|\leq 8^{-k}\eps r^4\}\geq 8^{-k-1}\eps r^4\}.  \]
Applying \eqref{eq:skinny_useful} and Markov's inequality yields that
\begin{multline*} \proba(F_k)\leq \frac{8^{k+1}}{\eps r^4} 
\E \left[\mathbbm{1}(r^4\leq |K|\leq \kappa r^4)  \sum_{v\in K} \mathbbm{1}(|B_\mathrm{chem}(v,2^{-k-1}\delta r^2)|\leq 8^{-k} \eps r^4)\right]
\\\leq C \frac{2^{4k+4} \kappa }{\eps \delta r^2}  \exp \left[- \frac{c \delta^2}{\eps} 2^{k-2}\right]
\end{multline*}
for every $k\geq 0$ and hence that
\begin{align*}
\proba \left (\exists v\in K \text{ such that } |B_\mathrm{chem}(v,\delta r^2)|\leq \eps r^4,\, r^4\leq |K|\leq \kappa r^4 \right ) 
&=\P(E_0) \leq  \sum_{k=0}^\infty \P(E_k \setminus E_{k+1}) \leq   \sum_{k=0}^\infty \P(F_{k})
\\& \leq \frac{C \kappa}{r^2}\sum_{k=0}^\infty \frac{2^{4k+4}}{\eps \delta}  \exp \left[- \frac{c \delta^2}{\eps} 2^{k-2}\right].
\end{align*}
To conclude, observe that there exists a constant $c'$ such that if $\eps < c' \delta^2$ then the ratios between successive terms in this sum are all bounded by $1/2$, so that we can bound the sum by twice the first term and obtain that
\begin{equation}
\proba \left (\exists v\in K \text{ such that } |B_\mathrm{chem}(v,\delta r^2)|\leq \eps r^4,\, r^4\leq |K|\leq \kappa r^4 \right ) \leq \frac{32 C \kappa}{\eps \delta r^2} e^{-c \delta^2/(4\eps)}
\label{eq:GHP_tightness_quantitative}
\end{equation}
whenever $\eps \leq c' \delta^2$, which is easily seen to imply the claim.
\end{proof}

We now prove \cref{lem:equicontinuity}.

\begin{proof}[Proof of \cref{lem:equicontinuity}]
As before, it suffices to prove that
\[ 
\lim_{\delta \downarrow0}\limsup_{r\to\infty} \proba \left (\exists u,v\in K \text{ such that } d_\mathrm{chem}(u,v) \leq \delta r^2 \text{ and } \|u-v\|\geq \eps r \; \middle| \; r^4\leq  |K|\leq \kappa r^4 \right )= 0\]
for every  $\eps>0$ and $\kappa>1$ sufficiently large that $\P(n \leq |K| \leq \kappa n) \geq \frac{1}{2} \P(|K|\geq n)$ for every $n\geq 1$.
First note that if $x$ and $y$ are connected by a path of length at most $\ell$ in critical percolation then they remain connected with constant probability in $p_c-1/\ell$ percolation (under the standard monotone coupling), and it follows by Bayes formula that there exist positive constants $C_1,C_2$, and $c_1$ such that
\[
  \P_{p_c}(x\leftrightarrow y, d_\mathrm{chem}(x,y) \leq \ell) \leq C_1 \P_{p_c-1/\ell}(x\leftrightarrow y) \leq C_2 \langle x-y\rangle^{-d+2} \exp\left[-c_1\frac{\langle x-y\rangle}{\sqrt{\ell}}\right]
\]
for every $x,y\in \Z^d$ and $\ell \geq 1$, with the estimate on the subcritical two-point function being proven in \cite{hutchcroft2023high}.
Using the mass-transport principle, we can use this estimate to obtain that there exists a constant $C_3$ such that
\begin{multline}
\E\left[\mathbbm{1}(r^4\leq |K|\leq \kappa r^4) \sum_{u,v\in K} \mathbbm{1}(d_\mathrm{chem}(u,v)\leq \ell, \|u-v\|\geq r) \right]
\\
=
\E\left[\mathbbm{1}(r^4\leq |K|\leq \kappa r^4) |K| \sum_{v\in K} \mathbbm{1}(d_\mathrm{chem}(0,v)\leq \ell, \|v\|\geq r) \right] 
\\
\leq C_2 \kappa r^4 \sum_{\|v\|\geq r} \langle v\rangle^{-d+2} e^{-c_1 \ell^{-1/2}\langle v\rangle} \leq C_3 \kappa r^6 e^{-c_1\ell^{-1/2} r}
\label{eq:small_chemical_distance_two_point_useful} 
\end{multline}
whenever $r\geq \ell^{1/2}\geq 1$ and $\kappa>1$. (For the final inequality, consider the contribution from dyadic scales $2^k r\leq \|v\|\leq 2^{k+1}r$ and note that the ratios between successive contributions of this form are at most $1/2$ once $k$ is larger than some constant $k_0$.)

We now apply \eqref{eq:small_chemical_distance_two_point_useful} to prove the lemma, following a similar strategy to the proof of \cref{lem:TightGHP}.
Fix $r\geq 1$, $\kappa>1$, and $0<\eps,\delta <1$, and for each $k\geq 0$ let $E_k$ be the event that $r^4\leq |K| \leq \kappa r^4$ and there exist $u,v\in K$ with $d_\mathrm{chem}(u,v) \leq 4^{-3k}\delta r^2$ and $\|u-v\|\geq 4^{-k}\eps r$. If $E_k$ holds but $E_{k+1}$ does not, it must be possible to find two vertices $u$ and $v$ having chemical distance in the interval $[4^{-3k-3}\delta r^2,4^{-3k}\delta r^2]$ with $\|u-v\|\geq 4^{-k}\eps r$, and any two points $x\in B_\mathrm{chem}(u, 4^{-3k-3}\delta r^2)$ and $y\in B_\mathrm{chem}(v, 4^{-3k-3}\delta r^2)$ must satisfy $\|x-u\|,\|y-v\|\leq 4^{-k-1} \eps r$ and hence $\|x-y\| \geq \|u-v\|-\|x-u\|-\|y-v\|\geq 4^{-k-1} \eps r$.
Let $F_k$ be the event that $E_k\setminus E_{k+1}$ holds and that every chemical ball of radius 
$4^{-3k-3}\delta r^2$ in $K$ has volume at least $c_2 4^{-7k-6}\delta^3 r^4$, where $c_2$ is the constant referred to as $c'$ in \eqref{eq:GHP_tightness_quantitative}. On the event $F_k$, there must exist at least $c_2^2 4^{-14k-12}\delta^6 r^8$ pairs of points $x,y$ with $d_\mathrm{chem}(x,y)\leq 4^{-3k+1} \delta r^2$ and $\|x-y\|\geq 4^{-k-1}\eps r$. Thus, it follows from \eqref{eq:small_chemical_distance_two_point_useful} and Markov's inequality that
\[
  \P(F_k) \leq  \frac{C_3 \kappa 4^{14k+12}}{c_2^2 \delta^6 r^8}  r^6 \exp\left[-c_1 \delta^{-1/2} \eps 2^{k-1} \right]=\frac{C_4 \kappa 4^{14k}}{ \delta^6 r^2} \exp\left[-c_1 \delta^{-1/2} \eps 2^{k-1} \right]
\] 
for some constant $C_4$.
On the other hand, we also have by \eqref{eq:GHP_tightness_quantitative} that there exist positive constants $C_5,C_6,C_7$ and $c_3,c_4,c_5$ such that 
\begin{multline*}
  \P((E_k\setminus E_{k+1})\setminus F_k) \leq \frac{C_5 \kappa}{(c_2 4^{-7k-6}\delta^3)(4^{-3k-3}\delta) r^2} e^{-c_3 (4^{-3k-3}\delta)^2 /(c_2 4^{-7k-6}\delta^3)} 
  \\\leq  \frac{C_6 \kappa  4^{10k+9}}{c_2 \delta^4  r^2} e^{-c_4  4^k \delta^{-1}} \leq \frac{C_7 \kappa}{\delta^4 r^2} e^{-c_5 4^k \delta^{-1}}.
\end{multline*}
Putting these estimates together, we obtain that there exists a constant $C_8$ such that
\begin{align*}
&\proba \left (\exists u,v\in K \text{ such that } d_\mathrm{chem}(u,v) \leq \delta r^2 \text{ and } \|u-v\|\geq \eps r,\, r^4 \leq |K|\leq \kappa r^4 \right ) 
\\&\hspace{7.4cm}=\P(E_0) \leq \sum_{k=0}^\infty (\P(F_k)+\P((E_k\setminus E_{k+1})\setminus F_k))
\\ &\hspace{7.4cm}\leq \frac{C_8 \kappa}{r^2} \left(\delta^{-6} \sum_{k=0}^\infty 4^{14k} e^{-c_1 \delta^{-1/2} \eps 2^{k-1}} + \delta^{-4} \sum_{k=0}^\infty e^{-c_5 4^k \delta^{-1}} \right).
\end{align*}
As before, there exists a positive constant $c_6$ such that if $\delta \leq c_6 \eps^2$ then the ratio between successive terms in both sums appearing here are at most $1/2$, so that there exist positive constants $C_9$ and $c_7$ such that
\begin{multline*}
  \proba \left (\exists u,v\in K \text{ such that } d_\mathrm{chem}(u,v) \leq \delta r^2 \text{ and } \|u-v\|\geq \eps r,\, r^4 \leq |K|\leq \kappa r^4 \right ) \\\leq \frac{C_9 \kappa}{r^2} \left( \delta^{-6} e^{-c_7 \delta^{-1/2}\eps }+ \delta^{-4} e^{-c_7 \delta^{-1}} \right)
\end{multline*}
whenever $ \delta \leq c_6 \eps^2$. This implies the claim in conjunction with \cref{thm:scaling_limit_intro} as before.
\end{proof}
We now prove \cref{lem:uniform_distance_comparison}.

\begin{proof}[Proof of \cref{lem:uniform_distance_comparison}]
Fix $\#\in \{\mathrm{piv},\mathrm{res}\}$ and suppose for contradiction that there exists $\eps>0$  such that the relevant probability does not converge to zero as $r\to\infty$. Passing to a subsequence if necessary, we may assume that the probability is bounded below by some positive constant $c$. Applying \cref{lem:TightGHP}, we deduce that there exists a positive constant $\delta>0$ such that, with probability at least $c/2$ along this subsequence, there exists a pair of points $u,v\in K$ with $|C_\# d_\mathrm{chem}(u,v)-C_\mathrm{chem} d_\#(u,v)| \geq  \eps r^2$ and every chemical distance ball of radius $\frac{1}{8}(C_\mathrm{chem}+C_\#)^{-1}\eps r^2$ in $K$ has volume at least $\delta r^4$. Since $d_\#(u,v)\leq d_\mathrm{chem}(u,v)$, the two vertices $u,v$ above must have $d_\mathrm{chem}(u,v) \geq (C_\mathrm{chem}+C_\#)^{-1}\eps r^2$.  As such, any two points $x\in B_\mathrm{chem}(u,\frac{1}{8}(C_\mathrm{chem}+C_\#)^{-1} \eps r^2)$ and $y\in B_\mathrm{chem}(v,\frac{1}{8}(C_\mathrm{chem}+C_\#)^{-1} \eps r^2)$ must have
\[
  |d_\mathrm{chem}(x,y) - d_\mathrm{chem}(u,v)|\leq d_\mathrm{chem}(u,x)+d_\mathrm{chem}(v,y) \leq \frac{1}{4}(C_\mathrm{chem}+C_\#)^{-1} \eps r^2
\]
and
\[
  |d_\#(x,y) - d_\#(u,v)| \leq d_\mathrm{chem}(u,x)+d_\mathrm{chem}(v,y) \leq \frac{1}{4} (C_\mathrm{chem}+C_\#)^{-1} \eps r^2
\]
and hence
\begin{align*}
  |C_\# d_\mathrm{chem}(x,y) - C_\mathrm{chem} d_\#(x,y)| &\geq |C_\# d_\mathrm{chem}(u,v) - C_\mathrm{chem} d_\#(u,v)|\\&\hspace{1cm}-C_\# | d_\mathrm{chem}(x,y) - d_\mathrm{chem}(u,v)| - C_\mathrm{chem} | d_\#(x,y) - d_\#(u,v)|
  \\
  &\geq \frac{\eps}{2}r^2.
\end{align*}
That is, in this scenario there must exist, with good probability, order $r^8$ pairs of points $x,y$ for which $|C_\# d_\mathrm{chem}(x,y) - C_\mathrm{chem} d_\#(x,y)|$ is of order at least $r^2$. Since $|K|$ has order $r^4$ with high probability conditioned on the event that $|K|\geq r^4$, this contradicts \cref{cor:CompareDistance}, which implies that $|C_\# d_\mathrm{chem}(X,Y) - C_\mathrm{chem} d_\#(X,Y)|$ is much smaller than $r^2$ with high probability when $X$ and $Y$ are chosen uniformly at random from $K$ and we condition on the event $|K|\geq r^4$.
\end{proof}

\begin{proof}[Proof of \cref{thm:scaling_limit_GHPf}]
Given \cref{lem:TightGHP,lem:equicontinuity,lem:uniform_distance_comparison} and \cref{thm:scaling_limit_intro}, this is just a standard ``convergence in a weak topology plus compactness in a strong topology implies convergence in the strong topology'' argument, and we keep the details light.
Fix $\lambda>0$. In light of \cref{prop:GHPf_compactness_lower_mass_version}, it follows immediately from \cref{lem:TightGHP,lem:equicontinuity,lem:uniform_distance_comparison} that the conditional distributions of 
\[\left(K,\frac{\bar d_\#}{C_\# r^2}, \frac{\mu_K}{C_\text{vol}r^4}, \frac{1}{r}\Phi_K,0 \right)\]
given $|K| \geq \lambda C_\mathrm{vol}r^4$ are tight in the rooted GHPf topology. (The diameter, total mass, and boundedness conditions follow for the chemical distance from the bounds $\P(|K|\geq n)\asymp n^{-1/2}$, $\P(\partial B_\mathrm{chem}(0,\ell) \neq \emptyset)\asymp \ell^{-1}$, and $\P(0\leftrightarrow \Z^d\setminus [-r,r]^d)\asymp r^{-2}$, and extend to the pivotal and resistance distances by \cref{lem:uniform_distance_comparison}.) As such, any divergent sequence $(r_n)$ has a subsequence along which these distributions converge to that of some random rooted metric-measure space equipped with a continuous function into $\R^d$, and moreover the measure on this space must be fully supported a.s. 

By the marked version of Gromov's reconstruction theorem
\cite[Theorem 1]{depperschmidt2011marked},
a compact metric-measure space $(X,d,\mu,f)$ with fully-supported finite measure and
continuous function into $\R^d$ is determined, up to function-preserving
measure-preserving isometry, by its total mass together with the joint law of
$(d(X_i,X_j),f(X_i))_{i,j\geq 1}$,
where the $X_i$ are i.i.d.\ samples from the normalized measure, with a similar statement holding in the rooted case by the same proof.
Since we also have convergence in the weak array topology along the same subsequence, it follows that the random metric-measure space arising as the rooted GHPf limit of $(K,\bar d_\# /C_\# r^2, \mu_K/ C_\mathrm{vol} r^4,\Phi_K/r,0)$ conditioned on $|K|\geq \lambda C_\mathrm{vol} r^4$ along the subsequence $r_n$ must be distributed as $(\sT,d,\nu,\Phi,o)$ under the measure $\N( \,\cdot\, \mid \nu(\sT) \geq \lambda)$. Since this subsequential limit does not depend on the choice of subsequence we must actually have convergence to the desired limit point, completing the proof.
\end{proof}

\subsection{First-order one-arm asymptotics}

We now deduce \cref{cor:first_order_arm} from \cref{thm:scaling_limit_intro,thm:scaling_limit_GHPf}. We begin by recalling the following estimates concerning the probability that a cluster has large diameter without having commensurately large volume.

\begin{lemma}
\label{lem:skinny_clusters}
Let $d>6$ and $L\geq 1$ be such that critical percolation on $\mathbb{Z}^d_L$ satisfies the estimate \eqref{eq:two_point_assumption}. For each $\eps >0$ there exists $\delta>0$ such that the estimates
\[
  \P_{p_c}(0 \leftrightarrow \Z^d \setminus [-r,r]^d \text{ and } |K| \leq \delta r^4) \leq \eps r^{-2}
\quad \text{ and } \quad 
  \P_{p_c}(\partial B_\mathrm{chem}(0,\ell) \neq \emptyset \text{ and } |K| \leq \delta \ell^2) \leq \eps \ell^{-1} 
\]
hold for all $r\geq 1$.
\end{lemma}

\begin{proof}[Proof of \cref{lem:skinny_clusters}]
The second estimate is an immediate consequence of \cite[Proposition 3.6 and Remark 3.7]{hutchcroft2020slightly} as stated in \eqref{eq:skinny}. For the first, note that if the origin is connected to the complement of $[-r,r]^d$ by a path of length at most $\eta r^2$ in critical percolation, then it has a constant probability to remain connected to the same complement in $p_c-(\eta r^2)^{-1}$ percolation in the standard monotone coupling. It follows from this observation and the slightly subcritical one-arm estimate \cite{hutchcroft2023high,chatterjee2021subcritical} that there exist positive constants $C_1,C_2,$ and $c_1$ such that
\[ \P_{p_c}(0 \leftrightarrow \Z^d \setminus [-r,r]^d, \partial B_\mathrm{chem}(0,\eta r^2) = \emptyset) \leq C_1 \P_{p_c-(\eta r^2)^{-1}}(0 \leftrightarrow \Z^d \setminus [-r,r]^d)
\leq \frac{C_2}{r^2} \exp\left[-c_1\eta^{-1/2} \right]  \]
for every $0<\eta\leq 1$ and $r\geq 1$. The claim regarding $\P(0 \leftrightarrow \Z^d \setminus [-r,r]^d \text{ and } |K| \leq \delta r^4)$ follows easily from this together with the corresponding estimate for $\P(\partial B_\mathrm{chem}(0,\ell) \neq \emptyset \text{ and } |K| \leq \delta \ell^2)$.
\end{proof}

\begin{proof}[Proof of \cref{cor:first_order_arm}]
As usual, we prove the theorem under the weaker hypothesis that $d>6$ and \eqref{eq:two_point_assumption} and \eqref{eq:two_blob_assumption} both hold.
It follows from \cref{thm:scaling_limit_intro,thm:scaling_limit_GHPf} that
\begin{multline*}
  \P(0 \leftrightarrow \Z^d \setminus [- r, r]^d \text{ and } |K|\geq \lambda C_\mathrm{vol} r^4) \sim \frac{\N( \mu(\R^d \setminus [-1,1]^d) \neq 0 \mid \nu(\sT) \geq \lambda )}{C_\mathrm{prob}\lambda^{1/2}} \cdot r^{-2}
  \\=\frac{\N( \mu(\R^d \setminus [-1,1]^d) \neq 0,\, \nu(\sT) \geq \lambda )}{C_\mathrm{prob} \N(\nu(\sT) \geq 1)} \cdot r^{-2} = \frac{\N( \mu(\R^d \setminus [-1,1]^d) \neq 0,\, \nu(\sT) \geq \lambda )}{AV} \cdot r^{-2}
\end{multline*}
as $r\to \infty$ for each fixed $\lambda>0$, where we used the scaling law $\N(\nu(\sT)\geq \lambda) = \lambda^{-1/2}\N(\nu(\sT)\geq 1)$ in the first equality on the second line and the  equality $C_\mathrm{prob}=AV/\N(\nu(\sT)\geq 1)$ in the second. (Here we use that the ``laws'' of $\nu(\sT)$ and $\sup\{r: \mu(\R^d\setminus [-r,r]^d)\neq 0\}$ are nonatomic, which in turn follows from the scale-invariance of the canonical measure, to apply the portmanteau theorem and deduce these asymptotic formulae from weak convergence.) A similar argument yields that
  \[
  \P(\partial B_\mathrm{chem}(0,C_\mathrm{chem} r^2) \neq \emptyset \text{ and } |K|\geq \lambda C_\mathrm{vol} r^4) \sim \frac{\N( \operatorname{rad}(\sT,o) \geq 1, \nu(\sT) \geq \lambda )}{AV} r^{-2}
\]
as $r\to \infty$. On the other hand, it follows from \cref{lem:skinny_clusters} and the facts that $\P(0 \leftrightarrow \Z^d \setminus [-r,r]^d)\asymp r^{-2}$ and $ \P(\partial B_\mathrm{chem}(0,\ell) \neq \emptyset) \asymp \ell^{-1}$  that for each $\eps>0$ there exists $\delta>0$ such that
\begin{multline*}
  \P(0 \leftrightarrow \Z^d \setminus [- r, r]^d \text{ and } |K|\geq \delta C_\mathrm{vol} r^4) \leq 
  \P(0 \leftrightarrow \Z^d \setminus [- r, r]^d )
  \\\leq (1+\eps)\P(0 \leftrightarrow \Z^d \setminus [- r, r]^d \text{ and } |K|\geq \delta C_\mathrm{vol} r^4)
\end{multline*}
and
\begin{multline*}
  \P(\partial B_\mathrm{chem}(0,C_\mathrm{chem} r^2) \neq \emptyset \text{ and } |K|\geq \delta C_\mathrm{vol} r^4) \leq 
  \P(\partial B_\mathrm{chem}(0,C_\mathrm{chem} r^2) \neq \emptyset )
  \\\leq (1+\eps)\P(\partial B_\mathrm{chem}(0,C_\mathrm{chem} r^2) \neq \emptyset \text{ and } |K|\geq \delta C_\mathrm{vol} r^4).
\end{multline*}
Using that $C_\mathrm{chem}=AG_\mathrm{chem}$, these estimates together imply that 
\[
  \P(0 \leftrightarrow \Z^d \setminus [- r, r]^d ) \sim \frac{\N(\mu(\R^d \setminus [-1,1]^d) \neq 0)}{AVr^2}
  \; \text{ and }\;
  \P(\partial B_\mathrm{chem}(0, r^2) \neq \emptyset ) \sim \frac{G_\mathrm{chem}\N(\operatorname{rad}(\sT,o)\geq 1)}{V r^2}
\]
as $r\to \infty$ as claimed.
\end{proof}

\section*{Acknowledgements}
 TH was supported by NSF grant DMS-1928930 and a Packard Fellowship for Science and Engineering. TH's work on the project was initiated during the workshop ``Random Interacting Systems, Scaling Limits, and Universality'' at the Institute for Mathematical Sciences at the National University of Singapore in December 2023, and he thanks Takashi Hara, Gordon Slade, and Xin Sun for stimulating discussions during that meeting. We also thank Gordon Slade for helpful comments on an earlier draft of the manuscript.

\appendix 

\section{Positivity of local factors}
\label{sec:positivity_of_local_factors}

In this section we gather the statements and proofs of the finite-energy lemmas needed to deduce positivity of the various limiting constants appearing in our analysis under minimal topological assumptions.

\medskip

\noindent
\textbf{Positivity of constants in \cref{sec:lace_expansion}.} 

\begin{lemma}
\label{lem:IIC_positivity}
Let $d>6$ and $L\geq 1$ be such that critical percolation on $\mathbb{Z}^d_L$ satisfies the estimate \eqref{eq:two_point_assumption}, and suppose that the IIC measure $\P_\mathrm{IIC}=\lim_{x\to \infty}\P_{p_c}( \,\cdot\, \mid 0 \leftrightarrow x)$ is well-defined.  If $W$ is a finite set of edges then $\P_\mathrm{IIC}(W \text{ \emph{closed}})=0$ if and only if every path from $0$ to $\infty$ passes through $W$.
\end{lemma}

\begin{proof}[Proof of \cref{lem:IIC_positivity}]
The fact that $\P_\mathrm{IIC}(W \text{ closed})=0$ whenever every path from $0$ to $\infty$ passes through $W$ is trivial; we focus on the converse. Fix a finite, non-empty set of edges $W$ for which there exists a path $\gamma$ from $0$ to $\infty$ in $\Z^d_L$ disjoint from $W$. For each $n\geq 1$ let $\gamma[n]$ and $\gamma(n)$ be the set of edges traversed by the first $n$ steps of $\gamma$ and the vertex visited at the $n$th step of $\gamma$ respectively. For each $x$ and $n\geq 1$ we have that
\begin{multline*}
  \P(0\leftrightarrow x,\, W \text{ closed}) \geq \P(\gamma[n] \text{ open}, W \text{ closed}, \gamma(n) \leftrightarrow x) \\\geq \P(\gamma[n] \text{ open, $W$  closed})\P(\gamma(n)\leftrightarrow x \text{ off $W$}),
\end{multline*}
where we applied the Harris-FKG inequality to the conditional measure $\P(\,\cdot \mid W $ closed$)$ in the second inequality. If $\gamma(n)$ is connected to $x$ only via $W$, there must exist an open simple path from $\gamma(n)$ to $x$ passing through $V(W)$, and applying the BK inequality we obtain that
\begin{multline*}
\P(\gamma(n)\leftrightarrow x \text{ off $W$}) = \P(\gamma(n)\leftrightarrow x) - \P(\gamma(n)\leftrightarrow x \text{ only via $W$})
\\\geq T_{p_c}(\gamma(n), x) - \sum_{w\in V(W)} T_{p_c}(\gamma(n),w)T_{p_c}(w,x).
\end{multline*}
When $n$ is fixed and $x\to \infty$, it follows from \eqref{eq:two_point_assumption} that the right-hand side of this inequality is asymptotic to $(1-\sum_{w\in V(W)} T_{p_c}(\gamma(n),w))T_{p_c}(0,x)$, and in particular that for each $n\geq 1$ there exists $R(n)<\infty$ such that $\P(\gamma(n)\leftrightarrow x \text{ off $W$}) \geq \frac{1}{2}(1-\sum_{w\in V(W)} T_{p_c}(\gamma(n),w))T_{p_c}(0,x)$ whenever $\|x\|\geq R(n)$. On the other hand, $\sum_{w\in V(W)} T_{p_c}(\gamma(n),w)$ is asymptotic to $|V(W)| T_{p_c}(0,\gamma(n))$ as $n\to \infty$, so that there exists $n<\infty$ such that $\sum_{w\in V(W)} T_{p_c}(\gamma(n),w)\leq 1/2$. Fixing this value of $n$ we obtain that
\[
  \P(0\leftrightarrow x,\, W \text{ closed}) \geq \frac{1}{4}\P(\gamma[n] \text{ open, $W$ closed}) T_{p_c}(0,x)
\]
for every $x\in \Z^d$ with $\|x\|\geq R(n)$. This is easily seen to imply the claim.
\end{proof}

\begin{lemma}
\label{lem:only_via_positivity}
Let $d>6$ and $L\geq 1$ be such that critical percolation on $\mathbb{Z}^d_L$ satisfies the estimate \eqref{eq:two_point_assumption}.
For each finite, non-empty set of edges and vertices $W$ there exists positive constants $c_W$ and $R_W$ such that
\[
  \P_{p_c}(0\leftrightarrow x \text{ only via $W$}) \geq c_W \P_{p_c}(0\leftrightarrow x)
\]
for every $x \in \Z^d$ with $\|x\|\geq R_W$.
\end{lemma}

\begin{proof}[Proof of \cref{lem:only_via_positivity}]
Fix a finite, non-empty set of edges and vertices $W$ and an infinite simple path $\gamma$ from $0$ to $\infty$ in $\Z^d_L$ that passes through $W$ (meaning that it either includes a vertex of $W\cap V$ or an edge of $W\cap E$), and let $F$ denote the set of edges that are incident to the part of $\gamma$ up until it first passes through $W$ but not contained in $\gamma$. If for each $n\geq 1$ we let 
$\gamma[n]$ and $\gamma(n)$ denote the set of edges traversed by the first $n$ steps of $\gamma$ and the location of the $n$th step of $\gamma$ respectively, then we can bound
\begin{multline*}
  \P(0\leftrightarrow x \text{ only via $W$}) \geq \P(\text{$\gamma[n]$ open, $F$ closed, $\gamma(n)\leftrightarrow x$})
  \\ 
  \geq \P(\text{$\gamma[n]$ open, $F$ closed}) \P(\gamma(n)\leftrightarrow x \text{ off $F$}),
\end{multline*}
where we again applied the Harris-FKG inequality to the conditional law of the percolation configuration given that $F$ is closed in the second inequality.
From here the proof can be concluded by the same argument as in the proof of \cref{lem:IIC_positivity}.
%
\end{proof}

\noindent
\textbf{Positivity of blob factors.}
We now prove the part of \cref{thm:scheme_function} concerning the positivity of the blob factor. 

\begin{lemma}
\label{lem:blob_factor_positivity}
Let $d>6$ and $L\geq 1$ be such that critical percolation on $\mathbb{Z}^d_L$ satisfies the estimates \eqref{eq:two_point_assumption} and \eqref{eq:two_blob_assumption}. Given a blob $B=(W,b_1,\ldots,b_k)$, the blob factor $\mathbf{W}(B)$ is positive  if and only if there exists a collection of vertex-disjoint paths connecting $b_1,\ldots,b_k$ to infinity in $\Z^d_L$ that are disjoint from $W$.
\end{lemma}

This proof is more topologically delicate than the previous proofs in this section since the different paths from the vertices $b_1,\ldots,b_k$ may be entangled with each other in a very complicated way. We begin by proving the following lemma which handles these topological difficulties by replacing the paths with a very well-behaved collection.

\begin{lemma}[Affine escape: vertex version]
\label{lem:affine_escape_vertex}
Let $d\geq 3$ and $L\geq 1$, and suppose $(W,b_1,\ldots,b_k)$ is a blob for which there exist vertex-disjoint paths $\gamma_1,\ldots,\gamma_k$ with $\gamma_i$ connecting $b_i$ to $\infty$ in $\Z^d_L \setminus W$. The paths $\gamma_i$ can always be chosen to be eventually affine with distinct slopes in the sense that there exist constant $a_i,q_i \in \N$ and $z_i\in \Z^d$ with $z_i/q_i \neq z_j/q_j$ for $i\neq j$ such that $\gamma_i(a_i+m q_i)=\gamma_i(a_i)+mz_i$ for every $1\leq i \leq k$ and $m\geq 1$.
\end{lemma}

\begin{proof}[Proof of \cref{lem:affine_escape_vertex}]
We first prove that the paths can be chosen to be eventually affine with possibly equal slopes before showing that these slopes can be taken to be distinct. Fix $(W,b_1,\ldots,b_k)$, let the vertex-disjoint paths $\gamma_1,\ldots,\gamma_k$ be as in the statement of the lemma, and let the integer $R$ be sufficiently large that every point $b_1,\ldots,b_k$ and every endpoint of an edge in $W$ belongs to $[-R,R]^d$. For each $1\leq i \leq k$, let $x_i$ be the first point of $\Z^d\setminus [-R,R]^d$ visited by $\gamma_i$, and let $e_i$ be the unit vector in $\Z^d$ pointing in the same direction as $x_i$, meaning that $e_i$ has one coordinate in $\{-1,1\}$, every other coordinate equal to zero, and satisfies $\|x_i+\ell e_i\|_\infty=\|x_i\|_\infty+\ell$ for every $\ell \geq 0$. We observe that the rays $\{x_i+L\ell e_i : \ell\geq 0\}$ in $\Z^d_L$ must be disjoint since $\|x_i\|_\infty \in [R+1,R+L]$ for every $1\leq i \leq k$ and the points $x_1,\ldots,x_k$ are all distinct. Replacing each path $\gamma_i$ with the path in $\Z^d_L$ that follows $\gamma_i$ until the point $x_i$ and then follows the ray $\{x_i+L\ell e_i : \ell\geq 0\}$ shows that the paths $\gamma_1,\ldots,\gamma_k$ can always be taken to be eventually affine, albeit possibly with identical slopes.

We now pass from eventually-affine rays to eventually-affine rays with distinct slopes.
We will prove more generally that if $(W,b_1,\ldots,b_k)$ is a blob for which there exist vertex-disjoint paths $\gamma_1,\ldots,\gamma_k$ connecting $b_i$ to $\infty$ in $\Z^d_L \setminus W$ and $\tilde \gamma_1,\ldots,\tilde \gamma_k$ is \emph{any} collection of vertex-disjoint half-infinite simple paths in $\Z^d_L \setminus (W \cup \bigcup \gamma_i)$ with the property that $\Z^d_L \setminus (\bigcup \tilde \gamma_i)$ is connected and one-ended (meaning that the complement of any finite set of vertices has a unique infinite connected component) and that no two of the rays $\tilde \gamma_i$ are adjacent then there exists a bijection $\pi:\{1,\ldots,k\}\to \{1,\ldots,k\}$ and a collection of vertex-disjoint paths $\gamma_1',\ldots,\gamma_k'$ connecting $b_i$ to $\infty$ in $\Z^d_L \setminus W$ such that $\gamma'_i$ eventually coincides with (a time-shift of) $\tilde \gamma_{\pi(i)}$ for each $1\leq i \leq k$. 
The conclusion of the lemma will follow from this claim since the rays $\tilde \gamma_1,\ldots,\tilde \gamma_k$ can obviously be taken to be affine with distinct slopes, disjoint from our collection of eventually-affine paths $\gamma_1,\ldots,\gamma_k$ constructed in the previous paragraph, not adjacent to one another, and such that $\Z^d_L \setminus (\bigcup \tilde \gamma_i)$ is connected and one-ended. (Here we are using that $d\geq 3$.)

Our proof will use the vertex form of Menger's theorem \cite[Theorem 3.3.1]{diestel2025graphtheory}, which states that the maximum size of a set of paths between two vertices $a$ and $z$ in a graph that are vertex-disjoint except at their endpoints is equal to the minimum size of a vertex cut, i.e., the minimum size of a set of vertices $C$ disjoint from $\{a,z\}$ such that any path between $a$ and $z$ intersects $C$. To apply this theorem in our context, we consider the graph formed from $\Z^d_L \setminus W$ by contracting each path $\tilde \gamma_i$ into a single vertex $[\tilde \gamma_i]$, adding two new vertices $a$ and $z$, connecting $a$ to each of the vertices $b_1,\ldots,b_k$ by an edge, and connecting $z$ to each of the vertices $[\tilde \gamma_1],\ldots,[\tilde \gamma_k]$ by an edge. A collection of $k$ paths between $a$ and $z$ in this graph that are vertex-disjoint other than at their endpoints can easily be turned into a collection of paths $\gamma_1',\ldots,\gamma_k'$ as required by the claim, and by Menger's theorem it suffices to prove that there does not exist a set $C$ of size $|C|<k$ disjoint from $\{a,z\}$ in this modified graph such that every path from $\{b_1,\ldots,b_k\}$ to $\bigcup \tilde\gamma_i$ intersects $C$. 
Let $C$ be a set of vertices of the modified graph with $|C|<k$, and let $C' \subseteq C$ be the vertices of $C$ corresponding to individual vertices of $\Z^d$ (rather than to one of the rays $\tilde \gamma_i$). 
 Since $\Z^d \setminus (\bigcup \tilde \gamma_i)$ is connected and one-ended, $\Z^d_L \setminus (W \cup C' \cup \bigcup \tilde \gamma_i)$ has a unique infinite connected component. Since the rays $\tilde \gamma_1,\ldots,\tilde \gamma_k$ are not adjacent, there are infinitely many vertices of $\Z^d_L \setminus (W \cup C' \cup \bigcup \tilde \gamma_i)$ adjacent to each of the rays $\tilde \gamma_i$, so that the unique infinite connected component of $\Z^d_L \setminus (W \cup C' \cup \bigcup \tilde \gamma_i)$ must contain infinitely many vertices adjacent to each of the rays $\tilde \gamma_1,\ldots,\tilde \gamma_k$.
Since $|C'| \leq |C|<k$ and the rays $\gamma_1,\ldots,\gamma_k$ are disjoint from the rays $\tilde \gamma_1,\ldots,\tilde \gamma_k$, at least one of the paths $\gamma_1,\ldots,\gamma_k$ must be disjoint from $C'$ and therefore connect the set $\{b_1,\ldots,b_k\}$ to $\infty$ in $\Z^d_L \setminus (W \cup C' \cup \bigcup \tilde \gamma_i)$. This path must connect $\{b_1,\ldots,b_k\}$ to the unique infinite connected component of $\Z^d_L \setminus (W \cup C' \cup \bigcup \tilde \gamma_i)$, so that every ray $\tilde \gamma_i$ is connected to $\{b_1,\ldots,b_k\}$ by a path in $\Z^d_L \setminus (W \cup C' \cup \bigcup \tilde \gamma_i)$ that is disjoint from $\bigcup \tilde \gamma_i$ except at its endpoint. Since $|C|<k$, it does not contain at least one of the vertices $[\tilde \gamma_u]$ and hence is not a vertex cut between $a$ and $z$. This shows that any vertex cut between $a$ and $z$ in the modified graph must have at least $k$ elements, concluding the proof.
\end{proof}

\begin{proof}[Proof of \cref{lem:blob_factor_positivity}]
Finally we prove the claim regarding the positive values of $\mathbf{W}$. We will be brief as the details are somewhat tedious. If $B=(W,b_1,\ldots,b_k)$ is a blob for which there do \emph{not} exist vertex-disjoint paths from $b_1,\ldots,b_k$ to $\infty$ in $\Z^d_L\setminus W$ then we have trivially that $\mathbf{W}(B)=0$. (Indeed, in this case it follows by a simple compactness argument that there exists $R$ such that there do not exist vertex-disjoint paths from $b_1,\ldots,b_k$ to the complement of the ball of radius $R$ in $\Z^d_L$, and hence that scheme functions involving the blob $B$ are identically zero when the input variables are sufficiently well-separated.) Suppose conversely that such paths do exist, let the eventually-affine paths $\gamma_1,\ldots,\gamma_k$ with distinct slopes be as in \cref{lem:affine_escape_vertex}, and consider the scheme $S_B$ whose plan is a $k$-star, with the vertex of degree $k$ assigned to the blob $B$ and leaves all assigned to the blob $(\emptyset,0)$, so that
\[
  \mathbf{W}(B) = \lim_{\min_{i\neq j} \|x_i-x_j\| \to \infty} \frac{S_B(x_1,\ldots,x_{k+1})}{\prod_{i=1}^k T_{p_c}(x_1,x_{i+1})}.
\]
To prove that $\mathbf{W}(B)$ is positive, it suffices to prove that there exist finite constants $N$ and $\lambda$ (which may depend on the blob and the choice of $\gamma_1,\ldots,\gamma_k$) such that $\|\gamma_i(m)\|\geq 3\|\gamma_i(n)\|$ and
\begin{multline}
  \P(\gamma_i(n) \leftrightarrow \gamma_i(m) \text{ and } \gamma_i(n) \nleftrightarrow (V(W) \cup \{\gamma_j(\ell): j\neq i, \ell\leq n\}) \text{ for every $1\leq i \leq k$}) \\\geq \frac{1}{2} \prod_{i=1}^k T_{p_c}(\gamma_i(n),\gamma_i(m))
  \label{eq:affine_surgery}
\end{multline}
for every $n\geq N$ and $m\geq \lambda n$, the right side of this inequality having the same order as the product $\prod_{i=1}^k T_{p_c}(0,\gamma_i(m))$ by \eqref{eq:two_point_assumption} and the assumption that $\|\gamma_i(m)\|\geq 3\|\gamma_i(n)\|$ for every $1\leq i \leq k$. Indeed, if we take a configuration in which the event whose probability is estimated in \eqref{eq:affine_surgery} holds, force $W$ to be closed, force the first $n$ steps in each of the paths $(\gamma_i)_i$ to be open, and force any edge that is incident to the first $n$ steps of one of the paths $\gamma_i$ but not contained in this path or in the cluster of the corresponding vertex $\gamma_i(n)$ in the original configuration to be closed then we obtain a configuration in which $W$ is closed and each $b_i$ is connected to $\gamma_i(m)$ in a distinct cluster for every $1\leq i \leq k$.

We now prove \eqref{eq:affine_surgery}. Since our paths are eventually affine, there exist finite constants $N_0$ and $\lambda$ such that if $n\geq N_0$ and $m\geq \lambda n$ then $\|\gamma_i(m)\| \geq 3 \|\gamma_i(n)\|$ and $\|\gamma_i(m)\| \geq 3 \|x\|$ for every $1\leq i \leq k$ and $x\in F_i = V(W) \cup \{\gamma_j(\ell):j\neq i, \ell\leq n\}$.  By \eqref{eq:two_point_assumption} and the tree-graph inequality, the probability that three points are connected is bounded by
\begin{multline*}
  T_{p_c}(x,y,z) \preceq \sum_{w \in \Z^d} \langle x-w\rangle^{-d+2}\langle w-y\rangle^{-d+2}\langle w-z\rangle^{-d+2} \\\asymp (d_\mathrm{max}(x,y,z)+1)^{-d+2}(d_\mathrm{min}(x,y,z)+1)^{-d+4},
\end{multline*}
where we write $d_\mathrm{min}(x,y,z)$ and $d_\mathrm{max}(x,y,z)$ for the minimal and maximal pairwise distances between the points $x,y$, and $z$ and where the convolution estimate in the second equality is proven in \cite[Remark III.5.3]{LRPpaper3}. As such, for each $1\leq i \leq k$ we have that
\begin{multline*}
  \P(\gamma_i(n) \leftrightarrow \gamma_i(m) \text{ and } \gamma_i(n)\leftrightarrow F_i) \preceq \langle \gamma_i(m) \rangle^{-d+2} \sum_{x\in F_i} \langle \gamma_i(n)-x\rangle^{-d+4} \\\asymp 
  T_{p_c}(\gamma_i(n),\gamma_i(m)) \sum_{x\in F_i} \langle \gamma_i(n)-x\rangle^{-d+4}.
\end{multline*}
Since $W$ is fixed, the paths $\gamma_i$ are affine, and $d>6$, we can easily verify that $\sum_{x\in F_i} \langle \gamma_i(n)-x\rangle^{-d+4}\to 0$ as $n\to \infty$. As such, there exists a constant $N_1\geq N_0$ such that if $n\geq N_1$ and $m\geq \lambda n$ then
\begin{equation}
\label{eq:connection_off_the_lines}
  \P(\gamma_i(n) \leftrightarrow \gamma_i(m) \text{ and } \gamma_i(n) \nleftrightarrow F_i) \geq 2^{-\frac{1}{10k}} T_{p_c}(\gamma_i(n),\gamma_i(m)).
\end{equation}
Now consider the usual coupling of the clusters of $\gamma_1(n),\ldots,\gamma_k(n)$ with a sequence of \emph{independent} clusters $\tilde K_1,\ldots,\tilde K_k$ (with $\tilde K_i$ distributed as the cluster of $\gamma_i(n)$), in which the status of an edge is determined by the first cluster that explores it (and by an independent coin toss if it is not explored by any of the clusters $\tilde K_1,\ldots,\tilde K_k$). By \eqref{eq:connection_off_the_lines}, we have that
\begin{equation}
\label{eq:connection_off_the_lines2}
  \P(\tilde K_i \cap F_i = \emptyset \text{ and } \gamma_i(m)\in \tilde K_i  \text{ for every $1\leq i \leq k$}) \geq \frac{1}{\sqrt{2}}\prod_{i=1}^k T_{p_c}(\gamma_i(n),\gamma_i(m))
\end{equation}
whenever $n\geq N_1$ and $m\geq \lambda n$. Now suppose that $\tilde K_i \cap F_i = \emptyset$ and $\gamma_i(m)\in \tilde K_i$ for every $1\leq i \leq k$ but that there exists $1\leq i_0\leq k$ such that $\gamma_{i_0}(n)$ is not connected to $\gamma_{i_0}(m)$ in the coupled configuration. Taking $2\leq i_0 \leq k$ to be minimal such that this occurs, the cluster $\tilde K_{i_0}$ must contain a path from $\gamma_{i_0}(n)$ to $\gamma_{i_0}(m)$ that visits one of the clusters $\tilde K_{j_0}$ with $j_0<i_0$. Letting $\mathscr{E}_{i_0,j_0}$ denote the event that that $\tilde K_i \cap F_i = \emptyset$ and $\gamma_i(m)\in \tilde K_i$ for every $1\leq i \leq k$ and that there is a simple path from $\gamma_{i_0}(n)$ to $\gamma_{i_0}(m)$ in $\tilde K_{i_0}$ that visits $\tilde K_{j_0}$, we have by a union bound that
\begin{multline}
  \P(\gamma_i(n) \leftrightarrow \gamma_i(m) \text{ and } \gamma_i(n) \nleftrightarrow F_i \text{ for every $1\leq i \leq k$}) \\\geq \P(\tilde K_i \cap F_i = \emptyset \text{ and } \gamma_i(m)\in \tilde K_i  \text{ for every $1\leq i \leq k$}) - \sum_{1\leq j_0<i_0\leq k} \P(\sE_{i_0,j_0}).
  \label{eq:affine_surgery2}
\end{multline}
Now, for each $1\leq j_0<i_0\leq k$, we can use the tree-graph inequality  to bound
\begin{align*}
  &\P(\sE_{i_0,j_0}) 
  \\&\leq \left[\prod_{i\neq i_0,j_0} T_{p_c}(\gamma_i(n),\gamma_i(m))\right] \sum_{a,b\in \Z^d} T_{p_c}(\gamma_{j_0}(n),a)T_{p_c}(a,\gamma_{j_0}(m)) T_{p_c}(a,b)T_{p_c}(\gamma_{i_0}(n),b)T_{p_c}(b,\gamma_{i_0}(m)).
\end{align*}
Using \eqref{eq:two_point_assumption} and the assumption that $d>6$, standard convolution estimates imply that $\P(\sE_{i_0,j_0})$ is much smaller than the product $\prod_{i=1}^kT_{p_c}(\gamma_i(n),\gamma_i(m))$ when $n$ is large and $m\geq \lambda n$, which concludes the proof together with \eqref{eq:connection_off_the_lines2} and \eqref{eq:affine_surgery2}. \qedhere

\end{proof}

\noindent
\textbf{Positivity of constants in \cref{sec:k_point_proof}.} We now prove a sharp topological characterisation of the positivity of the constant $\mathbf{C}(m,W)$ appearing in \cref{thm:scheme_function_variantW}.

\begin{lemma}
\label{lem:variant_positivity}
Let $d>6$ and $L\geq 1$ be such that critical percolation on $\mathbb{Z}^d_L$ satisfies the estimates \eqref{eq:two_point_assumption} and \eqref{eq:two_blob_assumption}. For each $m\geq 1$ and finite set of edges $W$, the constant $\mathbf{C}(m,W)$ appearing in \cref{thm:scheme_function_variantW} is positive if and only if there exist $m$ edge-disjoint paths from $0$ to $\infty$ in $\Z^d_L\setminus W$.
\end{lemma}

The proof of this lemma will require the following edge version of \cref{lem:affine_escape_vertex}.

\begin{lemma}[Affine escape: edge version]
\label{lem:affine_escape_edge}
Let $d\geq 3$ and $L \geq 1$, and suppose $(W,b_1,\ldots,b_k)$ is a blob for which there exist edge-disjoint paths $\gamma_1,\ldots,\gamma_k$ with $\gamma_i$ connecting $b_i$ to $\infty$ in $\Z^d_L \setminus W$. The paths $\gamma_i$ can always be chosen to be eventually affine with distinct slopes in the sense that there exist constant $a_i,q_i \in \N$ and $z_i\in \Z^d$ with $z_i/q_i \neq z_j/q_j$ for $i\neq j$ such that $\gamma_i(a_i+m q_i)=\gamma_i(a_i)+mz_i$ for every $1\leq i \leq k$ and $m\geq 1$.
\end{lemma}

\begin{proof}[Proof of \cref{lem:affine_escape_edge}]
The proof is very similar to that of \cref{lem:affine_escape_vertex}, with the main differences being in the initial reduction to eventually-affine paths of possibly equal slopes.
Fix $(W,b_1,\ldots,b_k)$, let the simple, edge-disjoint paths $\gamma_1,\ldots,\gamma_k$ be as in the statement of the lemma, and let the integer $R$ be sufficiently large that every point $b_1,\ldots,b_k$ and every endpoint of an edge in $W$ belongs to $[-R,R]^d$. For each $1\leq i \leq k$, let $x_i$ be the first point of $\Z^d\setminus [-R,R]^d$ visited by $\gamma_i$, and let $e_i$ be the vector traversed by $\gamma_i$ as it enters $x_i$, so that $x_i-e_i \in [-R,R]^d$. 
We claim that the rays $(\{x_i+\ell e_i : \ell \geq 0\})_{i=1}^k$ are edge-disjoint when considered as paths in $\Z^d_L$.  Once this is proven, replacing each path $\gamma_i$ by the path that follows $\gamma_i$ until the point $x_i$ and then follows the ray $\{x_i+\ell e_i:\ell\geq 0\}$ shows that we can always take the edge-disjoint paths $\gamma_i$ to be eventually affine, albeit with possibly equal slopes. The fact that the slopes can be taken to be distinct follows by an almost identical argument to the vertex case (using the edge version of Menger's theorem \cite[Theorem 3.3.6]{diestel2025graphtheory}) and is omitted. 

We now prove the claim that the rays $(\{x_i+\ell e_i : \ell \geq 0\})_{i=1}^k$ are edge-disjoint when considered as paths in $\Z^d_L$.
Note that $x_i$ must have a coordinate with absolute value strictly larger than $R$, the coordinate of $e_i$ in this direction must be non-zero and of the same sign as that of $x_i$, and the entire ray $\{x_i +\ell e_i : \ell \geq 0\}$ must be contained both in the half-space in which this coordinate has the same sign as that of $e_i$ and in the complement of the box $[-R,R]^d$. If two of the rays $\{x_i+\ell e_i : \ell \geq 0\}$ and $\{x_j+\ell e_j : \ell \geq 0\}$ are \emph{not} edge disjoint then they must have that either $e_i=e_j$ or $e_i=-e_j$ and that $x_i+\ell_1 e_i=x_j+\ell_2 e_j$ for some integers $\ell_1,\ell_2 \geq 0$. If $e_i=-e_j$, then $x_j-e_j=x_i+(\ell_1+\ell_2+1)e_i$ lies on the ray from $x_i$ and hence outside $[-R,R]^d$, contradicting $x_j-e_j\in[-R,R]^d$. Suppose instead that $e_i=e_j$ and that $x_i +\ell_1 e_i = x_j+\ell_2 e_i$ for some integers $\ell_1,\ell_2 \geq 0$. We may suppose without loss of generality that $\ell_1 \geq \ell_2$, in which case $x_j=x_i+(\ell_1 -\ell_2)e_i$. If $\ell_1>\ell_2$ then it follows that $x_j-e_j=x_j-e_i$ is contained in the ray $\{x_i+\ell e_i:\ell \geq 0\}$ and hence contained in the complement of the box $[-R,R]^d$, contradicting the definition of $x_j$ and $e_j$. We deduce that in fact $\ell_2=\ell_1$ so that $x_i=x_j$ and $e_i=e_j$, which is only possible if $i=j$ since the paths $\gamma_i$ were assumed to be edge-disjoint.
\end{proof}

\begin{proof}[Proof of \cref{lem:variant_positivity}]
This proof is very similar to that of \cref{lem:blob_factor_positivity} except that we use \cref{lem:affine_escape_edge} instead of \cref{lem:affine_escape_vertex}, and we omit the details. (Let us however note an interesting point that since we now need our connections to be disjoint, rather than belong to different clusters, it is now the fact that the \emph{open bubble} $T_{p_c}*T_{p_c}(0,x)$ converges to zero as $x\to \infty$ that is relevant for the proof, rather than the open triangle. See \cite[Remark 5.4]{hutchcroft20192} and \cite{HutchcroftReeves} for further discussion of related phenomena and their consequences for percolation \emph{at and below} the upper critical dimension.)
\end{proof}

 \setstretch{1}
 \footnotesize{
  \bibliographystyle{abbrv}
  \bibliography{unimodularthesis.bib}

@article{holmes2007weak,
  title={WEAK CONVERGENCE OF MEASURE-VALUED PROCESSES AND r-POINT FUNCTIONS},
  author={Holmes, Mark and Perkins, Edwin},
  journal={The Annals of Probability},
  volume={35},
  number={5},
  pages={1769--1782},
  year={2007}
}

@article{van2017criterion,
  title={A CRITERION FOR CONVERGENCE TO SUPER-{B}ROWNIAN MOTION ON PATH SPACE},
  author={\Hofstad, Remco and Holmes, Mark and Perkins, Edwin A},
  journal={The Annals of Probability},
  pages={278--376},
  year={2017},
  publisher={JSTOR}
}

@article{petersen1982relation,
  title={On the relation between the multidimensional moment problem and the one-dimensional moment problem},
  author={Petersen, LC690537},
  journal={Mathematica Scandinavica},
  pages={361--366},
  year={1982},
  publisher={JSTOR}
}

@article{athreya2016gap,
  title={The gap between {G}romov-vague and {G}romov--{H}ausdorff-vague topology},
  author={Athreya, Siva and L{\"o}hr, Wolfgang and Winter, Anita},
  journal={Stochastic Processes and their Applications},
  volume={126},
  number={9},
  pages={2527--2553},
  year={2016},
  publisher={Elsevier}
}

@article{van2014cycle,
  title={Cycle structure of percolation on high-dimensional tori},
  author={\Hofstad, Remco and Sapozhnikov, Art{\"e}m},
  journal={Annales de l'IHP Probabilit{\'e}s et statistiques},
  volume={50},
  number={3},
  pages={999--1027},
  year={2014}
}

@article{liu2026crossover,
  title={Crossover from subcritical to critical decay: random walk, self-avoiding walk, percolation},
  author={Liu, Yucheng and Slade, Gordon},
  journal={arXiv preprint arXiv:2605.15545},
  year={2026}
}

@article{panis2025sharp,
  title={Sharp bounds on the half-space two-point function for high-dimensional {B}ernoulli percolation},
  author={Panis, Romain and Schapira, Bruno},
  journal={arXiv preprint arXiv:2512.13624},
  year={2025}
}

@article{carpenter2025loops,
  title={On Loops in critical high-dimensional percolation},
  author={Carpenter, Amelia and Werner, Wendelin},
  journal={arXiv preprint arXiv:2507.02676},
  year={2025}
}

@article{blanc2024scaling,
  title={The scaling limit of critical hypercube percolation},
  author={Blanc-Renaudie, Arthur and Broutin, Nicolas and Nachmias, Asaf},
  journal={arXiv preprint arXiv:2401.16365},
  year={2024}
}

@article{depperschmidt2011marked,
  title={Marked metric measure spaces},
  author={Depperschmidt, Andrej and Greven, Andreas and Pfaffelhuber, Peter},
  journal={Electronic Communications in Probability},
  volume={16},
  year={2011},
  publisher={Institute of Mathematical Statistics}
}

@article{sheffield2006exploration,
  title={Exploration trees and conformal loop ensembles},
  author={Sheffield, Scott},
  journal={Duke Mathematical Journal},
  volume={147},
  number={1},
  pages={79--129},
  doi={10.1215/00127094-2009-007},
  year={2009}
}

@article{miller2025existence_BM,
  title={Existence and uniqueness of the canonical {B}rownian motion in non-simple conformal loop ensemble gaskets},
  author={Miller, Jason and Yuan, Yizheng},
  journal={arXiv preprint arXiv:2512.04807},
  year={2025}
}

@article{miller2025existence,
  title={Existence and uniqueness of the conformally covariant geodesic metric on simple conformal loop ensemble carpets},
  author={Miller, Jason and Tian, Yi},
  journal={arXiv preprint arXiv:2511.16208},
  year={2025}
}

@article{duminil2026random,
  title={A random walk approach to high-dimensional critical phenomena},
  author={Duminil-Copin, Hugo and Markar, Aman and Panis, Romain and Slade, Gordon},
  journal={arXiv preprint arXiv:2605.21438},
  year={2026}
}

@article{panis2026reversing,
  title={On reversing the {S}imon-{L}ieb inequality in high-dimensional percolation},
  author={Panis, Romain and Schapira, Bruno},
  journal={arXiv preprint arXiv:2605.30299},
  year={2026}
}

@article{chatterjee2025robust,
  title={Robust construction of the incipient infinite cluster in high dimensional critical percolation},
  author={Chatterjee, Shirshendu and Chinmay, Pranav and Hanson, Jack and Sosoe, Philippe},
  journal={arXiv preprint arXiv:2502.10882},
  year={2025}
}

@article{chatterjee2025limiting,
  title={Limiting distribution of the chemical distance in high dimensional critical percolation},
  author={Chatterjee, Shirshendu and Chinmay, Pranav and Hanson, Jack and Sosoe, Philippe},
  journal={arXiv preprint arXiv:2509.06236},
  year={2025}
}

@article{chatterjee2026convergence,
  title={Convergence of $k$-point functions in high dimensional percolation},
  author={Chatterjee, Shirshendu and Chinmay, Pranav and Hanson, Jack and Sosoe, Philippe},
  journal={arXiv preprint arXiv:2604.08462},
  year={2026}
}

@article{croydon2018scaling,
  title={Scaling limits of stochastic processes associated with resistance forms},
  author={Croydon, David},
  journal={Annales de l'Institut Henri Poincar{\'e}, Probabilit{\'e}s et Statistiques},
  volume={54},
  number={4},
  pages={1939--1968},
  doi={10.1214/17-AIHP861},
  year={2018}
}

@article{croydon2008convergence,
  title={Convergence of simple random walks on random discrete trees to Brownian motion on the continuum random tree},
  author={Croydon, David},
  journal={Annales de l'IHP Probabilit{\'e}s et statistiques},
  volume={44},
  number={6},
  pages={987--1019},
  doi={10.1214/07-AIHP153},
  year={2008}
}

@article{archer2024ghp,
  title={The {GHP} scaling limit of uniform spanning trees in high dimensions},
  author={Archer, Eleanor and Nachmias, Asaf and Shalev, Matan},
  journal={Communications in Mathematical Physics},
  volume={405},
  number={3},
  pages={73},
  year={2024},
  publisher={Springer}
}

@article{ambrosio2025tightness,
  title={Tightness of approximations to metrics on non-simple conformal loop ensemble gaskets},
  author={Ambrosio, Valeria and Miller, Jason and Yuan, Yizheng},
  journal={arXiv preprint arXiv:2507.15589},
  year={2025}
}

@article{djankovic2026scaling,
  title={The scaling limit of random walk and the intrinsic metric on planar critical percolation},
  author={Dankovi{\'c}, Irina and Markering, Maarten and Miller, Jason and Yuan, Yizheng},
  journal={arXiv preprint arXiv:2604.14122},
  year={2026}
}

@article{stenull2003logarithmic,
  title={Logarithmic corrections to scaling in critical percolation and random resistor networks},
  author={Stenull, Olaf and Janssen, Hans-Karl},
  journal={Physical Review E},
  volume={68},
  number={3},
  pages={036129},
  year={2003},
  publisher={APS}
}

@unpublished{HutchcroftReeves,
author={Hutchcroft, Tom and Reeves, Lily},
title={The intrinsic geometry of critical long-range percolation clusters at and above the upper critical dimension},
note={In preparation}
}

@book{diestel2025graphtheory,
  author    = {Diestel, Reinhard},
  title     = {Graph Theory},
  edition   = {6},
  series    = {Graduate Texts in Mathematics},
  publisher = {Springer Berlin Heidelberg},
  year      = {2025},
  isbn      = {978-3-662-70106-5},
  ean       = {9783662701065},
  doi       = {10.1007/978-3-662-70107-2},
  url       = {https://doi.org/10.1007/978-3-662-70107-2},
  pages     = {455},
  issn      = {0072-5285},
  eissn     = {2197-5612}
}

@article{angel2021tail,
  title={On the tail of the branching random walk local time},
  author={Angel, Omer and Hutchcroft, Tom and J{\'a}rai, Antal},
  journal={Probability Theory and Related Fields},
  volume={180},
  number={1},
  pages={467--494},
  year={2021},
  publisher={Springer}
}

@article{blanc2025critical,
  title={Critical percolation on the discrete torus in high dimensions},
  author={Blanc-Renaudie, Arthur and Nachmias, Asaf},
  journal={arXiv preprint arXiv:2512.19672},
  year={2025}
}

@article{cabezas2025bi,
  title={Bi-infinite incipient cluster in high dimensions},
  author={Cabezas, Manuel and Fribergh, Alexander and Heydenreich, Markus and J{\'a}rai, Antal A},
  journal={arXiv preprint arXiv:2506.06559},
  year={2025}
}

@article{van2010convergence,
  title={Convergence of the critical finite-range contact process to super-{B}rownian motion above the upper critical dimension: the higher-point functions},
  author={\Hofstad, Remco and Sakai, Akira},
  journal={Electronic Journal of Probability},
  volume={15},
  number={27},
  pages={801--894},
  doi={10.1214/EJP.v15-783},
  year={2010}
}

@article{camia2024conformal,
  title={Conformal covariance of connection probabilities and fields in {2D} critical percolation},
  author={Camia, Federico},
  journal={Communications on Pure and Applied Mathematics},
  volume={77},
  number={3},
  pages={2138--2176},
  year={2024},
  publisher={Wiley Online Library}
}

@book{le1999spatial,
  title={Spatial branching processes, random snakes and partial differential equations},
  author={Le Gall, Jean-Fran{\c{c}}ois},
  year={1999},
  publisher={Springer Science \& Business Media}
}

@book{dynkin1994introduction,
  title={An introduction to branching measure-valued processes},
  author={Dynkin, Eugene Borisovich},
  series={CRM Monograph Series},
  number={6},
  year={1994},
  publisher={American Mathematical Soc.}
}

@article{perkins2002part,
  title={Dawson-{W}atanabe superprocesses and measure-valued diffusions},
  author={Perkins, Edwin},
  journal={Lectures on probability theory and statistics},
  pages={125--329},
  year={2002},
  publisher={Springer}
}

@article{slade2002scaling,
  title={Scaling limits and super-{B}rownian motion},
  author={Slade, Gordon},
  journal={Notices AMS},
  volume={49},
  number={9},
  pages={1056--1067},
  year={2002}
}

@article{aldous1991continuum,
  title={The continuum random tree. {II}. An overview},
  author={Aldous, David},
  journal={Stochastic analysis (Durham, 1990)},
  volume={167},
  pages={23--70},
  year={1991},
  publisher={Citeseer}
}

@article{holmes2020range,
  title={On the range of lattice models in high dimensions},
  author={Holmes, Mark and Perkins, Edwin},
  journal={Probability Theory and Related Fields},
  volume={176},
  number={3},
  pages={941--1009},
  year={2020},
  publisher={Springer}
}

@article{holmes2008convergence,
  title={Convergence of lattice trees to super-{B}rownian motion above the critical dimension},
  author={Holmes, Mark P.},
  journal={Electronic Journal of Probability},
  volume={13},
  number={23},
  pages={671--755},
  doi={10.1214/EJP.v13-499},
  year={2008}
}

@article{borgs1999mean,
  title={Mean-field lattice trees},
  author={Borgs, Christian and Chayes, Jennifer and \Hofstad, Remco and Slade, Gordon},
  journal={Annals of Combinatorics},
  volume={3},
  pages={205--221},
  year={1999},
  publisher={Springer}
}

@article{cox1999rescaled,
  title={Rescaled particle systems converging to super-Brownian motion},
  author={Cox, Ted and Durrett, Richard and Perkins, Edwin A},
  journal={Perplexing Problems in Probability: Festschrift in Honor of Harry Kesten},
  pages={269--284},
  year={1999},
  publisher={Springer}
}

@article{derbez1998scaling,
  title={The scaling limit of lattice trees in high dimensions},
  author={Derbez, Eric and Slade, Gordon},
  journal={Communications in mathematical physics},
  volume={193},
  number={1},
  pages={69--104},
  year={1998},
  publisher={Springer}
}

@article{derbez1997lattice,
  title={Lattice trees and super-Brownian motion},
  author={Derbez, Eric and Slade, Gordon},
  journal={Canadian Mathematical Bulletin},
  volume={40},
  number={1},
  pages={19--38},
  year={1997},
  publisher={Cambridge University Press}
}

@article{van2013survival,
  title={The survival probability and r-point functions in high dimensions},
  author={\Hofstad, Remco and Holmes, Mark},
  journal={Annals of mathematics},
  pages={665--685},
  year={2013},
  publisher={JSTOR}
}

@article{hara1990mean,
  title={Mean-field critical behaviour for correlation length for percolation in high dimensions},
  author={Hara, Takashi},
  journal={Probability theory and related fields},
  volume={86},
  number={3},
  pages={337--385},
  year={1990},
  publisher={Springer}
}

@article{asselah2025capacity,
  title={Capacity in high dimensional percolation},
  author={Asselah, Amine and Schapira, Bruno and Sousi, Perla},
  journal={arXiv preprint arXiv:2509.21253},
  year={2025}
}

@article{van2025one,
  title={One-arm exponents of high-dimensional percolation revisited},
  author={van Engelenburg, Diederik and Garban, Christophe and Panis, Romain and Severo, Franco},
  journal={arXiv preprint arXiv:2510.21595},
  year={2025}
}

@article{hutchcroft2025dimension,
  title={Dimension dependence of critical phenomena in long-range percolation},
  author={Hutchcroft, Tom},
  journal={arXiv preprint arXiv:2510.03951},
  year={2025}
}

@article{durrett1999rescaled,
  title={Rescaled contact processes converge to super-{B}rownian motion in two or more dimensions},
  author={Durrett, Richard and Perkins, Edwin A},
  journal={Probability Theory and Related Fields},
  volume={114},
  number={3},
  pages={309--399},
  year={1999},
  publisher={Springer}
}

@article{van2006infinite,
  title={Infinite canonical super-{B}rownian motion and scaling limits},
  author={\Hofstad, Remco},
  journal={Communications in mathematical physics},
  volume={265},
  number={3},
  pages={547--583},
  year={2006},
  publisher={Springer}
}

@article{aldous1993continuum,
  title={The continuum random tree {III}},
  author={Aldous, David},
  journal={The Annals of Probability},
  pages={248--289},
  year={1993},
  publisher={JSTOR}
}

@article{watanabe1968limit,
  title={A limit theorem of branching processes and continuous state branching processes},
  author={Watanabe, Shinzo},
  journal={Journal of Mathematics of Kyoto University},
  volume={8},
  number={1},
  pages={141--167},
  year={1968},
  publisher={Duke University Press}
}

@article{hara1998incipient,
  title={The incipient infinite cluster in high-dimensional percolation},
  author={Hara, Takashi and Slade, Gordon},
  journal={Electronic Research Announcements of the American Mathematical Society},
  volume={4},
  number={8},
  pages={48--55},
  year={1998}
}

@article{van2003convergence,
  title={Convergence of critical oriented percolation to super-{B}rownian motion above $4+ 1$ dimensions},
  author={Van der Hofstad, Remco and Slade, Gordon},
  journal={Annales de l'IHP Probabilit{\'e}s et statistiques},
  volume={39},
  number={3},
  pages={413--485},
  doi={10.1016/S0246-0203(03)00008-6},
  year={2003}
}

@article{gladkov2024percolation,
  title={Percolation Inequalities and Decision Trees},
  author={Gladkov, Nikita},
  journal={arXiv preprint arXiv:2408.08457},
  year={2024}
}

@article{LRPpaper1,
  title={Critical long-range percolation {I}: High effective dimension},
  author={Hutchcroft, Tom},
  journal={arXiv preprint arXiv:2508.18807},
  year={2025}
}

@article{LRPpaper2,
  title={Critical long-range percolation {II}: Low effective dimension},
  author={Hutchcroft, Tom},
  journal={arXiv preprint arXiv:2508.18808},
  year={2025}
}

@article{LRPpaper3,
  title={Critical long-range percolation {III}: The upper critical dimension},
  author={Hutchcroft, Tom},
  journal={arXiv preprint arXiv:2508.18809},
  year={2025}
}

@article{ang2021integrability,
  title={Integrability of Conformal Loop Ensemble: {I}maginary {DOZZ} Formula and Beyond},
  author={Ang, Morris and Cai, Gefei and Sun, Xin and Wu, Baojun},
  journal={arXiv preprint arXiv:2107.01788},
  year={2021}
}

@article{duminil2024alternativeSAW,
  title={An alternative approach for the mean-field behaviour of weakly self-avoiding walks in dimensions $ d> 4$},
  author={Duminil-Copin, Hugo and Panis, Romain},
  journal={arXiv preprint arXiv:2410.03649},
  year={2024}
}

@article{duminil2024alternative,
  title={An alternative approach for the mean-field behaviour of spread-out {B}ernoulli percolation in dimensions $ d> 6$},
  author={Duminil-Copin, Hugo and Panis, Romain},
  journal={arXiv preprint arXiv:2410.03647},
  year={2024}
}

@article{heydenreich2014high,
  title={High-dimensional incipient infinite clusters revisited},
  author={Heydenreich, Markus and \Hofstad, Remco and Hulshof, Tim},
  journal={Journal of Statistical Physics},
  volume={155},
  pages={966--1025},
  year={2014},
  publisher={Springer}
}

@article{van2004incipient,
  title={The incipient infinite cluster for high-dimensional unoriented percolation},
  author={Van der Hofstad, Remco and J{\'a}rai, Antal A},
  journal={Journal of statistical physics},
  volume={114},
  number={3},
  pages={625--663},
  year={2004},
  publisher={Springer}
}

@article{chatterjee2020restricted,
  title={Restricted percolation critical exponents in high dimensions},
  author={Chatterjee, Shirshendu and Hanson, Jack},
  journal={Communications on Pure and Applied Mathematics},
  volume={73},
  number={11},
  pages={2370--2429},
  doi={10.1002/cpa.21938},
  year={2020},
  publisher={Wiley Online Library}
}

@article{hutchcroft2023high,
  title={High-dimensional near-critical percolation and the torus plateau},
  author={Hutchcroft, Tom and Michta, Emmanuel and Slade, Gordon},
  journal={The Annals of Probability},
  volume={51},
  number={2},
  pages={580--625},
  year={2023},
  publisher={Institute of Mathematical Statistics}
}

@article {MR1434128,
    AUTHOR = {Aldous, David},
     TITLE = {Brownian excursions, critical random graphs and the
              multiplicative coalescent},
   JOURNAL = {Ann. Probab.},
  FJOURNAL = {The Annals of Probability},
    VOLUME = {25},
      YEAR = {1997},
    NUMBER = {2},
     PAGES = {812--854},
      ISSN = {0091-1798},
   MRCLASS = {60C05 (60J50)},
  MRNUMBER = {1434128},
MRREVIEWER = {Endre Cs\'{a}ki},
       DOI = {10.1214/aop/1024404421},
       URL = {https://doi.org/10.1214/aop/1024404421},
}

@article{chatterjee2021subcritical,
  title={Subcritical Connectivity and Some Exact Tail Exponents in High Dimensional Percolation},
  author={Chatterjee, Shirshendu and Hanson, Jack and Sosoe, Philippe},
  journal={Communications in Mathematical Physics},
  volume={403},
  pages={83--153},
  doi={10.1007/s00220-023-04759-w},
  year={2023}
}

@article {MR2165583,
    AUTHOR = {Borgs, Christian and Chayes, Jennifer T. and van der Hofstad,
              Remco and Slade, Gordon and Spencer, Joel},
     TITLE = {Random subgraphs of finite graphs. {II}. {T}he lace expansion
              and the triangle condition},
   JOURNAL = {Ann. Probab.},
  FJOURNAL = {The Annals of Probability},
    VOLUME = {33},
      YEAR = {2005},
    NUMBER = {5},
     PAGES = {1886--1944},
      ISSN = {0091-1798},
   MRCLASS = {60C05 (05C80 60K35 82B43)},
  MRNUMBER = {2165583},
MRREVIEWER = {Sotiris E. Nikoletseas},
       DOI = {10.1214/009117905000000260},
       URL = {https://doi.org/10.1214/009117905000000260},
}

@article {MR2155704,
    AUTHOR = {Borgs, Christian and Chayes, Jennifer T. and van der Hofstad,
              Remco and Slade, Gordon and Spencer, Joel},
     TITLE = {Random subgraphs of finite graphs. {I}. {T}he scaling window
              under the triangle condition},
   JOURNAL = {Random Structures Algorithms},
  FJOURNAL = {Random Structures \& Algorithms},
    VOLUME = {27},
      YEAR = {2005},
    NUMBER = {2},
     PAGES = {137--184},
      ISSN = {1042-9832},
   MRCLASS = {05C80 (60C05 60D05 60K35)},
  MRNUMBER = {2155704},
MRREVIEWER = {Sotiris E. Nikoletseas},
       DOI = {10.1002/rsa.20051},
       URL = {https://doi.org/10.1002/rsa.20051},
}

@article {MR2276449,
    AUTHOR = {Heydenreich, Markus and \Hofstad, Remco},
     TITLE = {Random graph asymptotics on high-dimensional tori},
   JOURNAL = {Comm. Math. Phys.},
  FJOURNAL = {Communications in Mathematical Physics},
    VOLUME = {270},
      YEAR = {2007},
    NUMBER = {2},
     PAGES = {335--358},
      ISSN = {0010-3616},
   MRCLASS = {60K35 (05C80 82B43)},
  MRNUMBER = {2276449},
MRREVIEWER = {Antal A. J\'{a}rai},
       DOI = {10.1007/s00220-006-0152-8},
       URL = {https://doi.org/10.1007/s00220-006-0152-8},
}

@article {MR2653185,
    AUTHOR = {Nachmias, Asaf and Peres, Yuval},
     TITLE = {The critical random graph, with martingales},
   JOURNAL = {Israel J. Math.},
  FJOURNAL = {Israel Journal of Mathematics},
    VOLUME = {176},
      YEAR = {2010},
     PAGES = {29--41},
      ISSN = {0021-2172},
   MRCLASS = {05C80 (60C05)},
  MRNUMBER = {2653185},
MRREVIEWER = {David B. Penman},
       DOI = {10.1007/s11856-010-0019-8},
       URL = {https://doi.org/10.1007/s11856-010-0019-8},
}

@article {MR1431856,
    AUTHOR = {Aizenman, Michael},
     TITLE = {On the number of incipient spanning clusters},
   JOURNAL = {Nuclear Phys. B},
  FJOURNAL = {Nuclear Physics. B. Theoretical, Phenomenological, and
              Experimental High Energy Physics. Quantum Field Theory and
              Statistical Systems},
    VOLUME = {485},
      YEAR = {1997},
    NUMBER = {3},
     PAGES = {551--582},
      ISSN = {0550-3213},
   MRCLASS = {82B43},
  MRNUMBER = {1431856},
       DOI = {10.1016/S0550-3213(96)00626-8},
       URL = {https://doi.org/10.1016/S0550-3213(96)00626-8},
}

@article {MR1773141,
    AUTHOR = {Hara, Takashi and Slade, Gordon},
     TITLE = {The scaling limit of the incipient infinite cluster in
              high-dimensional percolation. {I}. {C}ritical exponents},
   JOURNAL = {J. Statist. Phys.},
  FJOURNAL = {Journal of Statistical Physics},
    VOLUME = {99},
      YEAR = {2000},
    NUMBER = {5-6},
     PAGES = {1075--1168},
      ISSN = {0022-4715},
   MRCLASS = {82B43 (60K35)},
  MRNUMBER = {1773141},
MRREVIEWER = {Remco van der Hofstad},
       DOI = {10.1023/A:1018628503898},
       URL = {https://doi.org/10.1023/A:1018628503898},
}

@article {MR1757958,
    AUTHOR = {Hara, Takashi and Slade, Gordon},
     TITLE = {The scaling limit of the incipient infinite cluster in
              high-dimensional percolation. {II}. {I}ntegrated
              super-{B}rownian excursion},
      NOTE = {Probabilistic techniques in equilibrium and nonequilibrium
              statistical physics},
   JOURNAL = {J. Math. Phys.},
  FJOURNAL = {Journal of Mathematical Physics},
    VOLUME = {41},
      YEAR = {2000},
    NUMBER = {3},
     PAGES = {1244--1293},
      ISSN = {0022-2488},
   MRCLASS = {82B43 (60K35)},
  MRNUMBER = {1757958},
MRREVIEWER = {Remco van der Hofstad},
       DOI = {10.1063/1.533186},
       URL = {https://doi.org/10.1063/1.533186},
}

@article{essam1978percolation,
  title={Percolation theory at the critical dimension},
  author={Essam, IW and Gaunt, DS and Guttmann, AJ},
  journal={Journal of Physics A: Mathematical and General},
  volume={11},
  number={10},
  pages={1983},
  year={1978},
  publisher={IOP Publishing}
}

@article{HutchcroftTriangle,
  title={On the derivation of mean-field percolation critical exponents from the triangle condition},
  author={Hutchcroft, Tom},
  journal={Journal of Statistical Physics},
  volume={189},
  number={1},
  pages={6},
  year={2022},
  publisher={Springer}
}

@article {MR782962,
    AUTHOR = {Brydges, David and Spencer, Thomas},
     TITLE = {Self-avoiding walk in {$5$} or more dimensions},
   JOURNAL = {Comm. Math. Phys.},
  FJOURNAL = {Communications in Mathematical Physics},
    VOLUME = {97},
      YEAR = {1985},
    NUMBER = {1-2},
     PAGES = {125--148},
      ISSN = {0010-3616},
   MRCLASS = {82A67 (60J15 82A25 82A43)},
  MRNUMBER = {782962},
MRREVIEWER = {Gregory F. Lawler},
       URL = {http://projecteuclid.org/euclid.cmp/1103941982},
}

@article{smirnov2001critical2,
  title={Critical percolation in the plane: conformal invariance, {C}ardy's formula, scaling limits},
  author={Smirnov, Stanislav},
  journal={Comptes Rendus de l'Acad{\'e}mie des Sciences-Series I-Mathematics},
  volume={333},
  number={3},
  pages={239--244},
  year={2001},
  publisher={Elsevier}
}

@article{hutchcroft2020slightly,
  title={Slightly supercritical percolation on non-amenable graphs {I}: {T}he distribution of finite clusters},
  author={Hutchcroft, Tom},
  journal={Proceedings of the London Mathematical Society},
  volume={125},
  number={4},
  pages={968--1013},
  doi={10.1112/plms.12474},
  year={2022}
}

@article{sakai2007lace,
  title={Lace expansion for the {I}sing model},
  author={Sakai, Akira},
  journal={Communications in mathematical physics},
  volume={272},
  number={2},
  pages={283--344},
  year={2007},
  publisher={Springer}
}

@article{hutchcroft20192,
  title={The {$L^2$} boundedness condition in nonamenable percolation},
  author={Hutchcroft, Tom},
  journal={Electronic Journal of Probability},
  volume={25},
  number={127},
  pages={1--27},
  doi={10.1214/20-EJP525},
  year={2020}
}

@article {MR1905854,
    AUTHOR = {Campanino, Massimo and Ioffe, Dmitry},
     TITLE = {Ornstein-{Z}ernike theory for the {B}ernoulli bond percolation
              on {$\Bbb Z^d$}},
   JOURNAL = {Ann. Probab.},
  FJOURNAL = {The Annals of Probability},
    VOLUME = {30},
      YEAR = {2002},
    NUMBER = {2},
     PAGES = {652--682},
      ISSN = {0091-1798},
   MRCLASS = {60K35 (60F15 60K15 82B43)},
  MRNUMBER = {1905854},
MRREVIEWER = {Yvan Velenik},
       DOI = {10.1214/aop/1023481005},
       URL = {https://doi.org/10.1214/aop/1023481005},
}

@article {MR1174248,
    AUTHOR = {Hara, Takashi and Slade, Gordon},
     TITLE = {The lace expansion for self-avoiding walk in five or more
              dimensions},
   JOURNAL = {Rev. Math. Phys.},
  FJOURNAL = {Reviews in Mathematical Physics. A Journal for Both Review and
              Original Research Papers in the Field of Mathematical Physics},
    VOLUME = {4},
      YEAR = {1992},
    NUMBER = {2},
     PAGES = {235--327},
      ISSN = {0129-055X},
   MRCLASS = {82B41 (60J15)},
  MRNUMBER = {1174248},
MRREVIEWER = {Neal Madras},
       DOI = {10.1142/S0129055X9200008X},
}

@article {MR1959796,
    AUTHOR = {Hara, Takashi and \Hofstad, Remco and Slade, Gordon},
     TITLE = {Critical two-point functions and the lace expansion for
              spread-out high-dimensional percolation and related models},
   JOURNAL = {Ann. Probab.},
  FJOURNAL = {The Annals of Probability},
    VOLUME = {31},
      YEAR = {2003},
    NUMBER = {1},
     PAGES = {349--408},
      ISSN = {0091-1798},
   MRCLASS = {60K35 (60K37 82B41 82B43)},
  MRNUMBER = {1959796},
MRREVIEWER = {Ilie A. Grigorescu},
       DOI = {10.1214/aop/1046294314},
       URL = {http://dx.doi.org.ezproxy.library.ubc.ca/10.1214/aop/1046294314},
}

@article {MR2393990,
    AUTHOR = {Hara, Takashi},
     TITLE = {Decay of correlations in nearest-neighbor self-avoiding walk,
              percolation, lattice trees and animals},
   JOURNAL = {Ann. Probab.},
  FJOURNAL = {The Annals of Probability},
    VOLUME = {36},
      YEAR = {2008},
    NUMBER = {2},
     PAGES = {530--593},
      ISSN = {0091-1798},
   MRCLASS = {82B27 (60K35 82B41 82B43 82C41)},
  MRNUMBER = {2393990},
MRREVIEWER = {Akira Sakai},
       DOI = {10.1214/009117907000000231},
       URL = {http://dx.doi.org.ezproxy.library.ubc.ca/10.1214/009117907000000231},
}

@book {MR2239599,
    AUTHOR = {Slade, G.},
     TITLE = {The lace expansion and its applications},
    SERIES = {Lecture Notes in Mathematics},
    VOLUME = {1879},
      NOTE = {Lectures from the 34th Summer School on Probability Theory
              held in Saint-Flour, July 6--24, 2004,
              Edited and with a foreword by Jean Picard},
 PUBLISHER = {Springer-Verlag, Berlin},
      YEAR = {2006},
     PAGES = {xiv+228},
      ISBN = {978-3-540-31189-8; 3-540-31189-0},
   MRCLASS = {60K35 (82B27)},
  MRNUMBER = {2239599},
MRREVIEWER = {Akira Sakai},
}

@book{heydenreich2015progress,
    AUTHOR = {Heydenreich, Markus and \Hofstad, Remco},
     TITLE = {Progress in high-dimensional percolation and random graphs},
    SERIES = {CRM Short Courses},
 PUBLISHER = {Springer, Cham; Centre de Recherches Math\'{e}matiques, Montreal,
              QC},
      YEAR = {2017},
     PAGES = {xii+285},
      ISBN = {978-3-319-62472-3; 978-3-319-62473-0},
   MRCLASS = {60K35 (05C80 82B41 82B43)},
  MRNUMBER = {3729454},
MRREVIEWER = {Christian M\"{o}nch},
}

@article {MR1043524,
    AUTHOR = {Hara, Takashi and Slade, Gordon},
     TITLE = {Mean-field critical behaviour for percolation in high
              dimensions},
   JOURNAL = {Comm. Math. Phys.},
  FJOURNAL = {Communications in Mathematical Physics},
    VOLUME = {128},
      YEAR = {1990},
    NUMBER = {2},
     PAGES = {333--391},
      ISSN = {0010-3616},
   MRCLASS = {82B43 (60K35)},
  MRNUMBER = {1043524},
MRREVIEWER = {R. T. Smythe},
}

@article {MR2748397,
    AUTHOR = {Kozma, Gady and Nachmias, Asaf},
     TITLE = {Arm exponents in high dimensional percolation},
   JOURNAL = {J. Amer. Math. Soc.},
  FJOURNAL = {Journal of the American Mathematical Society},
    VOLUME = {24},
      YEAR = {2011},
    NUMBER = {2},
     PAGES = {375--409},
      ISSN = {0894-0347},
   MRCLASS = {60K35},
  MRNUMBER = {2748397},
MRREVIEWER = {Antal A. J\~A!`rai},
       DOI = {10.1090/S0894-0347-2010-00684-4},
}

@article {MR2551766,
    AUTHOR = {Kozma, Gady and Nachmias, Asaf},
     TITLE = {The {A}lexander-{O}rbach conjecture holds in high dimensions},
   JOURNAL = {Invent. Math.},
  FJOURNAL = {Inventiones Mathematicae},
    VOLUME = {178},
      YEAR = {2009},
    NUMBER = {3},
     PAGES = {635--654},
      ISSN = {0020-9910},
   MRCLASS = {60K35},
  MRNUMBER = {2551766},
MRREVIEWER = {Thomas Polaski},
       DOI = {10.1007/s00222-009-0208-4},
}

@article {MR1127713,
    AUTHOR = {Barsky, D. J. and Aizenman, M.},
     TITLE = {Percolation critical exponents under the triangle condition},
   JOURNAL = {Ann. Probab.},
  FJOURNAL = {The Annals of Probability},
    VOLUME = {19},
      YEAR = {1991},
    NUMBER = {4},
     PAGES = {1520--1536},
      ISSN = {0091-1798},
   MRCLASS = {60K35 (82B43)},
  MRNUMBER = {1127713},
MRREVIEWER = {John C. Wierman},
}

@article {MR762034,
    AUTHOR = {Aizenman, Michael and Newman, Charles M.},
     TITLE = {Tree graph inequalities and critical behavior in percolation
              models},
   JOURNAL = {J. Statist. Phys.},
  FJOURNAL = {Journal of Statistical Physics},
    VOLUME = {36},
      YEAR = {1984},
    NUMBER = {1-2},
     PAGES = {107--143},
      ISSN = {0022-4715},
   MRCLASS = {82A43 (60K35 82A25)},
  MRNUMBER = {762034},
MRREVIEWER = {R. T. Smythe},
       DOI = {10.1007/BF01015729},
       URL = {http://dx.doi.org/10.1007/BF01015729},
}

@article {Schramm00,
    AUTHOR = {Schramm, Oded},
     TITLE = {Scaling limits of loop-erased random walks and uniform
              spanning trees},
   JOURNAL = {Israel J. Math.},
  FJOURNAL = {Israel Journal of Mathematics},
    VOLUME = {118},
      YEAR = {2000},
     PAGES = {221--288},
      ISSN = {0021-2172},
     CODEN = {ISJMAP},
   MRCLASS = {60K35 (30C85 60H15 82B27 82B44)},
  MRNUMBER = {1776084 (2001m:60227)},
MRREVIEWER = {Almut Burchard},
       DOI = {10.1007/BF02803524},
       URL = {http://dx.doi.org.ezproxy.library.ubc.ca/10.1007/BF02803524},
}

@article {Kesten86,
    AUTHOR = {Kesten, Harry},
     TITLE = {The incipient infinite cluster in two-dimensional percolation},
   JOURNAL = {Probab. Theory Related Fields},
  FJOURNAL = {Probability Theory and Related Fields},
    VOLUME = {73},
      YEAR = {1986},
    NUMBER = {3},
     PAGES = {369--394},
      ISSN = {0178-8051},
     CODEN = {PTRFEU},
   MRCLASS = {60K35},
  MRNUMBER = {859839 (88c:60196)},
MRREVIEWER = {R. T. Smythe},
       DOI = {10.1007/BF00776239},
       URL = {http://dx.doi.org.ezproxy.library.ubc.ca/10.1007/BF00776239},
}

@article {Jar03,
    AUTHOR = {J{\'a}rai, Antal A.},
     TITLE = {Incipient infinite percolation clusters in 2{D}},
   JOURNAL = {Ann. Probab.},
  FJOURNAL = {The Annals of Probability},
    VOLUME = {31},
      YEAR = {2003},
    NUMBER = {1},
     PAGES = {444--485},
      ISSN = {0091-1798},
     CODEN = {APBYAE},
   MRCLASS = {60K35 (82B43)},
  MRNUMBER = {1959799 (2004a:60159)},
MRREVIEWER = {Oded Schramm},
       DOI = {10.1214/aop/1046294317},
       URL = {http://dx.doi.org.ezproxy.library.ubc.ca/10.1214/aop/1046294317},
}

@article{le2005random,
  Author = {Le Gall, Jean-Fran{\c{c}}ois},
  Doi = {10.1214/154957805100000140},
  Fjournal = {Probability Surveys},
  Issn = {1549-5787},
  Journal = {Probab. Surv.},
  Mrclass = {60J80 (05C05 05C80 35J65 60C05 60J65)},
  Mrnumber = {2203728 (2007h:60078)},
  Mrreviewer = {Endre Cs{\'a}ki},
  Pages = {245--311},
  Title = {Random trees and applications},
  Url = {http://dx.doi.org/10.1214/154957805100000140},
  Volume = {2},
  Year = {2005}}

@article{CRT1,
  Author = {Aldous, David},
  Coden = {APBYAE},
  Fjournal = {The Annals of Probability},
  Issn = {0091-1798},
  Journal = {Ann. Probab.},
  Mrclass = {60C05 (05C80)},
  Mrnumber = {1085326 (91i:60024)},
  Mrreviewer = {Robin Pemantle},
  Number = {1},
  Pages = {1--28},
  Title = {The continuum random tree. {I}},
  Url = {http://links.jstor.org/sici?sici=0091-1798(199101)19:1<1:TCRTI>2.0.CO;2-B&origin=MSN},
  Volume = {19},
  Year = {1991}}

@article{addario2012,
  Author = {Addario-Berry, L. and Broutin, N. and Goldschmidt, C.},
  Coden = {PTRFEU},
  Doi = {10.1007/s00440-010-0325-4},
  Fjournal = {Probability Theory and Related Fields},
  Issn = {0178-8051},
  Journal = {Probab. Theory Related Fields},
  Mrclass = {05C80 (05C12 60C05 60F17)},
  Mrnumber = {2892951},
  Mrreviewer = {Jonathan Cutler},
  Number = {3-4},
  Pages = {367--406},
  Title = {The continuum limit of critical random graphs},
  Url = {http://dx.doi.org/10.1007/s00440-010-0325-4},
  Volume = {152},
  Year = {2012}}

@article{fitzner2015nearest,
    AUTHOR = {Fitzner, Robert and \Hofstad, Remco},
     TITLE = {Mean-field behavior for nearest-neighbor percolation in
              {$d>10$}},
   JOURNAL = {Electron. J. Probab.},
  FJOURNAL = {Electronic Journal of Probability},
    VOLUME = {22},
      YEAR = {2017},
     PAGES = {Paper No. 43, 65},
   MRCLASS = {60K35 (82B27 82B43)},
  MRNUMBER = {3646069},
       DOI = {10.1214/17-EJP56},
       URL = {https://doi.org/10.1214/17-EJP56},
}

@book{LP:book,
    AUTHOR = {Lyons, Russell and Peres, Yuval},
     TITLE = {Probability on Trees and Networks},
    SERIES = {Cambridge Series in Statistical and Probabilistic Mathematics},
    VOLUME = {42},
 PUBLISHER = {Cambridge University Press, New York},
      YEAR = {2016},
     PAGES = {xv+699},
      ISBN = {978-1-107-16015-6},
   MRCLASS = {60C05 (05C05 05C81 28A80 60J50 60J80 60K35 82B41)},
  MRNUMBER = {3616205},
       DOI = {10.1017/9781316672815},
       URL = {http://dx.doi.org/10.1017/9781316672815},
note = {Available at \url{http://pages.iu.edu/~rdlyons/}},
}
  }

\end{document}